\documentclass[12pt,a4paper]{amsart}
\usepackage[utf8]{inputenc}
\usepackage[english]{babel}
\usepackage{amsmath}
\usepackage{amsfonts}
\usepackage{amssymb}
\usepackage{color}
\usepackage{ifthen}
\IfFileExists{yhmath.sty}{\usepackage{yhmath}}{}

\usepackage{amsthm}

\newtheorem{theorem}{Theorem} [section]
\newtheorem{lemma}[theorem]{Lemma}

\newtheorem{corollary}[theorem]{Corollary}

\newtheorem{open}{Open Question}

\usepackage{graphics}
\usepackage{epsfig}

\usepackage{url}

\usepackage{hyperref}

\usepackage[a4paper,top=2.5cm,bottom=2.8cm,left=3.2cm,right=3.2cm]{geometry}
\markleft{\hfill \textsc{Properties of Ajima Circles} \hfill}
\begin{document}

Sangaku Journal of Mathematics (SJM) \copyright SJM \\
ISSN 2534-9562 \\
Volume 7 (2023), pp. xx--yy \\
Received XX September 2023 Published on-line XX XXX 2023 \\ 
web: \url{http://www.sangaku-journal.eu/} \\
\copyright The Author(s) This article is published 
with open access.\footnote{This article is distributed under the terms of the Creative Commons Attribution License which permits any use, distribution, and reproduction in any medium, provided the original author(s) and the source are credited.} \\
\bigskip
\bigskip

\newcommand\note[1]{\textbf{\textcolor{red}{#1}}}

\begin{center}
	{\Large \textbf{Properties of Ajima Circles}} \\
	\bigskip
	\textsc{Stanley Rabinowitz$^a$ and Ercole Suppa$^b$} \\

	$^a$ 545 Elm St Unit 1,  Milford, New Hampshire 03055, USA \\
	e-mail: \href{mailto:stan.rabinowitz@comcast.net}{stan.rabinowitz@comcast.net}\footnote{Corresponding author} \\
	web: \url{http://www.StanleyRabinowitz.com/} \\
	$^b$ Via B. Croce 54, 64100 Teramo, Italia \\
	e-mail: \href{mailto:ercolesuppa@gmail.com}{ercolesuppa@gmail.com} \\
	web: \url{http://www.esuppa.it/} \\

\end{center}

\bigskip

\textbf{Abstract.} We study properties of certain circles associated with a triangle.
Each circle is inside the triangle, tangent to two sides of the triangle, and
externally tangent to the arc of a circle erected internally on the third side.



\long\edef\void#1{}

\newcommand{\ifTexTures}[2]{\ifthenelse{\equal{\fmtversion}{2005/12/01}}{#1}{#2}}

\def\myarc#1{\setbox13=\hbox{$#1$}%
\setbox14=\hbox{$\displaystyle\frown\vphantom{x}$}%
\dimen15=1pt\advance\dimen15 by -.5\wd13\advance\dimen15 by -.5\wd14%
\dimen16=2pt\advance\dimen16 by 1ex%
\box13\kern\dimen15\raise\dimen16\box14}
\newcommand{\arc}{\ifTexTures{\myarc}{\wideparen}}

\newcommand{\degrees}{^\circ}
\newcommand{\W}{\mathbb{W}}



\section{Introduction}

The following figure appears in a Sangaku described in \cite{Honma}
and reprinted in \cite{Okumura}. 

\begin{figure}[ht]
\centering
\includegraphics[scale=0.5]{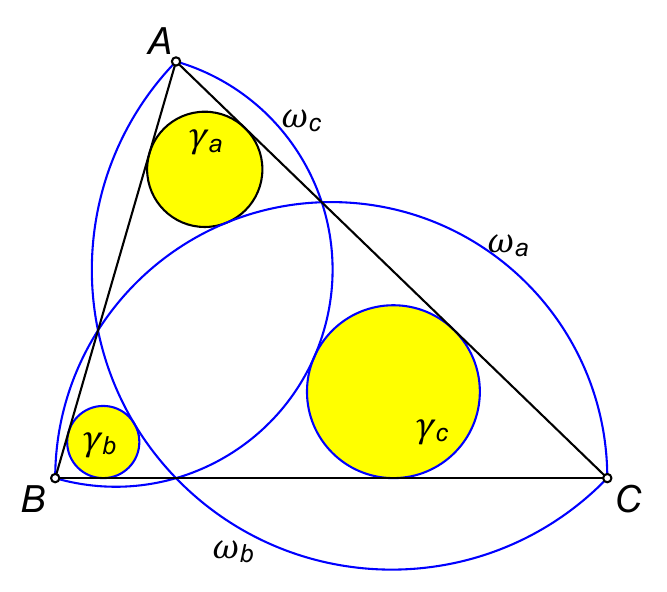}
\caption{Sangaku configuration}
\label{fig:Sangaku}
\end{figure}

In this figure, the semicircle erected inwardly on side $BC$ is named $\omega_a$.
Semicircles $\omega_b$ and $\omega_c$ are defined similarly.
The circle inside $\triangle ABC$, tangent to sides $AB$ and $AC$, and externally tangent to
semicircle $\omega_a$ is named $\gamma_a$.
Circles $\gamma_b$ and $\gamma_c$ are defined similarly.
The sangaku gave a relationship involving the radii of the three circles.

Additional properties of this configuration were given in \cite{Suppa} and \cite{RabSup}.
For example, in Figure~\ref{fig:properties} (left), the three blue common tangents are all congruent.
Their common length is $2r$, twice the inradius of $\triangle ABC$.
In Figure~\ref{fig:properties} (right), the six touch points lie on a circle with center $I$,
the incenter of $\triangle ABC$.

\begin{figure}[ht]
\centering
\includegraphics[scale=0.5]{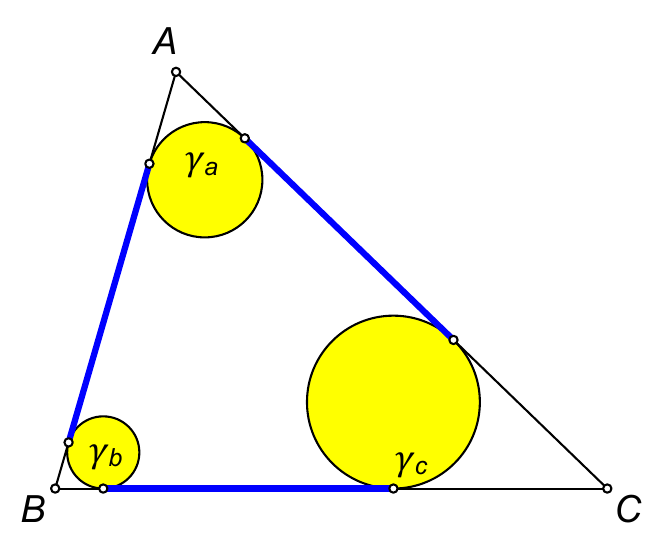}
\quad
\includegraphics[scale=0.5]{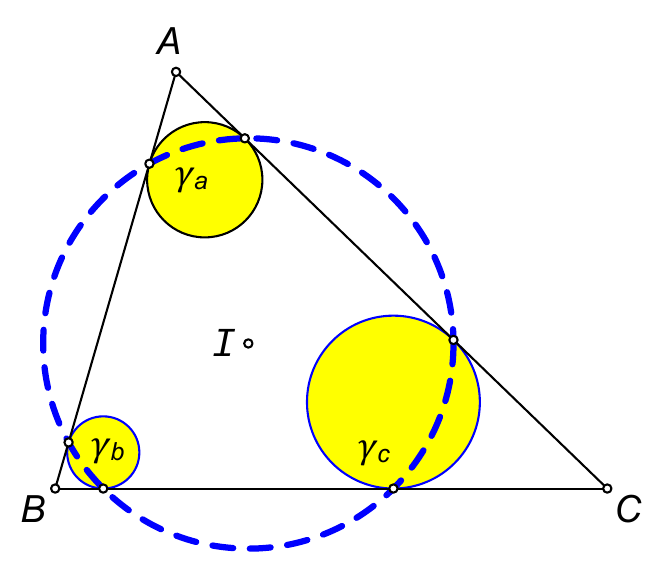}
\caption{properties}
\label{fig:properties}
\end{figure}

It is the purpose of this paper to present properties of circles such as $\gamma_a$ and also to
generalize these results by replacing the semicircles with
arcs having the same angular measure.

\section{Properties of $\omega_a$ and $\gamma_a$}

In this section we will discuss properties of the configuration shown in Figure~\ref{fig:configuration}
in which $\omega_a$ is any circle passing through vertices $B$ and $C$ of $\triangle ABC$.
The circle $\gamma_a$ is inside $\triangle ABC$, tangent to sides $AB$ and $AC$ and tangent to $\omega_a$ at $T$.
This circle is sometimes known in the literature as an Ajima circle \cite{Dergiades}.

\begin{figure}[ht]
\centering
\includegraphics[scale=0.45]{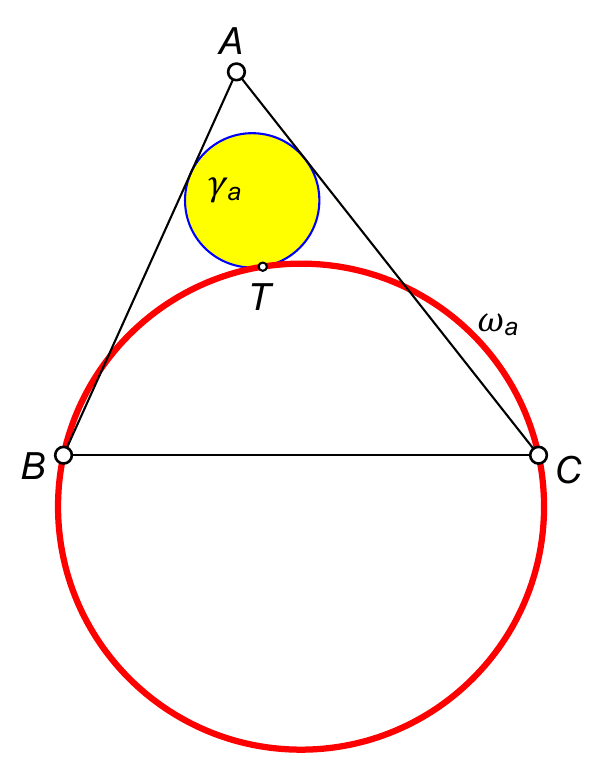}
\caption{The configuration we are studying}
\label{fig:configuration}
\end{figure}

An \emph{Ajima circle} of a triangle is a circle ($\gamma$) that is tangent to two sides of the triangle and
also tangent to a circle ($\omega$) passing through the endpoints of the third side.
In this paper, we are primarily interested in Ajima circles that lie inside the triangle and
for which $\gamma$ and $\omega$ are \emph{externally} tangent as shown in Figure~\ref{fig:configuration}.

Occasionally, we will generalize a result and present a theorem in which $\gamma_a$ is any circle tangent to $AB$
and $AC$ (not necessarily tangent to $\omega_a$). To help the reader recognize when a result applies to an Ajima circle,
we will color all Ajima circles yellow.

The standard notation related to our configuration that we use
throughout this paper is shown in the following table.

\begin{center}
\begin{tabular}{|c|l|}
\hline
\multicolumn{2}{|c|}{\color{red} \textbf{Standard Notation}}\\ \hline
\textbf{Symbol}&\textbf{Description} \\ \hline
$a, b, c$&lengths of sides of $\triangle ABC$\\ \hline
$\omega_a$&circle through points $B$ and $C$\\ \hline
$\gamma_a$&Ajima circle inscribed in $\angle BAC$ tangent to $\omega_a$\\ \hline
$D$&center of $\gamma_a$\\ \hline
$T$&Unless specified otherwise, $T$ is the point where $\gamma_a$ touches $\omega_a$.\\ \hline
$O_a$&center of $\omega_a$\\ \hline
$I$&incenter of $\triangle ABC$\\ \hline
$r$&inradius of $\triangle ABC$\\ \hline
$R$&circumradius of $\triangle ABC$\\ \hline
$p$&semiperimeter of $\triangle ABC$ $=(a+b+c)/2$\\ \hline
$\Delta$&area of $\triangle ABC$\\ \hline
$S$&twice the area of $\triangle ABC$ (i.e. $2\Delta$)\\ \hline
$G_e$&Gergonne point of $\triangle ABC$\\ \hline
$\theta$&angular measure of arc $\arc{BTC}$\\ \hline
\end{tabular}
\end{center}

\bigskip

\void{
The incenter of $\triangle ABC$ will be denoted by the letter $I$
and the inradius of $\triangle ABC$ will be denoted by the letter $r$.
The center of $\gamma_a$ will be denoted by the letter $D$ and
the center of $\omega_a$ will be denoted by $O_a$.
}

Without loss of generality, we will assume $AB<AC$.

We will survey some known results and give some new properties of this configuration.

The following result is due to Protasov \cite{Protassov}.
Proofs can be found in \cite{Ayme} and \cite[pp.~90--94]{Andreescu}.

\begin{theorem}[Protasov's Theorem]\label{thm:Protasov}
The segment $TI$ bisects $\angle BTC$ (Figure~\ref{fig:Protasov}).
\end{theorem}

\begin{figure}[ht]
\centering
\includegraphics[scale=0.6]{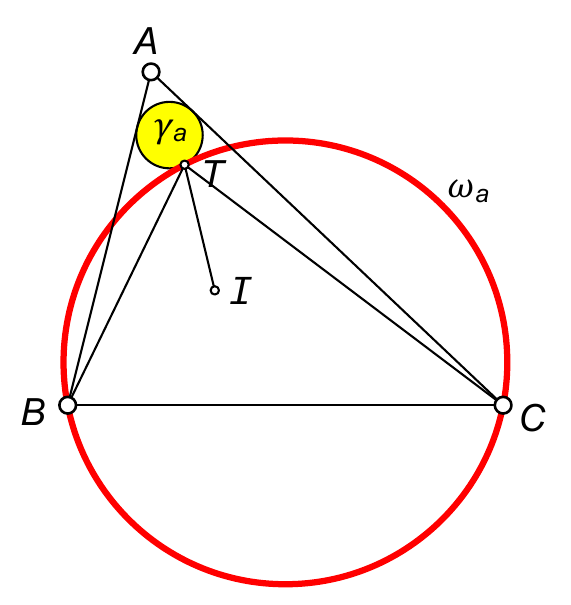}
\caption{$TI$ bisects $\angle BTC$}
\label{fig:Protasov}
\end{figure}

\newpage

The following result comes from \cite{Suppa13233}.

\begin{lemma}\label{lemma:Suppa'sLemma}
Let $\Gamma$ and $\Omega$ be two circles that are externally tangent at $T$.
Let $B$ and $C$ be points on $\Omega$ and let $BU$ and $CV$ be tangents to $\Gamma$
as shown in Figure~\ref{fig:Suppa'sLemma}.
Then
$$\frac{BU}{CV}=\frac{BT}{CT}.$$
\end{lemma}

\begin{figure}[ht]
\centering
\includegraphics[scale=0.5]{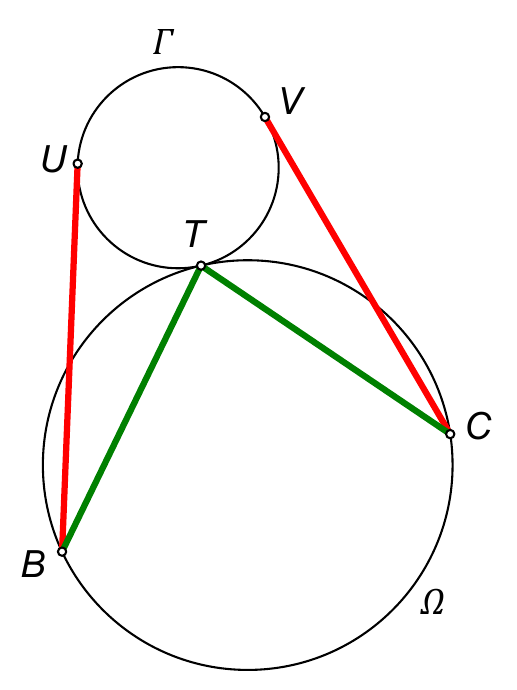}
\caption{$BU/CV=BT/CT$}
\label{fig:Suppa'sLemma}
\end{figure}

\begin{proof}
Let $BT$ meet $\Gamma$ at $Q$ and let $CT$ meet $\Gamma$ at $P$.
Let $XY$ be the tangent to both circles at $T$ (Figure~\ref{fig:Suppa'sLemmaProof}).
\begin{figure}[ht!]
\centering
\includegraphics[scale=0.5]{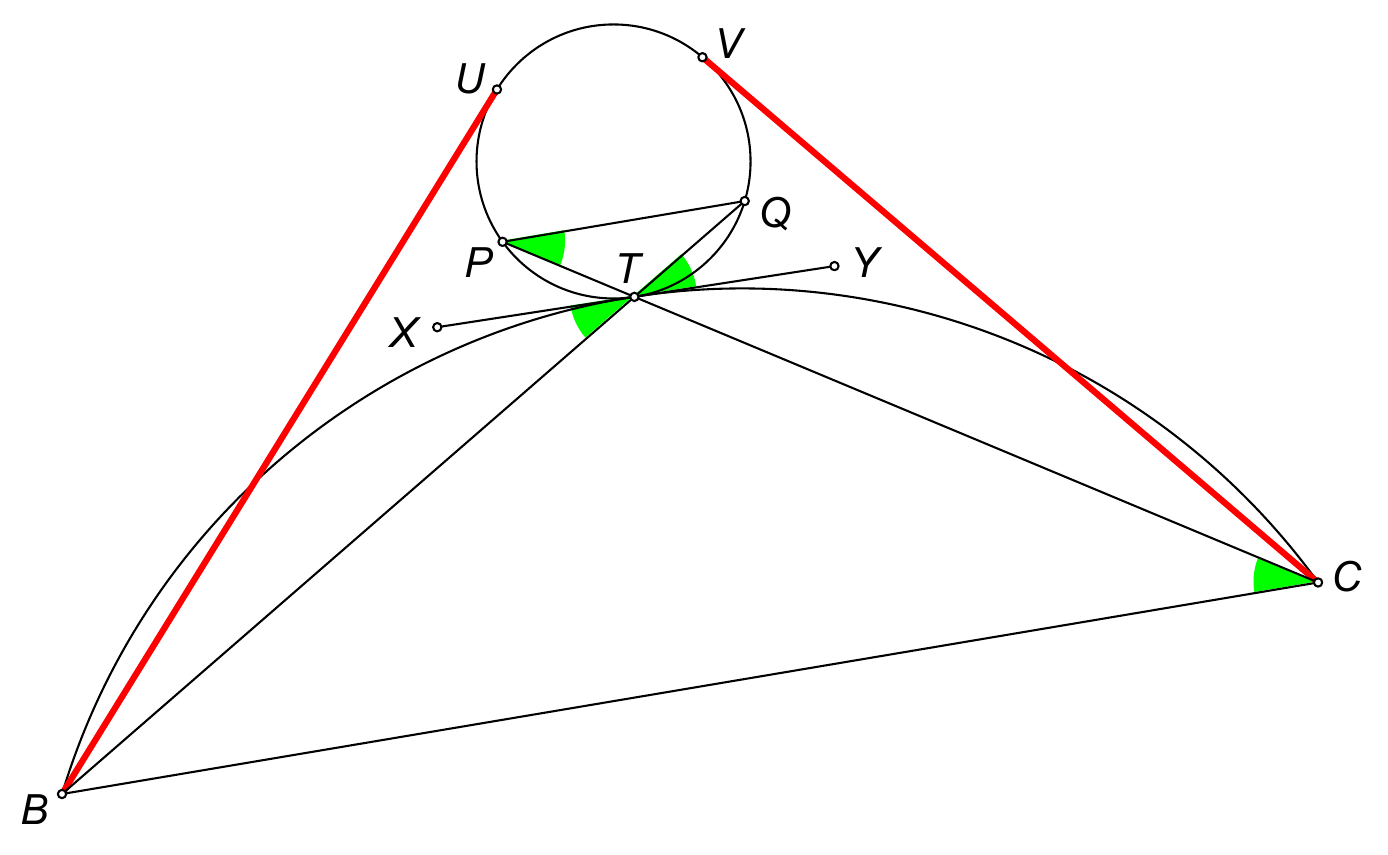}
\caption{}
\label{fig:Suppa'sLemmaProof}
\end{figure}

We have
$$\angle TPQ=\frac12\arc{QT}=\angle YTQ=\angle XTB=\frac12\arc{BT}=\angle TCB.$$
Since $\angle QTP=\angle BTC$, we find that $\triangle PQT\sim\triangle CBT$. Thus,
$$\frac{BT}{QT}=\frac{CT}{PT}$$
which implies
$$\frac{BT}{BQ}=\frac{CT}{CP}\quad\hbox{or}\quad \frac{BQ}{CP}=\frac{BT}{CT}.$$
Since $BU$ and $CV$ are tangents, we have $(BU)^2=BT\cdot BQ$ and $(CV)^2=CT\cdot CP$.
Combining, we get
$$\frac{(BU)^2}{(CV)^2}=\frac{BT\cdot BQ}{CT\cdot CP}=\left(\frac{BT}{CT}\right)\left(\frac{BQ}{CP}\right)
=\left(\frac{BT}{CT}\right)\left(\frac{BT}{CT}\right)=\frac{(BT)^2}{(CT)^2}$$
which implies $BU/CV=BT/CT$.
\end{proof}


\begin{theorem}\label{thm:result12}
Let $\gamma_a$ touch $AB$ and $AC$ at $F$ and $E$, respectively.
Let $TI$ meet $BC$ at $Z$ (Figure~\ref{fig:result12}).
Then $$\frac{BF}{CE}=\frac{BZ}{CZ}.$$
\end{theorem}

\begin{figure}[ht]
\centering
\includegraphics[scale=0.45]{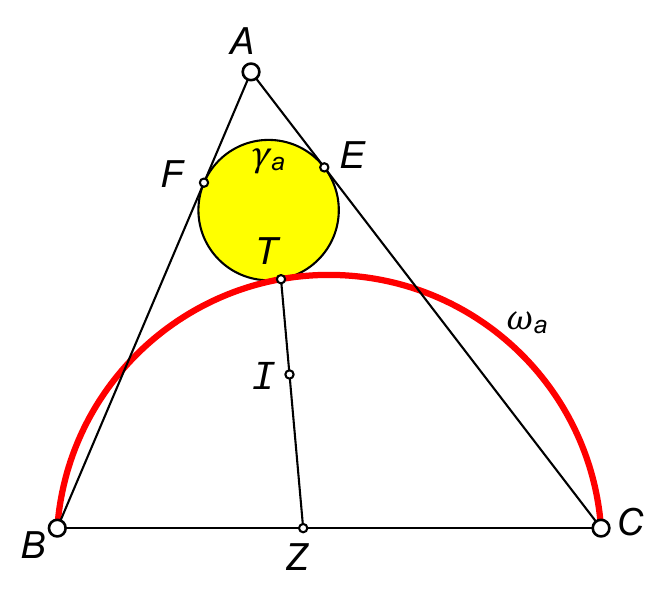}
\caption{$BF/CE=BZ/CZ$}
\label{fig:result12}
\end{figure}

\begin{proof}
By Lemma~\ref{lemma:Suppa'sLemma}, $BF/CE=BT/CT$.
By Protasov's Theorem, $TI$ bisects $\angle BTC$.
Since $TZ$ is an angle bisector of $\triangle BTC$, we have $BT/TC=BZ/CZ$.
Hence $BF/CE=BT/CT=BZ/CZ$.
\end{proof}


The following result comes from \cite{Yetti}.
A nice geometric proof can be found in \cite{Ayme}.
See also \cite{Pascual}.

\begin{theorem}[The Catalytic Lemma]\label{thm:catalytic}
Let $E$ be the point where $\gamma_a$ touches $AC$.
Then $E$, $T$, $I$, and $C$ are concyclic (Figure~\ref{fig:catalytique}).
\end{theorem}

\begin{figure}[ht]
\centering
\includegraphics[scale=0.47]{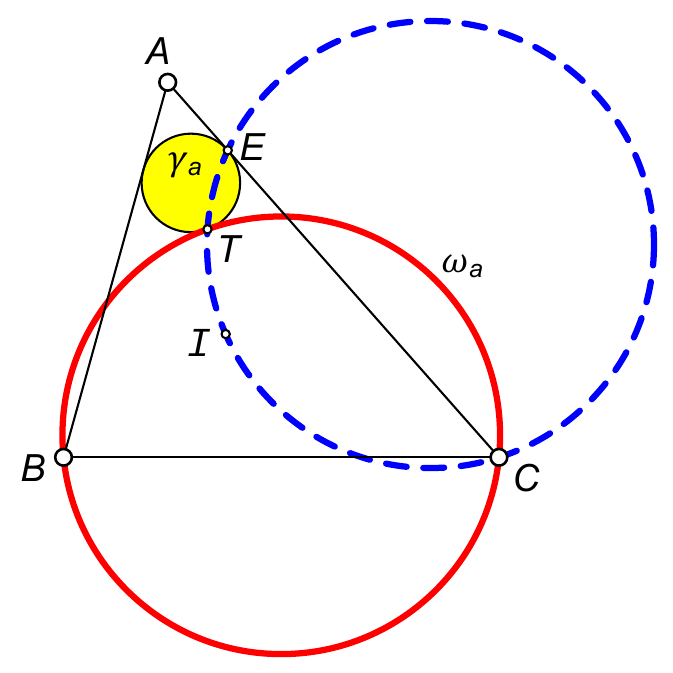}
\caption{$E$, $T$, $I$, $C$ are concyclic}
\label{fig:catalytique}
\end{figure}

\newpage

\begin{theorem}\label{thm:result3}
Let $E$ be the point where $\gamma_a$ touches $AC$.
Then $\angle BTC = 2\angle IEC$ (Figure~\ref{fig:result3}).
\end{theorem}

\begin{figure}[ht]
\centering
\includegraphics[scale=0.33]{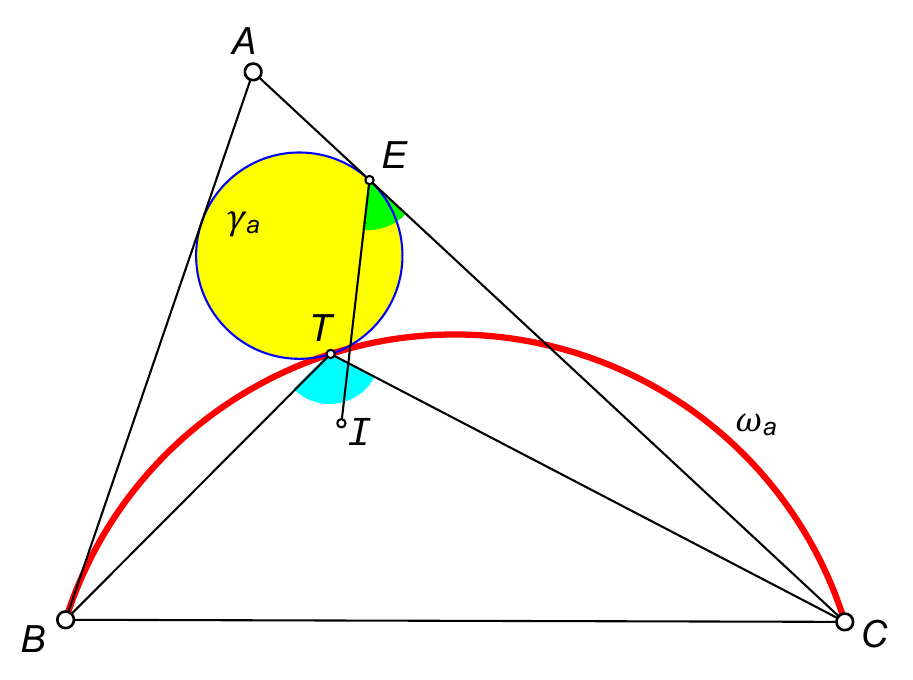}
\caption{blue angle is twice green angle}
\label{fig:result3}
\end{figure}

\begin{proof}
By the Catalytic Lemma,
$E$, $T$, $I$, and $C$ are concyclic (Figure~\ref{fig:result3proof}).
By Protasov's Theorem,
$TI$ bisects $\angle BTC$,
so $\angle BTC=2\angle ITC$.
But $\angle ITC=\angle IEC$ because both angles subtend the same arc $\arc{IC}$.
\begin{figure}[ht]
\centering
\includegraphics[scale=0.37]{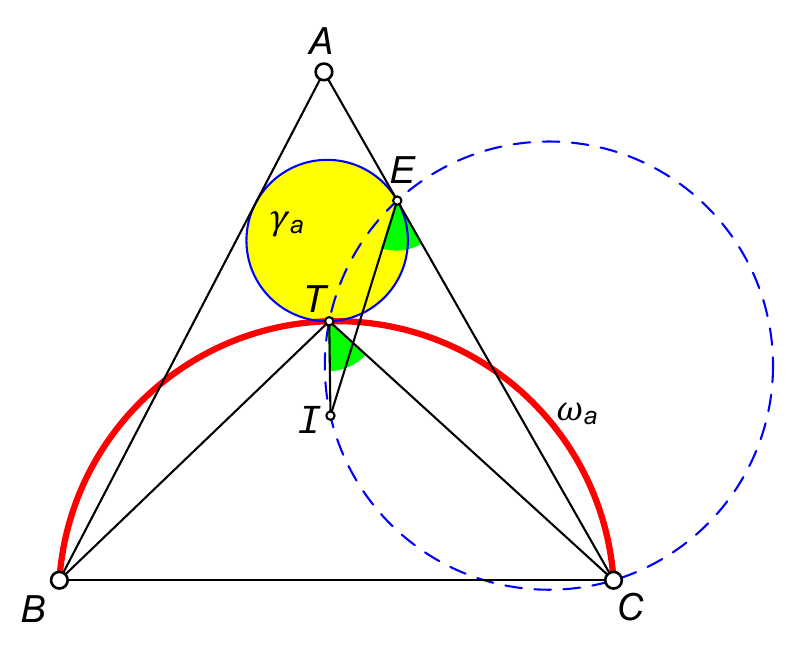}
\caption{green angles are equal}
\label{fig:result3proof}
\end{figure}
\end{proof}


The following result comes from \cite{Ayme}.

\begin{theorem}\label{thm:Thailand}
Let the touch points of circle $\gamma_a$ with $AC$ and $AB$ be $E$ and $F$, respectively.
Suppose $\omega_a$ meets $AC$ at $J$ between $A$ and $C$.
Let $X$ be the center of the excircle of $\triangle BJC$ opposite $C$.
Then $X$, $F$, and $E$ are collinear (Figure~\ref{fig:Thailand}).
\end{theorem}

\begin{figure}[ht]
\centering
\includegraphics[scale=0.5]{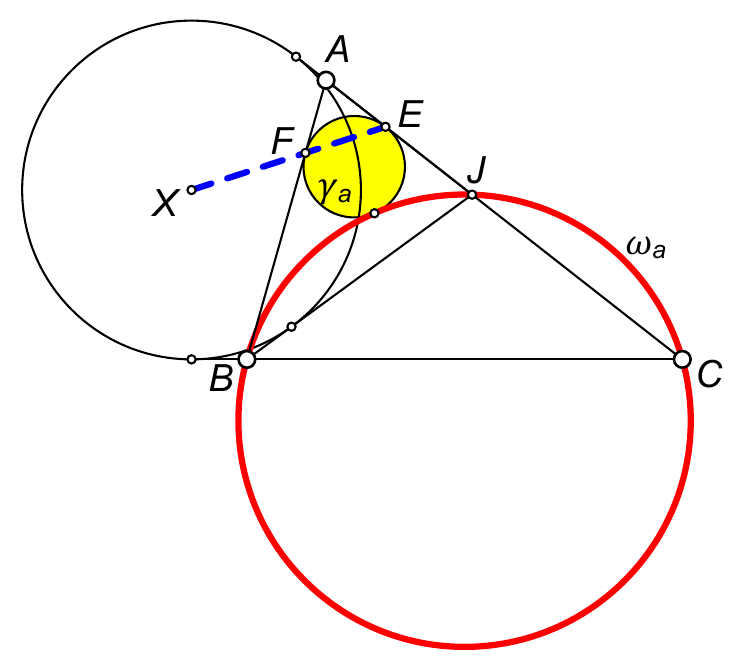}
\caption{$X$, $F$, and $E$ are collinear}
\label{fig:Thailand}
\end{figure}

\void{
The following result is well known \cite{AymeReim}.

\begin{lemma}[Reim's Theorem]\label{thm:Reim}
Let $C_1$ and $C_2$ be two circles intersecting at points $X$ and $Y$.
A line through $X$ meets $C_1$ again at $P_1$ and meets $C_2$ again at $P_2$.
A line through $Y$ meets $C_1$ again at $Q_1$ and meets $C_2$ again at $Q_2$.
Then $P_1Q_1\parallel P_2Q_2$ (Figure~\ref{fig:Reim}).
\end{lemma}

\begin{figure}[ht]
\centering
\includegraphics[scale=0.56]{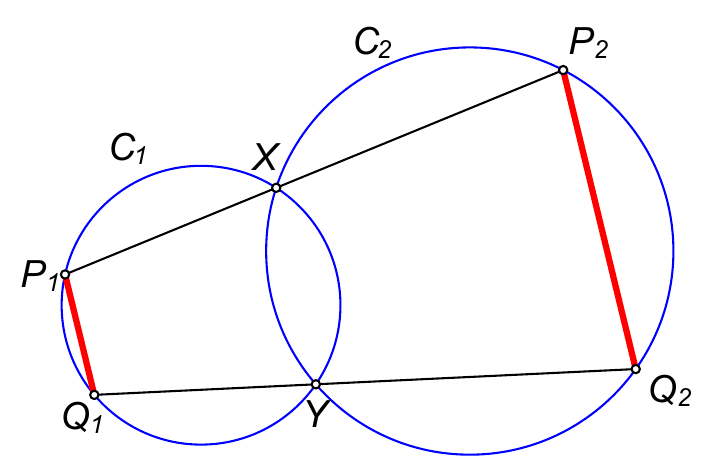}
\quad
\includegraphics[scale=0.56]{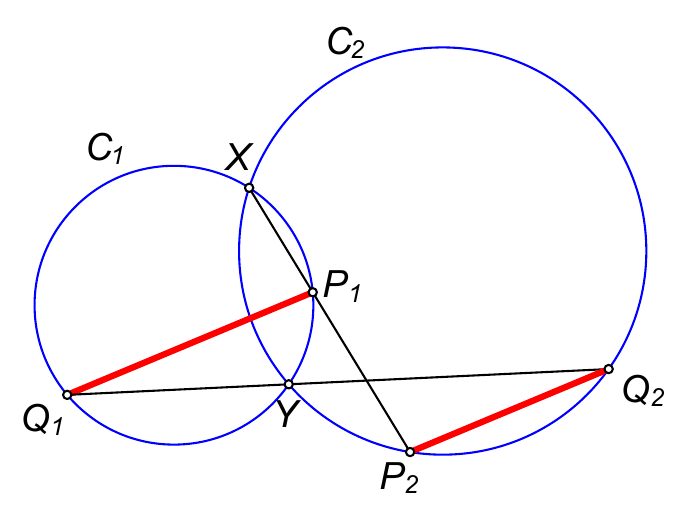}
\caption{red lines are parallel}
\label{fig:Reim}
\end{figure}
}

The following result comes from \cite{Suppa13069}.

\begin{theorem}\label{thm:RG13069}
The perpendicular bisector of $BC$ meets $\omega_a$ on the opposite side from $T$
at $N$ as shown in Figure~\ref{fig:RG13069}.
Then $T$, $I$, and $N$ are collinear.
\end{theorem}

\begin{figure}[ht]
\centering
\includegraphics[scale=0.45]{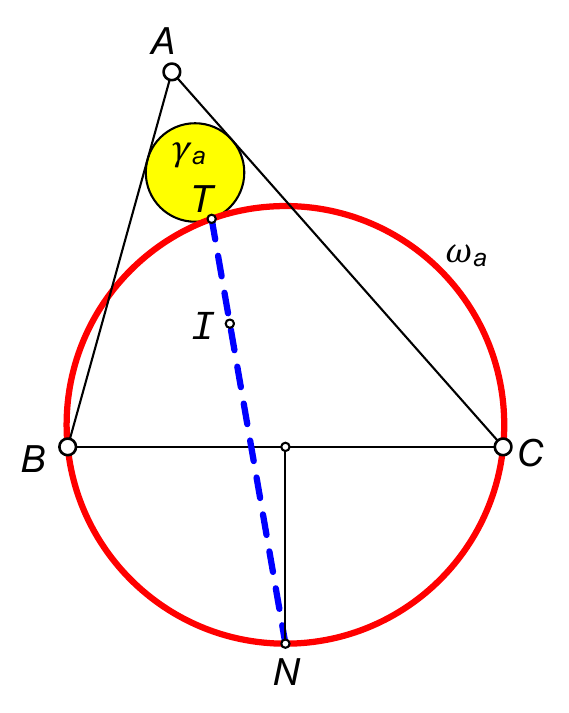}
\caption{}
\label{fig:RG13069}
\end{figure}

\begin{proof}
From Protasov's Theorem, $TI$ bisects $\angle BTC$.
Thus, $TI$ intersects the arc $\arc{BC}$ (not containing $T$) at its midpoint. This midpoint lies on
the perpendicular bisector of $BC$ and we are done.
\end{proof}

\textbf{Note.} This theorem provides a nice method for constructing $\gamma_a$.
First construct $N$ as the intersection of the perpendicular bisector of $BC$ with $\omega_a$.
Then construct $T$ as the intersection of $NI$ with $\omega_a$.
Finally, the center of $\gamma_a$ is found as the intersection of the line joining
the center of $\omega_a$ and $T$ with $AI$.

\begin{corollary}\label{cor:midarc}
We have $\angle NBC=\angle BCN$.
\end{corollary}

The following result is suggested by \cite{Istvan}.

\begin{theorem}\label{thm:RG13164}
Let $N$ be the midpoint of arc $\arc{BC}$ opposite $T$.
Let $E$ be the point where $\gamma_a$ touches side $AC$ (Figure~\ref{fig:RG13164}).
Let $J$ be the point where $\omega_a$ meets $AC$.
Then $IE\parallel NJ$.
\end{theorem}

\begin{figure}[ht]
\centering
\includegraphics[scale=0.4]{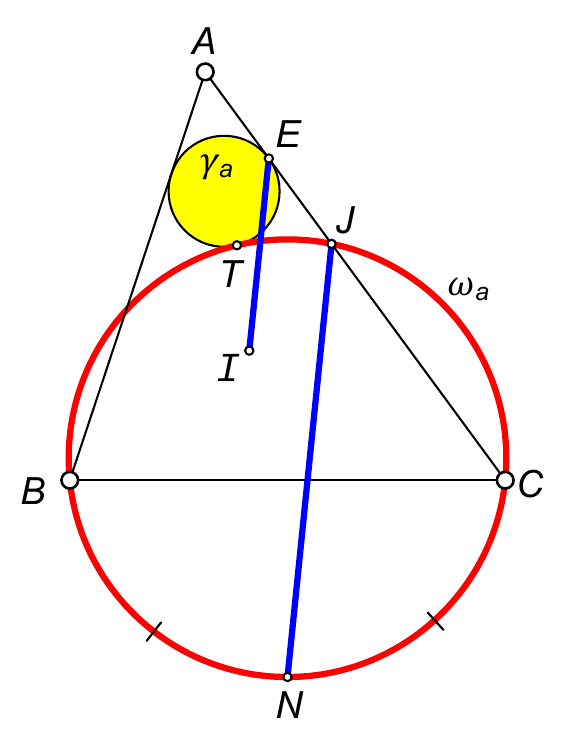}
\caption{blue lines are parallel}
\label{fig:RG13164}
\end{figure}

\begin{proof}
By Theorem~\ref{thm:result3}, $\angle IEC$ is half of $\angle BTC$.
But half of $\angle BTC$ is equal to $\angle NTC$ and $\angle NTC=\angle NJC$
because both angles are inscribed in arc $\arc{NC}$.
Thus, $\angle IEC=\angle NJC$ which makes $IE\parallel NJ$.
\end{proof}

The following result comes from \cite{Suppa13066}.

\begin{theorem}\label{thm:RG13066}
Let $\omega_a$ meet $AC$ at $J$ and let the line through $I$ parallel to $BJ$ meet $AC$ at $F$.
Let $E$ be the point where $\gamma_a$ touches $AC$.
Then $IF=FE$ (Figure~\ref{fig:RG13066}).
\end{theorem}

\begin{figure}[ht]
\centering
\includegraphics[scale=0.7]{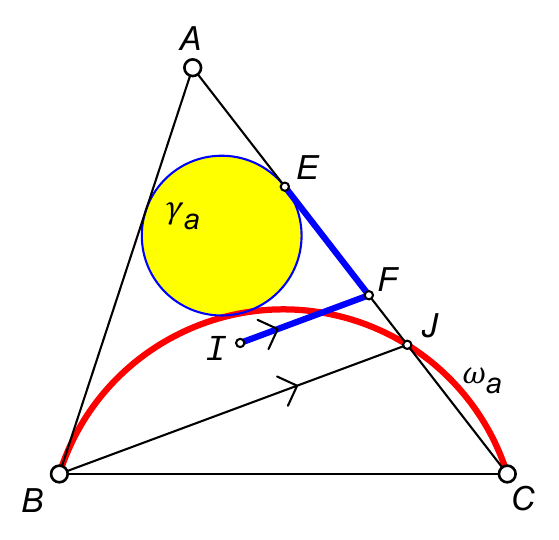}
\caption{blue segments are congruent}
\label{fig:RG13066}
\end{figure}



\begin{proof}
Let $TI$ meet $\omega_a$ again at $N$.
\begin{figure}[ht]
\centering
\includegraphics[scale=0.6]{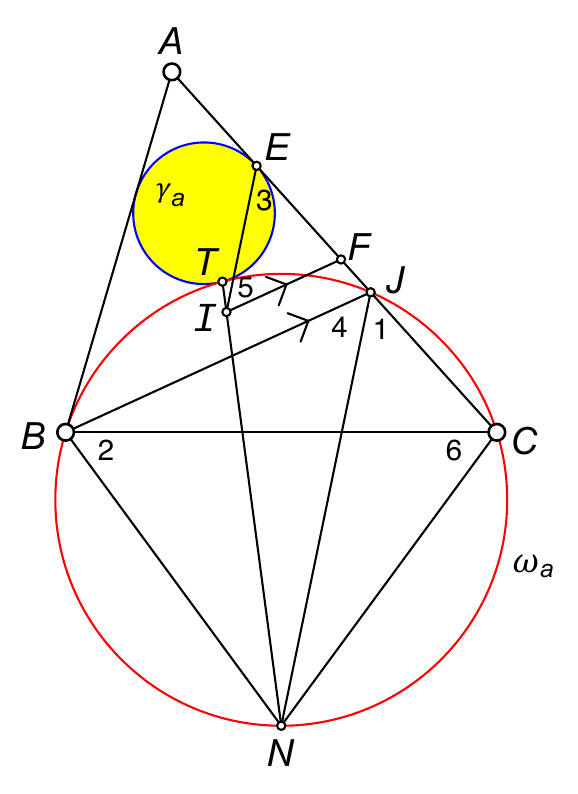}
\caption{}
\label{fig:RG13066proof}
\end{figure}
By Theorem~\ref{thm:RG13164}, $NJ\parallel IE$ (Figure~\ref{fig:RG13066proof}). Thus $\angle 3=\angle 1$. But $\angle 1=\angle 2$
since both subtend arc $\arc{NC}$ in circle $\omega_a$. Hence,
\begin{equation}
\angle 3=\angle 2.\label{eq:3=2}
\end{equation}
\void{
It is given that $IF\parallel BE$, so $\angle 4=\angle 5$. Since $\angle 5$ is measured by
half arc $\arc{AEC}$, we can deduce that $\angle 6=\angle 5$.
}
By Corollary~\ref{cor:midarc}, we have $\angle 2=\angle 6$. But $\angle 6=\angle 4$ since both
subtend arc $\arc{BN}$. Thus,
\begin{equation}
\angle 2=\angle 4.\label{eq:2=4}
\end{equation}
Since $IE\parallel NJ$ and $IF\parallel BJ$, we can conclude that
\begin{equation}
\angle 4=\angle 5.\label{eq:4=5}
\end{equation}
Combining equations (\ref{eq:3=2}), (\ref{eq:2=4}), and (\ref{eq:4=5}), we find that
$$\angle 3=\angle 2=\angle 4=\angle 5,$$
so $\angle 3=\angle 5$. Thus, $\triangle FIE$ is isosceles with $IF=FE$.
\end{proof}

For other proofs, see \cite{Ayme2} and \cite{Ayme3}.

This theorem provides another simple way to construct circle $\gamma_a$.
Draw the line through $I$ parallel to $BJ$ to get point $F$ where this line meets $AC$.
With center $F$, draw a circle with radius $FI$. Let this circle meet $AC$ (nearer $A$) at point $E$.
This is the touch point for circle $\gamma_a$.
Erect a perpendicular at $E$ to $AC$. This perpendicular meets $AI$
at the center of $\gamma_a$.


\begin{theorem}\label{thm:result8}
Let $T$ be any point on arc $\arc{BC}$.
Let $F$ be the foot of the perpendicular from $T$ to $BC$ (Figure~\ref{fig:result8}).
Then $\angle BTF=\angle O_aTC$.
\end{theorem}

\begin{figure}[ht]
\centering
\includegraphics[scale=0.45]{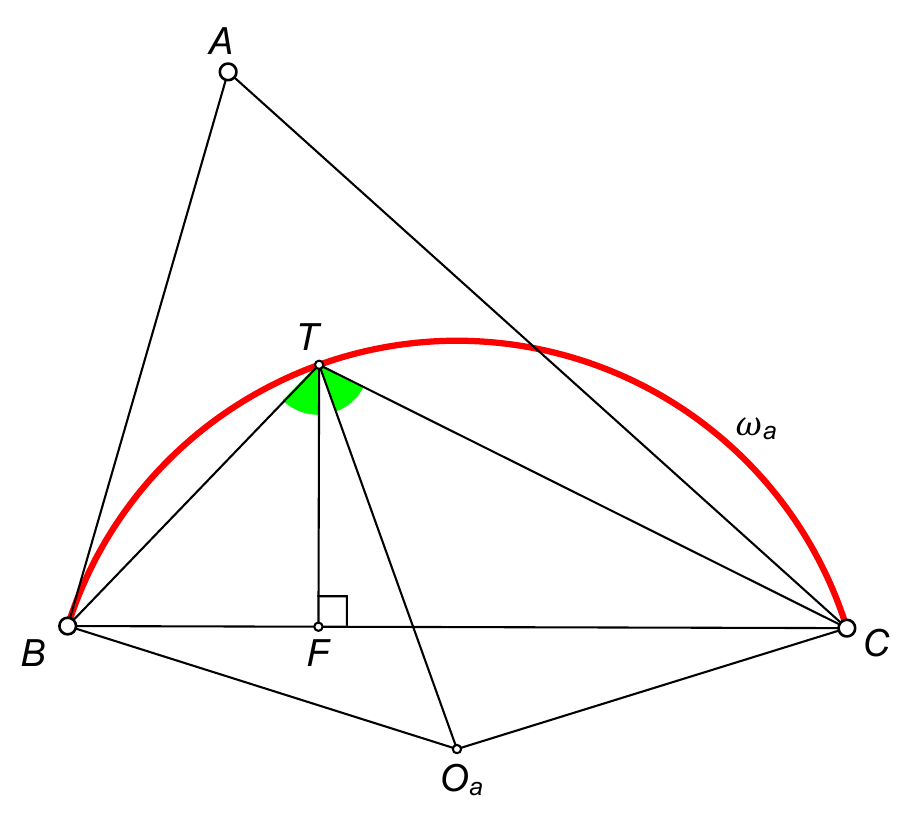}
\caption{green angles are equal}
\label{fig:result8}
\end{figure}

\bigskip

\begin{minipage}{3in}
\begin{proof}
Let $G$ be the foot of the perpendicular from $O_a$ to $TC$.
Since $\angle CBT$ is measured by half the measure of $\arc{TC}$ and 
$\angle TO_aC$ equals the measure of $\arc{TC}$, we have
$$\angle FBT=\frac12\angle CO_aT=\angle GO_aT.$$
Complements of equal angles are equal, so $\angle BTF=\angle O_aTC$.
\end{proof}
\end{minipage}
\hfil
\raisebox{-1.2in}{\includegraphics[scale=0.45]{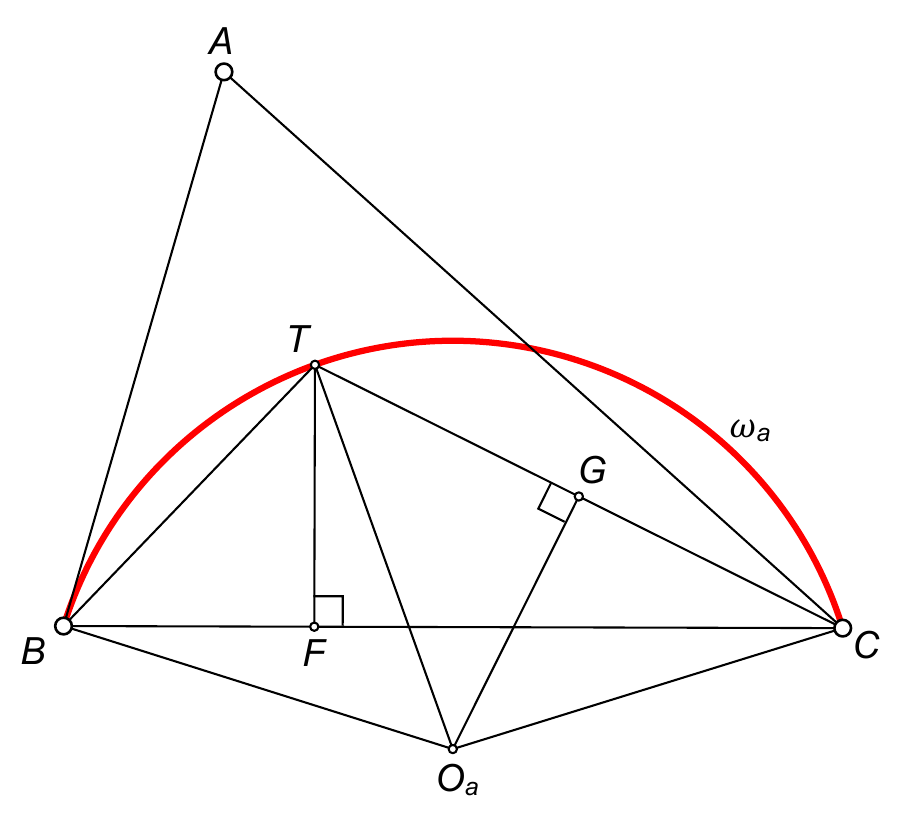}}

\newpage

\begin{theorem}\label{thm:result6}
Let $M$ be the midpoint of $BC$ (Figure~\ref{fig:result6}).
Then $\angle MO_aT=2\angle ITO_a$.
\end{theorem}

\begin{figure}[ht]
\centering
\includegraphics[scale=0.4]{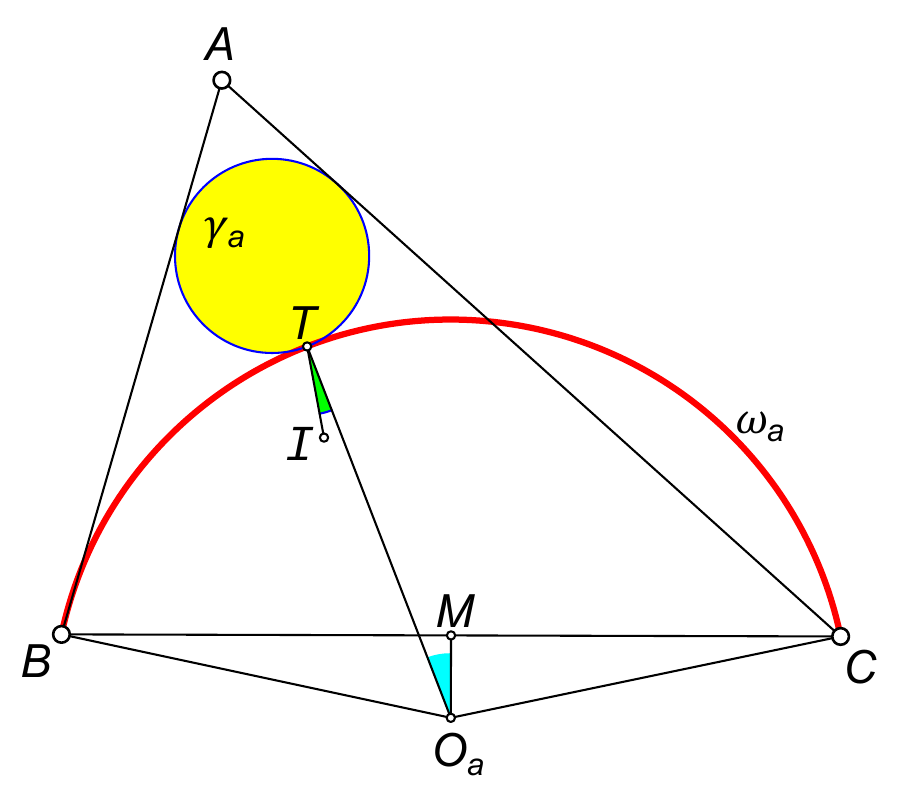}
\caption{blue angle = twice green angle}
\label{fig:result6}
\end{figure}

\begin{minipage}{3.2in}
\begin{proof}
Let $F$ be the foot of the perpendicular from $T$ to $BC$.
By Theorem~\ref{thm:result8}, $\angle 1=\angle 2$ in the figure to the right.

\medskip
By Protasov's Theorem, $\angle BTI=\angle ITC$.
Therefore $\angle FTI=\angle ITO_a$ or
$$\angle ITO_a=\frac12\angle FTO_a=\frac12\angle MO_aT$$
since $TF\parallel MO_a$.

\medskip
Hence, $\angle MO_aT=2\angle ITO_a$.
\end{proof}
\end{minipage}
\hfil
\raisebox{-1in}{\includegraphics[scale=0.4]{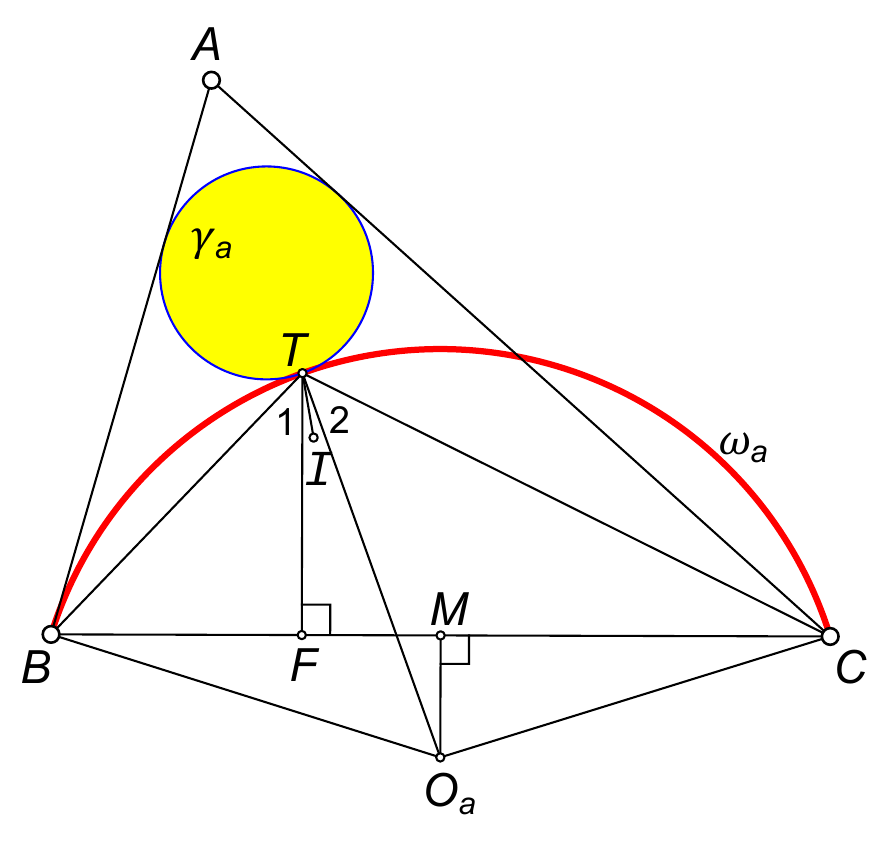}}

\begin{theorem}\label{thm:result9}
Let $IT$ meet $\gamma_a$ again at $T'$ (Figure~\ref{fig:result9}).
Then $T'D\perp BC$.
\end{theorem}

\begin{figure}[ht]
\centering
\includegraphics[scale=0.5]{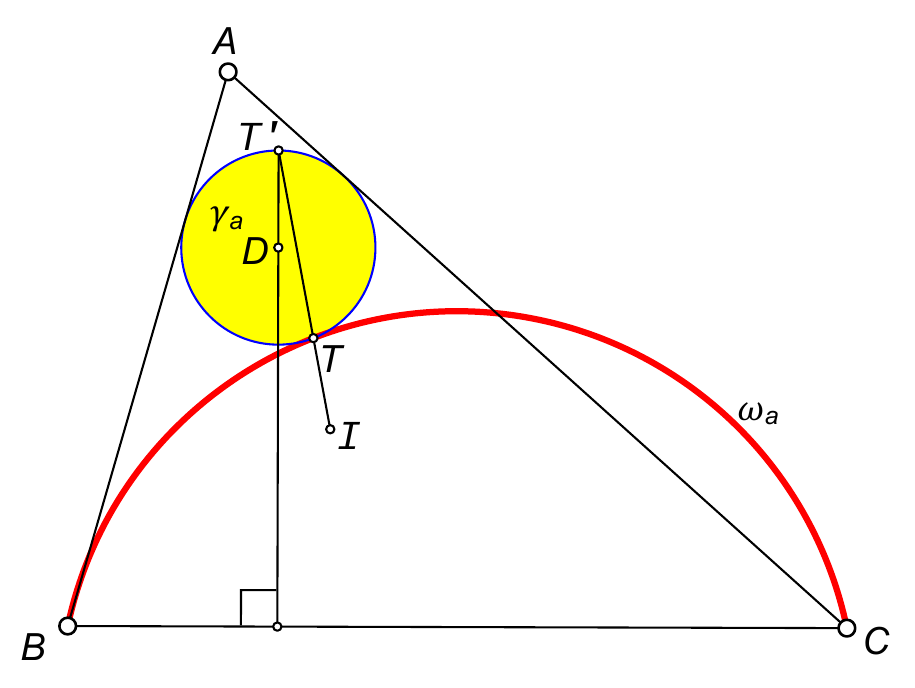}
\caption{$T'D\perp BC$}
\label{fig:result9}
\end{figure}


The following proof is due to Biro Istvan.

\newpage

\begin{proof}
Since $\gamma_a$ and $\omega_a$ are tangent at $T$, this means $D$, $T$, and $O_a$ are collinear.
Since $I$, $T$, and $T'$ are also collinear, we find $\angle T'TD=\angle ITO_a$.
Extend $TI$ until it meets $\omega_a$ again at $N$ (Figure~\ref{fig:result9proof}).

\begin{figure}[ht]
\centering
\includegraphics[scale=0.4]{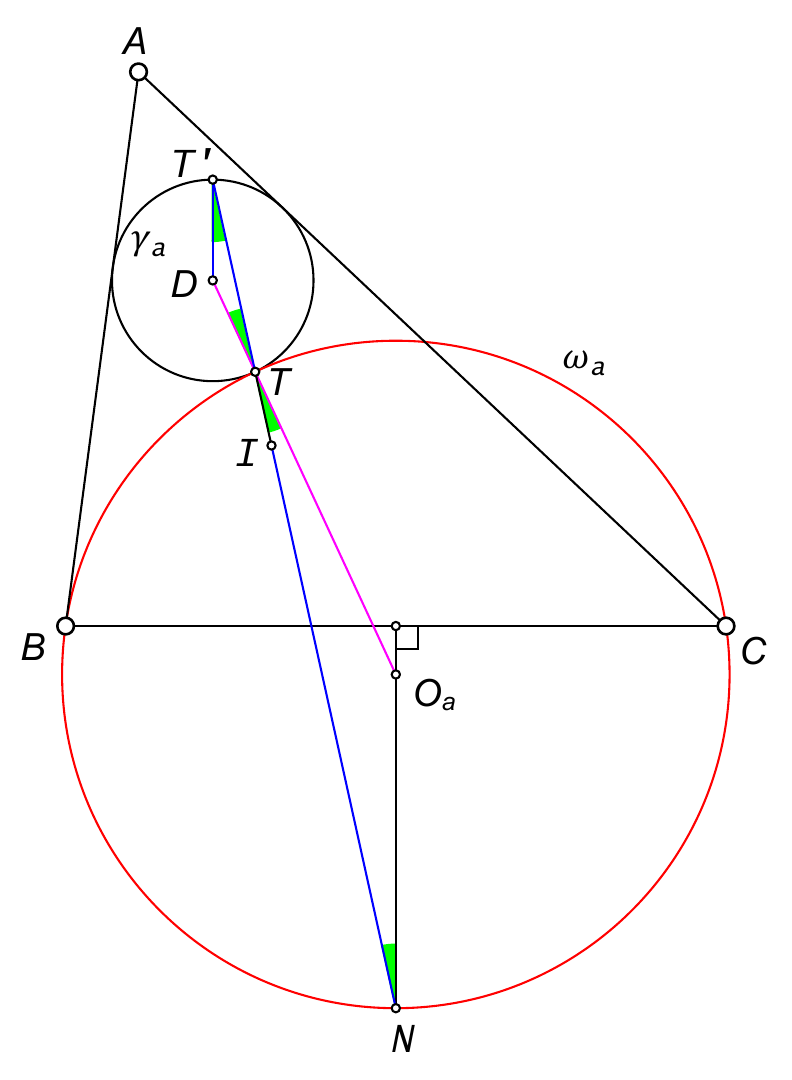}
\caption{}
\label{fig:result9proof}
\end{figure}

By Theorem~\ref{thm:RG13069}, $NO_a\perp BC$.
Base angles of an isosceles triangle are equal and vertical angles are equal,
so $\angle DT'T=\angle T'TD=\angle ITO_a=\angle O_aNT$.
So $T'D\parallel O_aN$ because $\angle DT'N=\angle O_aNT$.
Thus, $T'D\perp BC$.
\end{proof}


\begin{theorem}\label{thm:result10}
Let $IT$ meet $\gamma_a$ again at $T'$.
Let $E$ be the point where $\gamma_a$ touches side $AC$ (Figure~\ref{fig:result10}).
Then $\angle EDT'=\angle ACB$.
\end{theorem}

\begin{figure}[ht]
\centering
\includegraphics[scale=0.45]{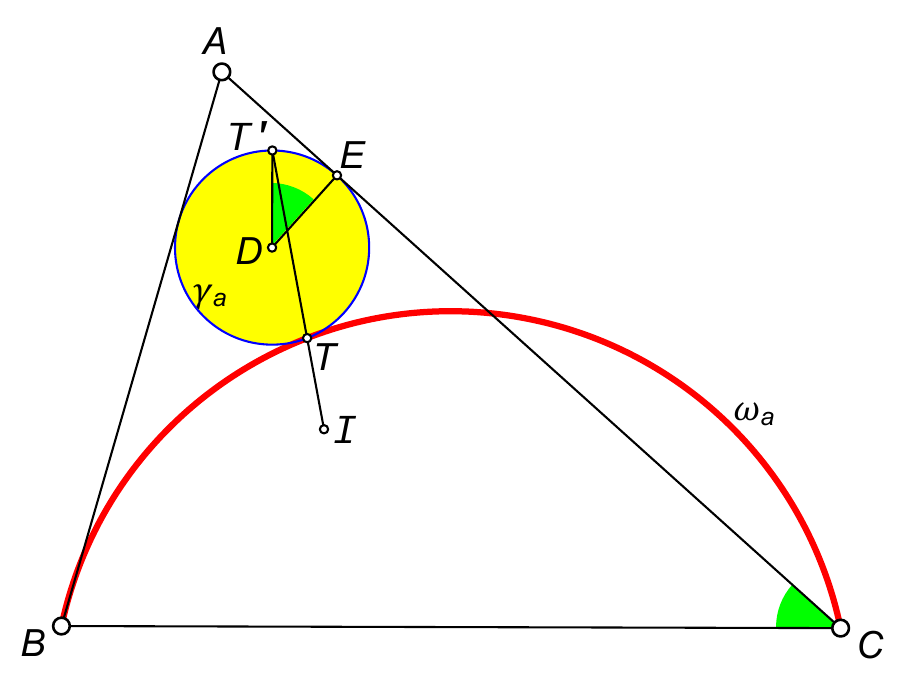}
\caption{green angles are equal}
\label{fig:result10}
\end{figure}

\begin{minipage}{3in}
\begin{proof}
Let $DT'$ meet $AC$ at $T_1$ and let $DT'$ meet $BC$ at $T_2$.
By Theorem~\ref{thm:result9}, $T_1T_2\perp BC$.
From right triangles $T_1ED$ and $CT_2T_1$, we see that $\angle EDT'=\angle ACB$
since they are both complementary to $\angle T_2T_1C$.
\end{proof}
\end{minipage}
\hfill
\raisebox{-1in}{\includegraphics[scale=0.4]{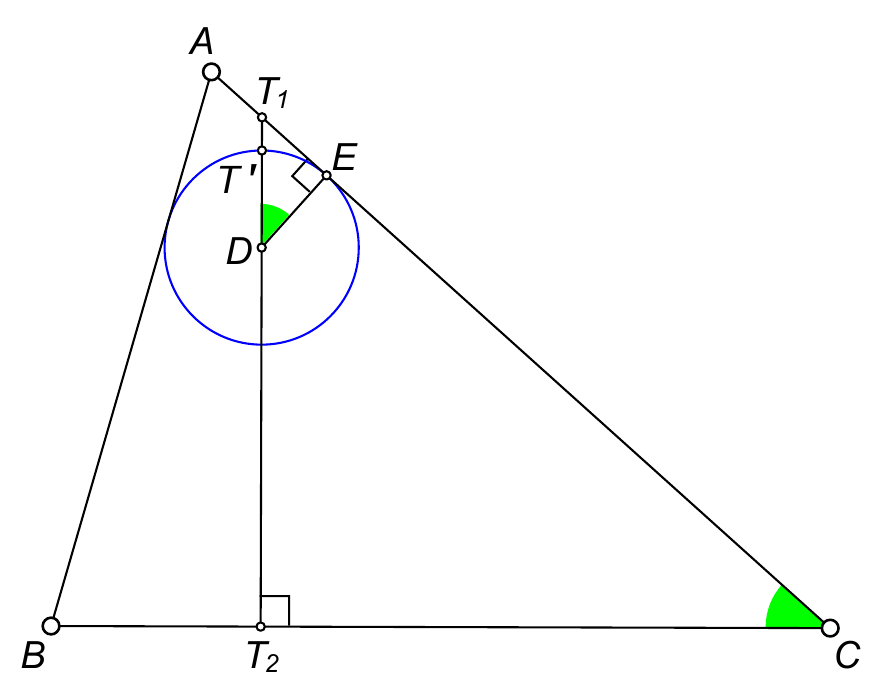}}

\newpage

\section{Properties Related to the Incircle}
\label{section:L}

In this section, we will discuss properties of Ajima circles that are related to the incircle.
As before, $I$ will denote the incenter of $\triangle ABC$.
Obviously, $IL\perp BC$.

Throughout this section, points will be labeled as shown in Figure~\ref{fig:configWithL}
and described in the following table.

\begin{figure}[ht]
\centering
\includegraphics[scale=0.4]{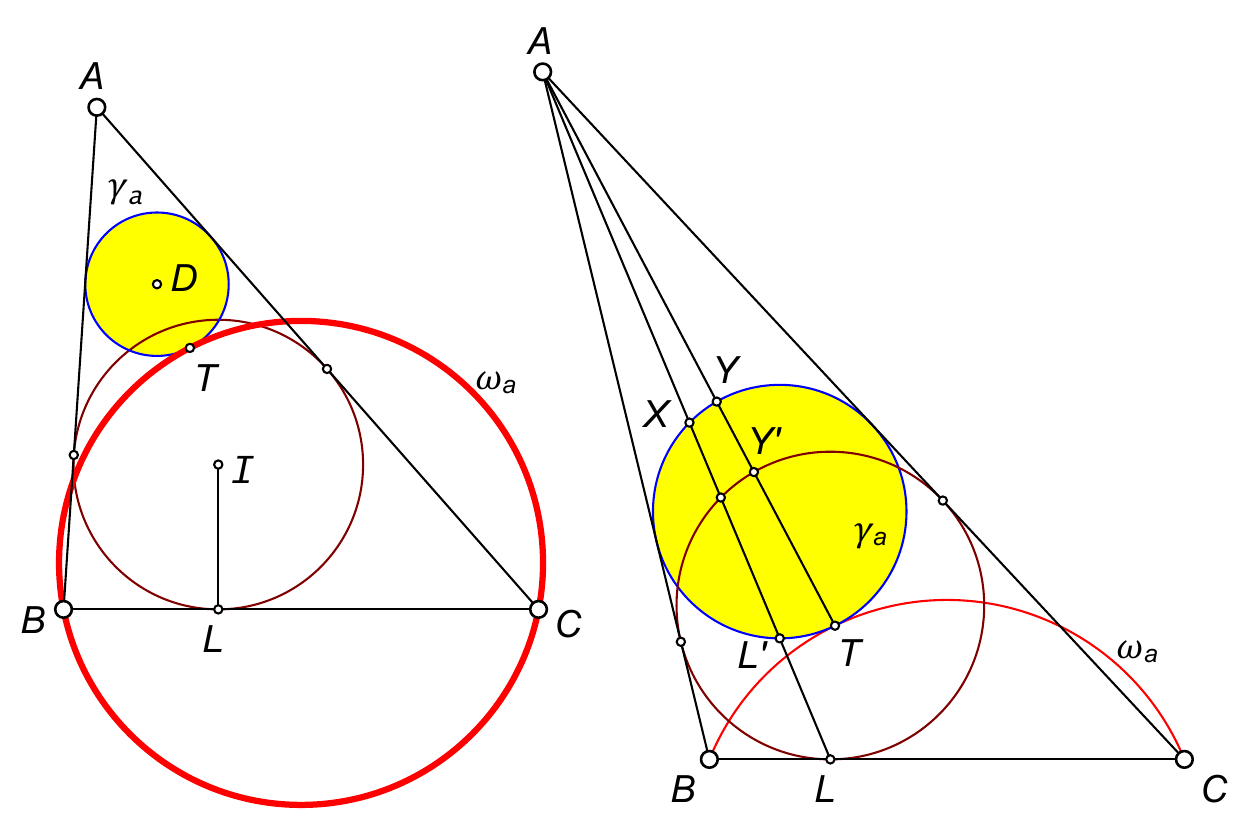}
\caption{basic configuration plus incircle}
\label{fig:configWithL}
\end{figure}

\begin{center}
\begin{tabular}{|c|l|}
\hline
\multicolumn{2}{|c|}{\color{red} \textbf{Notation for this Section}}\\ \hline
\textbf{Symbol}&\textbf{Description} \\ \hline
$I$&incenter of $\triangle ABC$\\ \hline
$D$&center of $\gamma_a$\\ \hline
$T$&Unless specified otherwise, $T$ is the point where $\gamma_a$ touches $\omega_a$.\\ \hline
$L$&point where incircle touches $BC$\\ \hline
$L'$&point closer to $L$ where $AL$ meets $\gamma_a$\\ \hline
$X$&point closer to $A$ where $AL$ meets $\gamma_a$\\ \hline
$Y$&point where $AT$ meets $\gamma_a$ again\\ \hline
$Y'$&point where $AT$ meets the incircle\\ \hline
\end{tabular}
\end{center}


\begin{theorem}\label{thm:result24}
Let $\gamma_a$ be any circle inscribed in $\angle BAC$.
Let $AL$ meet $\gamma_a$ at $L'$ (closer to $L$).
Let $D$ be the center of $\gamma_a$.
Then $DL'\perp BC$ (Figure~\ref{fig:result24}).
\end{theorem}

\begin{figure}[ht]
\centering
\includegraphics[scale=0.5]{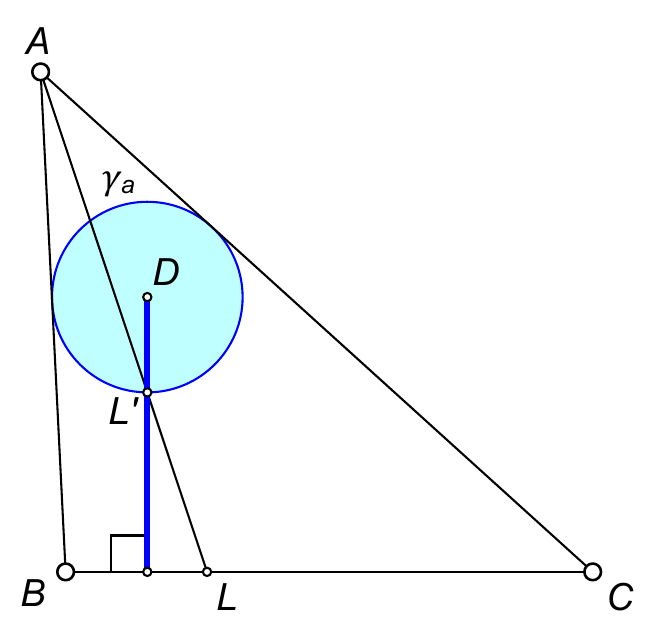}
\caption{$DL'\perp BC$}
\label{fig:result24}
\end{figure}

\begin{minipage}{3.4in}
\begin{proof}
The incircle and circle $\gamma_a$ are homothetic with $A$ being the center of the homothety.
This homothety maps $D$ to $I$ and maps $L'$ to $L$.
Since a homothety maps lines into parallel lines, we can conclude that $DL'\parallel IL$.
Since $IL\perp BC$, we therefore have $DL'\perp BC$.
\end{proof}
\end{minipage}
\hfil
\raisebox{-1in}{\includegraphics[scale=0.45]{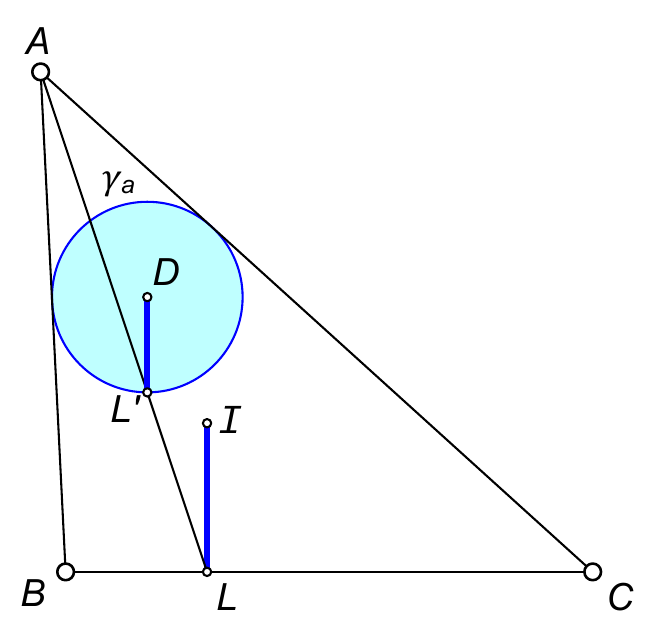}}

\begin{theorem}\label{thm:result27}
Let $\gamma_a$ be any circle inscribed in $\angle BAC$.
Let $AL$ meet $\gamma_a$ at $L'$ (closer to $L$).
Then the tangent to $\gamma_a$ at $L'$ is parallel to $BC$ (Figure~\ref{fig:result27}).
\end{theorem}

\begin{figure}[ht]
\centering
\includegraphics[scale=0.4]{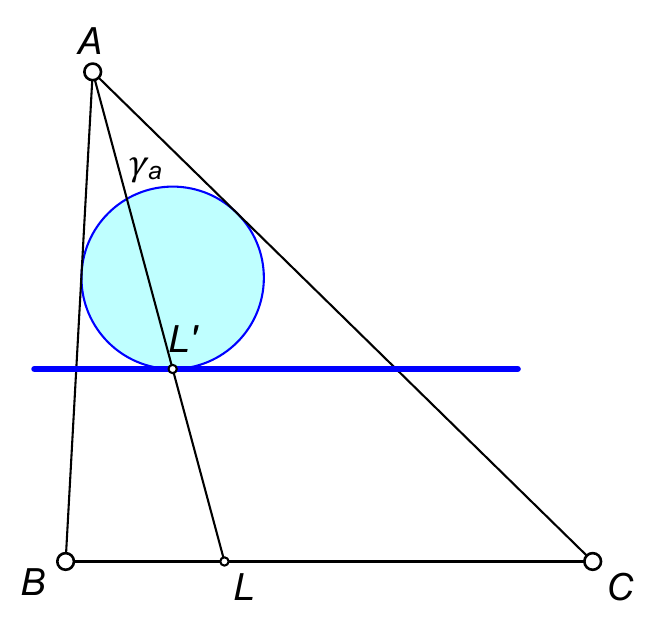}
\caption{blue tangent is parallel to $BC$}
\label{fig:result27}
\end{figure}

\begin{proof}
The tangent at $L'$ is perpendicular to $DL'$ (Figure~\ref{fig:result24}) which
is also perpendicular to $BC$ by Theorem~\ref{thm:result24}.
\end{proof}


\begin{theorem}\label{thm:result25}
Let $\gamma_a$ be any circle inscribed in $\angle BAC$.
Let $D$ be the center of $\gamma_a$.
Let $T$ be any point on $\gamma_a$.
Let $AL$ meet $\gamma_a$ at $L'$ (closer to $L$).
Let $AT$ meet $\gamma_a$ at $Y$ and $Y'$ with $Y$ closer to $A$.
Let $AT$ extended meet the incircle again at $T'$ (Figure~\ref{fig:result25}).
Then $YL'\parallel Y'L$, $L'T\parallel LT'$, and $YD\parallel Y'I$.
\end{theorem}

\begin{figure}[ht]
\centering
\includegraphics[scale=0.6]{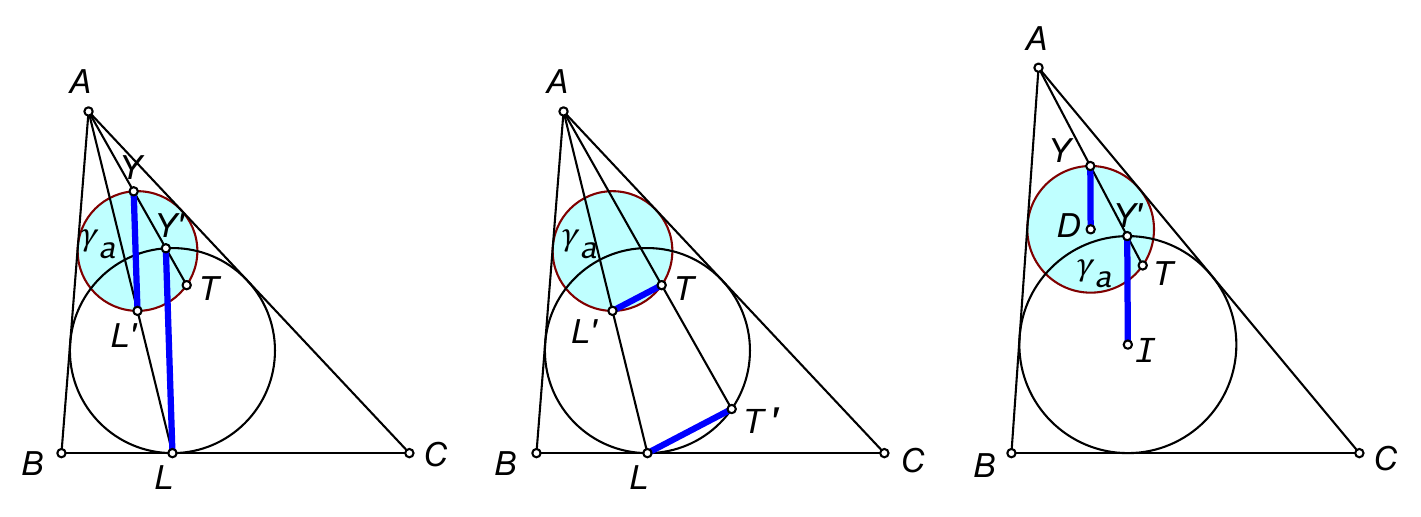}
\caption{blue lines are parallel}
\label{fig:result25}
\end{figure}

\begin{proof}
The incircle and circle $\gamma_a$ are homothetic with $A$ being the center of the homothety.
This homothety maps $D$ to $I$, $L'$ to $L$, $Y$ to $Y'$, and $T$ to $T'$.
These results then follow because a homothety maps a line into a parallel line.
\end{proof}

\begin{theorem}\label{thm:result13}
We have $\angle XTL'=\angle XLB$ (Figure~\ref{fig:result13}).

\end{theorem}

\begin{figure}[ht!]
\centering
\includegraphics[scale=0.5]{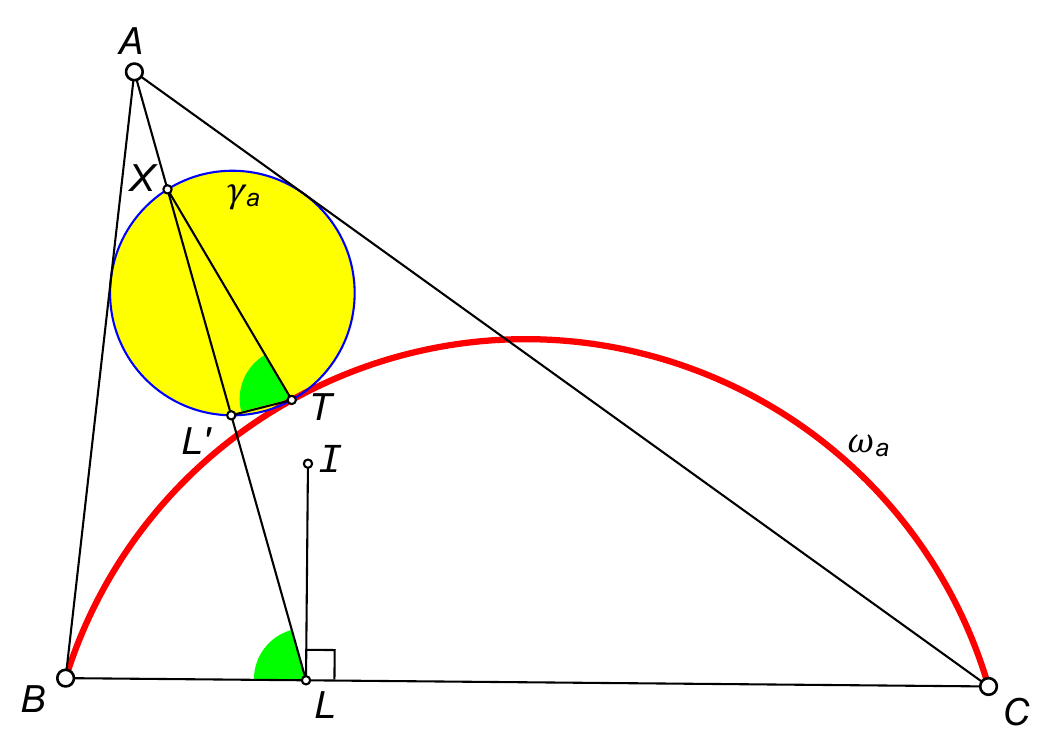}
\caption{green angles are equal}
\label{fig:result13}
\end{figure}

\begin{proof}
This is a special case of the following more general theorem.
\end{proof}


\begin{theorem}\label{thm:genresult13}
Let $T$ be any point on $\gamma_a$, on the opposite side of $AL$ from $B$.
Let $AL$ meet $\gamma_a$ at $X$ and $L'$ (with $X$ nearer $A$).
Then $\angle XTL'=\angle XLB$.
\end{theorem}

\begin{proof}
Let $L'Z$ be the tangent to $\gamma_a$ at $L'$ as shown in Figure~\ref{fig:result13proof}.

\begin{figure}[ht]
\centering
\includegraphics[scale=0.5]{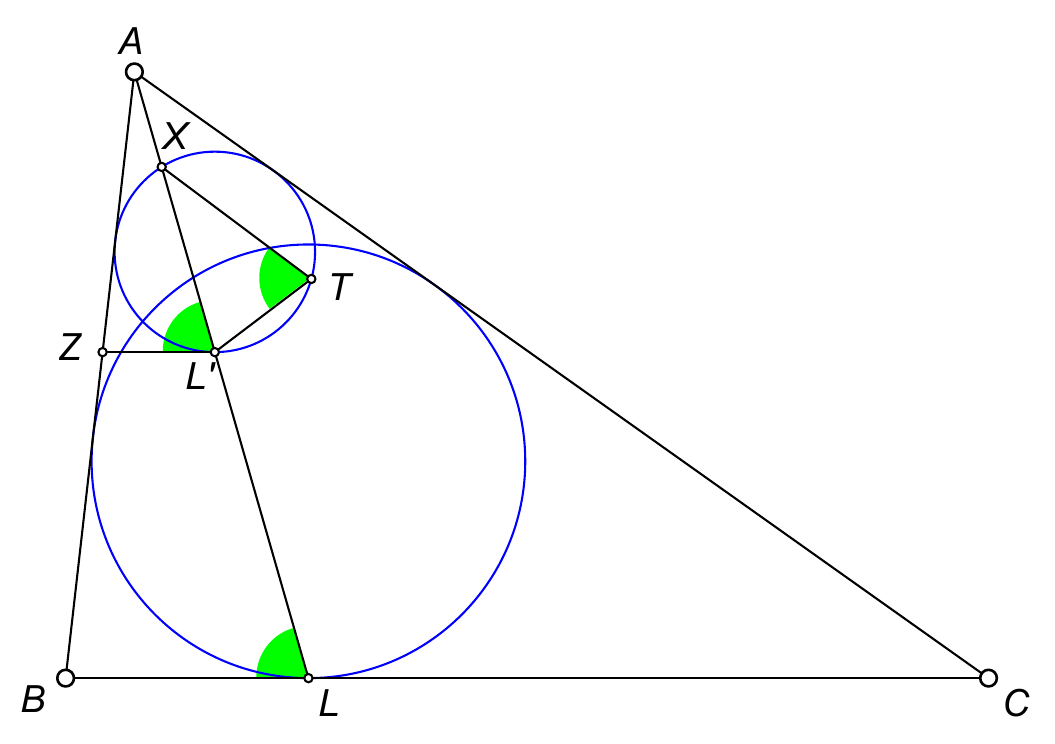}
\caption{green angles are equal}
\label{fig:result13proof}
\end{figure}
From Theorem~\ref{thm:result27}, $L'Z\parallel LB$, so $\angle AL'Z=\angle ALB$.
But $\angle XTL'=\angle XL'Z$ since both are measured by half of arc $\arc{XL'}$.
Thus $\angle XTL'=\angle ALB=\angle XLB$.
\end{proof}

\newpage

\begin{theorem}\label{thm:result17}
Let $\gamma_a$ be any circle inscribed in $\angle BAC$.
Let $T$ be any point on $\gamma_a$, on the opposite side of $AL$ from $B$.
Let $AL$ meet $\gamma_a$ at $X$ and $L'$ (with $X$ nearer $A$) as shown in Figure~\ref{fig:result17}.
Let $TL'$ meet $CB$ at $K$. Then $X$, $T$, $L$, and $K$ are concyclic.
\end{theorem}

\begin{figure}[ht!]
\centering
\includegraphics[scale=0.5]{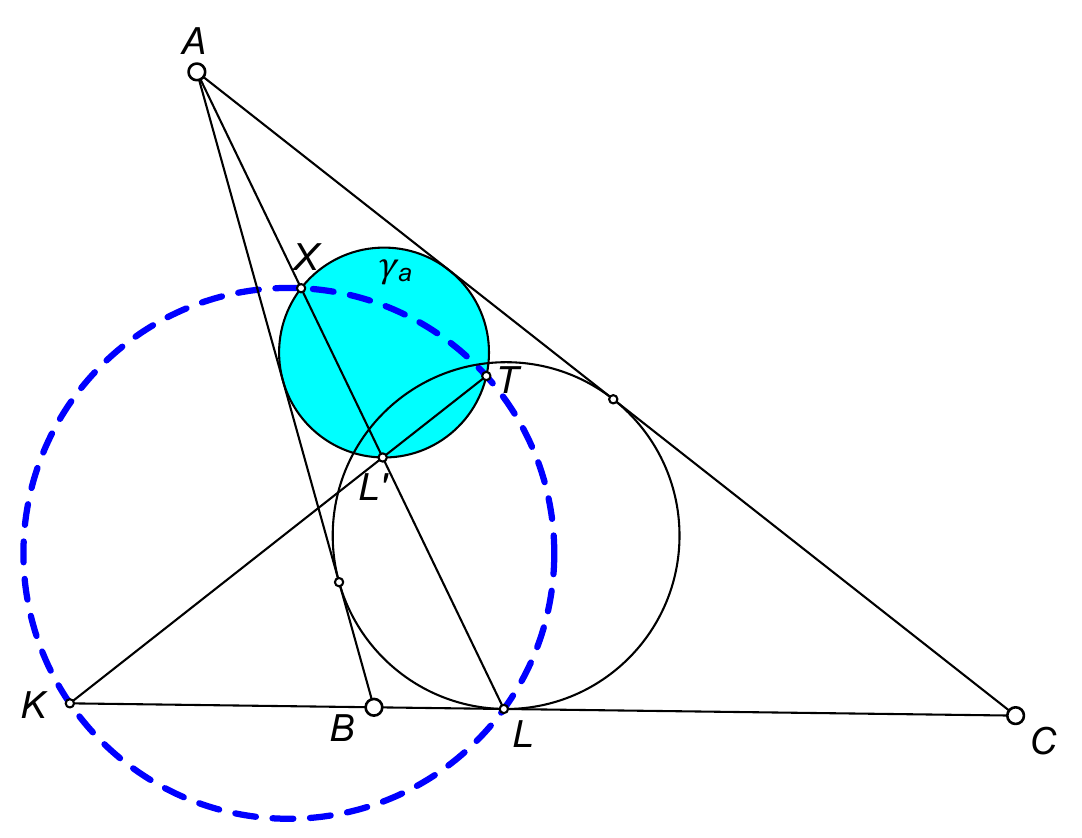}
\caption{$X,T,L,K$ lie on a circle.}
\label{fig:result17}
\end{figure}

\begin{proof}
From Theorem~\ref{thm:genresult13}, $\angle XTL'=\angle XLB$, or equivalently, $\angle XTK=\angle XLK$.
Thus, $X$, $T$, $L$, and $K$ are concyclic.
\end{proof}


The next five results have been suggested by Navid Safaei.

\begin{lemma}\label{angleBisectorLemma}
Let $N$ be the midpoint of arc $\arc{BC}$ of a circle.
Let $T$ be a point on arc $\arc{BN}$.
Then $TN$ is the external angle bisector of $\angle BTC$ (Figure~\ref{fig:angleBisectorLemma}).
\end{lemma}

\begin{figure}[ht!]
\centering
\includegraphics[scale=0.5]{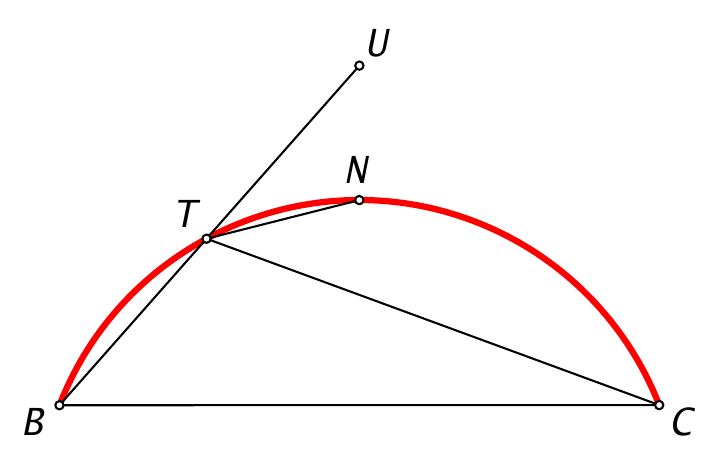}
\caption{}
\label{fig:angleBisectorLemma}
\end{figure}

\begin{proof}
Using properties of angles inscribed in a circle, we have
$$\angle NTU=\frac12(\arc{BT}+\arc{TN})=\frac12\arc{BN}=\frac12\arc{CN}=\angle CTN,$$
so $TN$ bisects $\angle CTU$.
\end{proof}

\newpage

\begin{lemma}\label{LTNlemma}
Let $N$ be the midpoint of arc $\arc{BC}$ of $\omega_a$.
Then $L'$, $T$, and $N$ are collinear (Figure~\ref{fig:LTNlemma}).
\end{lemma}

\begin{figure}[ht!]
\centering
\includegraphics[scale=0.5]{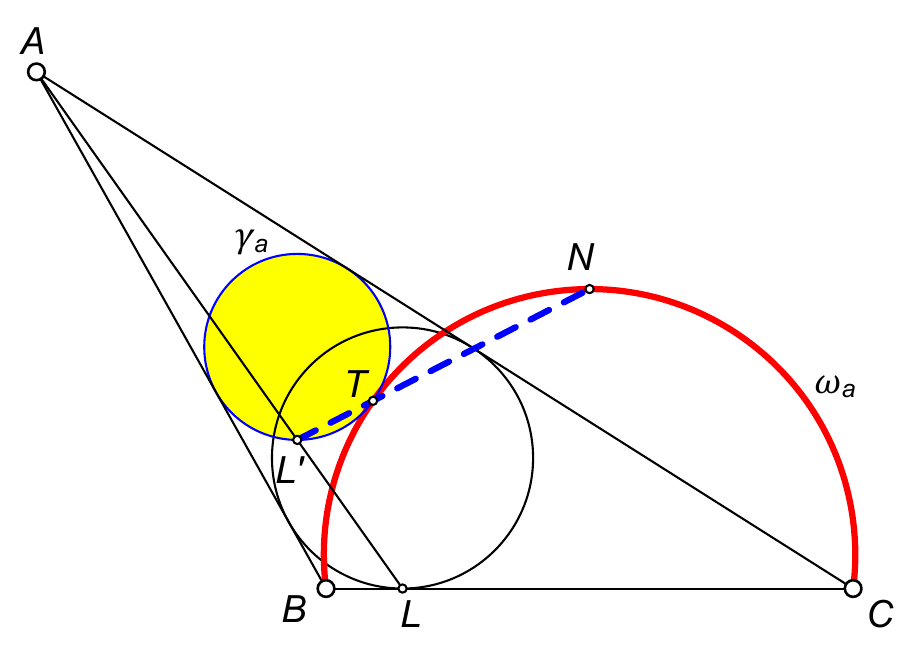}
\caption{}
\label{fig:LTNlemma}
\end{figure}

\begin{proof}
Note that $T$ is the center of a homothety between $\gamma_a$ and $\omega_a$.
Since the tangents at $L'$ and $N$ are parallel to $BC$ (by Theorem~\ref{thm:result27}),
this means that they are corresponding points of the homothety and hence $L'N$ passes through
the center of the homothety, $T$.
\end{proof}

\begin{theorem}\label{thm:extAngBisector}
The line $TL'$ is the exterior angle bisector of $\angle BTC$ in $\triangle BTC$. (Figure~\ref{fig:extAngBisector}).
\end{theorem}

\begin{figure}[ht!]
\centering
\includegraphics[scale=0.5]{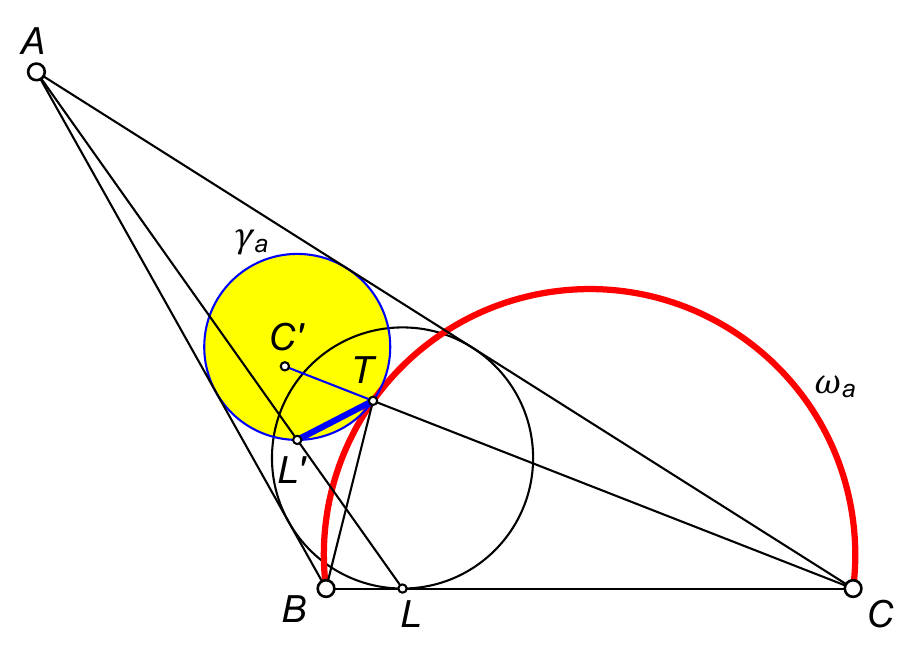}
\caption{$TL'$ bisects $\angle C'TB$}
\label{fig:extAngBisector}
\end{figure}

\begin{proof}
This follows immediately from Lemmas \ref{angleBisectorLemma} and \ref{LTNlemma}.
\end{proof}

\newpage

\begin{theorem}\label{thm:result18}
Let $E$ and $F$ be the points where $\gamma_a$ touches $AC$ and $AB$, respectively.
Then $EF$, $TL'$, and $CB$ are concurrent.
\end{theorem}

\begin{figure}[ht!]
\centering
\includegraphics[scale=0.55]{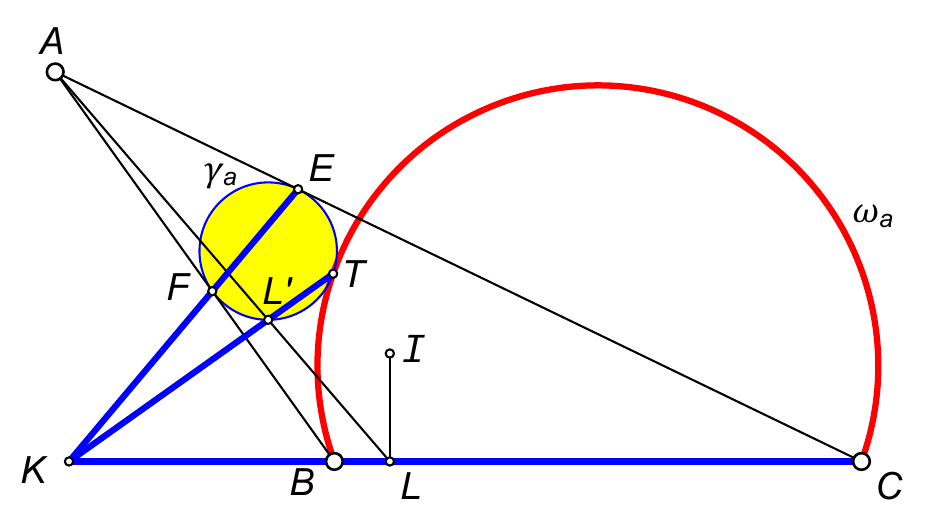}
\caption{Blue lines are concurrent.}
\label{fig:result18}
\end{figure}

\begin{proof}
Let $EF$ meet $CB$ at $K$.
From Theorem \ref{thm:result12}, we have 
\begin{equation}
   \frac{BT}{CT}=\frac{BF}{CE}.\label{eq:12}
\end{equation}

\begin{minipage}{3in}
By Menelaus' Theorem applied to $\triangle ABC$ and the transversal $FE$, we have
\begin{equation*}
   \frac{BF}{FA}\cdot \frac{AE}{EC}\cdot \frac{CK}{KB}=-1
\end{equation*}
from which we get 
\begin{equation}
   \frac{BF}{CE}=\frac{KB}{CK} \label{eq:13}
\end{equation}
because $FA=AE$.
From equations (\ref{eq:12}) and (\ref{eq:13}) we get
\begin{equation}
   \frac{KB}{CK}=\frac{BF}{CE}=\frac{BT}{CT}. \label{eq:14}
\end{equation}
\end{minipage}
\raisebox{-1in}{\includegraphics[scale=0.45]{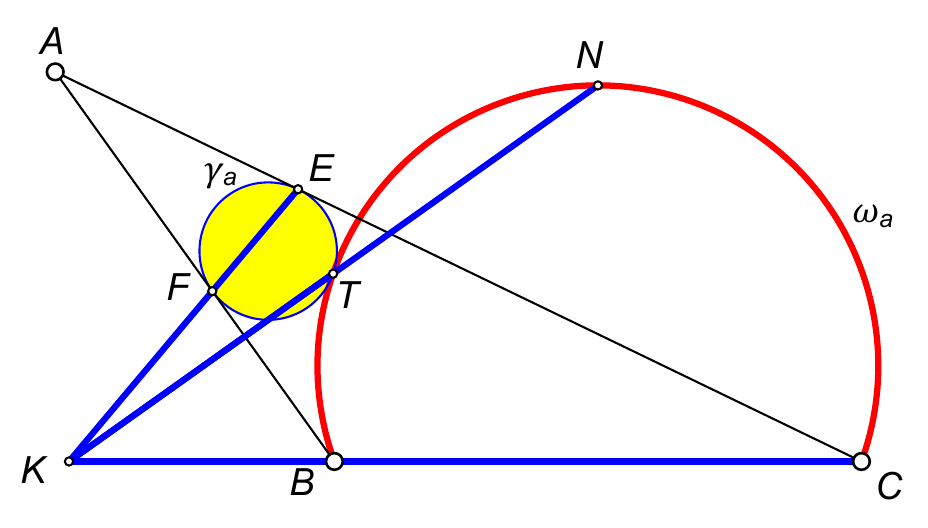}}

Hence (by a property of external angle bisectors), $TK$ is the external angle bisector of $\angle BTC$
in $\triangle BTC$.
Let $N$ be the midpoint of arc $\arc{BC}$ as shown in the figure above.
By Lemma~\ref{angleBisectorLemma}, $TN$ is also the external angle bisector of $\angle BTC$. 
It follows that $N$, $T$, and $K$ are collinear.
Since $TK$ passes through both $L'$ and $N$ (by Lemma~\ref{LTNlemma}), the four points, $K$, $L'$, $T$, and $N$
lie on a line, so $EF$, $TL'$, and $CB$ all pass through $K$.
\end{proof}

\newpage

\begin{theorem}\label{thm:result19}
The lines $YX$, $TL'$, and $CB$ are concurrent (Figure~\ref{fig:result19}).
\end{theorem}

\begin{figure}[ht!]
\centering
\includegraphics[scale=0.55]{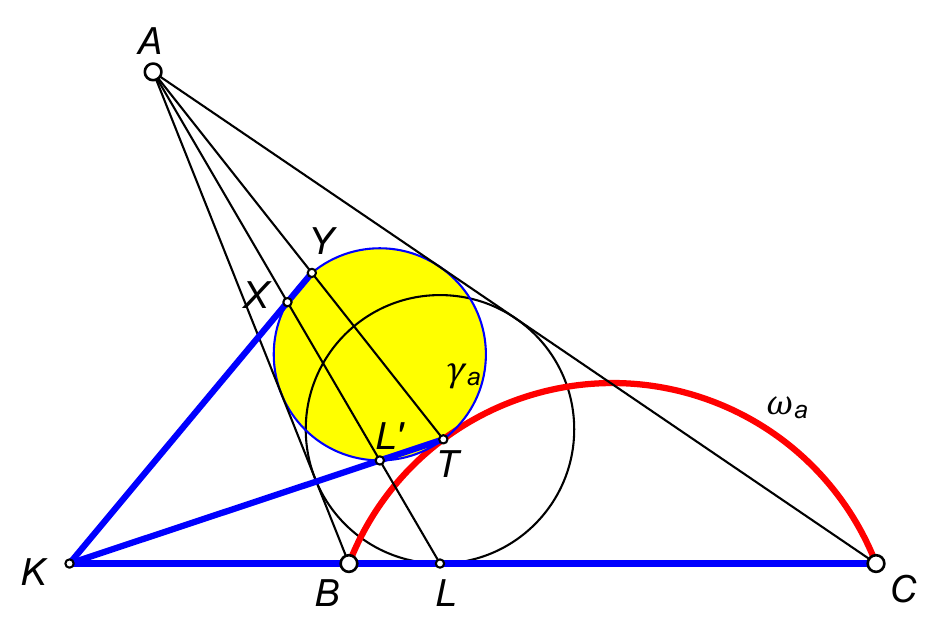}
\caption{Blue lines are concurrent.}
\label{fig:result19}
\end{figure}

\begin{proof}
Let $E$ and $F$ be the points where $\gamma_a$ touches $AC$ and $AB$, respectively.
Let $EF$ meet $CB$ at $K$.
Then $EF$, $TL'$, and $CB$ are concurrent at $K$
by Theorem~\ref{thm:result18} (Figure~\ref{fig:result19proof}).

\begin{figure}[ht!]
\centering
\includegraphics[scale=0.6]{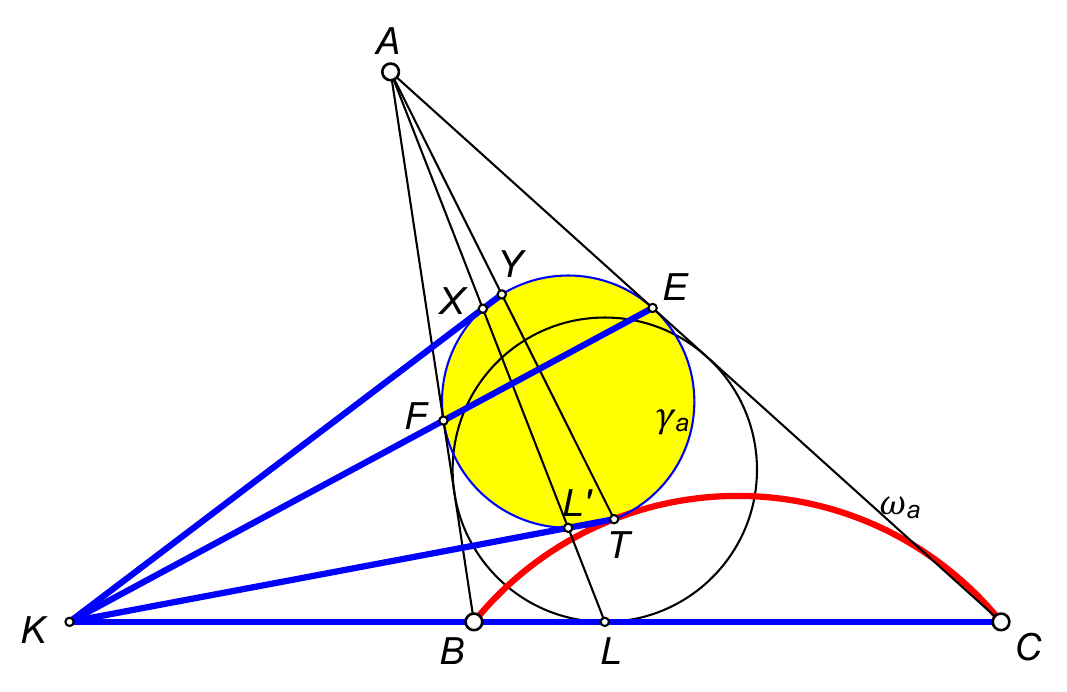}
\caption{Blue lines are concurrent.}
\label{fig:result19proof}
\end{figure}
Since $AF$ and $AE$ are tangents to circle $\gamma_a$, this means that $EF$ is the polar of $A$
with respect to $\gamma_a$.
Since $AXL'$ and $AYT$ are two secants from $A$, this means that $YX$ meets $TL'$ on the polar of $A$ (line $EF$). But $K$ is the only point on $TL'$ that lies on the polar of $A$.
Thus, $YX$ also passes through $K$.
\end{proof}

\newpage

The following result comes from \cite{Rabinowitz13268}.

\begin{theorem}\label{thm:result14}
The line $TL'$ bisects $\angle ATL$ (Figure~\ref{fig:result14}).
\end{theorem}

\begin{figure}[ht]
\centering
\includegraphics[scale=0.5]{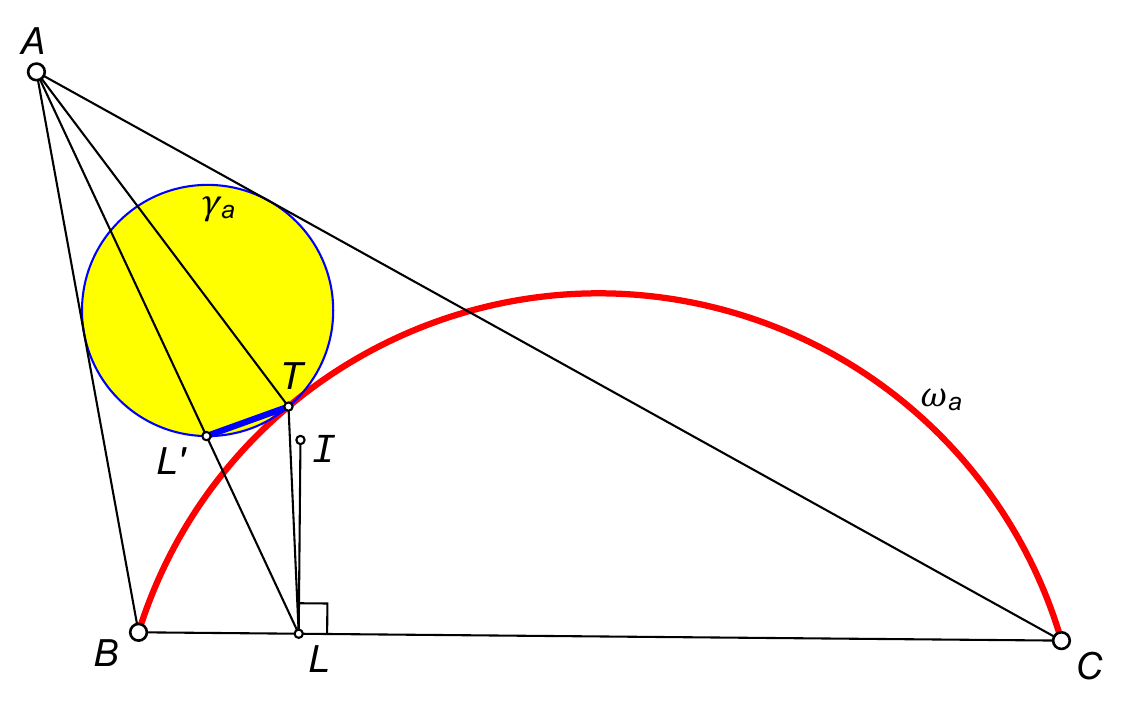}
\caption{$TL'$ bisects $\angle ATL$}
\label{fig:result14}
\end{figure}

\begin{proof}
Let $TL'$ meet $CB$ at $K$ (Figure~\ref{fig:result14proof}).
\begin{figure}[ht]
\centering
\includegraphics[scale=0.5]{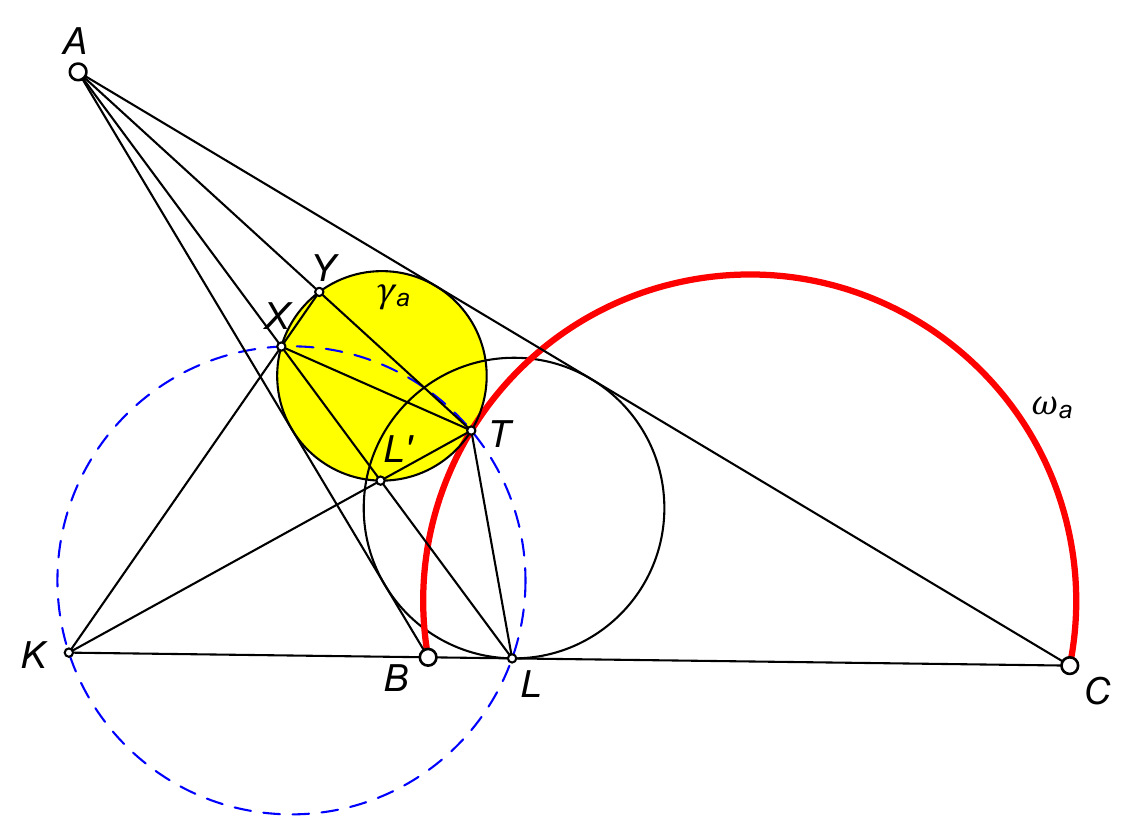}
\caption{}
\label{fig:result14proof}
\end{figure}

By Theorem~\ref{thm:result19}, $YX$ passes through $K$.
By Theorem~\ref{thm:result17}, $XTLK$ is a cyclic quadrilateral, so $\angle KTL=\angle KXL$.
But since $XYTL'$ is also a cyclic quadrilateral, $\angle KXL=\angle YTL'$.
Thus, $\angle KTL=\angle YTL'$ so $TL'$ bisects $\angle ATL$.
\end{proof}

\newpage

\begin{theorem}\label{thm:result16}
Extend $AT$ until it meets the incircle at $T'$ as shown in Figure~\ref{fig:result16}.
Then $TL=TT'$.
\end{theorem}

\begin{figure}[ht]
\centering
\includegraphics[scale=0.5]{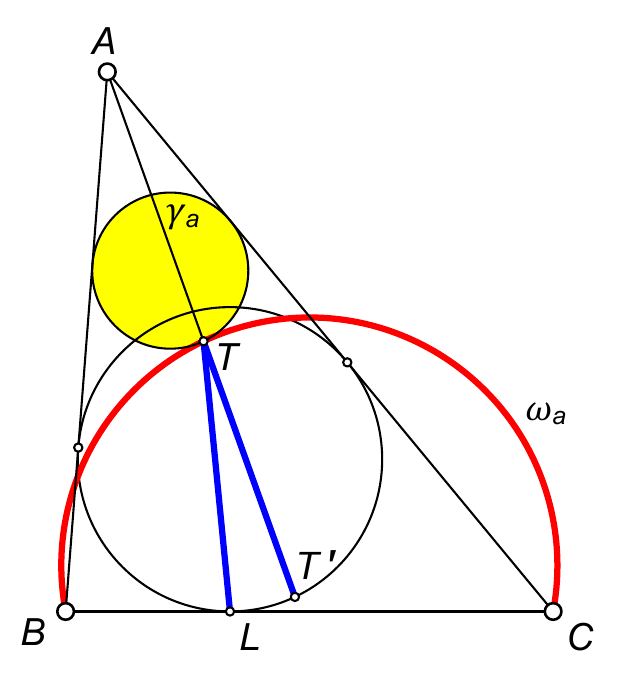}
\caption{blue lines are congruent}
\label{fig:result16}
\end{figure}

\begin{minipage}{2.5in}
\begin{proof}
Let $AL$ meet $\gamma_a$ at $L'$, closer to $L$, as shown in the figure to the right.
The incircle and $\gamma_a$ are homothetic with $A$ being the center of the homothety.
Since a homothety maps a line into a parallel line, $L'T\parallel LT'$.
Thus $\angle 2=\angle 3$ and $\angle 1=\angle 4$.
By Theorem~\ref{thm:result14}, $\angle 1=\angle 2$.
Hence $\angle 3=\angle 4$ making $\triangle TLT'$ an isosceles triangle.
\end{proof}
\end{minipage}
\hspace{-12pt}
\raisebox{-1in}{\includegraphics[scale=0.6]{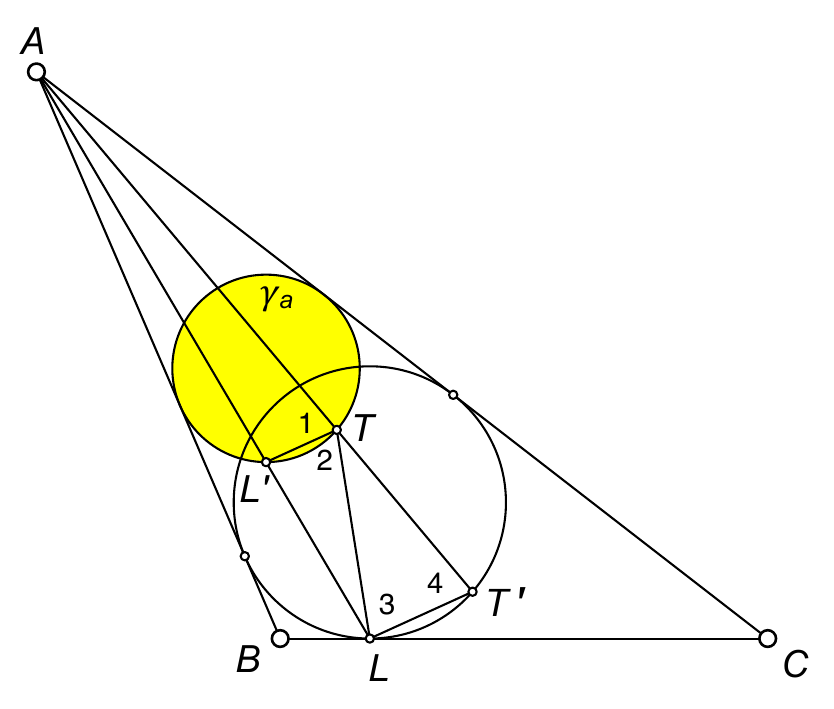}}


The following result comes from \cite{Suppa13196}.

\begin{theorem}\label{thm:result5} 
Let $AT$ meet $BC$ at $M$.
Then $TI$ bisects $\angle LTM$ (Figure~\ref{fig:result5}).
\end{theorem}

\begin{figure}[ht]
\centering
\includegraphics[scale=0.6]{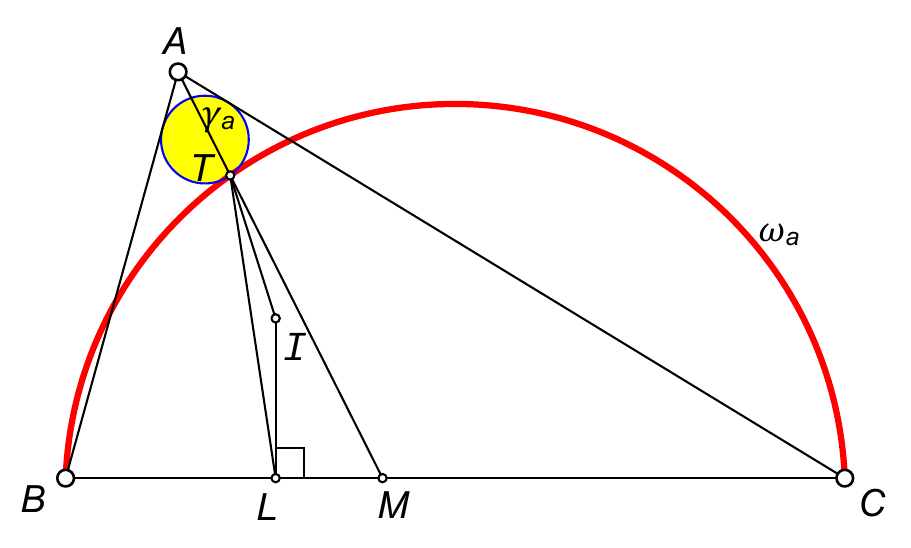}
\caption{$TI$ bisects $\angle LTM$}
\label{fig:result5}
\end{figure}

\begin{minipage}{3in}
\begin{proof}
Let $AM$ meet the incircle at $T'$ (closer to $M$)
as shown in the figure to the right.
By Theorem~\ref{thm:result16}, $TL=TT'$.
Since both $L$ and $T'$ lie on the incircle, we must have $IL=IT'$.
Thus $\triangle LTI\cong \triangle T'TI$ by SSS.
Hence $\angle LTI=\angle ITT'$.
\end{proof}
\end{minipage}
\raisebox{-1in}{\includegraphics[scale=0.5]{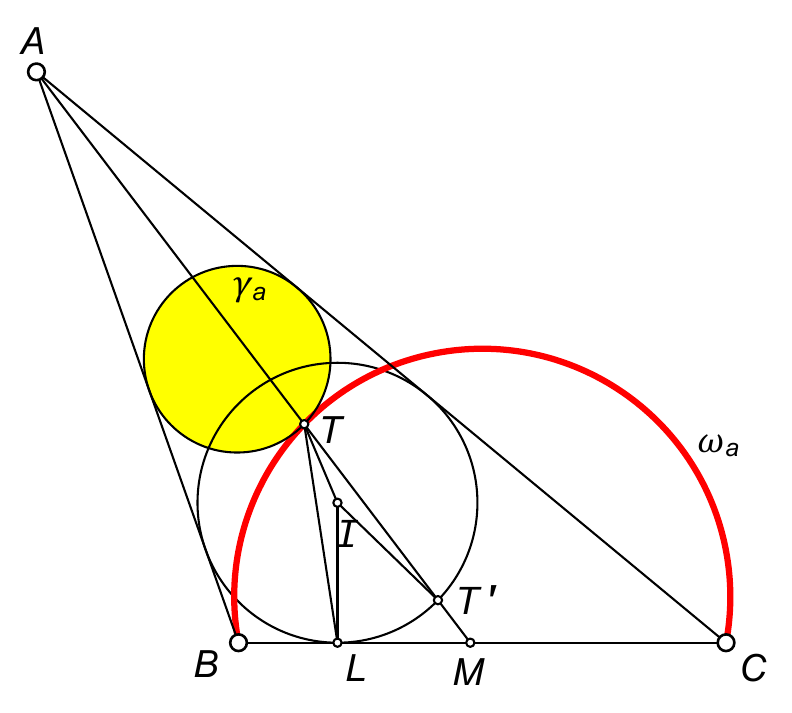}}

\begin{theorem}\label{thm:result1}
We have $\angle ATD=\angle ILT$ (Figure~\ref{fig:result1}).
\end{theorem}

\begin{figure}[ht]
\centering
\includegraphics[scale=0.5]{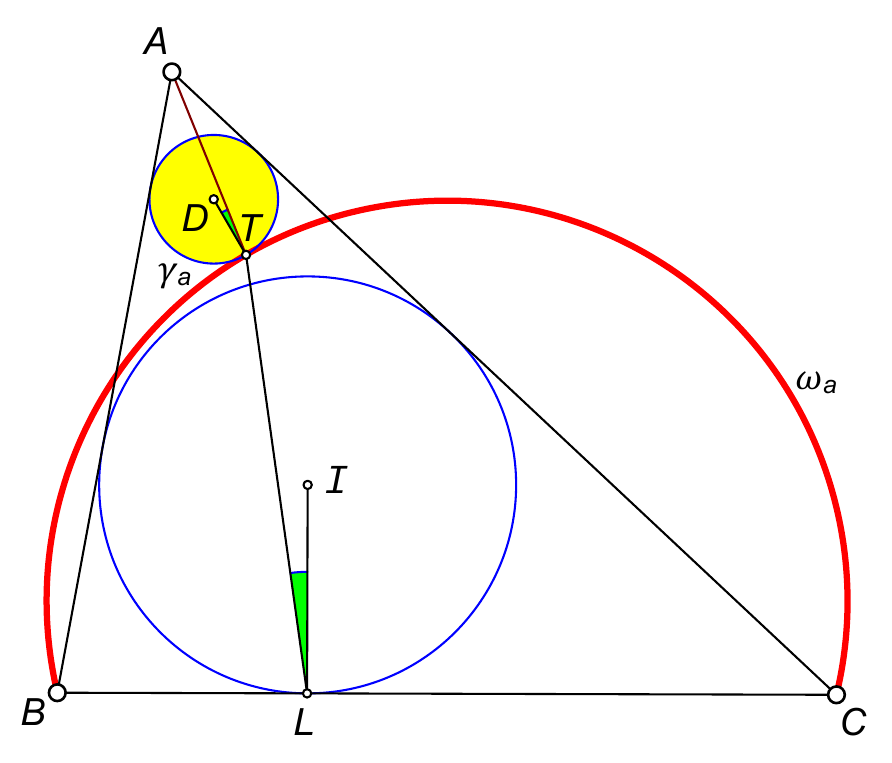}
\caption{green angles are equal}
\label{fig:result1}
\end{figure}

\begin{proof}
Let $F$ be the foot of the perpendicular from $T$ to $BC$.
Let $AT$ meet $BC$ at $M$.
Since $\gamma_a$ is tangent to $\omega_a$ at $T$,
this means that $DTO_a$ is a straight line.

\begin{minipage}{3in}

Number the resulting angles as shown in the figure to the right.
Lines $AM$ and $DO_a$ meet at $T$ forming equal vertical angles.
These are labeled $x$ in the figure.
From Theorem~\ref{thm:result8}, $\angle BTF=\angle O_aTC$.
These are labeled 1 in the figure.
Since $TF\parallel IL$, $\angle FTL=\angle ILT$.
These are labeled $y$ in the figure.
From Theorem~\ref{thm:result5}, $\angle LTI=\angle ITM$.
These are labeled 2 in the figure.

\end{minipage}
\hfill
\raisebox{-0.95in}{\includegraphics[scale=0.4]{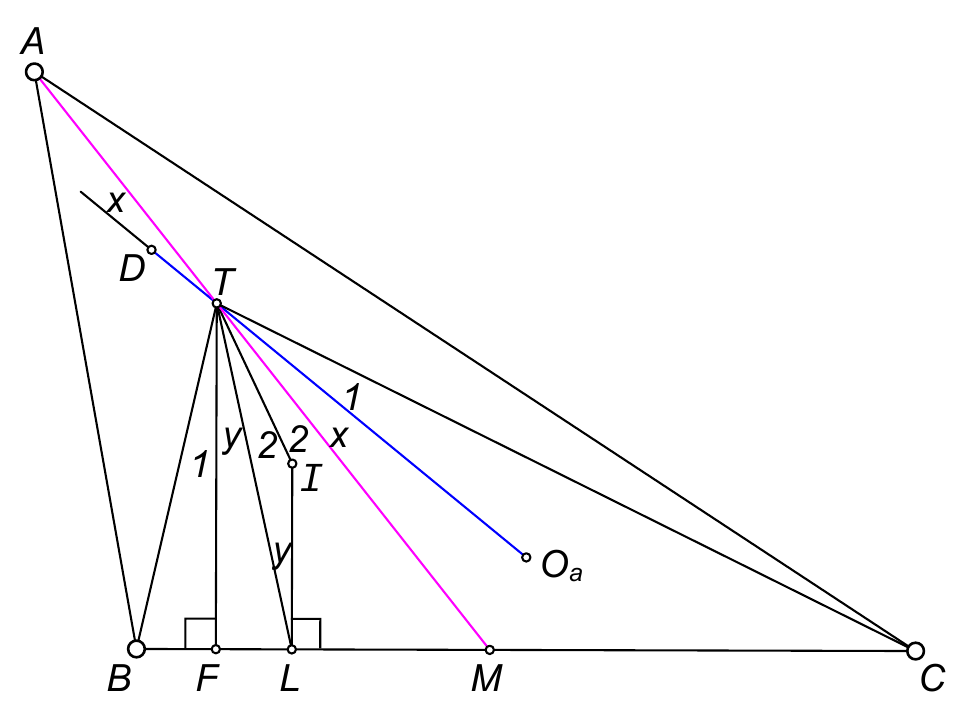}}

By Protasov's Theorem, $1+y+2=2+x+1$.
Thus $x=y$ and $\angle ATD=\angle ILT$.
\end{proof}

\newpage

\begin{lemma}\label{lemma:touchingCircles}
Two circles, $C_1$ and $C_2$, are internally tangent at $P$.
A chord $AB$ of $C_1$ meets $C_2$ at points $C$ and $D$ as shown in Figure~\ref{fig:equalAngles}.
Then $\angle APC=\angle DPB$.
\end{lemma}

\begin{figure}[ht]
\centering
\includegraphics[scale=0.5]{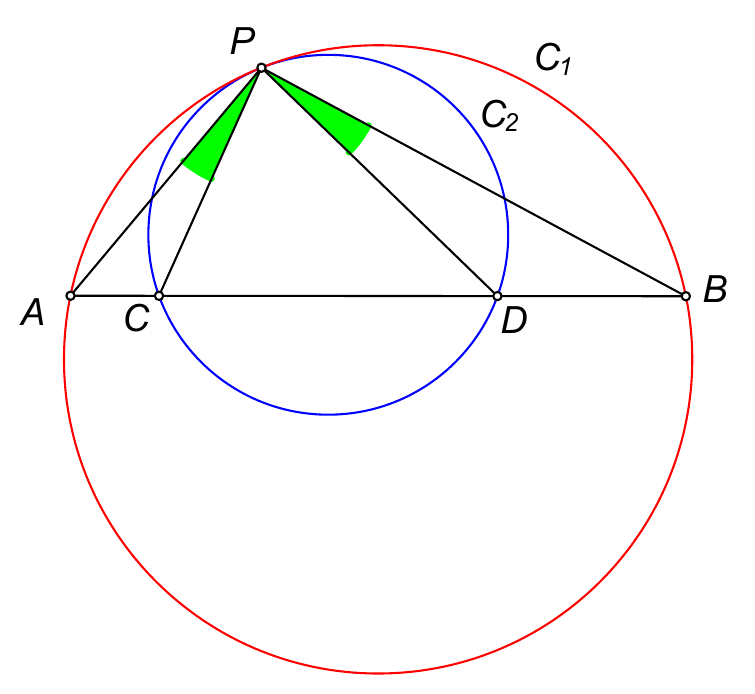}
\caption{green angles are equal}
\label{fig:equalAngles}
\end{figure}

\begin{minipage}{3in}
\begin{proof}
Let $t$ be the common tangent at $P$.
Let $PA$ meet $C_2$ at $E$ and let $PB$ meet $C_2$ at $F$.
Label the angles as shown in the figure to the right.

\smallskip
In the blue circle, $\angle 1=\angle 2$ since both are measured by half of arc $\arc{PF}$.
In the red circle, $\angle 1=\angle 3$ since both are measured by half of arc $\arc{PB}$.

\smallskip
Thus $\angle 2=\angle 3$ which makes $EF\parallel AB$.
Parallel chords intercept equal arcs, so $\arc{CE}=\arc{FD}$
which implies $\angle x=\angle y$.
\end{proof}
\end{minipage}
\raisebox{-1.2in}{\includegraphics[scale=0.5]{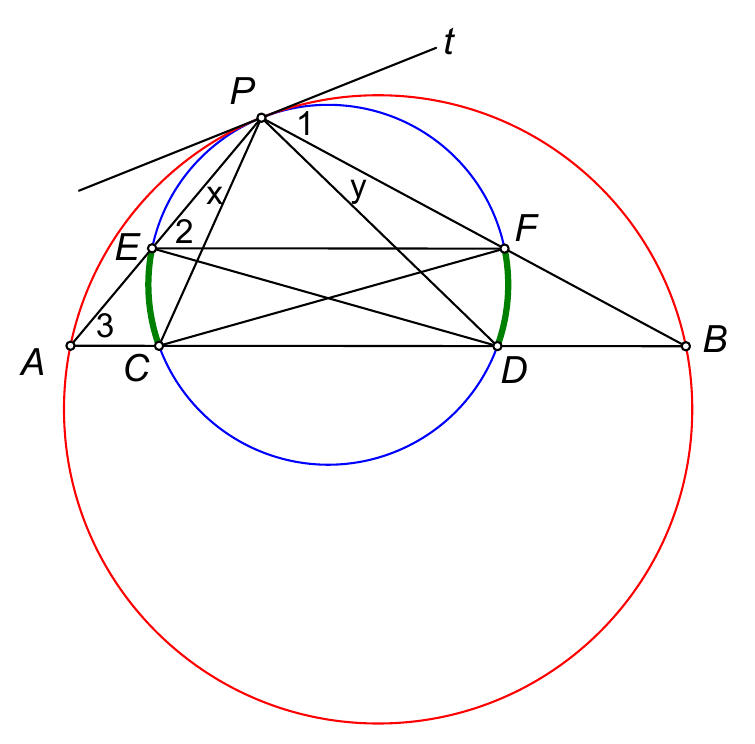}}


The following result comes from \cite{Coulon}.

\begin{theorem}
Let $AT$ meet $BC$ at $M$.
The $\odot TLM$ is tangent to $\gamma_a$ (Figure~\ref{fig:result15}).
\end{theorem}

\begin{figure}[ht]
\centering
\includegraphics[scale=0.5]{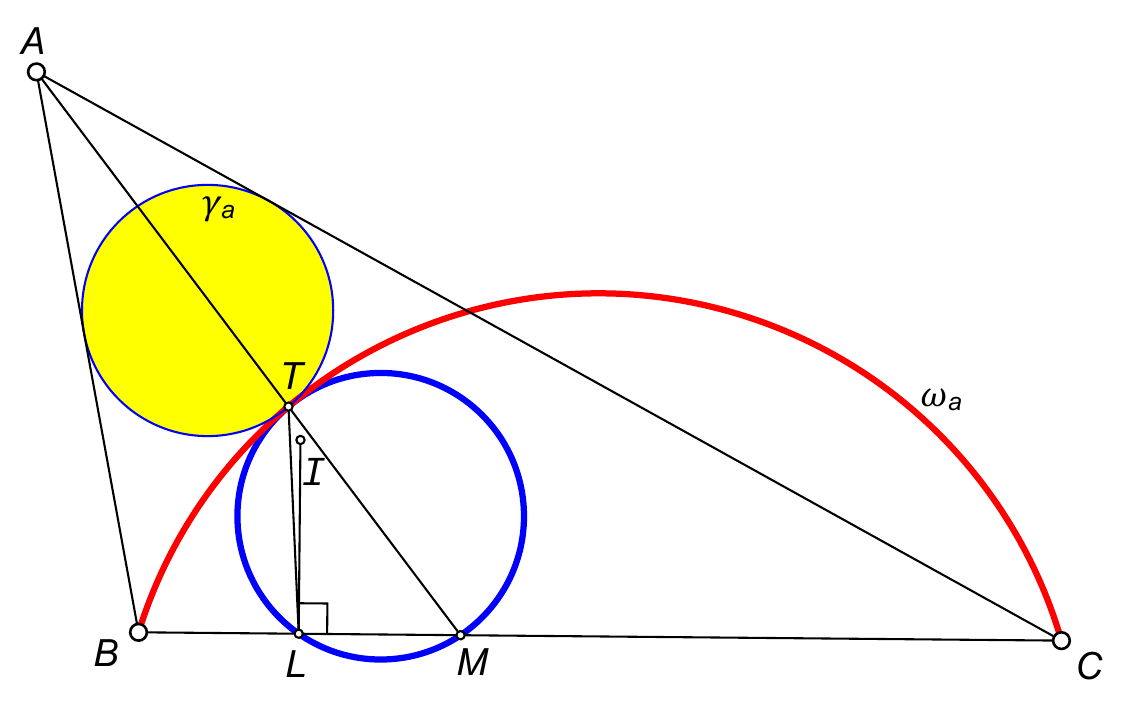}
\caption{three circles touch at $T$}
\label{fig:result15}
\end{figure}

\begin{proof}
Let $\Gamma$ be the circle tangent to $\omega_a$ at $T$ and passing through $L$.
Let $LC$ meet $\Gamma$ again at $M'$ as shown in Figure~\ref{fig:result15proof}.
By Lemma~\ref{lemma:touchingCircles},
$$\angle BTL=\angle M'TC.$$
These are labeled ``1'' in the figure.
By Protasov's Theorem,
$$\angle BTI=\angle ITC.$$
Subtracting shows that
$$\angle 2=\angle 3.$$
So $TI$ bisects $\angle LTM'$.
But by Theorem~\ref{thm:result5}, $TI$ bisects $\angle LTM$.
This implies that $M'=M$, so $\Gamma=\odot(TLM)$ and we are done.
\end{proof}

\begin{figure}[ht]
\centering
\includegraphics[scale=0.5]{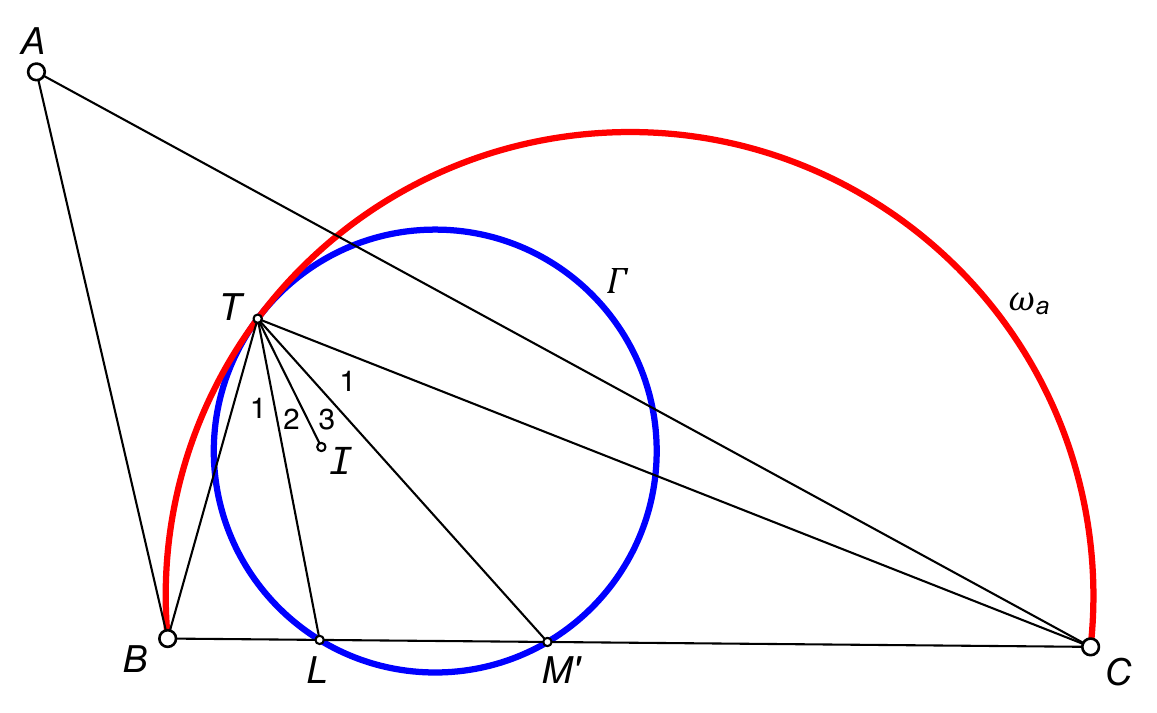}
\caption{}
\label{fig:result15proof}
\end{figure}


\begin{theorem}\label{thm:result7}
We have $\angle DTI=\angle TIL$ (Figure~\ref{fig:result7}).
\end{theorem}

\begin{figure}[ht]
\centering
\includegraphics[scale=0.5]{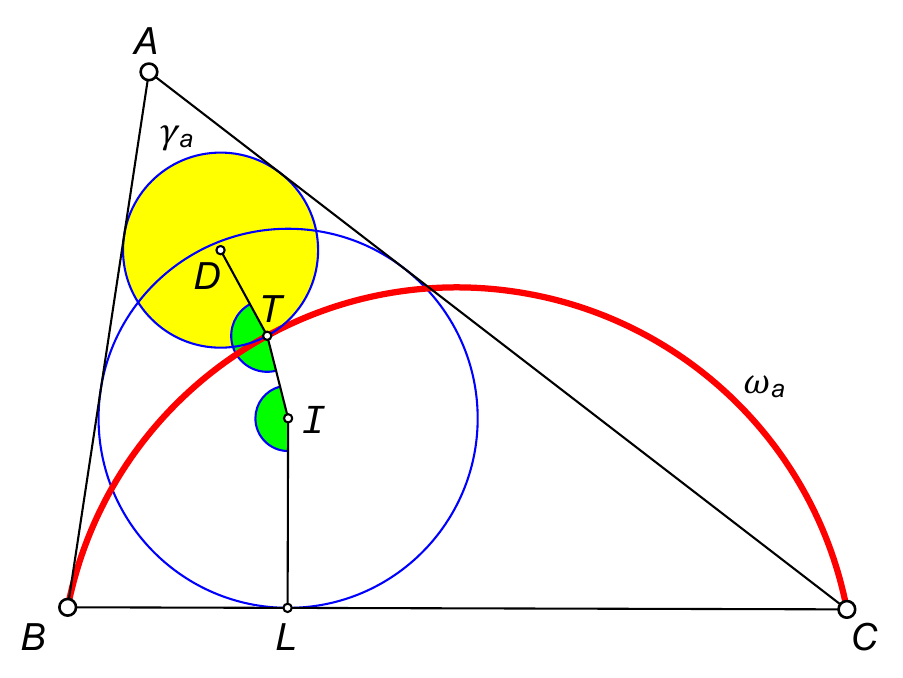}
\caption{green angles are equal}
\label{fig:result7}
\end{figure}

\begin{minipage}{2.8in}
\begin{proof}
Let $AT$ meet $BC$ at $M$.
The sum of the angles of $\triangle TIL$ is $180\degrees$, so
\begin{equation}
180\degrees-\angle LTI=\angle TIL+\angle TLI.\label{eq:7}
\end{equation}

Since $ATM$ is a straight line, we have
$$\angle DTI=180\degrees-\angle ITM-\angle ATD.$$
By Theorem~\ref{thm:result5}, $\angle ITM=\angle LTI$, so
$$\angle DTI=180\degrees-\angle LTI-\angle ATD.$$
From equation (\ref{eq:7}), we get
$$\angle DTI=\angle TIL+\angle TLI-\angle ATD.$$
From Theorem~\ref{thm:result1}, $\angle TLI=\angle ATD$. Hence
$\angle DTI=\angle TIL$.
\end{proof}
\end{minipage}
\hfill
\raisebox{-1in}{\includegraphics[scale=0.4]{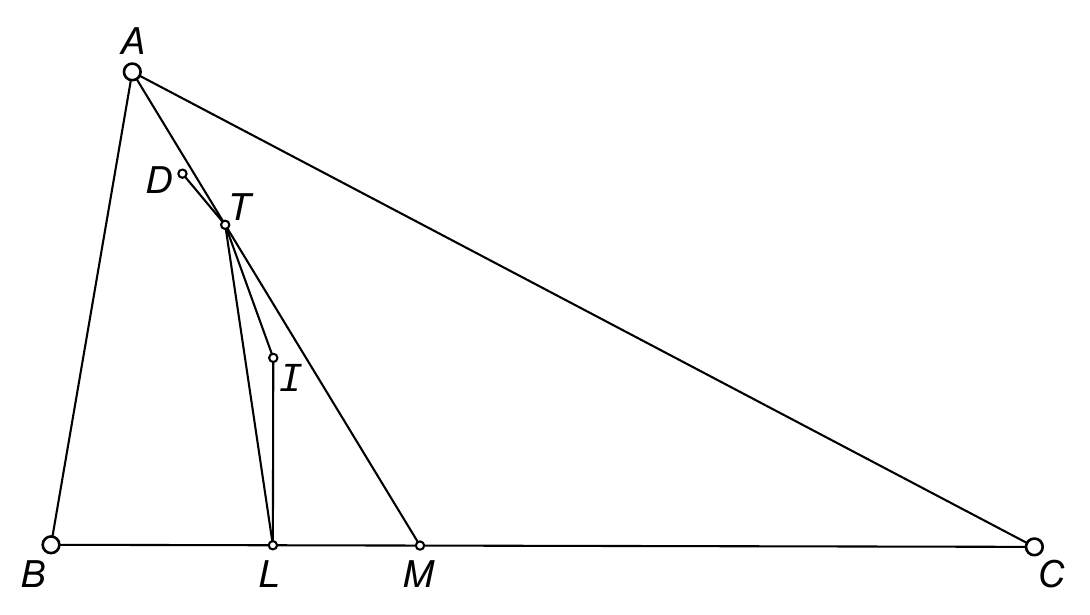}}


\begin{lemma}\label{lemma:result21}
Let $\gamma_a$ be any circle inscribed in $\angle BAC$.
Let $T$ be any point on $\gamma_a$ on the other side of $AL$ from $B$.
Let $AL$ meet $\gamma_a$ at $X$ (nearer $A$).
Let $AT$ meet the incircle at $Y'$.
Then $X$, $Y'$, $T$, and $L$ are concyclic (Figure~\ref{fig:result21}).
\end{lemma}

\begin{figure}[ht]
\centering
\includegraphics[scale=0.35]{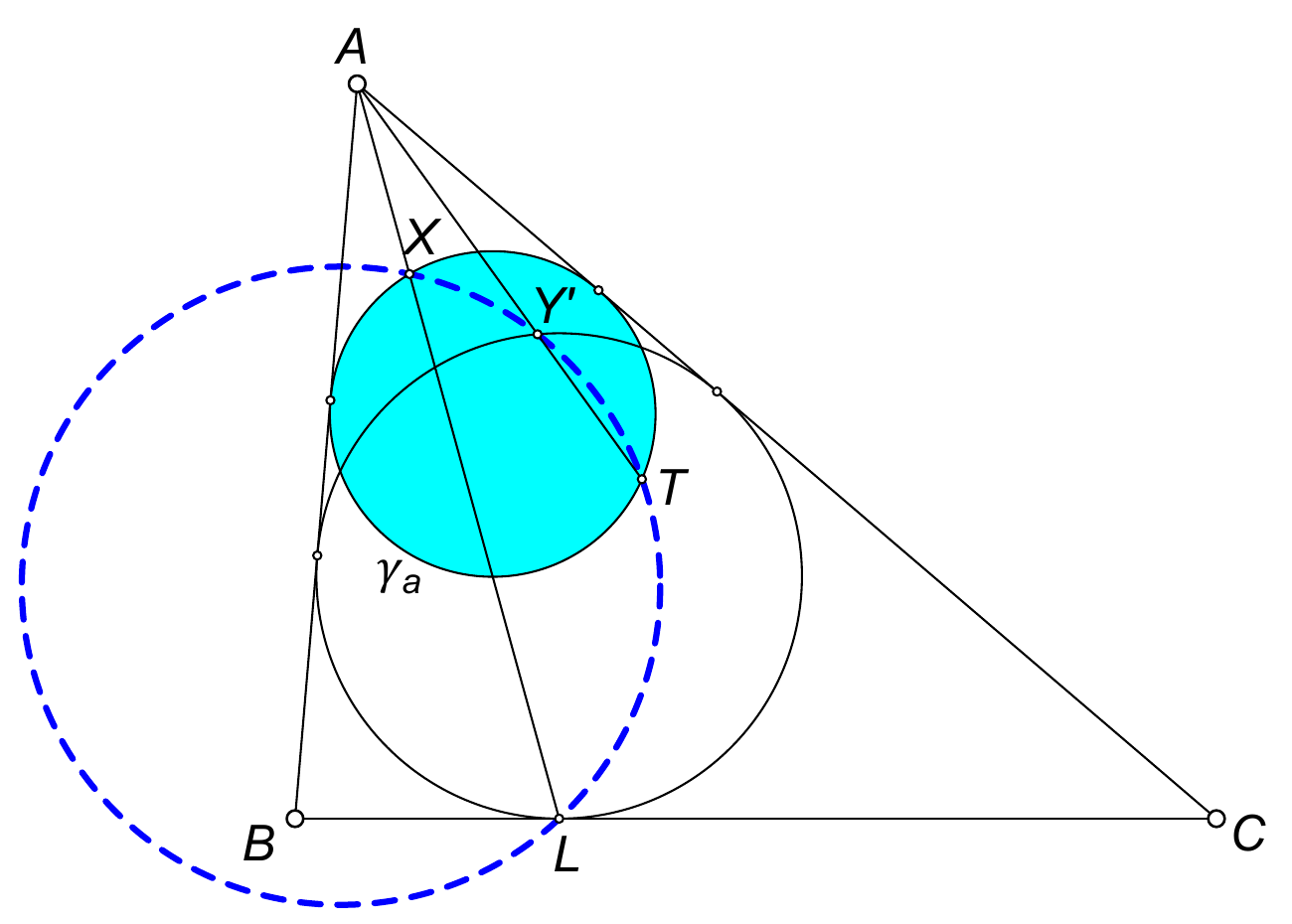}
\caption{four points lie on a circle}
\label{fig:result21}
\end{figure}

\begin{minipage}{2.9in}
\begin{proof}
Let $L'$ be the point (nearer $L$) where $AL$ meets $\gamma_a$
as shown in the figure to the right.
Lines $AXL'$ and $AYT$ are both secants to $\gamma_a$,
so $$AX\cdot AL'=AY\cdot AT.$$
Note that the incircle and circle $\gamma_a$ are homothetic with $A$ as the center of the homothety.
This homothety maps $L'$ to $L$ and maps $Y$ to $Y'$.
Therefore,
$$\frac{AL'}{AL}=\frac{AY}{AY'}.$$
Hence
$$\frac{AX}{AT}=\frac{AY}{AL'}=\frac{AY'}{AL}.$$

\bigskip
Thus $AX\cdot AL=AY'\cdot AT$ which implies that $X$, $Y'$, $T$, and $L$ lie on a circle.
\end{proof}
\end{minipage}
\raisebox{-1in}{\includegraphics[scale=0.4]{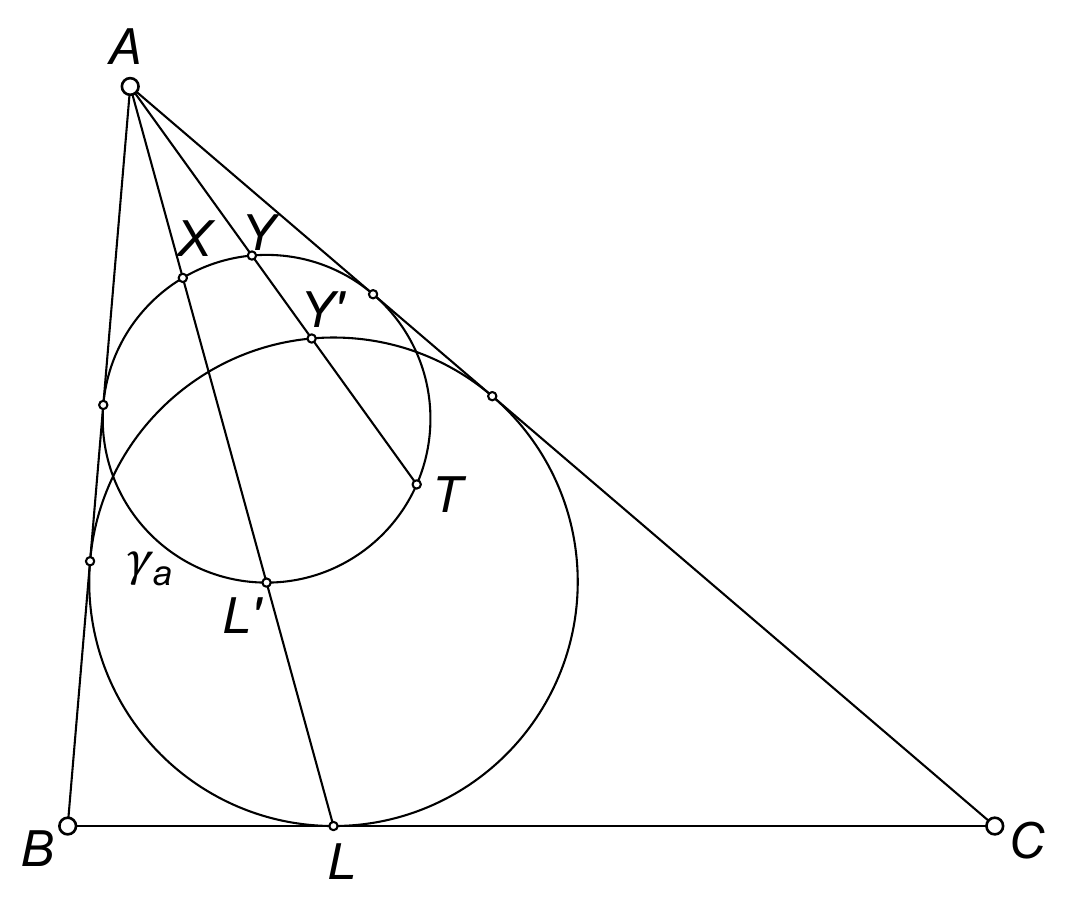}}


The following result comes from \cite{Suppa13249}.

\begin{theorem}\label{thm:result20}
We have $IT\perp TL'$ (Figure~\ref{fig:result20}).
\end{theorem}

\begin{figure}[ht]
\centering
\includegraphics[scale=0.6]{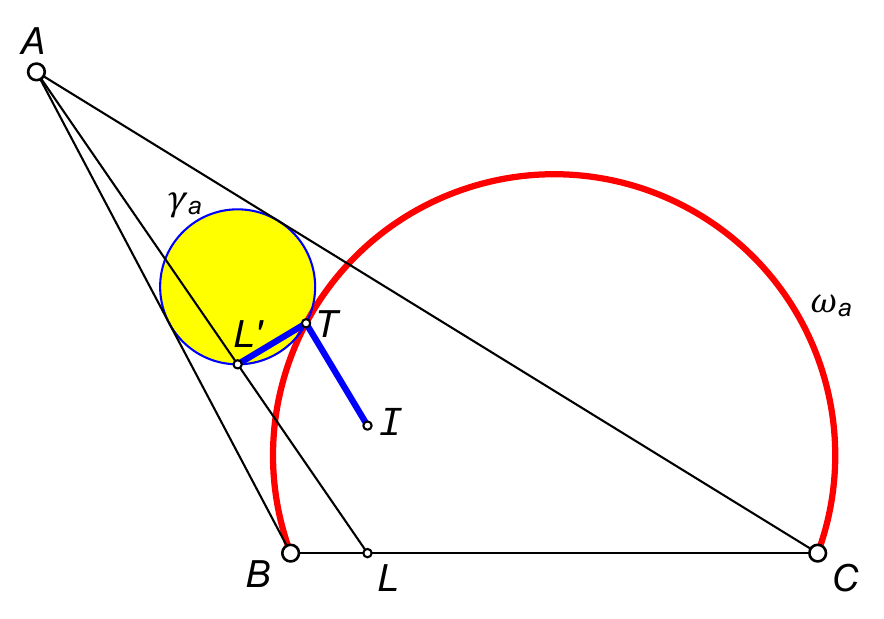}
\caption{blue lines are perpendicular}
\label{fig:result20}
\end{figure}

\begin{proof}
Line $TL'$ is the external angle bisector of $\angle BTC$ by
Theorem \ref{thm:extAngBisector}.
Line  $TI$ is the internal angle bisector of $\angle BTC$ by Protasov's Theorem.
Thus, $IT\perp TL'$.
\end{proof}


\begin{theorem}\label{thm:result22}
Let $\gamma_a$ be any circle inscribed in $\angle BAC$.
Let $AL$ meet $\gamma_a$ at $X$ (closer to $A$).
Let $\gamma_a$ touch $AC$ and $AB$ at $E$ and $F$, respectively.
Let $AI$ meet $EF$ at $G$.
Then $X$, $G$, $I$, and $L$ lie on a circle (Figure~\ref{fig:result22}).
\end{theorem}

\begin{figure}[ht]
\centering
\includegraphics[scale=0.6]{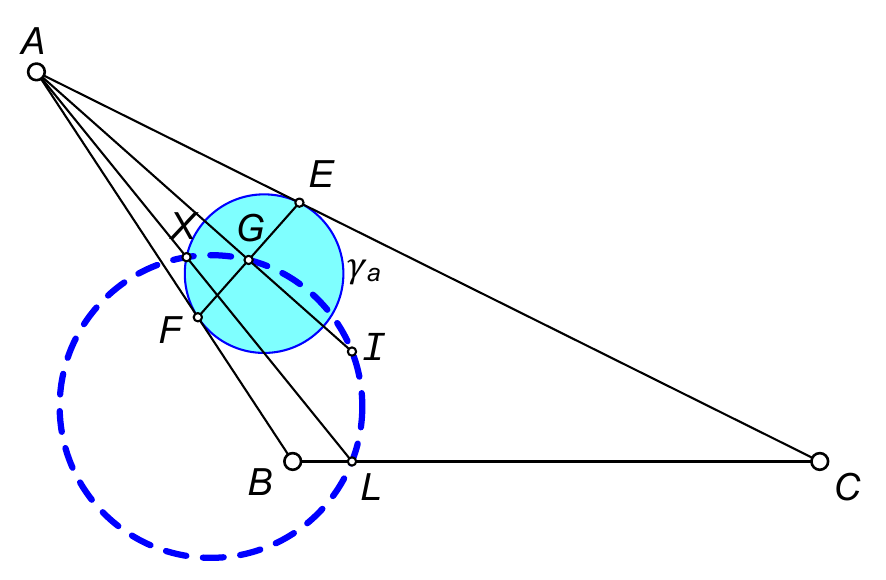}
\caption{four points lie on a circle}
\label{fig:result22}
\end{figure}

\begin{proof}
Let $XL$ meet $\gamma_a$ again at $L'$.
Line $AXL'$ is a secant to circle $\gamma_a$, and $AE$ a tangent.
So $AX\cdot AL'=(AE)^2$.

\begin{minipage}{3in}
Triangles $AGE$ and $AED$ are similar right triangles,
so $$\frac{AE}{AG}=\frac{AD}{AE}\quad\hbox{or}\quad(AE)^2=AD\cdot AG.$$
Thus,
$$AX\cdot AL'=AD\cdot AG\quad\hbox{or}\quad\frac{AX}{AG}=\frac{AD}{AL'}.$$
From Theorem~\ref{thm:result24}, $DL'\parallel IL$, so
$$\frac{AD}{AL'}=\frac{AI}{AL}.$$
\end{minipage}
\raisebox{-1in}{\includegraphics[scale=0.5]{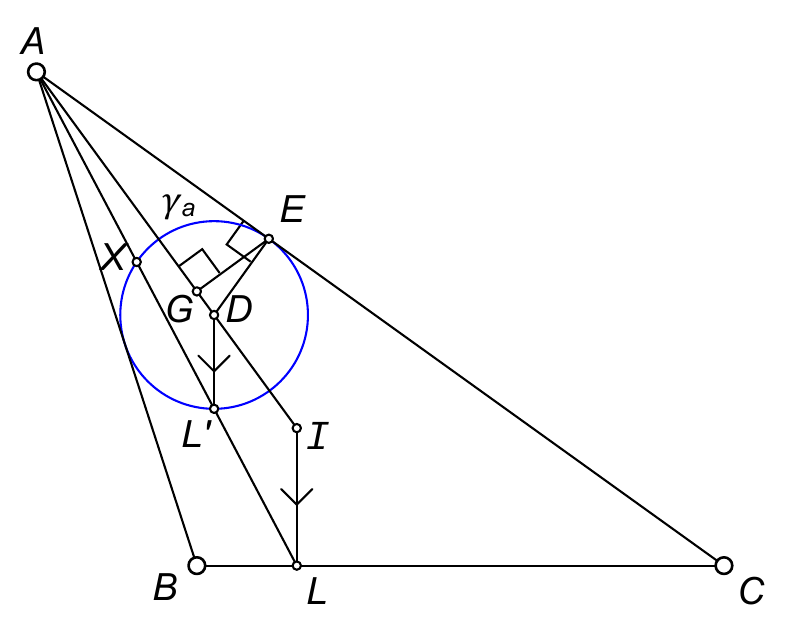}}

Therefore,
$$\frac{AX}{AG}=\frac{AI}{AL}\quad\hbox{or}\quad AX\cdot AL=AG\cdot AI.$$

This implies that $X$, $G$, $I$, and $L$ lie on a circle.
\end{proof}


\begin{theorem}\label{thm:result23}
Let $\gamma_a$ touch $AC$ and $AB$ at $E$ and $F$, respectively.
Let $AI$ meet $EF$ at $G$.
Let $EF$ meet $CB$ at $K$.
Then $X$, $G$, $Y'$, $T$, $I$, $L$, and $K$ lie on a circle
with diameter $KI$ (Figure~\ref{fig:result23}).
\end{theorem}

\begin{figure}[ht]
\centering
\includegraphics[scale=0.5]{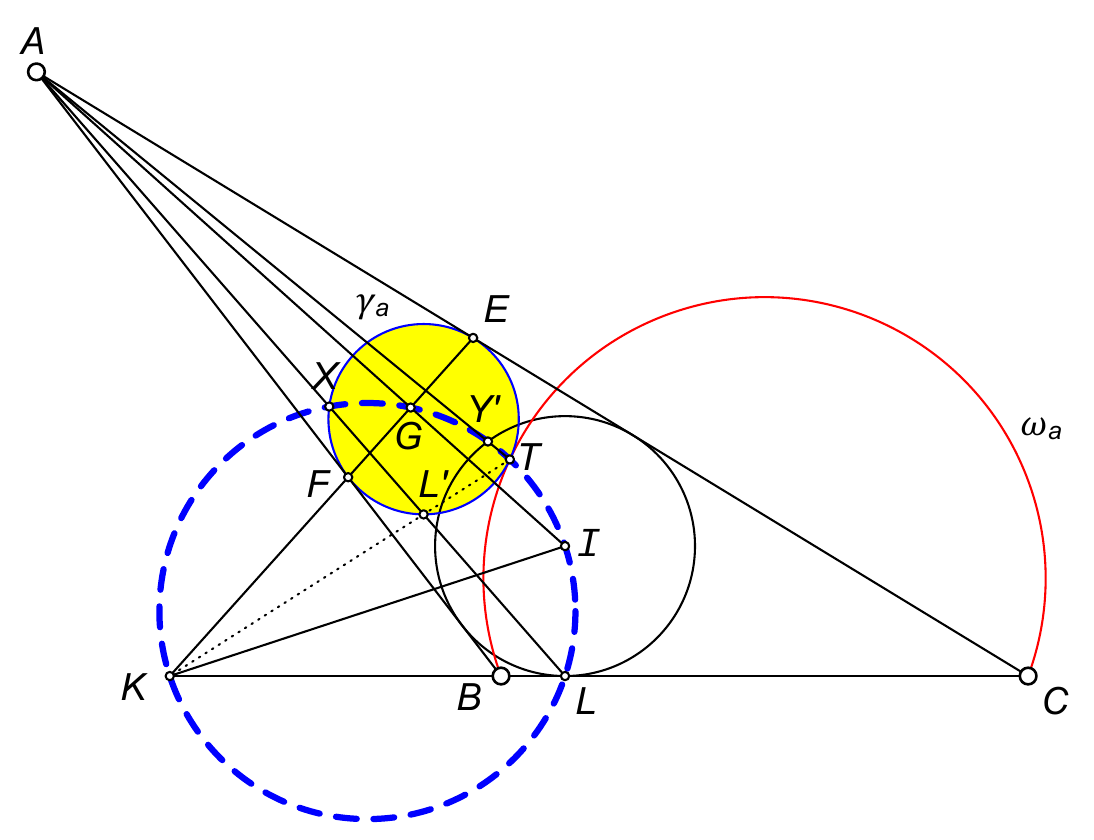}
\caption{seven points lie on a circle}
\label{fig:result23}
\end{figure}

\begin{proof}
From Theorem~\ref{thm:result18}, $TK$ passes through $L'$. But from Theorem~\ref{thm:result20}, $TL'\perp TI$,
so $\angle KTI$ is a right angle. This means that $T$ lies on the circle with diameter $KI$.
Since $\angle ILK$ is also a right angle, this means that $L$ is also on this circle.

From Theorem~\ref{thm:result17}, the circle through $T$, $L$, and $K$ also passes through $X$.


From Lemma~\ref{lemma:result21}, the circle through $X$, $T$, and $L$ passes through $Y'$.

Since $AI$ bisects $\angle BAC$, $G$ is the midpoint of $EF$ and $AG\perp EF$.
Hence $\angle KGI$ is a right angle.
Thus, $G$ lies on the circle with diameter $KI$.

Therefore, all seven points lie on the blue circle shown in Figure~\ref{fig:result23}.
\end{proof}

See also \cite{Chuong} for a proof that $X$, $G$, $T$, $I$, $L$, and $K$ lie on a circle.

\newpage

\begin{theorem}
We have $\angle ATD=\angle IXT$ (Figure~\ref{fig:result26}).
\end{theorem}

\begin{figure}[ht]
\centering
\includegraphics[scale=0.5]{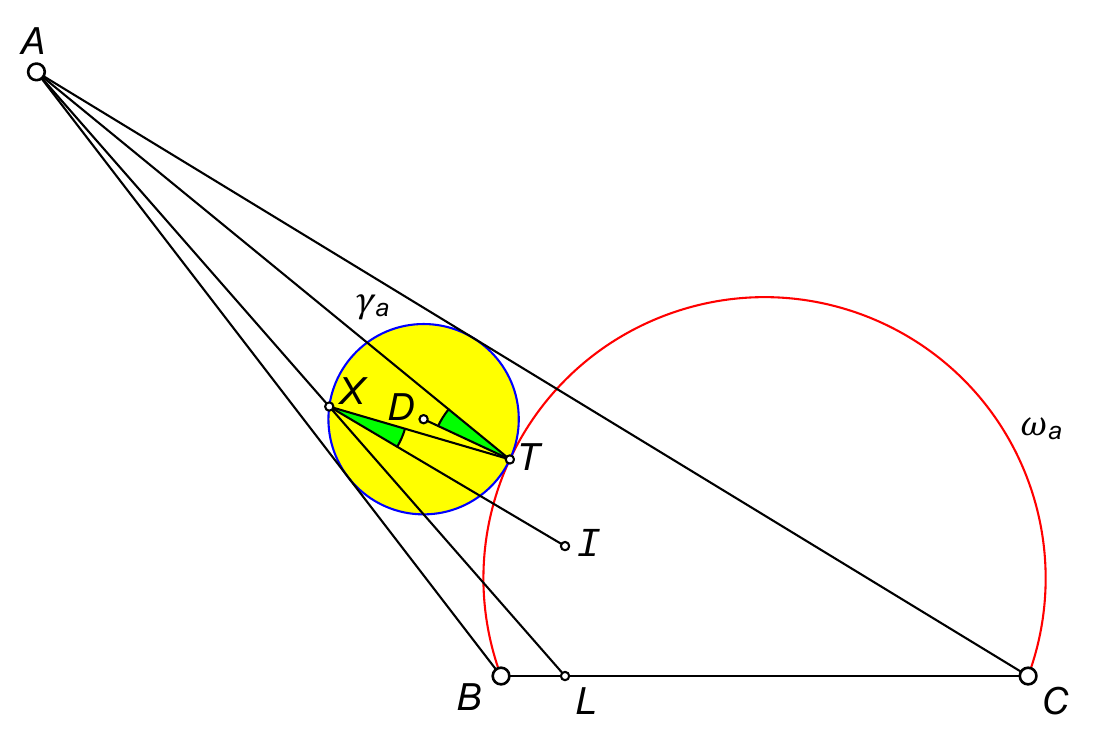}
\caption{green angles are equal}
\label{fig:result26}
\end{figure}

\begin{minipage}{2in}
\begin{proof}
By Theorem~\ref{thm:result23}, $X$, $T$, $I$, and $L$ lie on a circle
as shown in the figure to the right.
Then $$\angle IXT=\angle ILT$$ because both subtend arc $\arc{TI}$.
By Theorem~\ref{thm:result1},
$$\angle ILT=\angle ATD$$
where $D$ is the center of $\gamma_a$ as seen in Figure~\ref{fig:result26}.
Thus, $\angle ATD=\angle IXT$.
\end{proof}
\end{minipage}
\hspace{-0.22in}
\hfill
\raisebox{-1in}{\includegraphics[scale=0.55]{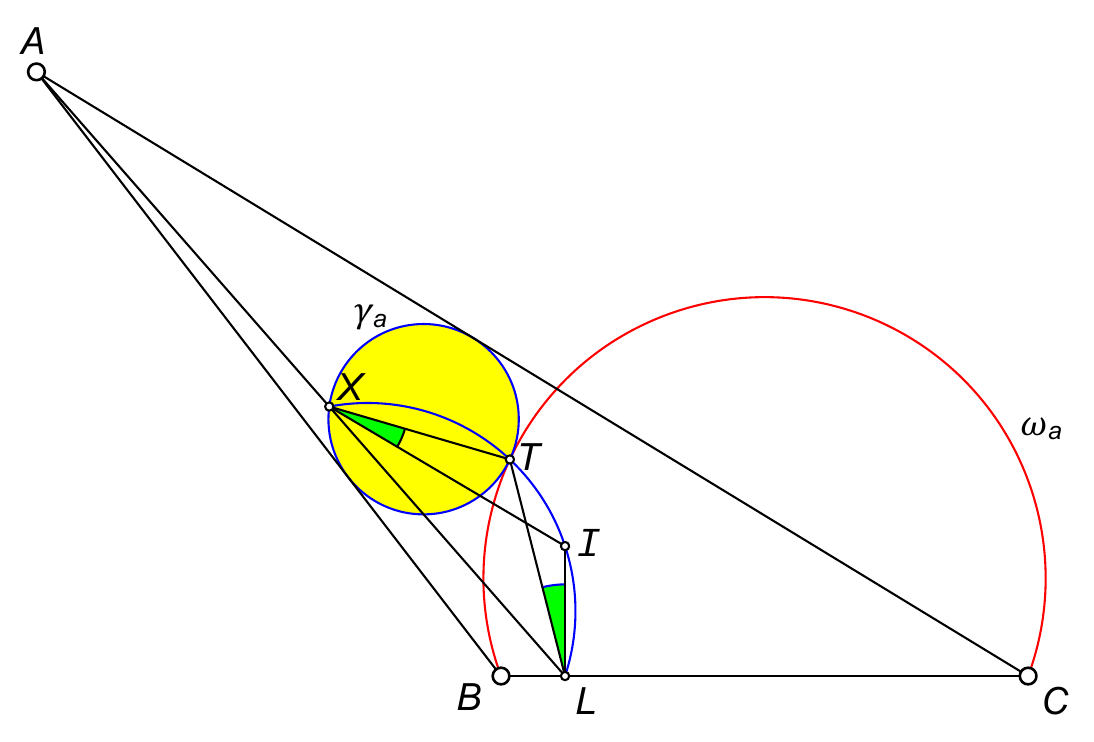}}


\section{Arcs with a Given Angular Measure}

We can generalize many of the results in \cite{RabSup} by replacing the
semicircles with arcs having the same angular measure.
Let $\omega_a$, $\omega_b$, and $\omega_c$ be arcs with the same angular measure $\theta$
erected internally on the sides of $\triangle ABC$ as shown in Figure~\ref{fig:genArcs}.

\begin{figure}[ht]
\centering
\includegraphics[scale=0.4]{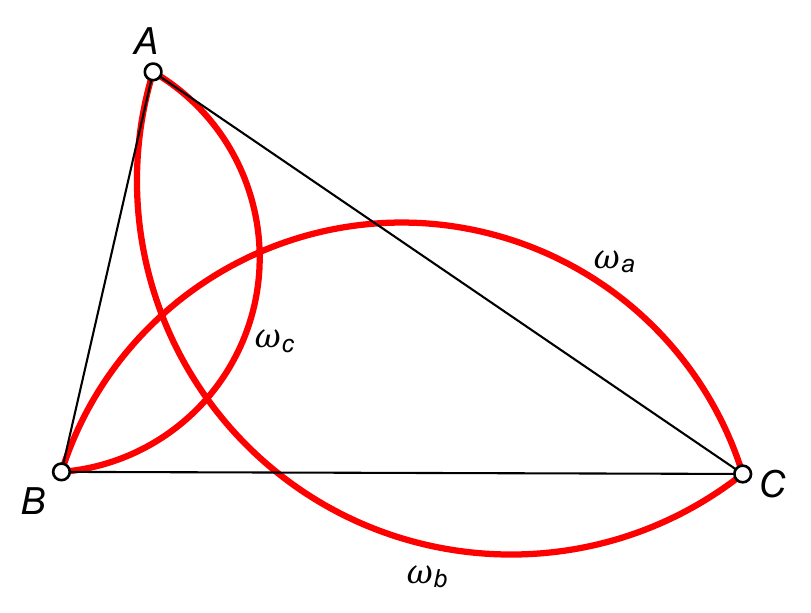}
\caption{arcs have same angular measure}
\label{fig:genArcs}
\end{figure}

Throughout this section, we will use the symbols shown in the following table.

\begin{center}
\begin{tabular}{|c|l|}
\hline
\multicolumn{2}{|c|}{\color{red} \textbf{Notation for this Section}}\\ \hline
\textbf{Symbol}&\textbf{Description} \\ \hline
$\theta$&angular measure of arc $\arc{BTC}$\\ \hline
$t$&$\tan(\theta/4)$\\ \hline
$\rho_a$&radius of $\gamma_a$\\ \hline
$\omega_a$&arc with angular measure $\theta$ that passes through $B$ and $C$\\ \hline
$R_a$&radius of $\omega_a$\\ \hline
$H$&point where incircle touches $AC$\\ \hline
$K$&point where $\gamma_a$ touches $AC$\\ \hline
$E$&point where $AC$ meets $\omega_a$ again\\ \hline
$O_a$&center of $\omega_a$\\ \hline
$L$&point where incircle touches $BC$\\ \hline
$L'$&point closer to $L$ where $AL$ meets $\gamma_a$\\ \hline
$X$&point closer to $A$ where $AL$ meets $\gamma_a$\\ \hline
\end{tabular}
\end{center}


\bigskip

\begin{theorem}\label{thm:thetaOver2}
Let the line through $I$ parallel to $BE$ meet $AC$ at $F$
(Figure~\ref{fig:IF}).
Then
$$\angle IFA=\frac{\theta}{2}.$$
\end{theorem}

\begin{figure}[ht]
\centering
\includegraphics[scale=0.5]{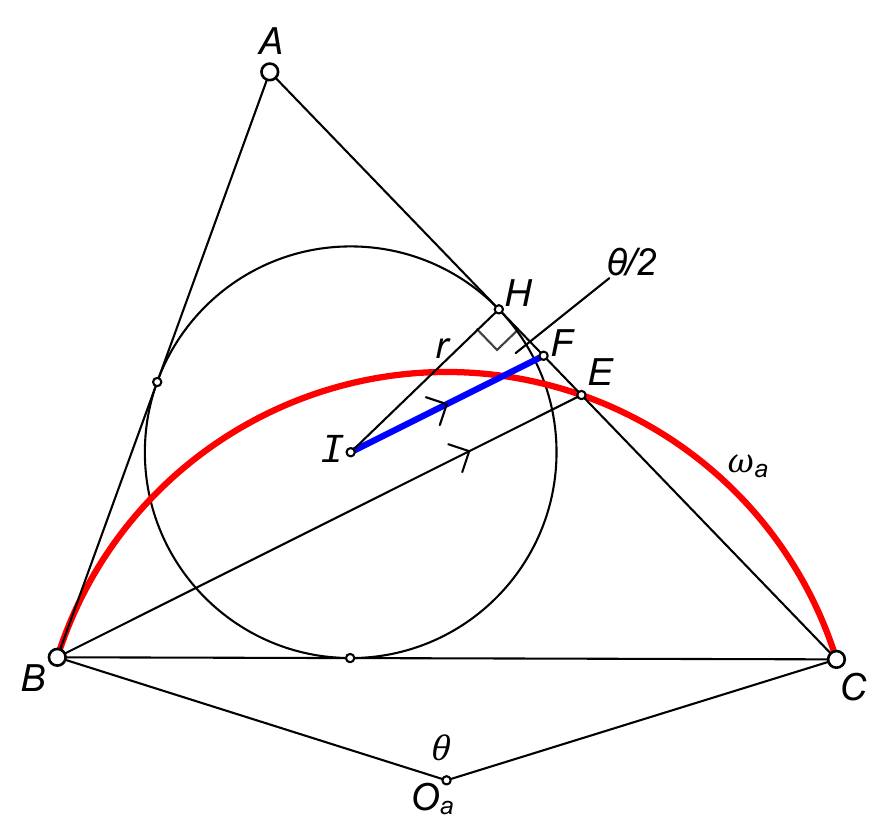}
\caption{}
\label{fig:IF}
\end{figure}

\begin{proof}
Since arc $\arc{BC}$ has measure $\theta$, $\angle CO_aB=\theta$ and the remaining arc on the circle $(O_a)$ outside the triangle
must have measure $360\degrees-\theta$. An inscribed angle is measured by half
its intercepted arc, so $\angle BEC=180\degrees-\theta/2$.
Consequently, $\angle AEB=\theta/2$ and since $BE\parallel IF$, we have $\angle AFI=\theta/2$.
\end{proof}

\newpage

\textbf{Note 1.} If $\theta<2C$, the figure looks different (Figure~\ref{fig:IF2}).
In this case, the arc (extended) meets $AC$ at a point $E$ such that $C$ lies between $A$ and $E$.
In this case, $\angle CEB$ is measured by half arc $\arc{BC}$ and
$BE\parallel IF$ implies that $\angle AFI=\theta/2$.

\begin{figure}[ht]
\centering
\includegraphics[scale=0.37]{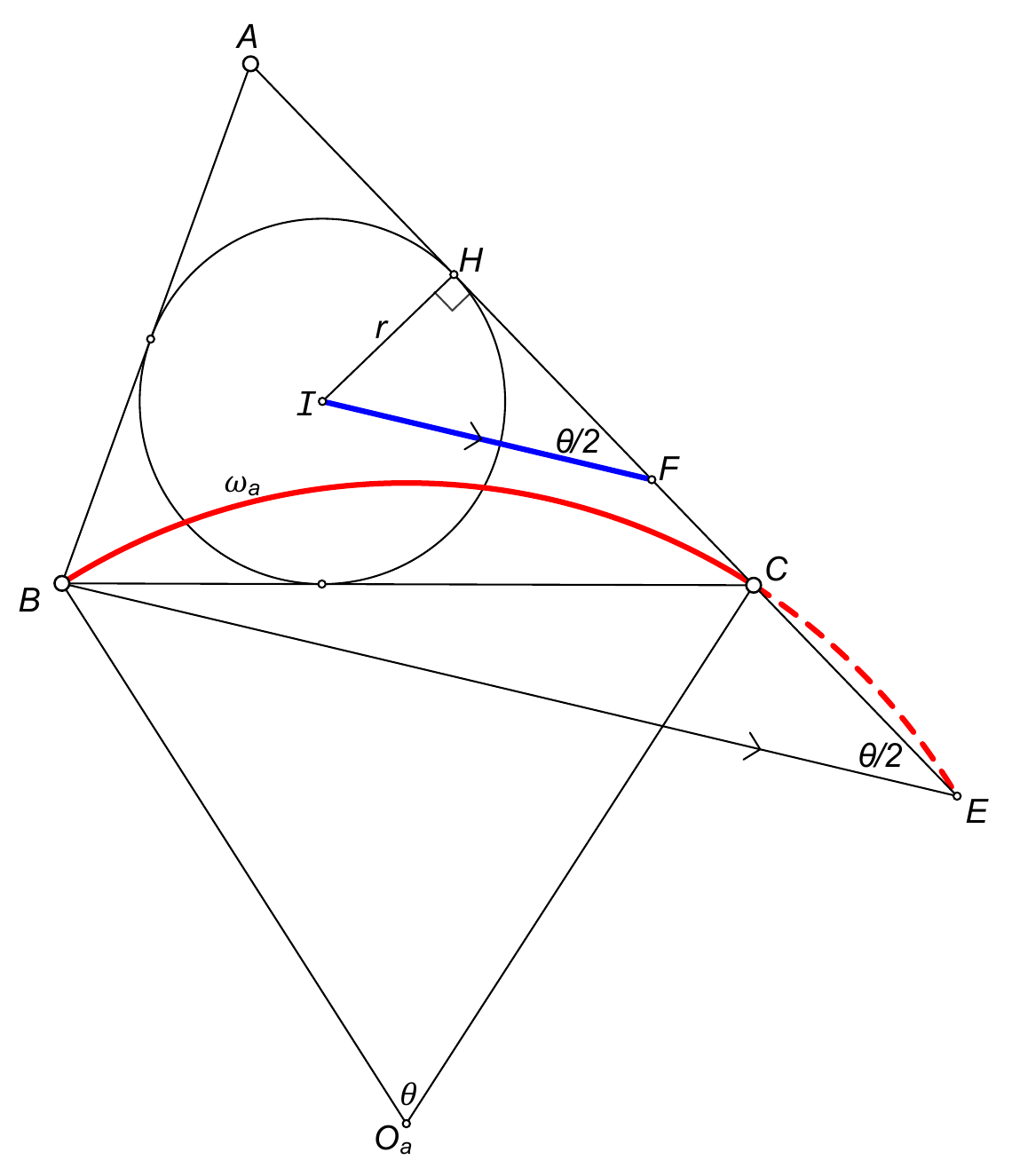}
\caption{Case $\theta<2C$}
\label{fig:IF2}
\end{figure}

\textbf{Note 2.} If $F$ lies between $A$ and $H$ or if $\theta>2(180\degrees-A)$ which causes
$A$ to lie between $E$ and $F$, the figure also looks different (Figure~\ref{fig:IF3}).
In this case, the arc meets $AC$ (possibly extended) at a point $E$ such that $F$
lies between $E$ and $H$.
In this case, the red arc has measure $\theta$ and the remaining arc (below $BC$)
has measure $360\degrees-\theta$. Then $\angle BEC$ is measured by half that arc
and so $\angle BEC=180\degrees-\theta/2$.
So $BE\parallel IF$ implies that $\angle AFI=\theta/2$.

\begin{figure}[ht]
\centering
\includegraphics[scale=0.4]{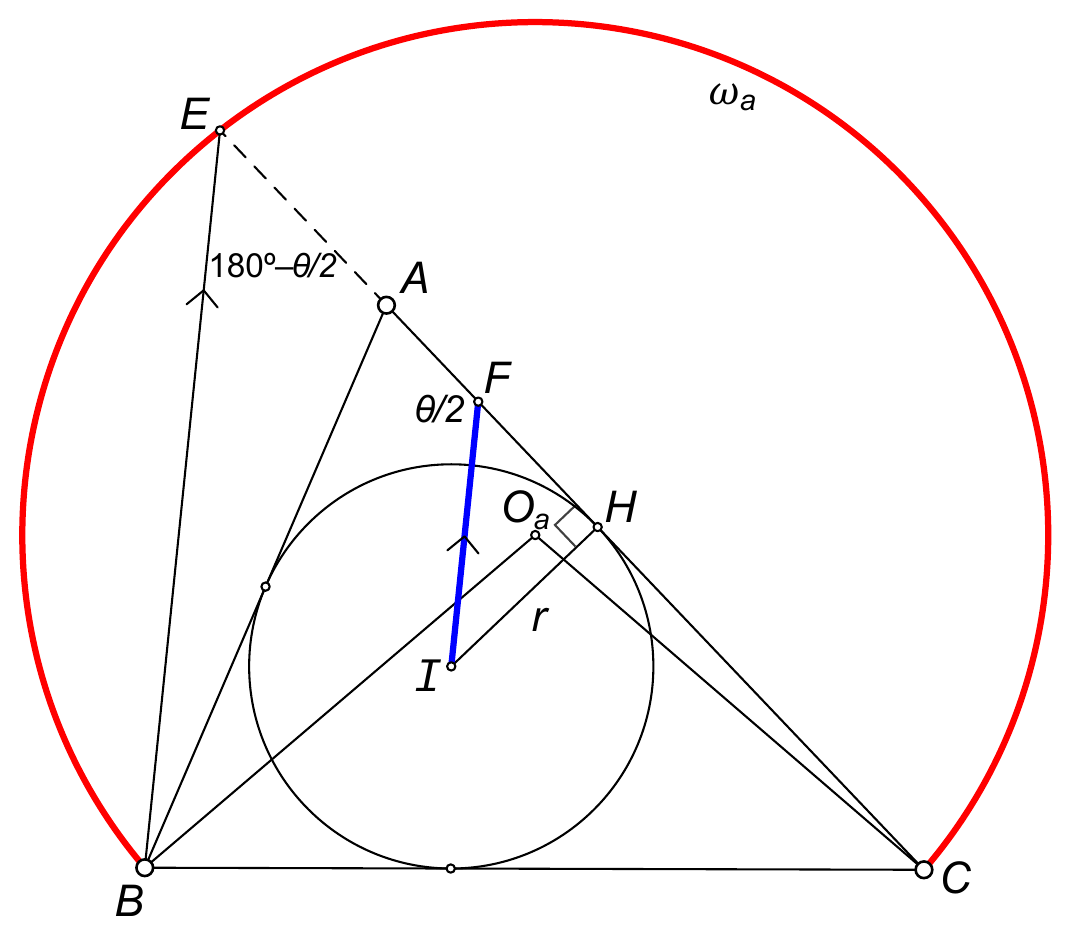}
\caption{Case $\theta>2(180\degrees-A)$}
\label{fig:IF3}
\end{figure}

\begin{corollary}\label{thm:IF}
Let the line through $I$ parallel to $BE$ meet $AC$ at $F$.
Then $$IF=r\csc\frac{\theta}{2}.$$
\end{corollary}

\begin{proof}
From right triangle $IHF$, we have
$$\sin\frac{\theta}{2}=\frac{IH}{IF}=\frac{r}{IF},$$
so $IF=r\csc\frac{\theta}{2}$.
\end{proof}

Let $\gamma_a$ be the circle inside $\triangle ABC$ tangent to sides $AB$ and $AC$
and also tangent to $\omega_a$.
The radii of circles $\gamma_a$, $\gamma_b$, and $\gamma_c$ are denoted by $\rho_a$, $\rho_b$, and $\rho_c$, respectively.

\begin{theorem}\label{thm:thetaOver4}
We have (Figure~\ref{fig:thetaOver4})
$$\angle HIK = \frac{\theta}{4}.$$
\end{theorem}

\begin{figure}[ht]
\centering
\includegraphics[scale=0.4]{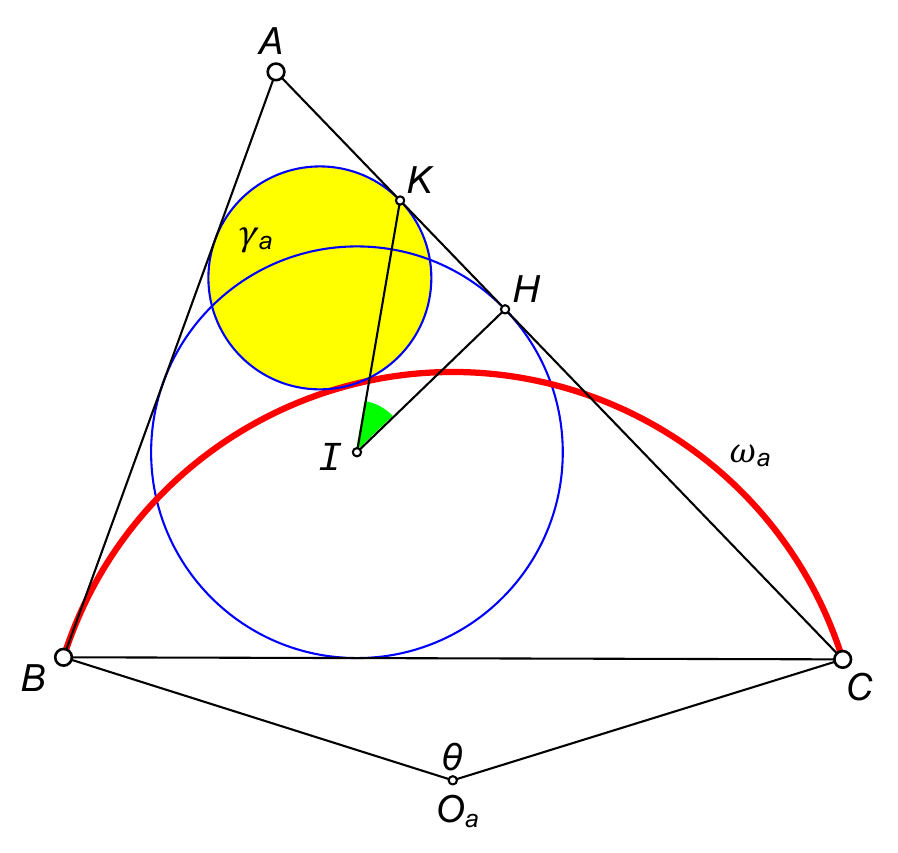}
\caption{green angle = $\theta/4$}
\label{fig:thetaOver4}
\end{figure}

\begin{proof}
A line through $I$ parallel to $BE$ meets $AC$ at $F$ (Figure~\ref{fig:thetaOver4proof}).
\begin{figure}[ht]
\centering
\includegraphics[scale=0.5]{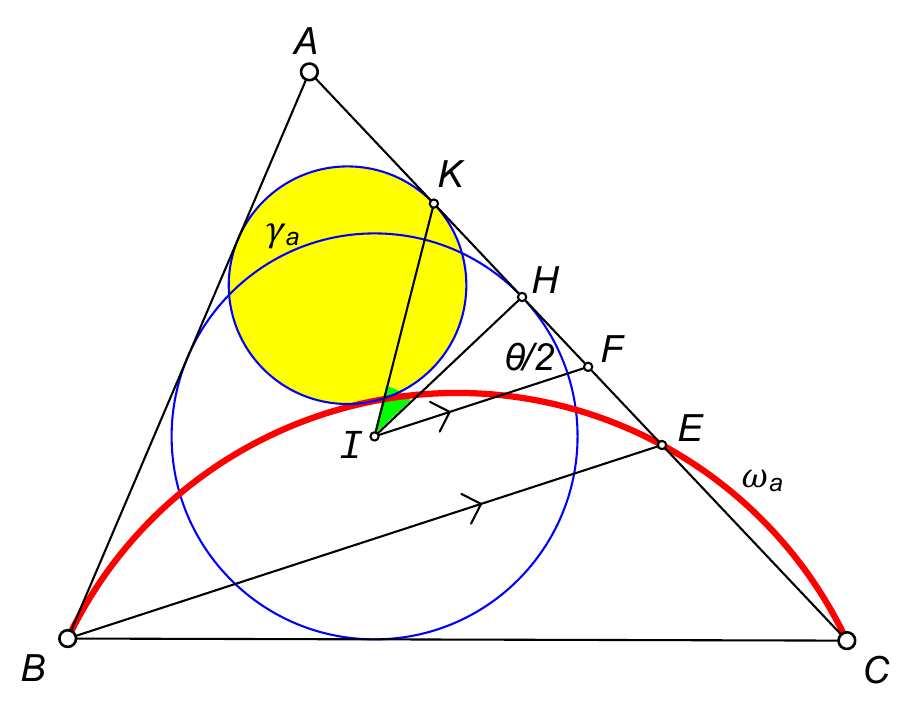}
\caption{}
\label{fig:thetaOver4proof}
\end{figure}

From Theorem~\ref{thm:thetaOver2}, $\angle AFI=\theta/2$.
From Theorem~\ref{thm:RG13066}, $IF=FK$, so $\triangle KFI$ is isosceles
and $\angle FIK=\angle IKF=(180\degrees-\theta/2)/2=90\degrees-\theta/4$.
From right triangle $FHI$, we see that $\angle FIH=90\degrees-\theta/2$.

Thus, $\angle HIK=\angle FIK-\angle FIH=(90\degrees-\theta/4)-(90\degrees-\theta/2)=\theta/4$.
\end{proof}

\newpage

\begin{theorem}\label{thm:result11}
We have $\angle DKI=\theta/4$ (Figure~\ref{fig:result11}).
\end{theorem}

\begin{figure}[ht]
\centering
\includegraphics[scale=0.5]{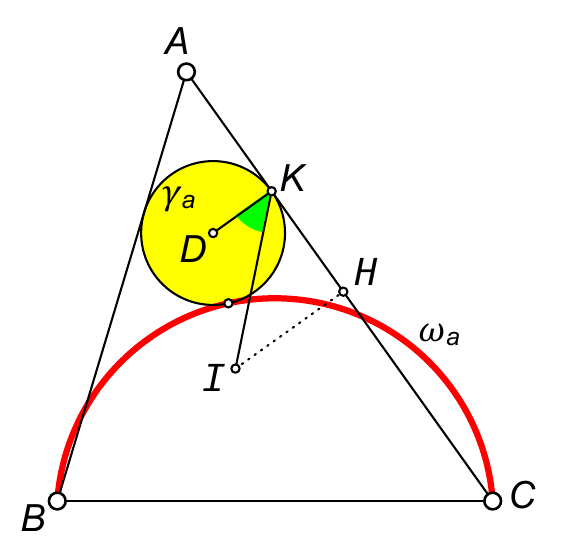}
\caption{$\angle DKI=\theta/4$}
\label{fig:result11}
\end{figure}

\begin{proof}
Since $DK$ and $IH$ are both perpendicular to $AC$, we have $DK\parallel IH$.
Thus $\angle DKI=\angle HIK=\theta/4$.
\end{proof}

\begin{corollary}\label{cor:KH}
The length of the common external tangent between $\rho_a$ and the incircle is $r\tan\frac{\theta}{4}$.
\end{corollary}

\begin{proof}
In Figure~\ref{fig:thetaOver4}, we see that $HK$ is the common external tangent
between $\rho_a$ and the incircle. Since $I$ is the incenter, $IH=r$.
From right triangle $IHK$, we see that 
$$\tan \angle HIK=\tan\frac{\theta}{4}=\frac{HK}{IH}=\frac{HK}{r}$$
and the result follows.
\end{proof}

\textbf{Note.}
If the arc $\omega_a$ gets large enough, point $A$ will lie inside $\omega_a$ and
circle $\gamma_a$ will not exist. However, if we expand the definition of $\gamma_a$ in that case
so that it refers to the circle outside $\triangle ABC$, tangent to sides $AB$ and $AC$ extended,
and tangent \emph{internally} to $\omega_a$ as shown in Figure~\ref{fig:thetaOver4a}, then
Theorem~\ref{thm:thetaOver4} still holds.

\begin{figure}[ht]
\centering
\includegraphics[scale=0.5]{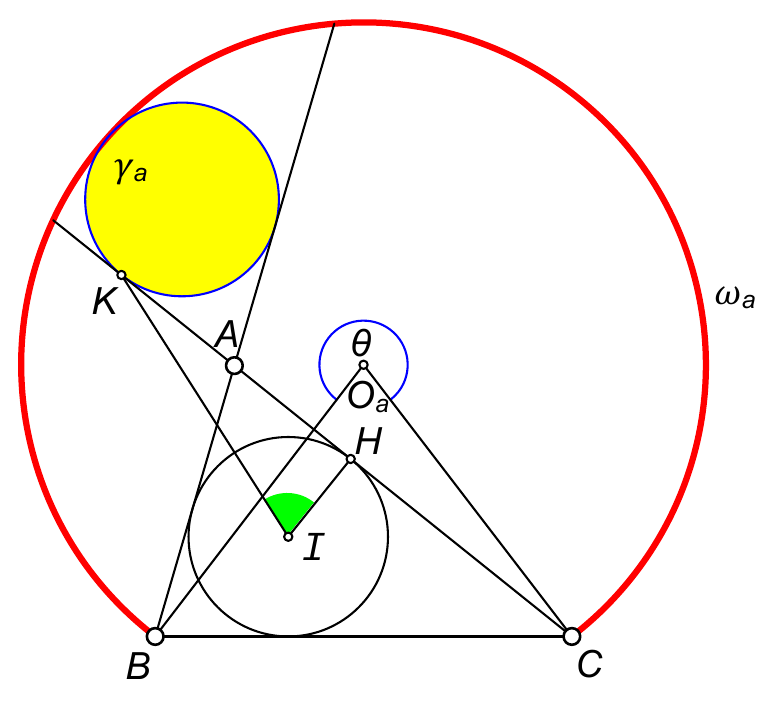}
\caption{green angle = $\theta/4$}
\label{fig:thetaOver4a}
\end{figure}

\newpage

\begin{theorem}[Ajima's Theorem]\label{thm:rho1}
We have
\begin{equation}
\rho_a=r\left(1-\tan\frac{A}{2}\tan\frac{\theta}{4}\right).\label{eq:rho1}
\end{equation}
\end{theorem}

\begin{figure}[ht]
\centering
\includegraphics[scale=0.5]{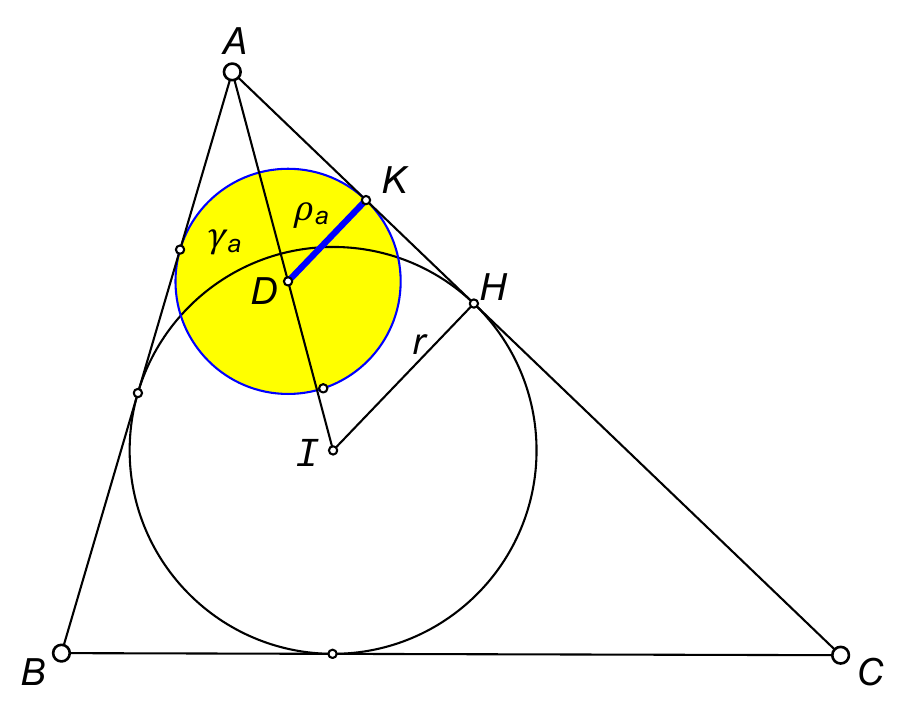}
\caption{}
\label{fig:rho}
\end{figure}

\begin{proof}
See Figure~\ref{fig:rho}.
By Corollary~\ref{cor:KH},
$$HK=r\tan\frac{\theta}{4}.$$
In right triangle $AIH$, we have $\angle IAH=A/2$, so
$$AH=r\cot\frac{A}{2}.$$
Therefore,
$$AK=AH-HK=r\cot\frac{A}{2}-r\tan\frac{\theta}{4}.$$
From right triangle $AKD$ with $DK=\rho_a$, we have
$$\rho_a=AK\tan\frac{A}{2}=\left(r\cot\frac{A}{2}-r\tan\frac{\theta}{4}\right)\tan\frac{A}{2}$$
which is the desired result by the identity $\tan x\cot x=1$.
\end{proof}

The wasan geometer Naonobu Ajima
found this result in 1781 (see \cite[p.~32]{Fukagawa-Rigby}).
More info about Ajima's Theorem can be found in \cite[pp.~96--97]{Fukagawa-Rigby}.
It has been said \cite[p.~103]{Fukagawa-Pedoe} that this result is of great importance because
it is used in the solution of many Japanese temple geometry problems.

\begin{corollary}\label{cor:rho}
We have
$$\rho_a=r\left(1-\frac{rt}{p-a}\right).$$
\end{corollary}

\begin{proof}
This follows immediately from the well-known fact \cite{WikiIncircle} that in Figure~\ref{fig:thetaOver4},
$AH=p-a$, so $\tan(A/2)=r/(p-a)$. 
\end{proof}

\begin{corollary}\label{cor:rho1}
We have
$$\rho_a=\frac{\Delta-(p-b)(p-c)t}{p}=r-\frac{(p-b)(p-c)t}{p}.$$
\end{corollary}

\begin{proof}
We use the well-known formulas $r=\frac{\Delta}{p}$ and $\Delta=\sqrt{(p(p-a)(p-b)(p-c)}$.
From Corollary~\ref{cor:rho}, we have
\begin{align*}
\rho_a&=r-\frac{r^2t}{p-a}\\
&=r-\left(\frac{\Delta^2}{p^2}\right)\frac{t}{p-a}\\
&=r-\left(\frac{p(p-a)(p-b)(p-c)}{p^2}\right)\frac{t}{p-a}\\
&=r-\frac{(p-b)(p-c)t}{p}\\
&=\frac{\Delta-(p-b)(p-c)t}{p}.\qedhere
\end{align*}
\end{proof}

Theorem~\ref{thm:rho1} remains true, with a sign change, if we allow the extended position for $\gamma_a$.

\begin{theorem}\label{thm:rho1-a}
Let $\omega_a$ be an arc of a circle with angular measure $\theta$
that passes through points $B$ and $C$ of $\triangle ABC$.
Suppose $\theta>2(180\degrees-A)$ so that $A$ lies inside $\omega_a$.
Let $\gamma_a$ be the circle outside the triangle tangent to sides $AB$ and $AC$ extended
and also internally tangent to $\omega_a$ as shown in Figure~\ref{fig:rho-a}.
Let $\rho_a$ be the radius of $\gamma_a$.
Then
$$\rho_a=-r\left(1-\tan\frac{A}{2}\tan\frac{\theta}{4}\right).$$
\end{theorem}

\begin{figure}[ht]
\centering
\includegraphics[scale=0.45]{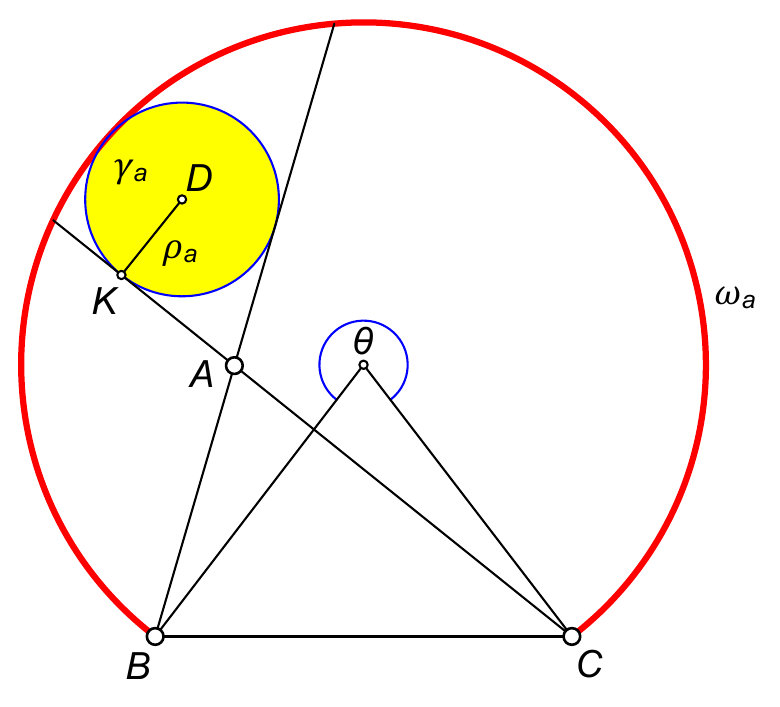}  
\caption{}
\label{fig:rho-a}
\end{figure}

In all cases, we could say that
$$\rho_a=r\left|1-\tan\frac{A}{2}\tan\frac{\theta}{4}\right|.$$

A unifying discussion about circles tangent to arcs with a given angular measure can be found
in \cite{Protassov2}.

\newpage

\begin{lemma}\label{lemma:sin2x}
For any $x$,
$$\sin 2x=\frac{2\tan x}{\tan^2 x+1}.$$
\end{lemma}

\begin{proof}
We have
\begin{equation*}
\frac{2\tan x}{\tan^2 x+1}=\frac{2\tan x}{\sec^2 x}=\frac{2(\sin x)/(\cos x)}{1/\cos^2 x}
=2\sin x\cos x=\sin 2x.\qedhere
\end{equation*}
\end{proof}


\begin{theorem}[Radius of $\omega_a$]\label{thm:Ra}
We have
$$R_a=\frac{a}{2}\csc\frac{\theta}{2}.$$
\end{theorem}

\begin{proof}
Let $M$ be the foot of the perpendicular
from $O_a$ to $BC$ (Figure~\ref{fig:Ra}).

\begin{figure}[ht]
\centering
\includegraphics[scale=0.7]{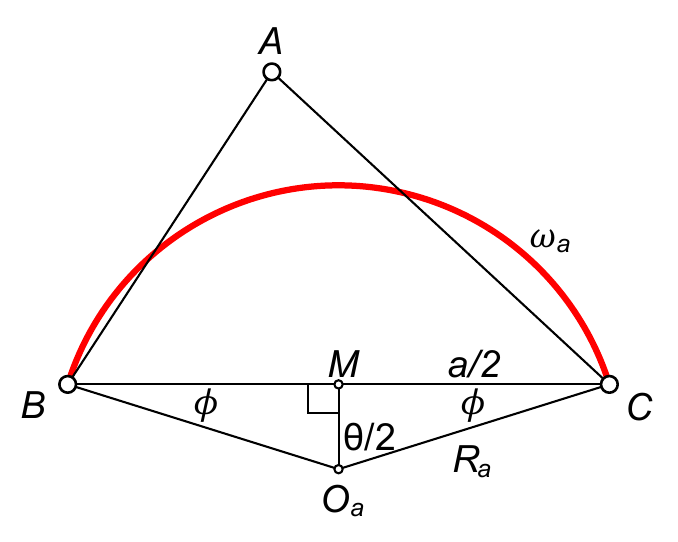}
\caption{}
\label{fig:Ra}
\end{figure}

Then $O_aC=R_a$ and $MC=a/2$.
We have $\angle CO_aB=\theta$ since the angular measure of the arc is $\theta$.
Thus $\angle CO_aM=\theta/2$ and hence $\sin(\theta/2)=(a/2)/R_a$ and the result follows.
\end{proof}

\begin{corollary}\label{thm:Ra2}
We have
\begin{equation}
R_a=\frac{a(t^2+1)}{4t}\label{eq:Ra2}.
\end{equation}
\end{corollary}

\begin{proof}
This follows from Lemma~\ref{lemma:sin2x}.
\end{proof}

\begin{corollary}\label{thm:Ra3}
We have
\begin{equation}
R_a=\frac{R(t^2+1)\sin A}{2t}\label{eq:Ra3}.
\end{equation}
\end{corollary}

\begin{proof}
From the Extended Law of Sines, we have $a/\sin A=2R$.
Substituting $a=2R\sin A$ into equation~(\ref{eq:Ra2}) gives the desired result.
\end{proof}

\newpage

\begin{theorem}\label{thm:AL'L}
We have (Figure~\ref{fig:AL'L})
$$\frac{AL'}{AL}=\frac{\rho_a}{r}.$$
\end{theorem}

\begin{figure}[ht]
\centering
\includegraphics[scale=0.5]{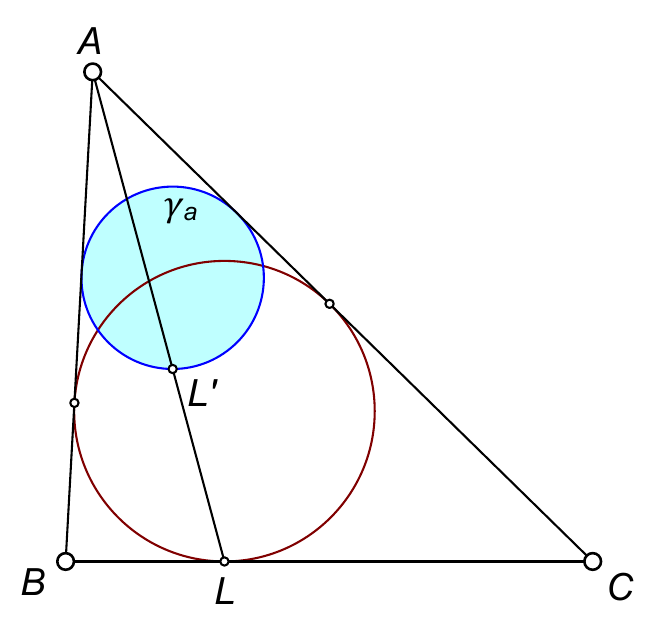}
\caption{}
\label{fig:AL'L}
\end{figure}

\begin{proof}
This follows from the fact that $L$ and $L'$ are corresponding points
in the homothety with center $A$ that maps $\gamma_a$ into the incircle.
\end{proof}

\begin{theorem}[Length of $AL'$]\label{thm:lengthOfAL'}
We have (Figure~\ref{fig:lengthOfAL'})
$$AL'=\left(1-\frac{rt}{p-a}\right)\sqrt{\frac{(p-a)[ap-(b-c)^2]}{a}}.$$
\end{theorem}

\begin{figure}[ht]
\centering
\includegraphics[scale=0.6]{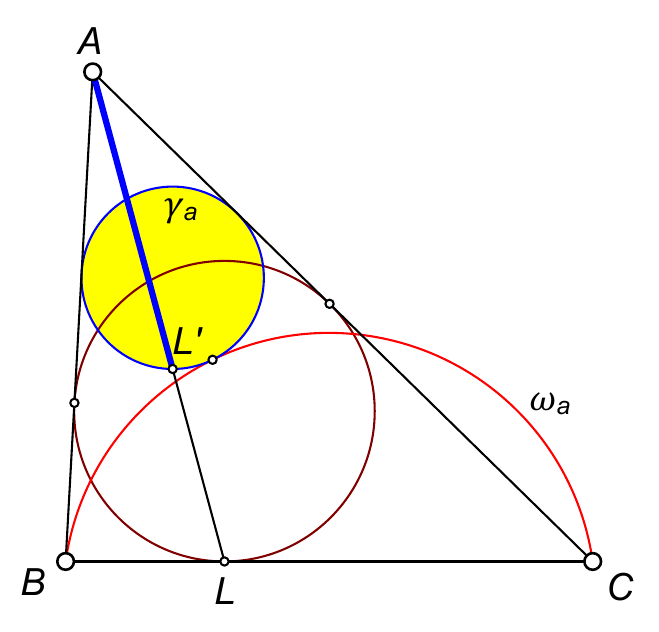}
\caption{}
\label{fig:lengthOfAL'}
\end{figure}

\begin{proof}
Since $L$ is the point where the incircle touches $BC$, $AL$ is a Gergonne cevian of $\triangle ABC$.
The length of a Gergonne cevian is known. From Property 3.1.3 in \cite{GergonneFormulary}, we have
\begin{equation}
AL=\sqrt{\frac{(p-a)[ap-(b-c)^2]}{a}}.\label{eq:AL}
\end{equation}
Circle $\gamma_a$ and the incircle are homothetic, with $A$ being the center of the homothety.
Since $L'$ and $L$ are corresponding points of the homothety, we have
$$\frac{AL'}{AL}=\frac{\rho_a}{r}.$$
Thus, $AL'=(\rho_a/r)\cdot AL$.
Combining this with the value of $\rho_a/r$ from Corollary~\ref{cor:rho} gives us our result.
\end{proof}


\begin{lemma}\label{lemma:tau}
We have $AK=p-a-rt$.
\end{lemma}

\begin{figure}[ht]
\centering
\includegraphics[scale=0.5]{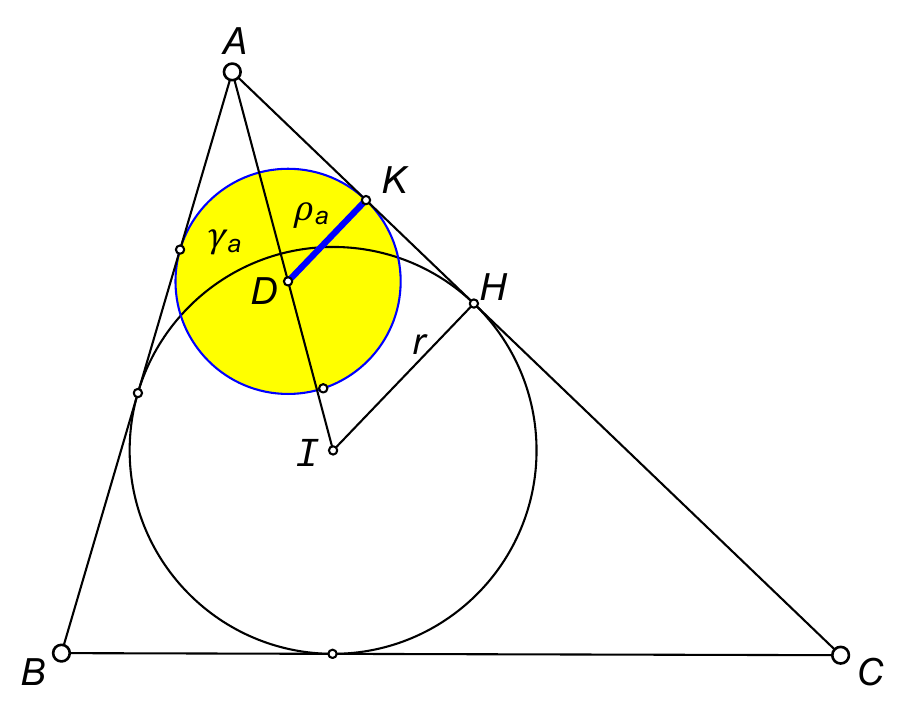}
\caption{}
\label{fig:rhodup}
\end{figure}

\begin{proof}
It is well known that $AH=p-a$ (Figure~\ref{fig:rhodup}).
From Corollary~\ref{cor:KH}, we have $HK=rt$.
Thus, $AK=AH-HK=p-a-rt$.
\end{proof}


\begin{theorem}[Length of $AX$]\label{thm:lengthOfAX}
We have (Figure~\ref{fig:lengthOfAX}).
$$AX=\frac{(p-a-rt)\sqrt{a(p-a)}}{\sqrt{ap-(b-c)^2}}.$$
\end{theorem}

\begin{figure}[ht]
\centering
\includegraphics[scale=0.5]{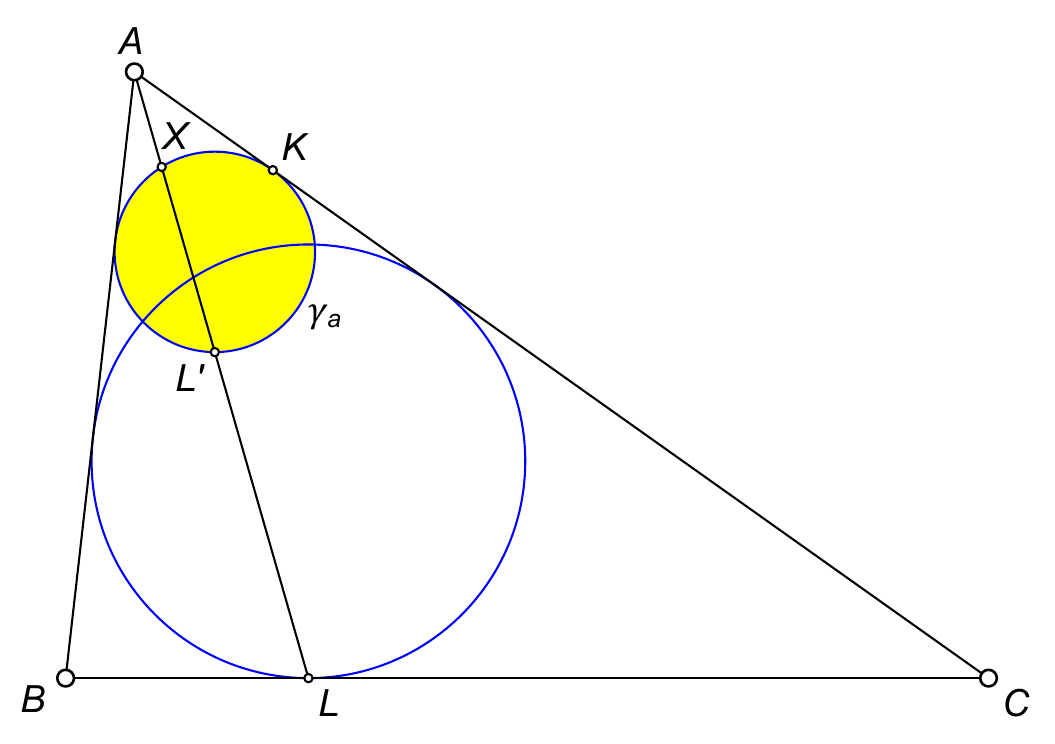}
\caption{}
\label{fig:lengthOfAX}
\end{figure}

\begin{proof}
Since $AXL'$ is a secant to $\gamma_a$ and $AK$ is a tangent, we have $AX\cdot AL'=(AK)^2$.
From Lemma~\ref{lemma:tau},
$$AK=p-a-rt.$$
From Theorem~\ref{thm:lengthOfAL'}, we have
$$AL'=\left(\frac{p-a-rt}{p-a}\right)\sqrt{\frac{(p-a)[ap-(b-c)^2]}{a}}.$$
So
\begin{align*}
AX&=\frac{(AK)^2}{AL'}\\
&=\frac{(p-a-rt)(p-a)}{\sqrt{\frac{(p-a)[ap-(b-c)^2]}{a}}}\\
&=\frac{(p-a-rt)\sqrt{a(p-a)}}{\sqrt{ap-(b-c)^2}}.\qedhere
\end{align*}
\end{proof}

\begin{corollary}\label{cor:AX/AL'}
We have
$$\frac{AX}{AL'}=\frac{a(p-a)}{ap-(b-c)^2}.$$
\end{corollary}

\newpage

\section{Barycentric Coordinates}\label{section:bary}

In this section, we will find the barycentric coordinates for various points associated with
our configuration.

\begin{theorem}[Coordinates for $D$]\label{thm:baryD}
The barycentric coordinates for $D$ are
$$D=\Bigl(ap(p-a)+(b+c)t\Delta):bp(p-a)-bt\Delta:cp(p-a)-ct\Delta\Bigr)$$
where $\Delta$ denotes the area of $\triangle ABC$, $p$ denotes the semiperimeter,
and $t=\tan\frac{\theta}{4}$.
\end{theorem}

\begin{proof}
Let $y$ be the distance between $D$ and $BC$.
Summing the areas of triangles $DBC$, $DCA$ and $DAB$ we obtain
$$ay+b\rho_a +c\rho_a=2\Delta.$$
Thus,
$$ay=2\Delta-(b+c)\rho_a.$$
Letting $[XYZ]$ denote the area of $\triangle XYZ$, we find that
the barycentric coordinates for $D$ are therefore
\begin{align*}
   D&=\Bigl([DBC]:[DCA]:[DAB]\Bigr)=\left(ay:b\rho_a:c\rho_a\right)\\
	  &=\left(2\Delta-(b+c)\rho_a:b\rho_a:c\rho_a\right)\\
	  &=\left(\frac{2\Delta}{\rho_a}-(b+c):b:c\right).
\end{align*}
Replacing $\rho_a$ by its value given by Corollary~\ref{cor:rho}, we get
\begin{align*}
D&=\left(\frac{2\Delta}{r\left(1-\frac{rt}{p-a}\right)}-(b+c):b:c\right)\\
&=\left(\frac{2(p-a)\Delta}{r\left(p-a-rt\right)}-(b+c):b:c\right)\\
&=\Bigl((b+c-a)\Delta-(b+c)r(p-a-rt):br(p-a-rt):cr(p-a-rt)\Bigr).
\end{align*}
Replacing $r$ by $\Delta/p$, then multiplying all coordinates by $p^2/\Delta$ gives
$$D=\Bigl(ap(b+c-p)+(b+c)t\Delta:bp(p-a)-bt\Delta:cp(p-a)-ct\Delta\Bigr).$$
Finally, noting that $b+c-p=p-a$, gives the desired result.
\end{proof}

\begin{theorem}[Coordinates for $O_a$]\label{thm:baryOa}
The barycentric coordinates for $O_a$ are
$$O_a=\left(-a^2:S_c+S\cot\phi:S_b+S\cot\phi\right).$$
where $\phi=90\degrees-\theta/2$, $S=2\Delta$,
$S_b=(c^2+a^2-b^2)/2$, and $S_c=(a^2+b^2-c^2)/2$.

\end{theorem}

\begin{proof}
The result follows from Conway's Formula \cite[p.~34]{Yiu}.
\end{proof}

\newpage

\begin{theorem}[Coordinates for $T$]\label{thm:baryT}
The barycentric coordinates for $T$ are\\$(T_x:T_y:T_z)$ where
\begin{equation*}
\begin{aligned}
T_x&=2 a \sin \frac{\theta }{4} \left(au \cos\frac{\theta }{2} +(b+c)
   u+2 a S \sin \frac{\theta }{2}\right),\\
T_y&=- u
   \left(2\left(a^2-bc-c^2\right) \cos \frac{\theta }{2} +a^2+2 ab-(b+c)^2\right)\sin \frac{\theta}{4}\\
   &+2 S\left(a^2+bc-c^2\right) \cos \frac{3 \theta }{4} +2 b S (2 a+b-c)\cos \frac{\theta}{4} ,\\
T_z&=-u
   \left(2\left(a^2-bc-b^2\right) \cos \frac{\theta}{2} +a^2+2 a c-(b+c)^2\right)\sin \frac{\theta }{4} \\
   &+2 S\left(a^2+bc-b^2\right) \cos \frac{3 \theta }{4} +2 c S(2 a-b+c) \cos \frac{\theta }{4},\\
\end{aligned}
\end{equation*}
where $S=2\Delta$ and $u=a^2-(b-c)^2$.
\end{theorem}

\begin{proof}
The barycentric coordinates for $D$ were found in Theorem~\ref{thm:baryD}. This can
be simplified to $D=(D_x:D_y:D_z)$ where
\begin{equation*}
\begin{aligned}
D_x&=a^3-a (b+c)^2-2 S (b+c) t,\\
D_y&=-b \left(-a^2+b^2+2 b c+c^2-2 S t\right),\\
D_z&=-c \left(-a^2+b^2+2 b c+c^2-2 S t\right)
\end{aligned}
\end{equation*}
by using the substitutions $r=S/(a+b+c)$, $p=(a+b+c)/2$, and $\Delta=S/2$.

The barycentric coordinates for $O_a$ were found in Theorem~\ref{thm:baryOa}, namely
$$O_a=\left(-a^2:S_c+S\cot\phi:S_b+S\cot\phi\right)$$
where $\phi=90\degrees-\theta/2$,
$S_b=(c^2+a^2-b^2)/2$, and $S_c=(a^2+b^2-c^2)/2$.

From Corollary~\ref{cor:rho1}, we have
$$\rho_a=\frac{S-2(p-b)(p-c)t}{2p}.$$
The radius of $\omega_a$ was found in Theorem~\ref{thm:Ra}, namely
$$R_a=\frac{a}{2}\csc\frac{\theta}{2}.$$

The touch point $T$ divides the segment $DO_a$ in the ratio $\rho_a:R_a$.
This fact allows us to use \textsc{Mathematica} to find the barycentric coordinates for $T$
from the known barycentric coordinates for $D$ and $O_a$.
\end{proof}

\newpage

\section{Properties of Three Ajima Circles}

Let $\omega_a$, $\omega_b$, and $\omega_c$, be arcs of angular measure $\theta$
erected internally on the sides of $\triangle ABC$.
Let $\gamma_a$ be the circle inscribed in $\angle BAC$ and tangent externally to $\omega_a$.
Define $\gamma_b$ and $\gamma_c$ similarly.
The three circles, $\gamma_a$, $\gamma_b$ and $\gamma_c$ will be called a \emph{general triad of circles}
associated with $\triangle ABC$ (Figure~\ref{fig:genTriad}).

\begin{figure}[ht]
\centering
\includegraphics[scale=0.5]{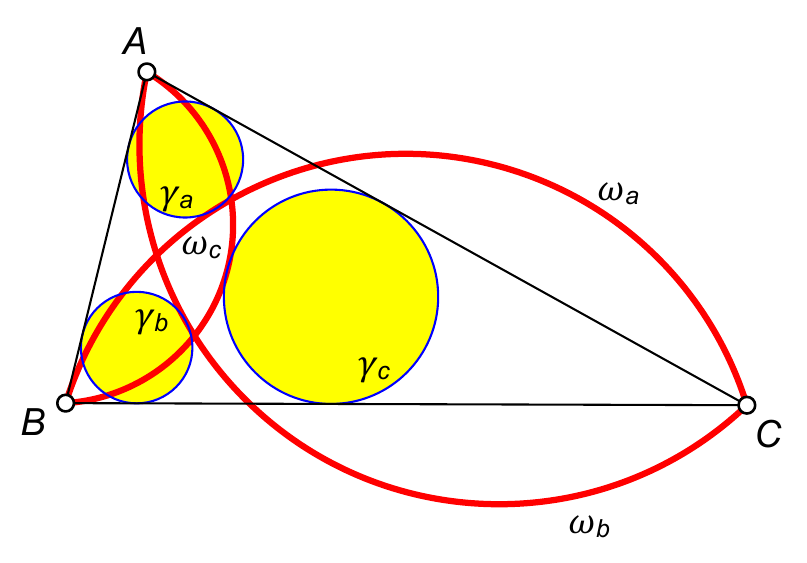}
\caption{general triad of circles}
\label{fig:genTriad}
\end{figure}

For the remainder of this paper, we will assume that the three circles
$\gamma_a$, $\gamma_b$ and $\gamma_c$ all lie inside $\triangle ABC$.
An equivalent condition is that all angles of $\triangle ABC$ have measure less
than $180\degrees-\frac{\theta}{2}$.

\begin{theorem}\label{thm:genTangents}
The common external tangents to any pair of circles in a general triad are congruent (Figure~\ref{fig:genTangents}).
The common length is $2r\tan\frac{\theta}{4}$.
\end{theorem}

\begin{figure}[ht]
\centering
\includegraphics[scale=0.5]{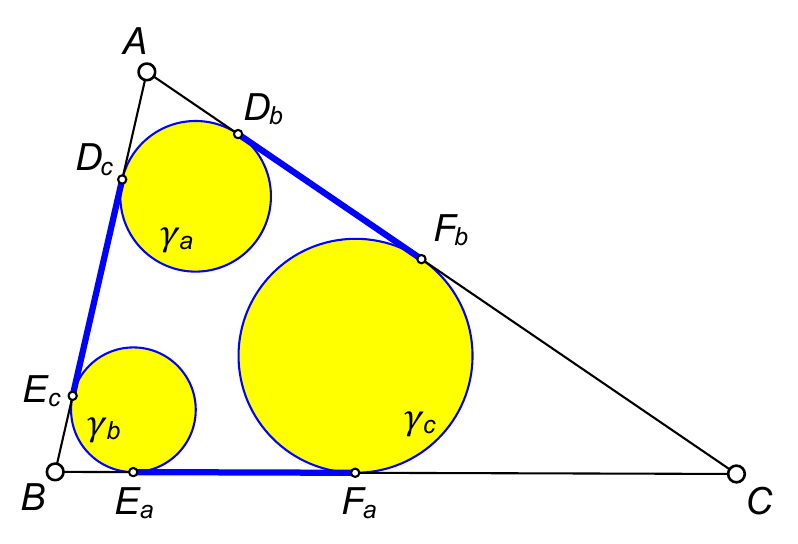}
\caption{blue lines are congruent}
\label{fig:genTangents}
\end{figure}

\begin{proof}
The common length is twice $KH$ (Figure~\ref{fig:thetaOver4}) whose value is given by Corollary~\ref{cor:KH}.
\end{proof}

\newpage

\textbf{Note.} The theorem remains true if some or all of the yellow circles are outside of $\triangle ABC$
as shown in Figure~\ref{fig:genTangents-123}.

\begin{figure}[ht]
\centering
\includegraphics[scale=0.4]{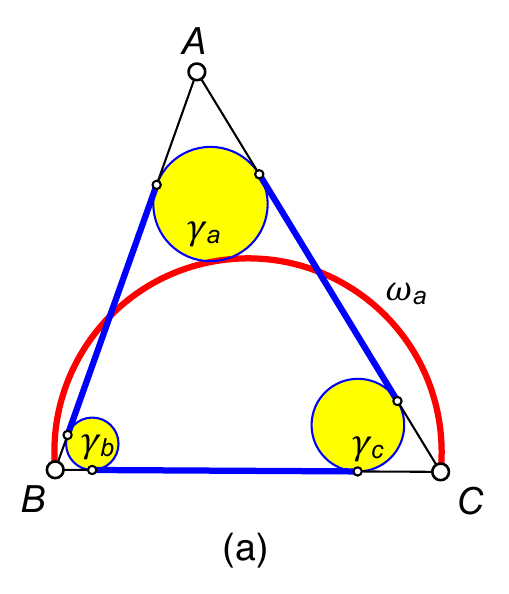}
\includegraphics[scale=0.4]{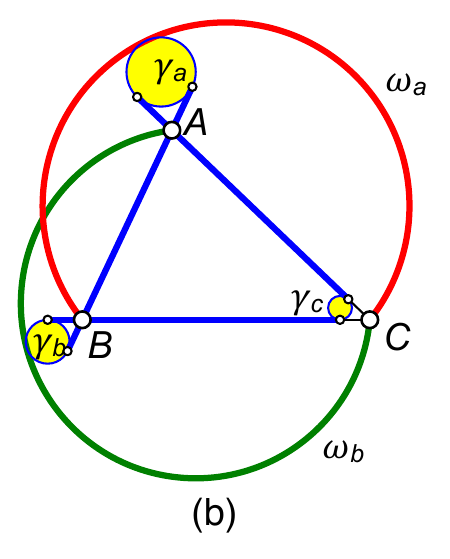}
\includegraphics[scale=0.4]{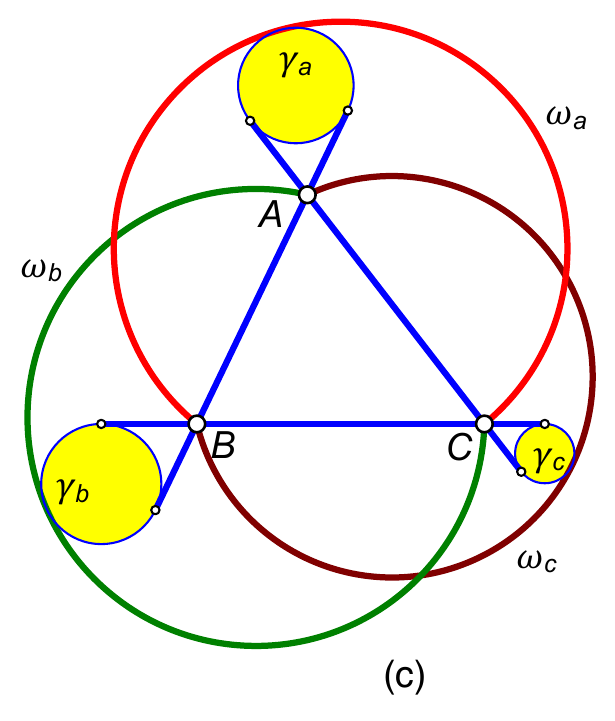}
\caption{blue lines are congruent}
\label{fig:genTangents-123}
\end{figure}


\begin{theorem}\label{thm:midpoints}
Let $M_a$, $M_b$, and $M_c$ be the midpoints of the common tangents (lying along the sides of $\triangle ABC$) to
a general triad of circles associated with that triangle.
Then $M_a$, $M_b$, and $M_c$ are the touch points of the incircle of $\triangle ABC$
with the sides of the triangle (Figure \ref{fig:midpoints}).
\end{theorem}

\begin{figure}[ht]
\centering
\includegraphics[scale=0.55]{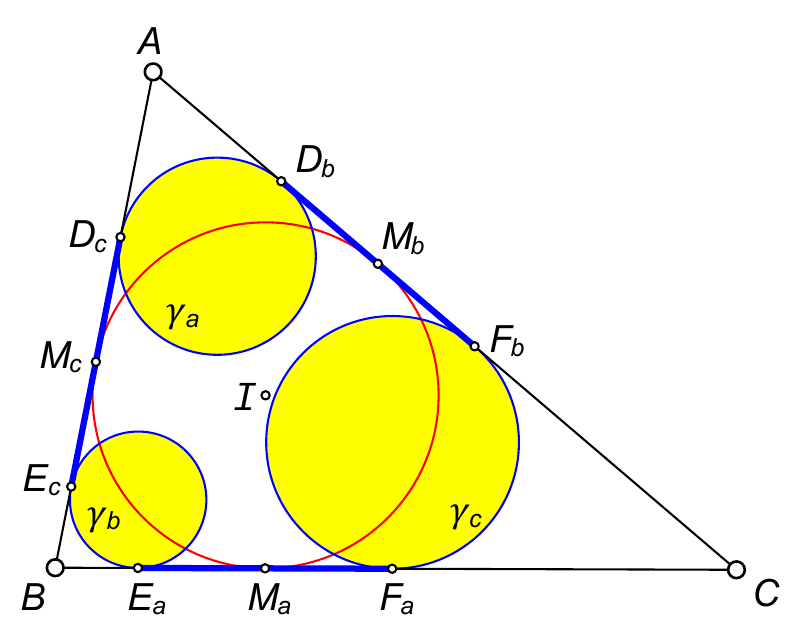}  
\caption{}
\label{fig:midpoints}
\end{figure}

\begin{proof}
This follows from Corollary~\ref{cor:KH}.
\end{proof}

\newpage

\begin{theorem}\label{radixAxis}
Let $\gamma_a$, $\gamma_b$, and $\gamma_c$ be
a general triad of circles associated with triangle $\triangle ABC$.
Let $M_a$, $M_b$, and $M_c$ be the points where the incircle of $\triangle ABC$ touches the sides. 
Then the radical axis of $\gamma_b$, and $\gamma_c$ is $AM_a$
(Figure~\ref{fig:radicalAxis}).
\end{theorem}

\begin{figure}[ht]
\centering
\includegraphics[scale=0.5]{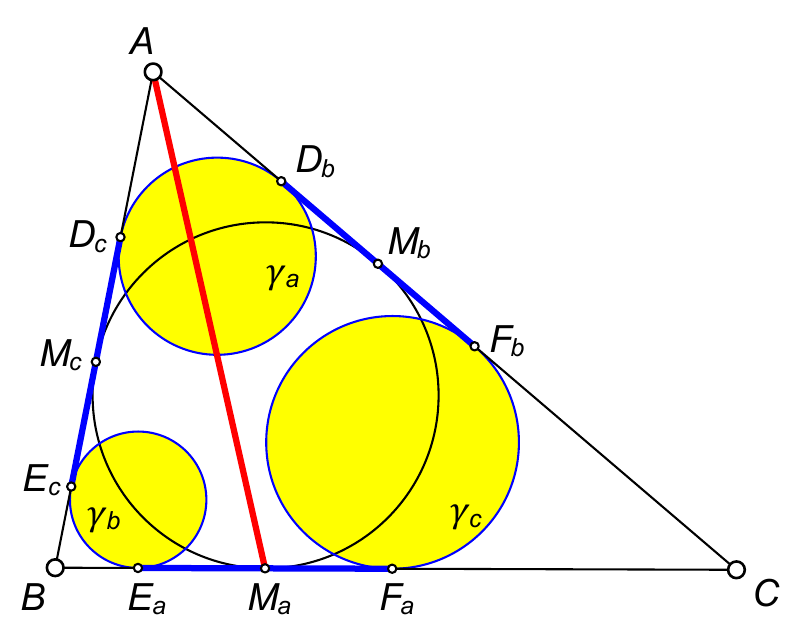}  
\caption{}
\label{fig:radicalAxis}
\end{figure}

\begin{proof}
From Theorem~\ref{thm:midpoints}, $M_aE_a=M_aF_a$.
Thus, the tangents from $M_a$ to $\gamma_b$ and $\gamma_c$ are equal.
Since $AD_c=AD_b$ and $D_cE_c=D_bF_b$ (Theorem~\ref{thm:genTangents}), this means $AE_c=AF_b$.
Hence the tangents from $A$ to $\gamma_b$ and $\gamma_c$ are equal.
The radical axis of circles $\gamma_b$ and $\gamma_c$ is the locus of points
such that the lengths of the tangents to the two circles from that point are equal.
The radical axis of two circles is a straight line.
Therefore, the radical axis of circles $\gamma_b$ and $\gamma_c$ is $AM_a$,
the Gergonne cevian from $A$.
\end{proof}


\begin{theorem}\label{thm:radicalCenter}
Let $\gamma_a$, $\gamma_b$, and $\gamma_c$ be
a general triad of circles associated with triangle $\triangle ABC$.
Then the radical center of the three circles of the triad is the Gergonne point of $\triangle ABC$ (Figure~\ref{fig:radicalCenterGe}).
\end{theorem}

\begin{figure}[ht]
\centering
\includegraphics[scale=0.5]{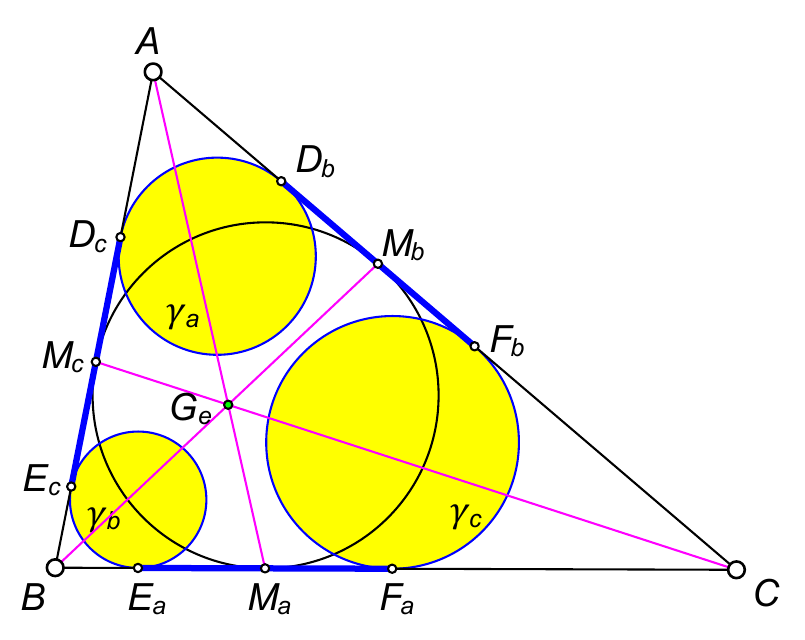}
\caption{}
\label{fig:radicalCenterGe}
\end{figure}

\begin{proof}
By Theorem~\ref{radixAxis}, the radical axis of circles $\gamma_b$ and $\gamma_c$ is $AM_a$,
the Gergonne cevian from $A$.
Similarly, the radical axis of circles $\gamma_a$ and $\gamma_c$ is the Gergonne
cevian from $B$ and the radical axis of circles $\gamma_a$ and $\gamma_b$ is the Gergonne
cevian from $C$. Hence, the radical center of the general triad of circles is the intersection point
of the three Gergonne cevians, namely, the Gergonne point of $\triangle ABC$.
\end{proof}

\newpage

\begin{theorem}
The six points of contact of a general triad of circles lie on a circle with center $I$,
the incenter of $\triangle ABC$ (Figure \ref{fig:genConcyclic}).
\end{theorem}

\begin{figure}[ht]
\centering
\includegraphics[scale=0.5]{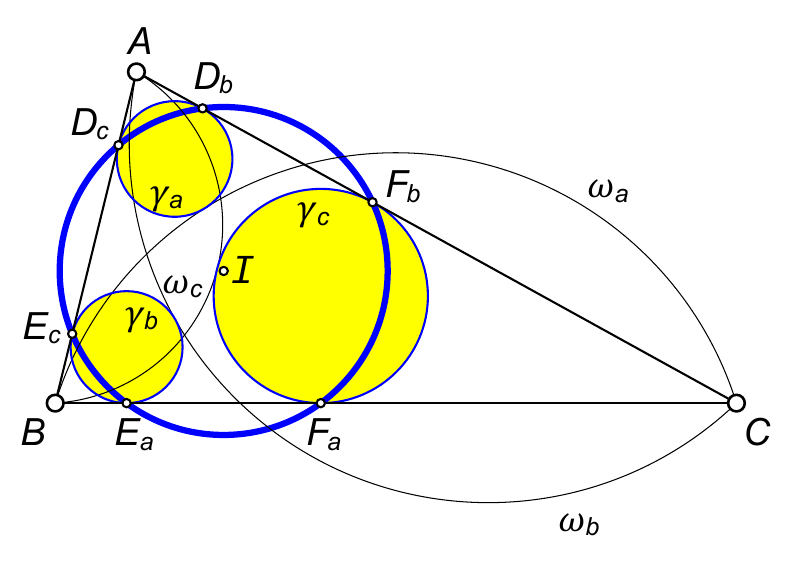}
\caption{touch points are concyclic}
\label{fig:genConcyclic}
\end{figure}

\begin{proof}
This follows from Theorem~\ref{thm:thetaOver4} from which we can deduce that
\begin{equation*}
ID_b=ID_c=IE_a=IE_c=IF_a=IF_b=r\sec\frac{\theta}{4}.\qedhere
\end{equation*}
\end{proof}

\void{
The following result is well known \cite{Knot}.

\begin{lemma}[Trigonometric Form of Ceva's Theorem]\label{lemma:CevaTrig}
Let $P_a$, $P_b$, and $P_c$ be three points inside $\triangle ABC$
and number the angles formed by the sides of the triangle with $AP_a$, $BP_b$, and $CP_c$
as shown in Figure~\ref{fig:CevaTrig}.
Then $AP_a$, $BP_b$, and $CP_c$ are concurrent if and only if
$$\sin\angle1\sin\angle3\sin\angle5=\sin\angle2\sin\angle4\sin\angle6.$$
\end{lemma}

\begin{figure}[ht]
\centering
\includegraphics[scale=0.5]{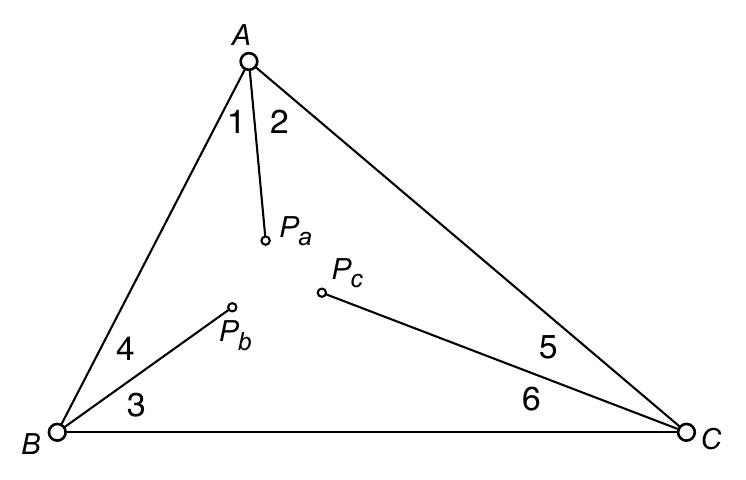}
\caption{}
\label{fig:CevaTrig}
\end{figure}

\begin{lemma}[Exact Trilinear Form of Ceva's Theorem]\label{lemma:CevaTrilinear}
Let $P_a$, $P_b$, and $P_c$ be three points inside $\triangle ABC$.
Let the distances from these points to the sides of the triangle be
$h_1$ through $h_6$
as shown in Figure~\ref{fig:CevaTrilinear}.
Then $AP_a$, $BP_b$, and $CP_c$ are concurrent if and only if
$$h_1h_3h_5=h_2h_4h_6.$$
\end{lemma}

\begin{figure}[ht]
\centering
\includegraphics[scale=0.5]{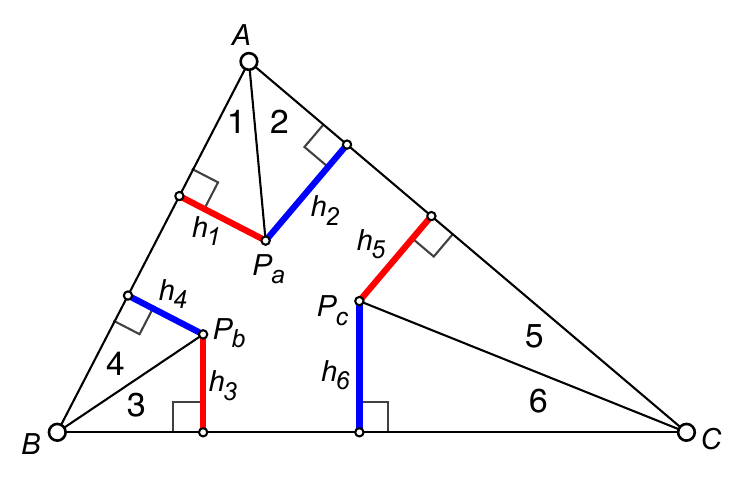}
\caption{}
\label{fig:CevaTrilinear}
\end{figure}

\begin{proof}
We have
\begin{equation*}
\begin{aligned}
h_1h_3h_5&=(AP_a\cdot\sin\angle1)(AP_b\cdot\sin\angle3)(AP_c\cdot\sin\angle5)\\
&=(AP_a\cdot\sin\angle2)(AP_b\cdot\sin\angle4)(AP_c\cdot\sin\angle6)\\
&=h_2h_4h_6\\
\end{aligned}
\end{equation*}
so Lemmas \ref{lemma:CevaTrig} and \ref{lemma:CevaTrilinear} are equivalent.
\end{proof}

\begin{corollary}[Trilinear Form of Ceva's Theorem]\label{cor:CevaTrilinear}
Let $P_a$, $P_b$, and $P_c$ be three points inside $\triangle ABC$.
Suppose the trilinear coordinates for these points are as follows.
$$P_a=(a_x:a_y:a_z),\quad P_b=(b_x:b_y:b_z),\quad P_c=(c_x:c_y:c_z).$$
Then $AP_a$, $BP_b$, and $CP_c$ are concurrent if and only if
$$\frac{a_z}{a_y}\cdot\frac{b_x}{b_z}\cdot\frac{c_y}{c_x}\cdot=1.$$
\end{corollary}

\begin{proof}
Note that the exact trilinear coordinates for these points are
$$P_a=(d_a:h_2:h_1),\quad P_b=(h_3:d_b:h_4),\quad P_c=(h_6:h_5:d_c)$$
where the $h$'s are as given in Lemma~\ref{lemma:CevaTrilinear}
and $d_a$ is the distance of $P_a$ from $BC$ with analogous definitions for $d_b$ and $d_c$.
By Lemma~\ref{lemma:CevaTrilinear}, $AP_a$, $BP_b$, and $CP_c$ are concurrent if and only if
$$h_1h_3h_5=h_2h_4h_6.$$
Since the trilinear coordinates of a point are proportional to the exact trilinear coordinates,
this means that $AP_a$, $BP_b$, and $CP_c$ are concurrent if and only if
$$a_zb_xc_y=a_yb_zc_x$$
which is equivalent to the desired result.
\end{proof}

\begin{corollary}[Barycentric Form of Ceva's Theorem]\label{cor:CevaBary}
Let $P_a$, $P_b$, and $P_c$ be three points inside $\triangle ABC$.
Suppose the barycentric coordinates for these points are as follows.
$$P_a=(a'_x:a'_y:a'_z),\quad P_b=(b'_x:b'_y:b'_z),\quad P_c=(c'_x:c'_y:c'_z).$$
Then $AP_a$, $BP_b$, and $CP_c$ are concurrent if and only if
$$\frac{a'_z}{a'_y}\cdot\frac{b'_x}{b'_z}\cdot\frac{c'_y}{c'_x}=1.$$
\end{corollary}

\begin{proof}
This follows from the relationship between barycentric coordinates and trilinear coordinates.
We have $a'_x=aa_x$, etc. Thus
$$\frac{a'_z}{a'_y}\cdot\frac{b'_x}{b'_z}\cdot\frac{c'_y}{c'_x}
=\frac{ca_z}{ba_y}\cdot\frac{ab_x}{cb_z}\cdot\frac{bc_y}{ac_x}
=\frac{a_z}{a_y}\cdot\frac{b_x}{b_z}\cdot\frac{c_y}{c_x}$$
and the result follows from Corollary~\ref{cor:CevaBary}.
\end{proof}
}

\begin{theorem}\label{thm:Jacobi}
Let the centers of $\omega_a$, $\omega_b$, and $\omega_b$, be $O_a$, $O_b$, and $O_c$, respectively.
Then $AO_a$, $BO_b$, and $CO_c$ are concurrent (Figure~\ref{fig:Jacobi}).
\end{theorem}

\begin{figure}[ht]
\centering
\includegraphics[scale=0.8]{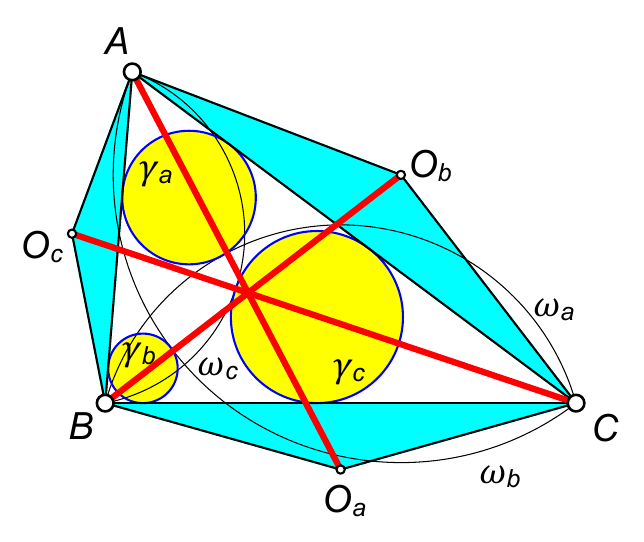}
\caption{red lines are concurrent}
\label{fig:Jacobi}
\end{figure}

\begin{proof}
Note that isosceles triangles $BCO_a$, $CAO_b$, and $ABO_c$ are similar.
Therefore $AO_a$, $BO_b$, and $CO_c$ are concurrent by Jacobi's Theorem \cite{Jacobi}.
\end{proof}

\newpage

\begin{theorem}[Paasche Analog]\label{thm:Paasche}
Let $\gamma_a$, $\gamma_b$, and $\gamma_c$ be
a general triad of circles associated with triangle $\triangle ABC$.
Let $T_a$, $T_b$, and $T_c$ be the points where they touch the three arcs
having the same angular measure (Figure~\ref{fig:Paasche}).
Then $AT_a$, $BT_b$, and $CT_c$ are concurrent.
\end{theorem}

\begin{figure}[ht]
\centering
\includegraphics[scale=0.5]{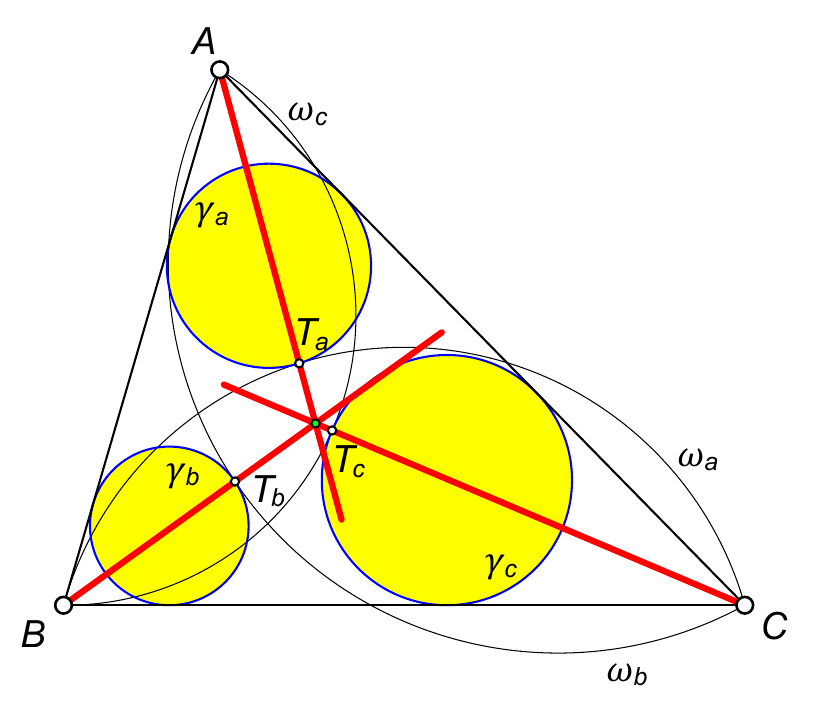}
\raise 0.15in\hbox{\includegraphics[scale=0.5]{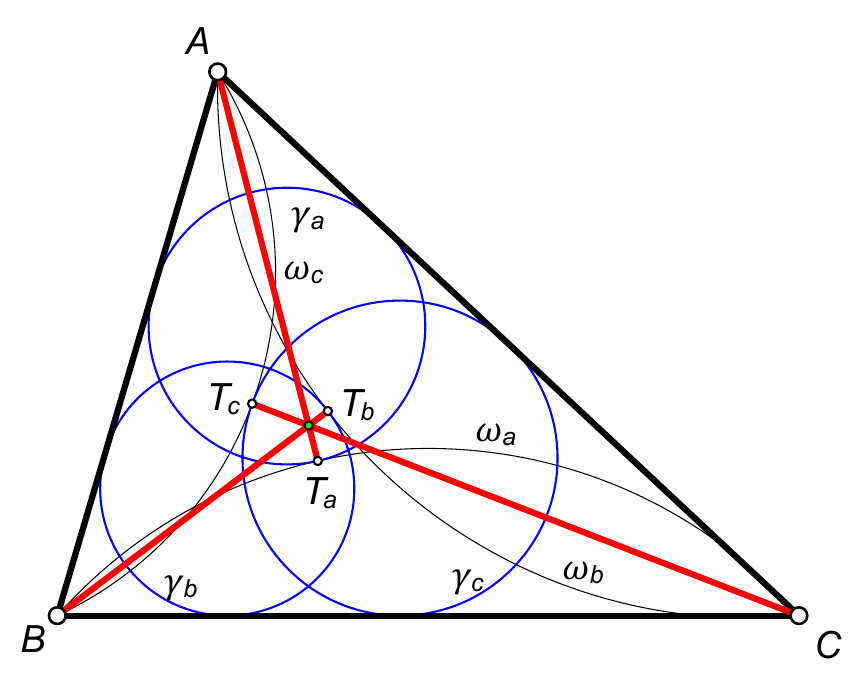}}
\caption{red lines are concurrent}
\label{fig:Paasche}
\end{figure}

\begin{proof}
The barycentric coordinates for $T_a$ were found in
Theorem~\ref{thm:baryT}. The barycentric coordinates for $A$ are $(1:0:0)$.
We can thus find the equation of the line $AT_a$ using formula (3) from \cite{Grozdev}.
Similarly, we can find the equations for the lines $BT_b$ and $CT_c$.
Then, using \textsc{Mathematica}, we can use the condition that three lines are concurrent (formula (6) from \cite{Grozdev}) to prove that $AT_a$, $BT_b$ and $CT_c$ are concurrent.
\end{proof}

We call this theorem the Paasche Analog because when $\theta=180\degrees$, the point of
concurrence is the Paasche point of the triangle \cite{ETC1123}.

\begin{open}
Is there a purely geometric proof for Theorem~\ref{thm:Paasche}?
\end{open}

The coordinates for the point of concurrence are complicated and we do not give them here.
However, we did find the following interesting result.

\begin{theorem}
When $\theta=120\degrees$, the point of concurrence of $AT_a$, $BT_b$, and $CT_c$ is
the isogonal conjugate of $X_{7005}$.
When $\theta=240\degrees$, the point of concurrence of $AT_a$, $BT_b$, and $CT_c$ is
$X_{14358}$.
\end{theorem}

\void{
\begin{proof}
From Theorem~\ref{thm:baryT}, we found that the barycentric coordinates for $D_a$ are of the form
$(A_x:A_y:A_z)$ where
$$\frac{A_y}{A_z}=-\frac{\cot \theta \Bigl((b c-\text{SA})
  \left[u+a b \sec ^2\frac{\theta }{4}\right]+2 a b S \csc\frac{\theta }{2}\Bigr)-2 \text{SC}
  \left(t(\text{SA}-b c)+S\right)}{\cot \theta \Bigl((b c-\text{SA})
  \left[u-a c \sec^2\frac{\theta }{4}\right]+2 a c S \csc\frac{\theta }{2}\Bigr)+2 \text{SB}
  \left(t(\text{SA}-b c)+S\right)}$$
where $S$, SA, SB, and SC are as given in Theorem~\ref{thm:baryT}, and where $t=\tan(\theta/4)$ and
$u=2 b c-2 S t+2 \text{SA}$.

If $D_b=(B_x:B_y:B_z)$ and $D_c=(C_x:C_y:C_z)$, then we can find similar expressions
for $B_z/B_x$ and $C_x/C_y$. Using \textsc{Mathematica}, we find that
$$\frac{A_y}{A_z}\cdot \frac{B_z}{B_x}\cdot \frac{C_x}{C_y}=1,$$
Thus, by Corollary~\ref{cor:CevaBary}, $AT_a$, $BT_b$, and $CT_c$ concur.
\end{proof}
}

\void{
Note that the expression for $A_y/A_x$ has the form
$$\frac{A_y}{A_z}=-\frac{f(a,b,c)}{f(a,c,b)}.$$
In other words, exchanging $b$ and $c$ in the numerator of the fraction gives the denominator.

By symmetry, we get similar results for the barycentric coordinates for $D_b$ and~$D_c$.
For $D_b=(B_x:B_y:B_z)$, we find
$$\frac{B_z}{B_x}=-\frac{f(a,b,c)}{f(c,b,a)}$$
and for $D_c=(C_x:C_y:C_z)$, we find
$$\frac{C_x}{C_y}=-\frac{f(a,b,c)}{f(b,a,c)}.$$

Since
$$\frac{A_y}{A_z}\cdot \frac{B_z}{B_x}\cdot \frac{C_x}{C_y}=1,$$
we see from Corollary~\ref{cor:CevaBary}, that $AT_a$, $BT_b$, and $CT_c$ concur.
}

\newpage

\section{Some Metric Identities}

Throughout this section, we will let
$$t=\tan\frac{\theta}{4}$$
and
$$\W=\frac{4R+r}{p}.$$

The following three identities were given in \cite[Lemma~3]{Suppa} and we will need them here as well.

\begin{lemma}\label{lemma:1}
Let $A$, $B$, and $C$ be the angles of a triangle with inradius $r$, circumradius $R$,
and semiperimeter $p$. Then
$$\tan\frac{A}{2}+\tan\frac{B}{2}+\tan\frac{C}{2}=\W.$$
\end{lemma}

\begin{lemma}\label{lemma:2}
Let $A$, $B$, and $C$ be the angles of a triangle. Then
$$\tan\frac{A}{2}\tan\frac{B}{2}+\tan\frac{B}{2}\tan\frac{C}{2}+\tan\frac{C}{2}\tan\frac{A}{2}=1.$$
\end{lemma}

\begin{lemma}\label{lemma:3}
Let $A$, $B$, and $C$ be the angles of a triangle with inradius $r$ and semiperimeter $p$. Then
$$\tan\frac{A}{2}\tan\frac{B}{2}\tan\frac{C}{2}=\frac{r}{p}.$$
\end{lemma}

From Theorem~\ref{thm:rho1},
$$\rho_a=r\left(1-t\tan\frac{A}{2}\right),$$
so
\begin{equation}
r-\rho_a=rt\tan\frac{A}{2}\label{eq:frac1}
\end{equation}
with similar formulas for $r-\rho_b$ and $r-\rho_c$.
Also,
\begin{equation}
\tan\frac{A}{2}=\frac{r-\rho_a}{rt}\label{eq:frac}
\end{equation}
with similar formulas for $\tan\frac{B}{2}$ and $\tan\frac{C}{2}$.
Using equation (\ref{eq:frac1}) gives us the following corollary
to these lemmas.

\begin{corollary}
For a general triad of circles associated with $\triangle ABC$, we have
\begin{align}
\sum(r-\rho_a)&=rt\W,\label{eq:x1}\\
\sum(r-\rho_a)(r-\rho_b)&=r^2t^2,\label{eq:x2}\\
\prod(r-\rho_a)&=\frac{r^4t^3}{p}\label{eq:x3}.
\end{align}
\end{corollary}

\begin{theorem}\label{thm:1}
For a general triad of circles associated with $\triangle ABC$, we have
$$\rho_a+\rho_b+\rho_c=3r-rt\W.$$
\end{theorem}

\begin{proof}
This follows immediately from equation (\ref{eq:x1}).
\end{proof}

When $\theta=180\degrees$, the arcs become semicircles, $t=1$, and this result agrees with formula (6) in \cite{Suppa}.

\begin{theorem}\label{thm:2p}
For a general triad of circles associated with $\triangle ABC$, we have
$$3r^2-2r\sum\rho_a+\sum\rho_a\rho_b=r^2t^2.$$
\end{theorem}

\begin{proof}
Expanding the left side of equation (\ref{eq:x2}) gives the desired result.
\end{proof}

\begin{theorem}\label{thm:2}
For a general triad of circles associated with $\triangle ABC$, we have
$$\rho_a\rho_b+\rho_b\rho_c+\rho_c\rho_a=r^2\left(t^2-2t\W+3\right).$$
\end{theorem}

\begin{proof}
From Theorem~\ref{thm:2p}, we have
$$3r^2-2r\sum\rho_a+\sum\rho_a\rho_b=r^2t^2.$$

Using Theorem~\ref{thm:1}, we get
$$3r^2-2r\left(3r-rt\W\right)+\sum\rho_a\rho_b=r^2t^2.$$
Thus,
$$\sum\rho_a\rho_b=r^2t^2+2r\left(3r-rt\W\right)-3r^2$$
which simplifies to
$$\sum\rho_a\rho_b=r^2t^2-2r^2t\W+3r^2$$
as desired.
\end{proof}

When $\theta=180\degrees$, the arcs become semicircles and this result agrees with formula (7) in \cite{Suppa}.

\begin{theorem}\label{thm:3}
We have
$$\rho_a^2+\rho_b^2+\rho_c^2=r^2\left[3-2t\W+\left(\W^2-2\right)t^2\right]$$
\end{theorem}

\begin{proof}
Using the identity
$$\left(\sum\rho_a\right)^2=\sum\rho_a^2+2\sum\rho_a\rho_b,$$
we find that
$$\rho_a^2+\rho_b^2+\rho_c^2=\left(3r-rt\W\right)^2
-2r^2\left(t^2-2t\W+3\right).$$
Simplifying gives
$$\rho_a^2+\rho_b^2+\rho_c^2=r^2\left[3-2t\W+\left(\W^2-2\right)t^2\right]$$
which is the desired result.
\end{proof}

When $\theta=180\degrees$, the arcs become semicircles and this result agrees with formula (8) in \cite{Suppa}.

\begin{theorem}\label{thm:4}
For a general triad of circles associated with $\triangle ABC$, we have
$$\rho_a\rho_b\rho_c=r^3\left(1-t\W+t^2-\frac{r}{p}t^3\right).$$
\end{theorem}

\begin{proof}
Start with equation (\ref{eq:x3}). Expand and use Theorems \ref{thm:1} and \ref{thm:2}
to substitute known values for $\sum \rho_a$ and $\sum\rho_a\rho_b$. Solving for $\rho_a\rho_b\rho_c$ then gives
the desired formula.
\end{proof}

\void{
\begin{theorem}\label{thm:5}
For a general triad of circles associated with $\triangle ABC$, we have
$$r^2(3-t^2)+r(\rho_a\rho_b+\rho_b\rho_c+\rho_c\rho_a)=2(\rho_a+\rho_b+\rho_c).$$
\end{theorem}

\begin{proof}
From equation (\ref{eq:frac}), we have
\begin{equation*}
\begin{aligned}
\tan\frac{A}{2}&=\frac{r-\rho_a}{rt}\\
\tan\frac{B}{2}&=\frac{r-\rho_b}{rt}\\
\tan\frac{C}{2}&=\frac{r-\rho_c}{rt}.
\end{aligned}
\end{equation*}
Using Lemma~\ref{lemma:2}, we get
$$\left(\frac{r-\rho_a}{rt}\right)\left(\frac{r-\rho_b}{rt}\right)
+\left(\frac{r-\rho_b}{rt}\right)\left(\frac{r-\rho_c}{rt}\right)+
\left(\frac{r-\rho_c}{rt}\right)\left(\frac{r-\rho_a}{rt}\right)=1.$$
Expanding out gives the desired result.
\end{proof}
}

\goodbreak
The following result was found empirically using the program ``OK Geometry''.\footnote{
OK Geometry is a tool for analyzing dynamic geometric constructions, developed by Zlatan Magajna which can be freely downloaded from \url{https://www.ok-geometry.com/}.}

\begin{theorem}\label{thm:formula}
We have
$$a^2 \rho_a ^2 (2 r-\rho_a )^2+16 r (r-\rho_a ) (r R-r R_a-\rho_a  R) (r R-r \text{Ra}+\rho_a R_a)=0.$$
\end{theorem}

\begin{proof}
Starting with the left side of the equation, we make the following substitutions, in succession.
\begin{align*}
\rho_a&=r\left(1-\frac{r}{p-a}\tan\frac{\theta}{4}\right)\\
R&=\frac{abc}{4\Delta}\\
r&=\frac{\Delta}{p}\\
\Delta&=\sqrt{p(p-a)(p-b)(p-c)}\\
p&=\frac{a+b+c}{2}\\
R_a&=\frac{a}{2\cos(90\degrees-\frac{\theta}{2})}
\end{align*}
Simplifying the resulting expression using \textsc{Mathematica}, we find that the expression is equal to 0.
\end{proof}

In some special cases, this formula can be simplified.

\begin{theorem}
If $\theta=360\degrees-4A$, then
$$\rho_a=\frac{2rR_a}{R+2R_a}.$$
\end{theorem}

\begin{proof}
The proof is the same as the proof of Theorem~\ref{thm:formula}.
\end{proof}

For a fixed $\theta$, we can find a relationship between $r$, $R$, $R_a$ and $\rho_a$,
not involving $a$.

\begin{theorem}
We have
\begin{equation}
\frac{R_a}{rR}=\frac{(r-\rho_a)(1+t^2)}{(r-\rho_a)^2+r^2t^2}.\label{eq:RarR}
\end{equation}
\end{theorem}

\begin{proof}
This follows by eliminating $\tan(A/2)$ from equations (\ref{eq:rho1}) and (\ref{eq:Ra3}).
The expression $\sin A$ is expressed in terms of $\tan(A/2)$ using Lemma~\ref{lemma:sin2x}.
\end{proof}

Solving for $t^2$ in equation (\ref{eq:RarR}) gives us the following.

\begin{theorem}
We have
\begin{equation*}
t^2=\frac{(r-\rho_a)(\rho_aR_a+rR-rR_a)}{r(R\rho_a+rR_a-rR)}.
\end{equation*}
\end{theorem}

\newpage

\void{
Let $T$ denote the configuration consisting of the three circles, $\gamma_a$, $\gamma_b$,
and $\gamma_c$ and their common external tangents as shown in Figure~\ref{fig:triad}.
Let $T^*$ denote the same configuration and same shape triangle when $\theta=180\degrees$.

\begin{figure}[ht]
\centering
\includegraphics[scale=0.7]{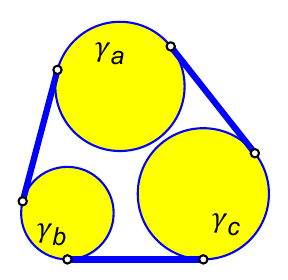}
\caption{Configuration $T$ \note{Do we need this?}}
\label{fig:triad}
\end{figure}
}

\void{

\begin{theorem}
Configurations $T$ and $T^*$ can be placed so that they are homothetic. The homothetic ratio is $t=\tan\frac{\theta}{2}$
and the center of the homothety is the incenter, $I$.
\end{theorem}

\note{This isn't quite right. I can't seem to express what I want to say.}

}

\void{
Note that the configuration is uniquely determined (up to congruence) by specifying the radii of the circles
and the length of the common tangent.
}

\newpage

\section{Apollonius Circles of the Three Ajima Circles}

A circle that is tangent to three given circles is called an \emph{Apollonius circle} of those three circles.

If all three circles lie inside an Apollonius circle, then the Apollonius circle
is called the \emph{outer Apollonius circle} of the three circles.
The outer Apollonius circle surrounds the three circles and is internally tangent to all three.

If all three circles lie outside an Apollonius circle, then the Apollonius circle
is called the \emph{inner Apollonius circle} of the three circles.
The inner Apollonius circle will either be internally tangent to the three given circles
or it will be externally tangent to all the circles. Figure~\ref{fig:innerApollonius}
shows various configurations. In each case, the red circle is the inner Apollonius circle of
the three blue circles.

\begin{figure}[ht]
\centering
\includegraphics[scale=0.38]{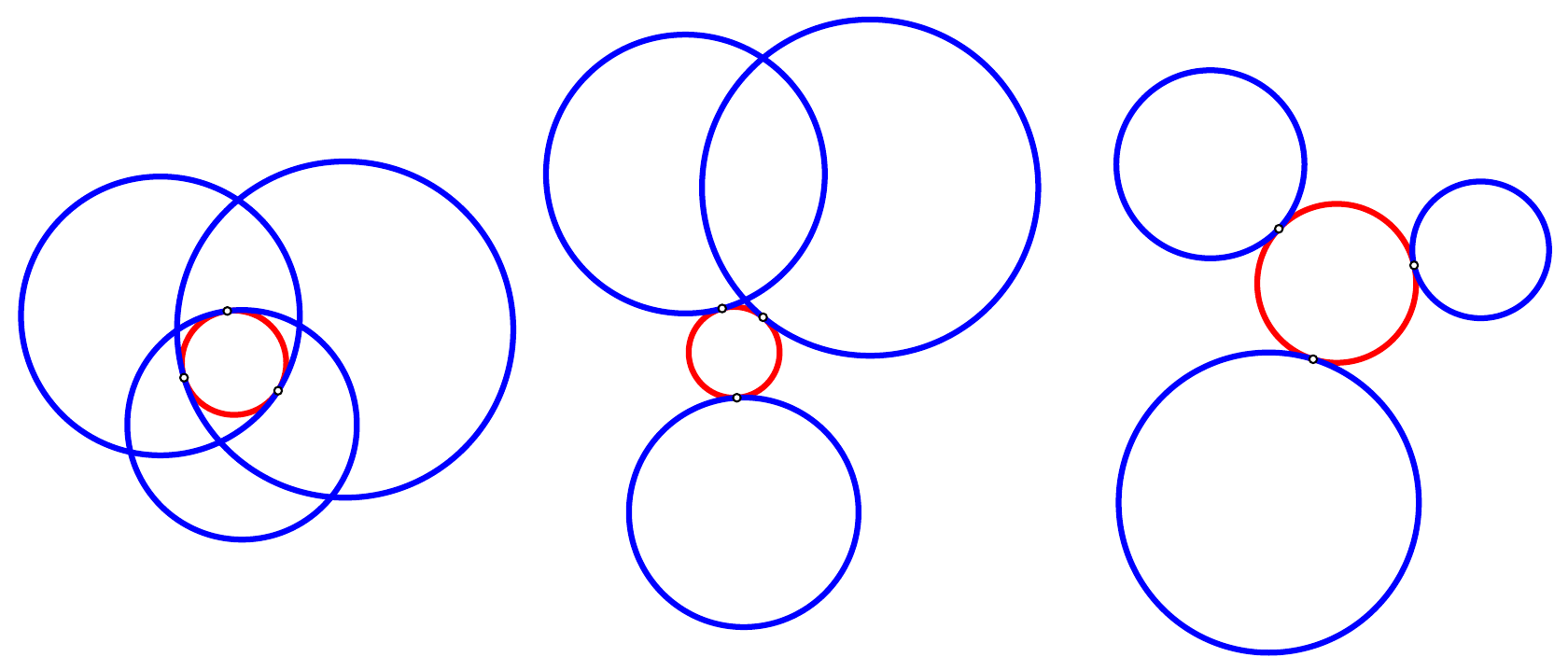}
\caption{inner Apollonius circle of three circles}
\label{fig:innerApollonius}
\end{figure}

We will be looking at the inner and outer Apollonius circles of a general triad
of circles associated with $\triangle ABC$.
But first, let us review some known facts about tangent circles.


\begin{lemma}\label{lemma:twoCircles}
Let $U(r_1)$ and $V(r_2)$ be two circles in the plane.
Let $S$ be a center of similarity of the two circles (Figure~\ref{fig:twoCircles}).
Then
$$\frac{US}{SV}=\frac{r_1}{r_2}.$$
\end{lemma}

\begin{figure}[ht]
\centering
\includegraphics[scale=0.35]{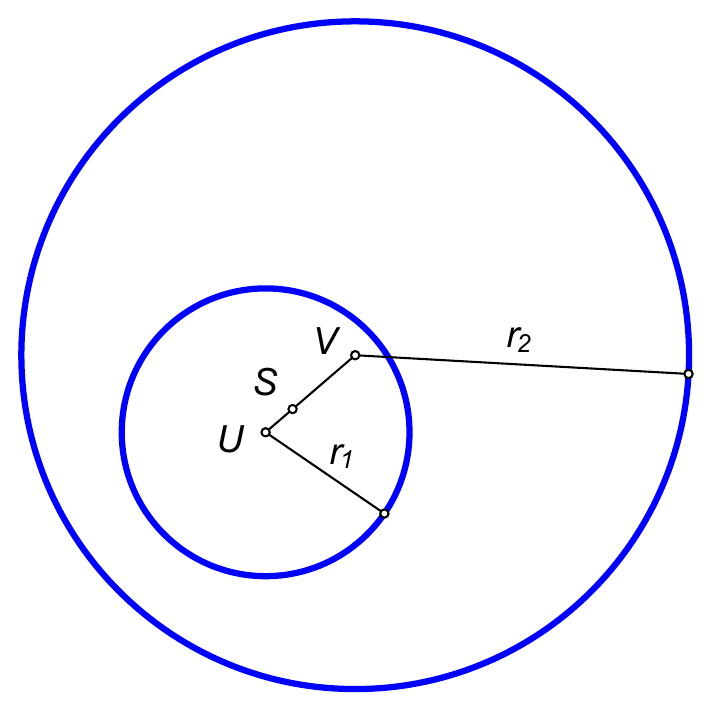}
\caption{}
\label{fig:twoCircles}
\end{figure}

\begin{proof}
The line of centers of two circles points passes through the
center of similitude. So $S$ lies on $UV$.
In a similarity, corresponding distances in two similar figures are in proportion to their ratio of similitude.
Their ratio of similitude is the ratio of their radii, namely $r_1/r_2$.
So $SU/SV=r_1/r_2$.
\end{proof}

\newpage

When we say that a circle is inscribed in an angle $ABC$, we mean that the circle
is tangent to the rays $\overrightarrow{BA}$ and $\overrightarrow{BC}$.

The following result comes from \cite[Theorem~2]{Rabinowitz}.

\begin{lemma}\label{lemma:triConcur}
Let $C_a$ be an arbitrary circle inscribed in $\angle BAC$ of $\triangle ABC$.
Let $C_b$ be an arbitrary circle inscribed in $\angle CBA$.
Let $C_c$ be an arbitrary circle inscribed in $\angle ACB$.
Let $S$ be the inner (respectively outer) Apollonius circle of $C_a$, $C_b$, and $C_c$.
Let $T_a$ be the point where $C_a$ touches $(S)$. Define $T_b$ and $T_c$ similarly.
Then $AT_a$, $BT_b$, and $CT_c$ are concurrent at a point $P$ (Figure~\ref{fig:triConcur}).
The point $P$ is the internal (external) center of similitude of the incircle of $\triangle ABC$
and circle $(S)$.
\end{lemma}

\begin{figure}[ht]
\centering
\includegraphics[scale=0.5]{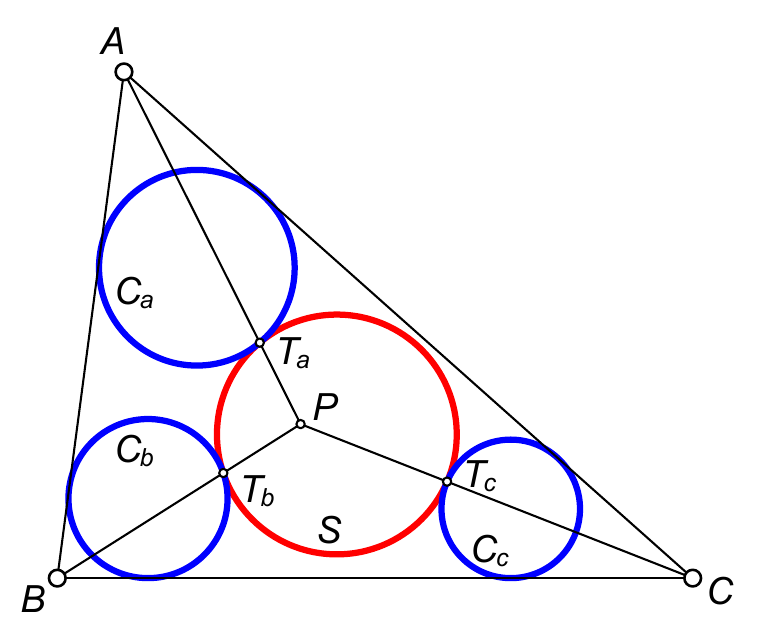}
\caption{}
\label{fig:triConcur}
\end{figure}

The following lemma comes from \cite[p.~85]{Casey}.

\begin{lemma}\label{lemma:Casey}
If two circles touch two others, then the radical axis of either pair
passes through a center of similitude of the other pair (Figure~\ref{fig:Casey}).
\end{lemma}

\begin{figure}[ht]
\centering
\includegraphics[scale=0.5]{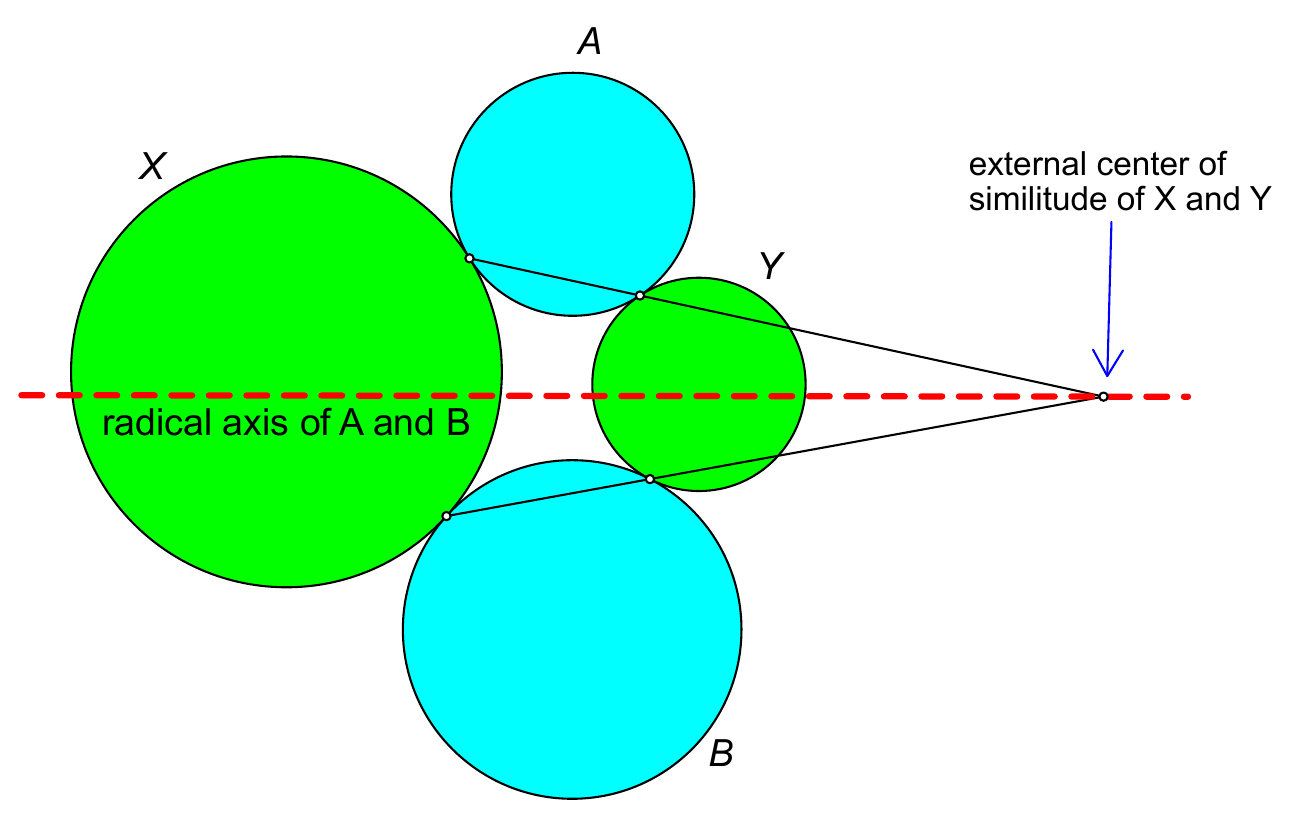}
\caption{}
\label{fig:Casey}
\end{figure}

\goodbreak
The following lemma comes from Gergonne's construction of Apollonius circles.
(See \cite[pp.~159--160]{Dorrie}.)

\begin{lemma}\label{lemma:Gergonne}
Let $(O_1)$, $(O_2)$, and $(O_3)$ be three circles in the plane.
Let $C_i$ be the inner Apollonius circle of the circles $(O_1)$, $(O_2)$, and $(O_3)$.
Let $C_o$ be the outer Apollonius circle of the circles $(O_1)$, $(O_2)$, and $(O_3)$.
Let $U_1$ be the point where $C_i$ touches $(O_1)$. Define $U_2$ and $U_3$ similarly.
Let $V_1$ be the point where $C_o$ touches $(O_1)$. Define $V_2$ and $V_3$ similarly.
Then $V_1U_1$, $V_2U_2$, and $V_3U_3$ are concurrent at the radical center, $R$,
of $(O_1)$, $(O_2)$, and $(O_3)$ (Figure~\ref{fig:radicalCenter}).
\end{lemma}

\begin{figure}[ht]
\centering
\includegraphics[scale=0.35]{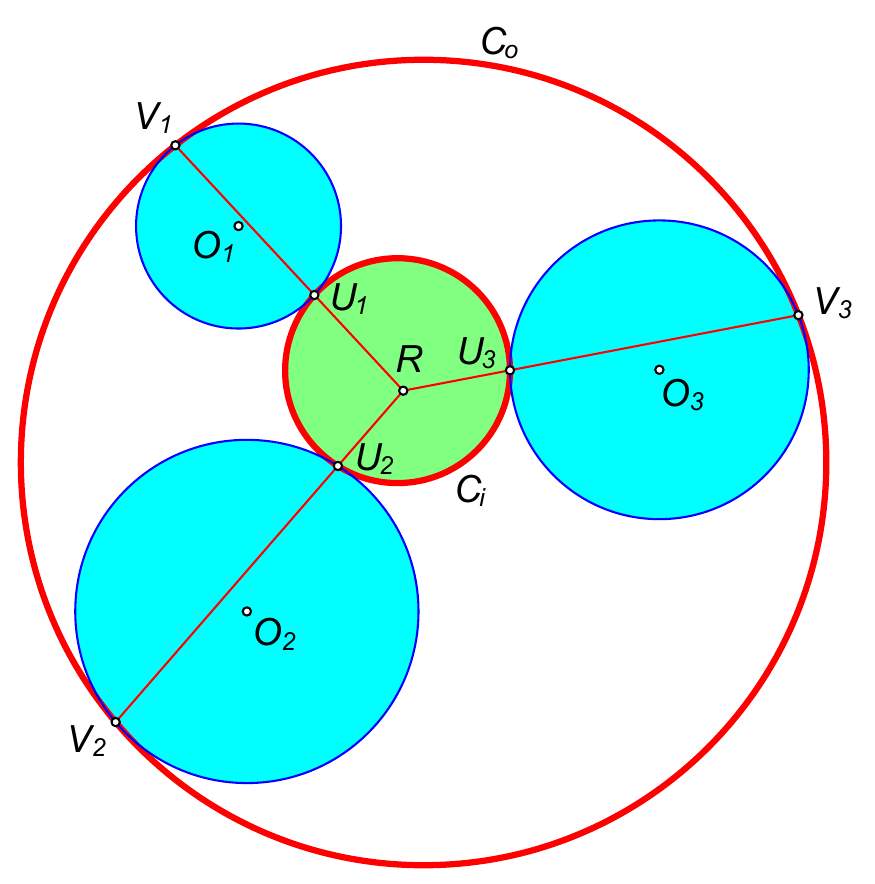}
\caption{}
\label{fig:radicalCenter}
\end{figure}


\begin{theorem}\label{thm:oi}
Let $C_1$, $C_2$, and $C_3$ be three circles as shown in Figure~\ref{fig:genSoddy}.
Let $U(\rho_i)$ and $V(\rho_o)$ be the inner and outer
Apollonius circles of $C_1$, $C_2$, and $C_3$, respectively.
Let $S$ be the radical center of the three circles.
Then $S$ lies on $UV$ and
$$\frac{SU}{SV}=\frac{\rho_i}{\rho_o}.$$
\end{theorem}

\begin{figure}[ht]
\centering
\includegraphics[scale=0.4]{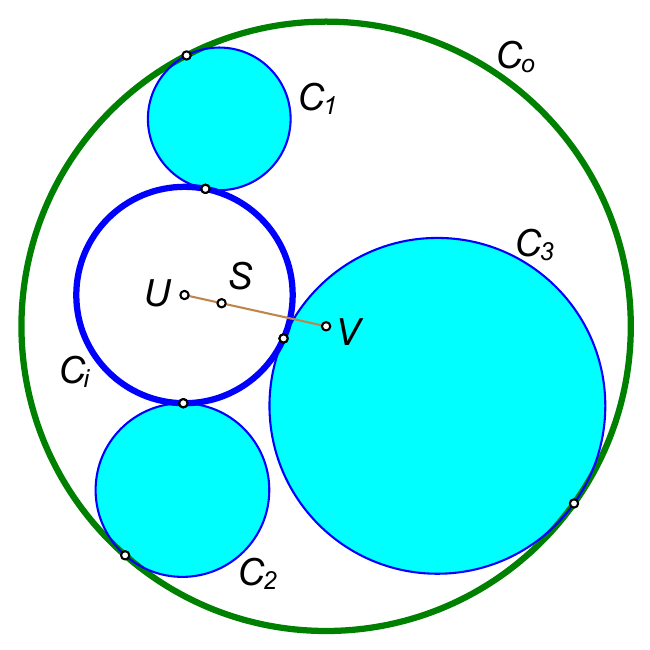}
\caption{$SU/SV=\rho_i/\rho_o$}
\label{fig:genSoddy}
\end{figure}

\begin{proof}
Circles $C_1$ and $C_2$ each touch circles $C_i$ and $C_o$.
By Lemma~\ref{lemma:Casey}, the radical axis of $C_1$ and $C_2$ passes through
a center of similarity, $S^*$, of $C_i$ and $C_o$.
Similarly, the radical axis of $C_2$ and $C_3$ passes through $S^*$.
These two radical axes meet at $S$, so $S=S^*$.

Note that $C_i$ and $C_o$ are two circles with center of similarity $S$.
By Lemma~\ref{lemma:twoCircles}, $S$ lies on $UV$ and $SU/SV=\rho_i/\rho_o$.
\end{proof}

\newpage


The following results were found via complex calculations
carried out with \textsc{Mathematica}. The details are omitted.

\begin{theorem}[Coordinates for $U_a$]\label{thm:coordUa}
Let $\gamma_a$, $\gamma_b$, and $\gamma_c$ be a general triad of circles associated
with $\triangle ABC$.
Let $C_i$ be the inner Apollonius circle of the circles in the triad.
Let $U_a$ be the point where $C_i$ touches $\gamma_a$.
Then the barycentric coordinates for $U_a$ are $(x:y:z)$ where
\begin{align*}
x&=2a(p-b)(p-c)t\\
y&=(p-c)[S-2(p-b)(p-c)t]\\
z&=(p-b)[S-2(p-b)(p-c)t]
\end{align*}
and where $p$ is the semiperimeter of $\triangle ABC$, $S$ is twice the area,
and $t=\tan(\theta/4)$.
\end{theorem}

The coordinates for $U_b$ and $U_c$ are similar.

\begin{theorem}[Coordinates for $U$]\label{thm:coordU}
Let $\gamma_a$, $\gamma_b$, and $\gamma_c$ be a general triad of circles associated
with $\triangle ABC$.
Let $C_i$ be the inner Apollonius circle of the circles in the triad.
Let $U$ be the center of $C_i$. Then the barycentric coordinates for $U$ are $(X:Y:Z)$ where
\begin{align*}
X&=\left(-2 a^3+a^2 (b+c)+(b-c)^2 (b+c)\right)t-2a S\\
Y&=\left(a^3-a^2 c+a \left(b^2-c^2\right)+c^3+b^2 c-2b^3\right)t-2 b S\\
Z&=\left(a^3-a^2 b+a\left(c^2-b^2\right)+b^3+b c^2-2 c^3\right)t-2 c S
\end{align*}
and where $S$ is twice the area of $\triangle ABC$ and $t=\tan(\theta/4)$.
\end{theorem}

\begin{theorem}[Coordinates for $V_a$]\label{thm:coordVa}
Let $\gamma_a$, $\gamma_b$, and $\gamma_c$ be a general triad of circles associated
with $\triangle ABC$.
Let $C_o$ be the outer Apollonius circle of the circles in the triad.
Let $V_a$ be the point where $C_o$ touches $\gamma_a$.
Then the barycentric coordinates for $V_a$ are $(x:y:z)$ where
\begin{align*}
x&=2(p-b)(p-c)[2S+a(p-a)t]\\
y&=(p-a)(p-c)[S-2(p-b)(p-c)t]\\
z&=(p-a)(p-b)[S-2(p-b)(p-c)t]
\end{align*}
and where $p$ is the semiperimeter of $\triangle ABC$, $S$ is twice the area,
and $t=\tan(\theta/4)$.
\end{theorem}

The coordinates for $V_b$ and $V_c$ are similar.

\begin{theorem}[Coordinates for $V$]\label{thm:coordV}
Let $\gamma_a$, $\gamma_b$, and $\gamma_c$ be a general triad of circles associated
with $\triangle ABC$.
Let $C_o$ be the outer Apollonius circle of the circles in the triad.
Let $V$ be the center of $C_o$. Then the barycentric coordinates for $V$ are $(X:Y:Z)$ where
\begin{align*}
X&=\left(-2 a^3+a^2 (b+c)+(b-c)^2 (b+c)\right)t+6 aS\\
Y&=\left(a^3-a^2 c+a \left(b^2-c^2\right)+c^3+b^2 c-2b^3\right)t+6 b S\\
Z&=\left(a^3-a^2 b+a\left(c^2-b^2\right)+b^3+b c^2-2 c^3\right)t+6 c S
\end{align*}
and where $S$ is twice the area of $\triangle ABC$ and $t=\tan(\theta/4)$.
\end{theorem}

\void{
\begin{proof}
Let $P=(x:y:z)$ and $Q=(X:Y:Z)$.
We will prove that $P=U_a$ and $Q=U$.

Let $D$ be the center of $\gamma_a$.
From Theorem~\ref{thm:baryD2}, we know that the barycentric coordinates for $D$ are
$$D=\Bigl(ap(p-a)+(b+c)t\Delta):bp(p-a)-bt\Delta:cp(p-a)-ct\Delta\Bigr).$$
where $S=2\Delta$.

Let $\rho_a$ be the radius of $\gamma_a$.
From Corollary~\ref{cor:rho1}, we know that
$$\rho_a=\frac{\Delta-(p-b)(p-c)t}{p}.$$
Using the formula for the distance between two points in barycentric coordinates \cite[formula (9)]{Grozdev},
we can find the distance $d$ between $P$ and $D$. Using \textsc{Mathematica}, it is straightforward to
show that $d=\rho_a$. Thus $P$ lies on $\gamma_a$.

Using formula (4) from \cite{Grozdev}, which is the condition that three points are collinear,
we can verify, using \textsc{Mathematica}, that $D$, $P$, and $Q$ are collinear.

Calculating the distance between $P$ and $Q$, we find
$$\mathrm{dist}(P,Q)=\frac{2S+(a^2+b^2+c^2-2ab-2bc-2ca)t}{2(a+b+c)}.$$
Since this expression is symmetrical in $a$, $b$ and $c$, this means that
$$\mathrm{dist}(Q,U_a)=\mathrm{dist}(Q,U_b)=\mathrm{dist}(Q,U_b)$$
and so $Q$ is equidistant from $U_a$, $U_b$, and $U_c$.
This proves that $Q$ is the center of an Apollonius circle of $\gamma_a$, $\gamma_b$, and
$\gamma_c$. \note{(Does it?)}
\note{and how do we know that $Q$ is the center of the \underline{inner} Apollonius circle?}
\end{proof}
}

\void{
Let $O_a$ be the center of $\omega_a$.
From Theorem~\ref{thm:baryOa}, we know that the barycentric coordinates for $O_a$ are
$$O_a=\left(-a^2:S_c+S\cot\phi:S_b+S\cot\phi\right).$$
where $\phi=90\degrees-\theta/2$, $S=2\Delta$,
$S_b=(c^2+a^2-b^2)/2$, and $S_c=(a^2+b^2-c^2)/2$.
}

\newpage

\begin{theorem}\label{thm:AUaGe}
Let $\gamma_a$, $\gamma_b$, and $\gamma_c$ be a general triad of circles associated
with $\triangle ABC$.
Let $C_i$ be the inner Apollonius circle of the circles in the triad.
Let $U_a$ be the point where $C_i$ touches $\gamma_a$. Define $U_b$ and $U_c$ similarly.
Then $A$, $U_a$, and $G_e$ are collinear (Figure~\ref{fig:AUaGe}).
Similarly, $B$, $U_b$, and $G_e$ are collinear; and $C$, $U_c$, and $G_e$ are collinear.
\end{theorem}

\begin{figure}[ht]
\centering
\includegraphics[scale=0.5]{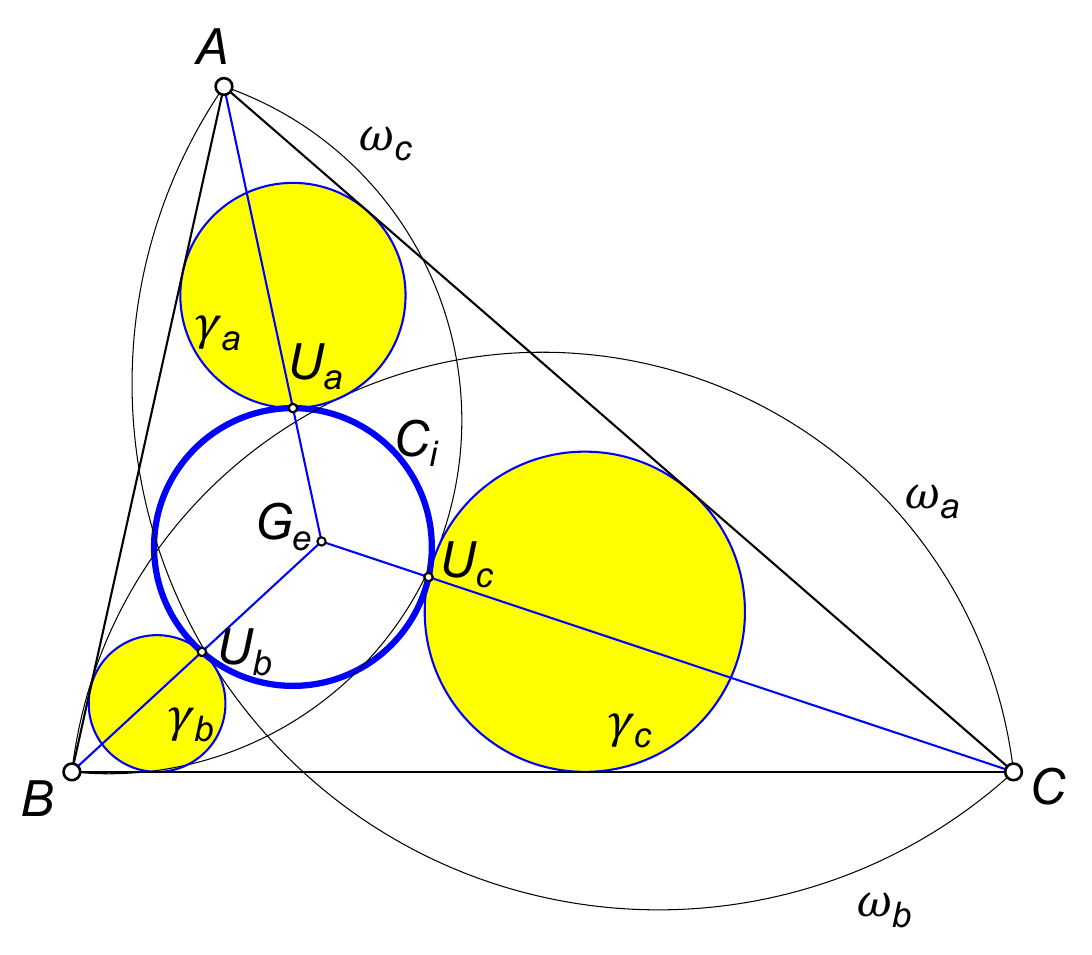}
\caption{lines concur at the Gergonne point}
\label{fig:AUaGe}
\end{figure}

\void{
\begin{proof}
By symmetry, it suffices to prove that $A$, $U_a$, and $G_e$ are collinear.
By Theorem~\ref{thm:result27}, the tangent to $\gamma_a$ at $U_a$ is parallel to $BC$.
Let $L$ be the point where the incircle of $\triangle ABC$ touches $BC$.
\end{proof}
}

\begin{proof}
By symmetry, it suffices to prove that $A$, $U_a$, and $G_e$ are collinear.
The barycentric coordinates for $A$ are $(1:0:0)$.
The barycentric coordinates for $G_e$ are well known to be
$$G_e=\left(\frac{1}{b+c-a}:\frac{1}{c+a-b}:\frac{1}{a+b-c}\right).$$
The barycentric coordinates for $U_a$ were given in Theorem~\ref{thm:coordUa}.
Using these coordinates and the condition for three points to be collinear
(formula (4) from \cite{Grozdev}), it is straightforward to confirm that $A$, $U_a$, and $G_e$ are collinear.
\end{proof}

\begin{open}
Is there a purely geometric proof for Theorem~\ref{thm:AUaGe}?
\end{open}

\begin{corollary}
The point we called $L'$ in Section~\ref{section:L}
(the intersection of $AL$ with $\gamma_a$ nearer $L$)
coincides with $U_a$, the point where the inner Apollonius circle touches $\gamma_a$.
\end{corollary}

\newpage

\begin{theorem}\label{thm:AVaGe}
Let $\gamma_a$, $\gamma_b$, and $\gamma_c$ be a general triad of circles associated
with $\triangle ABC$.
Let $C_o$ be the outer Apollonius circle of the circles in the triad.
Let $V_a$ be the point where $C_o$ touches $\gamma_a$. Define $V_b$ and $V_c$ similarly.
Then $A$, $V_a$, and $G_e$ are collinear (Figure~\ref{fig:AVaGe}).
Similarly, $B$, $V_b$, and $G_e$ are collinear; and $C$, $V_c$, and $G_e$ are collinear.
\end{theorem}

\begin{figure}[ht]
\centering
\includegraphics[scale=0.5]{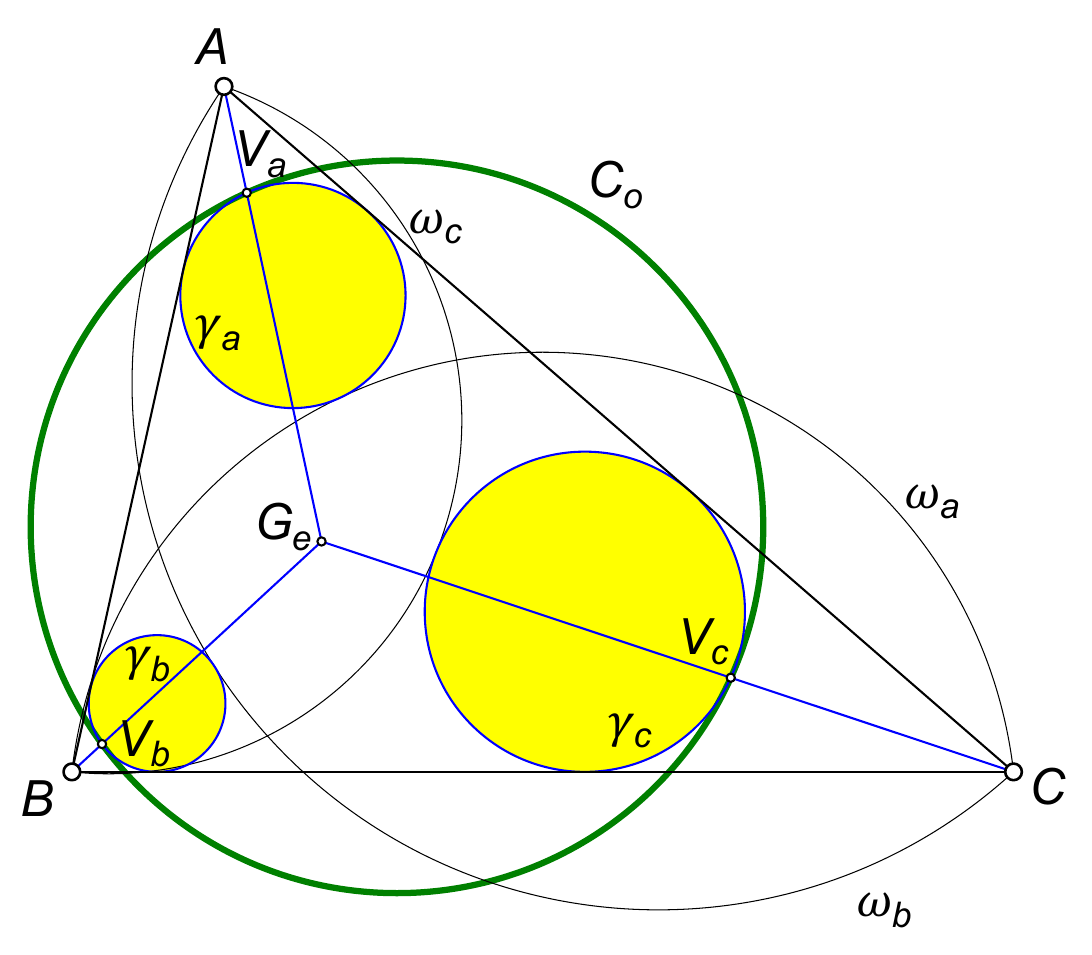}
\caption{lines concur at the Gergonne point}
\label{fig:AVaGe}
\end{figure}

\begin{proof}
By Lemma~\ref{lemma:Gergonne}, $U_aV_a$, $U_bV_b$, and $U_cV_c$ concur
at the radical center of $\gamma_a$, $\gamma_b$, and $\gamma_c$.
By Theorem~\ref{thm:radicalCenter}, this radical center is the Gergonne point of the triangle.
So $V_a$ lies on $G_eU_a$. 
By Theorem~\ref{thm:AUaGe}, $A$, $U_a$, and $G_e$ are collinear.
So $A$ lies on $G_eU_a$.
Since both $A$ and $V_a$ lie on $G_eU_a$,
we see that $V_a$ lies on $AU_a$.
Similarly, $V_b$ lies on $BU_b$ and $V_c$ lies on $CU_c$.
\end{proof}

\begin{corollary}
The point we called $X$ in Section~\ref{section:L}
(the intersection of $AL$ with $\gamma_a$ nearer $A$)
coincides with $V_a$, the point where the outer Apollonius circle touches $\gamma_a$.
\end{corollary}


Combining Theorems \ref{thm:AUaGe} and \ref{thm:AVaGe} lets us state the following result.

\begin{theorem}\label{thm:Ge}
Let $\gamma_a$, $\gamma_b$, and $\gamma_c$ be a general triad of circles associated
with $\triangle ABC$.
Let $C_i$ be the inner Apollonius circle of the circles in the triad.
Let $C_o$ be the outer Apollonius circle of the circles in the triad.
Let $U_a$ be the point where $C_i$ touches $\gamma_a$. Define $U_b$ and $U_c$ similarly.
Let $V_a$ be the point where $C_o$ touches $\gamma_a$. Define $V_b$ and $V_c$ similarly.
Then $A$, $V_a$, and $U_a$ are collinear.
Similarly, $B$, $V_b$, and $U_b$ and $C$, $V_c$, and $U_c$ are collinear.
The three lines meet at $G_e$, the Gergonne point of $\triangle ABC$
(Figure~\ref{fig:genGe}).
\end{theorem}

\begin{figure}[ht]
\centering
\includegraphics[scale=0.43]{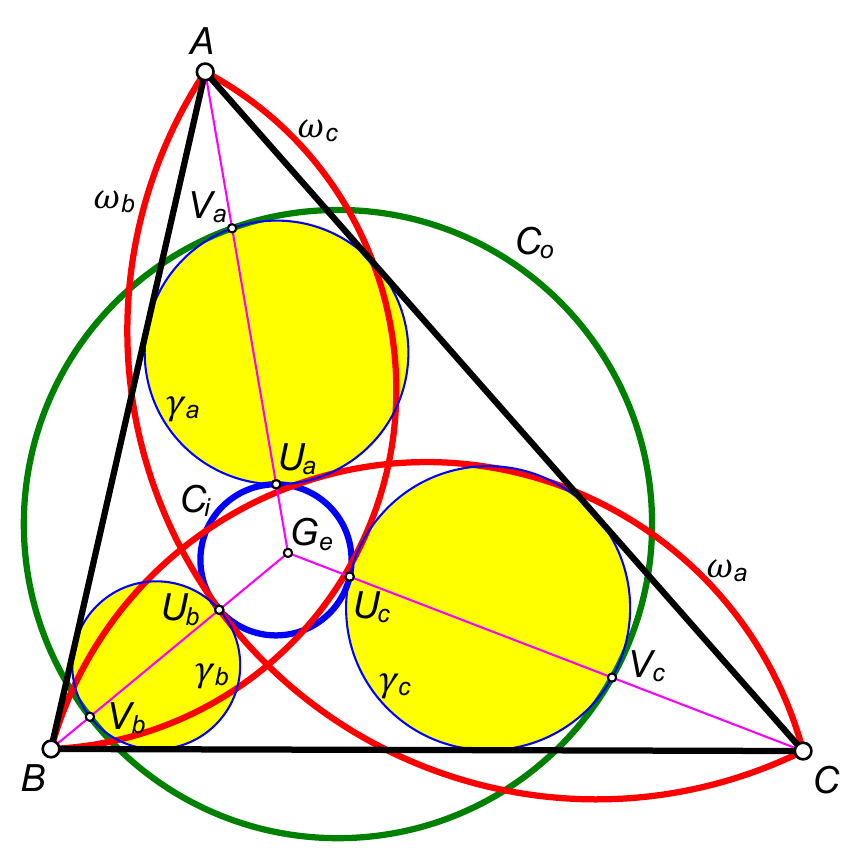}
\includegraphics[scale=0.43]{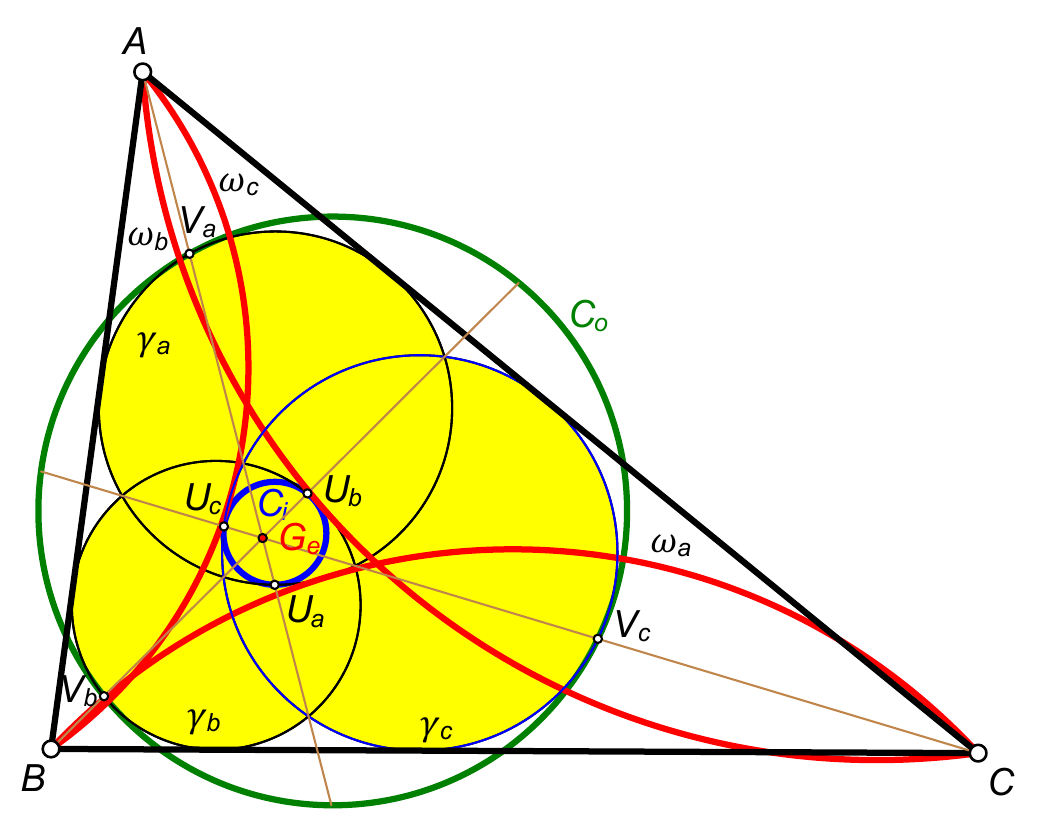}
\caption{lines concur at the Gergonne point}
\label{fig:genGe}
\end{figure}

\void{
\begin{figure}[ht]
\centering
\includegraphics[scale=0.4]{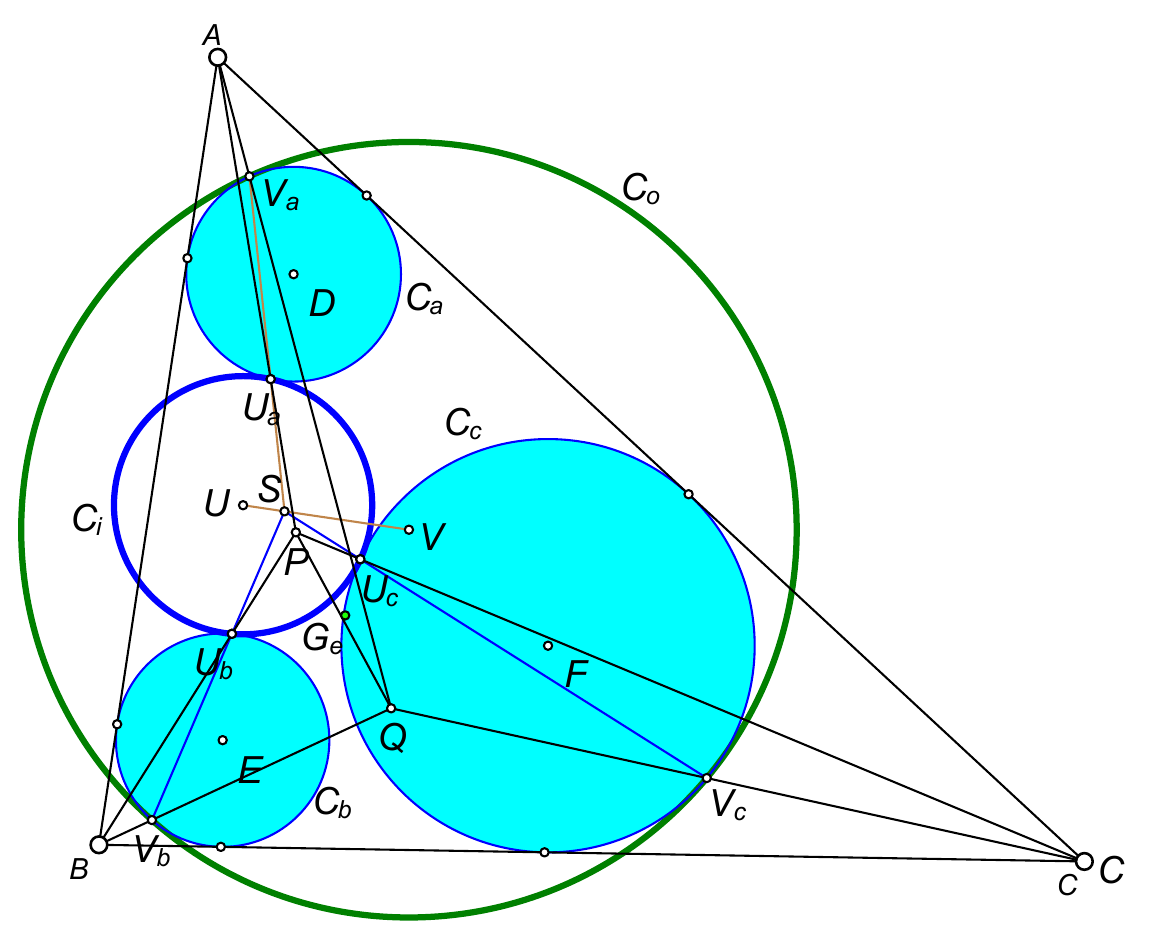}
\caption{counterexample to ``proof''}
\label{fig:counter}
\end{figure}
}

\newpage

\begin{theorem}\label{thm:Ua=L'}
Let $\gamma_a$, $\gamma_b$, and $\gamma_c$ be a general triad of circles associated
with $\triangle ABC$.
Let $C_i$ be the inner Apollonius circle of the circles in the triad.
Let $U_a$ be the point where $C_i$ touches $\gamma_a$.
Let $L$ be the point where the incircle of $\triangle ABC$ touches $BC$ (Figure~\ref{fig:Ua=L'}).
Then $A$, $U_a$, and $L$ are collinear.
\end{theorem}

\begin{figure}[ht]
\centering
\includegraphics[scale=0.4]{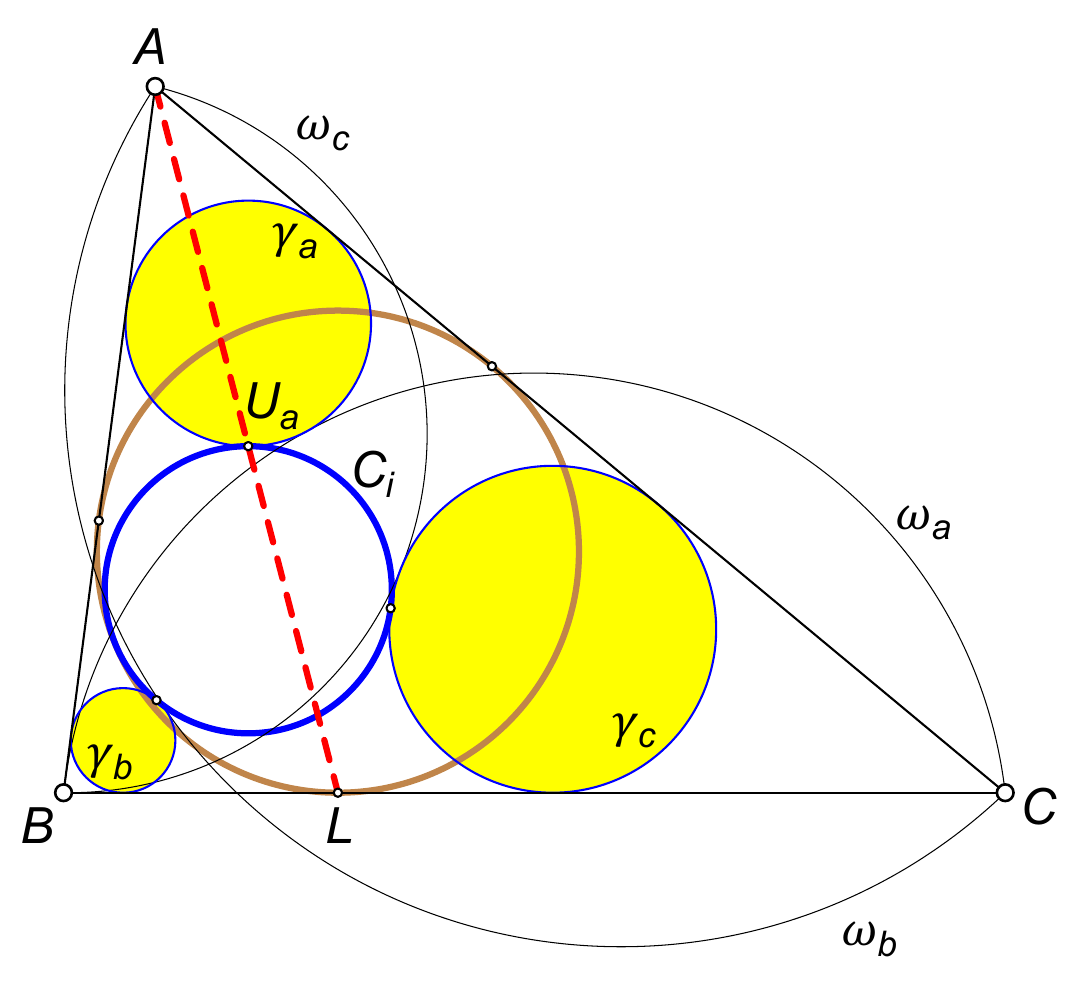}
\caption{$A$, $U_a$, and $L$ are collinear}
\label{fig:Ua=L'}
\end{figure}

\begin{proof}
By Theorem~\ref{thm:Ge}, $AU_a$ passes through $G_e$, the Gergonne point of $\triangle ABC$.
But by definition, $AL$ also passes through $G_e$. Thus, $AU_a$ coincides with $AL$.
\end{proof}

This gives us an easy way to construct the inner Apollonius circle of a general triad
of circles. Let the incircle of $\triangle ABC$ touch $BC$ at $L$.
Then $AL$ meets $\gamma_a$ (closer to $L$) at $U_a$.
Construct $U_b$ and $U_c$ in the same manner.
Then the circumcircle of $\triangle U_aU_bU_c$ is the inner Apollonius circle.

To construct the outer Apollonius circle, find the point $V_a$ where
$AL$ meets $\gamma_a$ (closer to $A$).
Construct $V_b$ and $V_c$ in the same manner.
Then the circumcircle of $\triangle V_aV_bV_c$ is the outer Apollonius circle.

\newpage

\begin{theorem}\label{thm:UaTriangle}
Let $\gamma_a$, $\gamma_b$, and $\gamma_c$ be a general triad of circles associated
with $\triangle ABC$.
Let $C_i$ be the inner Apollonius circle of the circles in the triad.
Let $U_a$ be the point where $C_i$ touches $\gamma_a$.
Let $t_a$ be the tangent to $\gamma_a$ at $U_a$.
Define $t_b$ and $t_c$ similarly (Figure~\ref{fig:UaTriangle}).
Then $t_a$, $t_b$, and $t_c$ form a triangle homothetic to $\triangle ABC$.
The center of the homothety is $G_e$, the Gergonne point of $\triangle ABC$.
\end{theorem}

\begin{figure}[ht]
\centering
\includegraphics[scale=0.35]{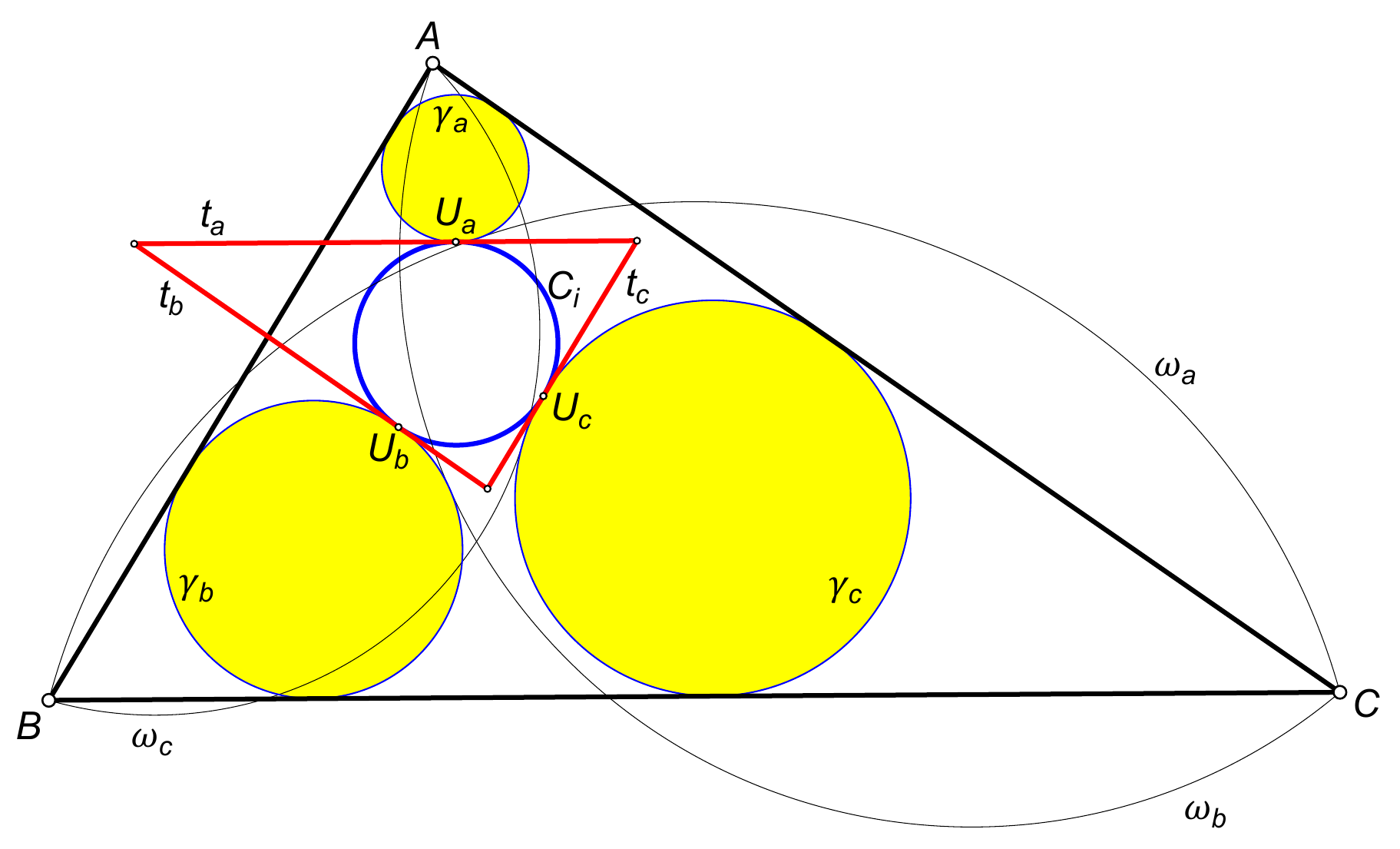}
\caption{red triangle is homothetic to $\triangle ABC$}
\label{fig:UaTriangle}
\end{figure}

\begin{proof}
By Theorem~\ref{thm:result27}, the tangent to $\gamma_a$ at $U_a$ is parallel to $BC$.
So $t_a\parallel BC$, $t_b\parallel CA$, and $t_c\parallel AB$.
Thus, the triangle formed by $t_a$, $t_b$, and $t_c$, is similar to $\triangle ABC$.
Let $A'$, $B'$, and $C'$ be the vertices of this triangle.
By a well-known theorem \cite[Art.~24]{Johnson}, this implies that $\triangle ABC$ is homothetic to $\triangle A'B'C'$
(with $A$ mapping to $A'$, $B$ to $B'$, and $C$ to $C'$).
Let $L$, $M$, and $N$ be the points where the incircle of $\triangle ABC$
touches $BC$, $CA$, and $AB$ (Figure~\ref{fig:incircleOfUaFig}).
\begin{figure}[ht]
\centering
\includegraphics[scale=0.35]{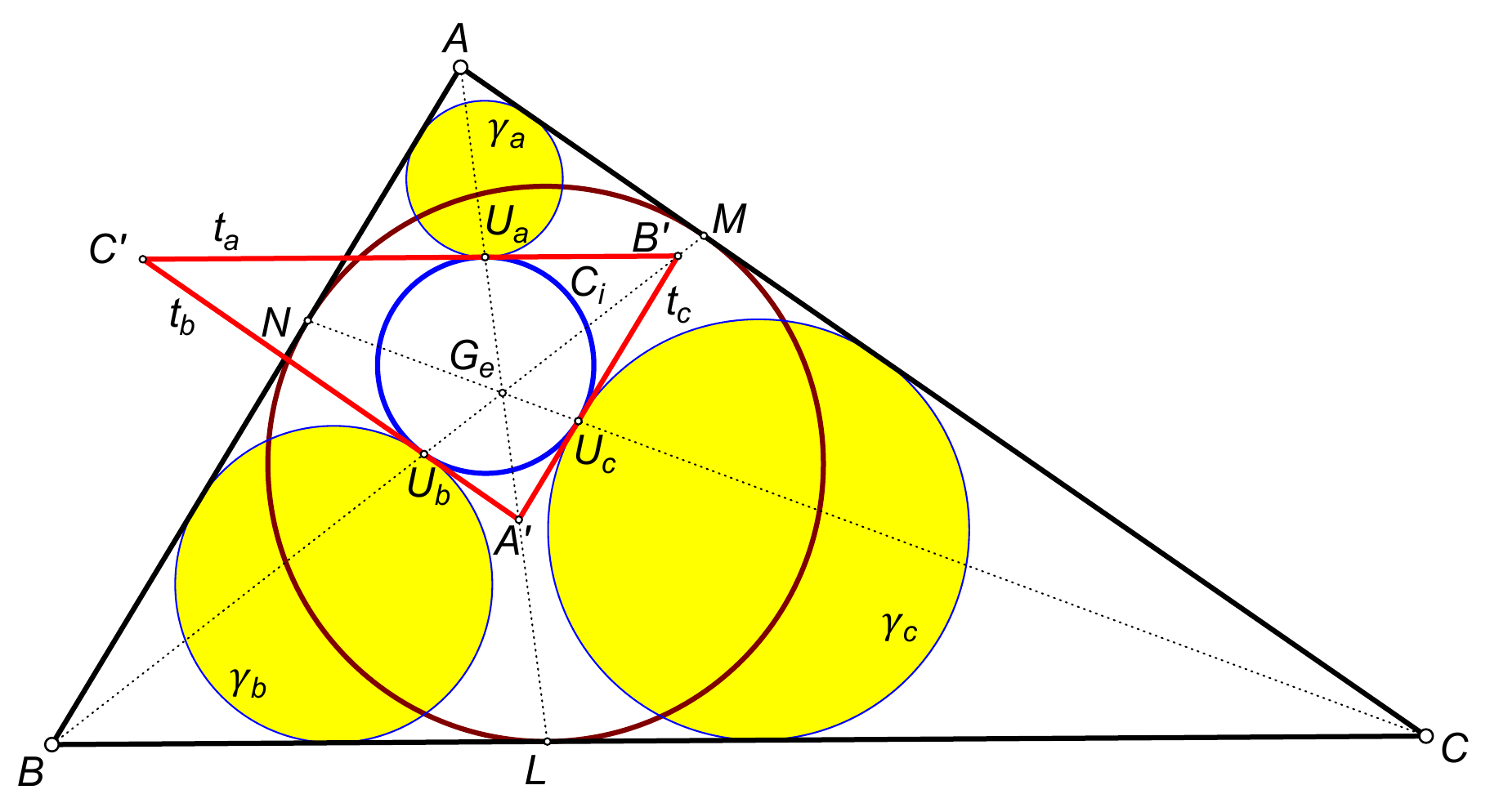}
\caption{}
\label{fig:incircleOfUaFig}
\end{figure}

The homothety maps the incircle of $\triangle ABC$ into the incircle of $\triangle A'B'C'$,
and the touch points into the touch points, i.e. $L$ maps to $U_a$, $M$ maps to $U_b$, and $N$ maps to $U_c$.
Hence, the center of the homothety is the point of concurrence of lines $LU_a$, $MU_b$, and $NU_c$.
By Theorem~\ref{thm:Ua=L'}, the line $LU_a$ coincides with the line $AL$, the line $MU_b$ coincides with $BM$,
and the line $NU_c$ coincides with line $CN$.
Therefore, the center of the homothety is $G_e$, the Gergonne point of $\triangle ABC$.
\void{
By Theorem~\ref{thm:Ge}, $AU_a$, $BU_b$, and $CU_c$ all pass through $G_e$.
Thus, $\triangle ABC$ is perspective with $\triangle A'B'C'$ with perspector $G_e$.
The circle $\odot LMN$ is the incircle of $\triangle ABC$ and the circle $C_i$
is the incircle of $\triangle A'B'C'$.
Consequently, a homothety with center $G_e$ and ratio of similitude $-BC/B'C'$ maps $\triangle A'B'C'$ into $\triangle ABC$
and also maps $C_i$ into $\odot LMN$. Point $U_a$ maps to $L$, $U_b$ maps to $M$, and $U_c$ maps to $N$.
}
\end{proof}

\void{
\textbf{Note.} The triangle formed by $t_a$, $t_b$, and $t_c$ is also homothetic
to the intouch triangle of $\triangle ABC$, with $G_e$ being the center of the homothety.
}

\newpage

\begin{corollary}\label{cor:intouch}
Let $\gamma_a$, $\gamma_b$, and $\gamma_c$ be a general triad of circles associated
with $\triangle ABC$.
Let $C_i$ be the inner Apollonius circle of the circles in the triad.
Then $C_i$ and the incircle of $\triangle ABC$ are homothetic
with $G_e$ as the center of the homothety (Figure~\ref{fig:intouch}).
\end{corollary}

\begin{figure}[ht]
\centering
\includegraphics[scale=0.4]{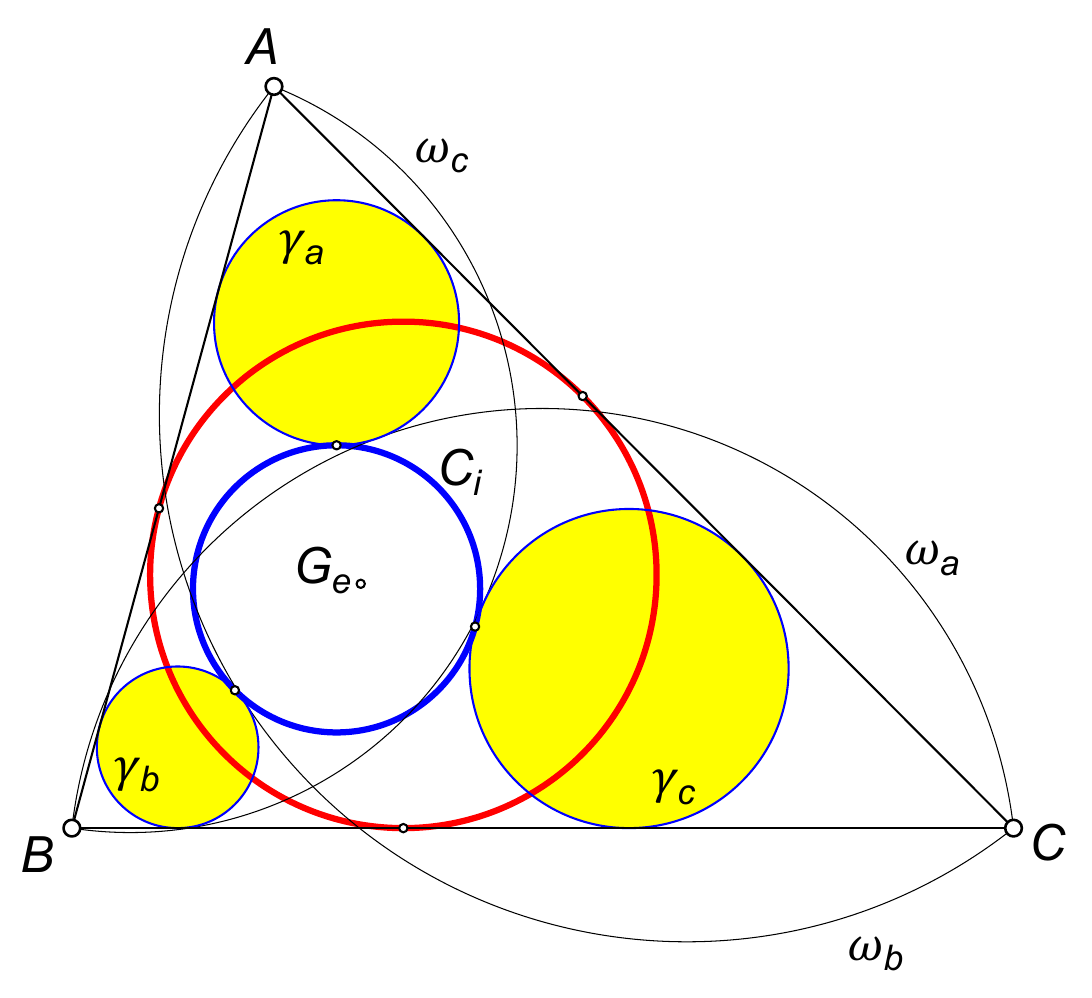}
\caption{$G_e$ is the center of similarity between the incircle and $C_i$}
\label{fig:intouch}
\end{figure}

\begin{theorem}\label{thm:GeUI}
For a general triad of circles associated with $\triangle ABC$,
let $U(\rho_i)$ be the inner Apollonius circle of $\gamma_a$, $\gamma_b$, and $\gamma_c$
(Figure~\ref{fig:rhoir}).
Let $G_e$ be the Gergonne point of $\triangle ABC$.
Then
$$\frac{G_eU}{G_eI}=\frac{\rho_i}{r}.$$
\end{theorem}

\begin{figure}[ht]
\centering
\includegraphics[scale=0.45]{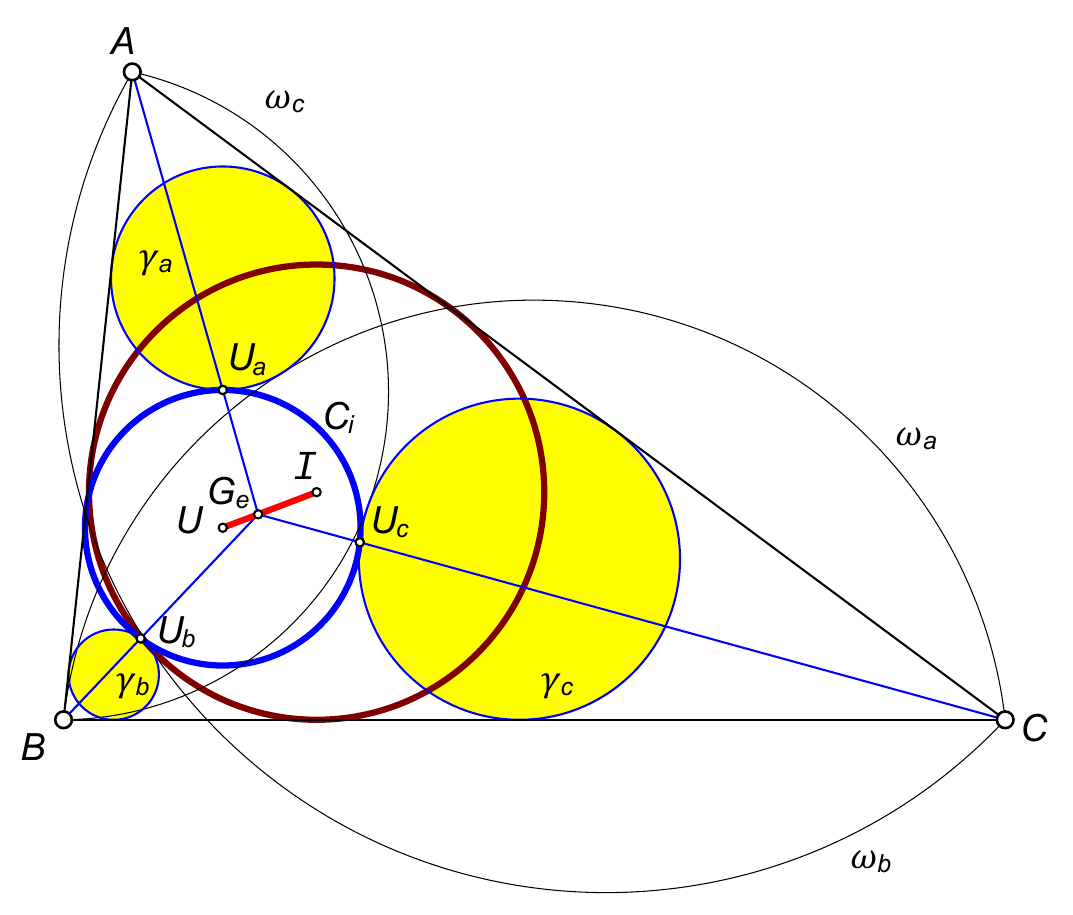}
\caption{$G_eU/G_eI=\rho_i/r$}
\label{fig:rhoir}
\end{figure}

\begin{proof}
Let the touch points of circle $(U)$ with $\gamma_a$, $\gamma_b$, and $\gamma_c$ be
$U_a$, $U_b$, and $U_c$, respectively.
By Lemma~\ref{lemma:triConcur}, $AU_a$, $BU_b$, and $CU_c$ concur at a point $P$
that is a center of similitude of circle $(U)$ and the incircle, $(I)$.
By Theorem~\ref{thm:Ge}, $P=G_e$.
\void{
Corresponding distances in two similar figures are in proportion to their ratio of similitude.
In this case, the two similar figures are circles $(U)$ and $(I)$. Their ratio of similitude is $\rho_i/r$.
Thus, $G_eU:G_eI=\rho_i:r$.
}
Note that $(U)$ and $(I)$ are two circles with center of similarity $G_e$.
By Lemma~\ref{lemma:twoCircles}, $G_eU/G_eI=\rho_i/r$.
\end{proof}


\begin{theorem}\label{thm:GeVI}
For a general triad of circles associated with $\triangle ABC$,
let $V(\rho_o)$ be the outer Apollonius circle of $\gamma_a$, $\gamma_b$, and $\gamma_c$ (Figure~\ref{fig:rhoor}).
Let $G_e$ be the Gergonne point of $\triangle ABC$.
Then
$$\frac{G_eV}{G_eI}=\frac{\rho_o}{r}.$$
\end{theorem}

\begin{figure}[ht]
\centering
\includegraphics[scale=0.4]{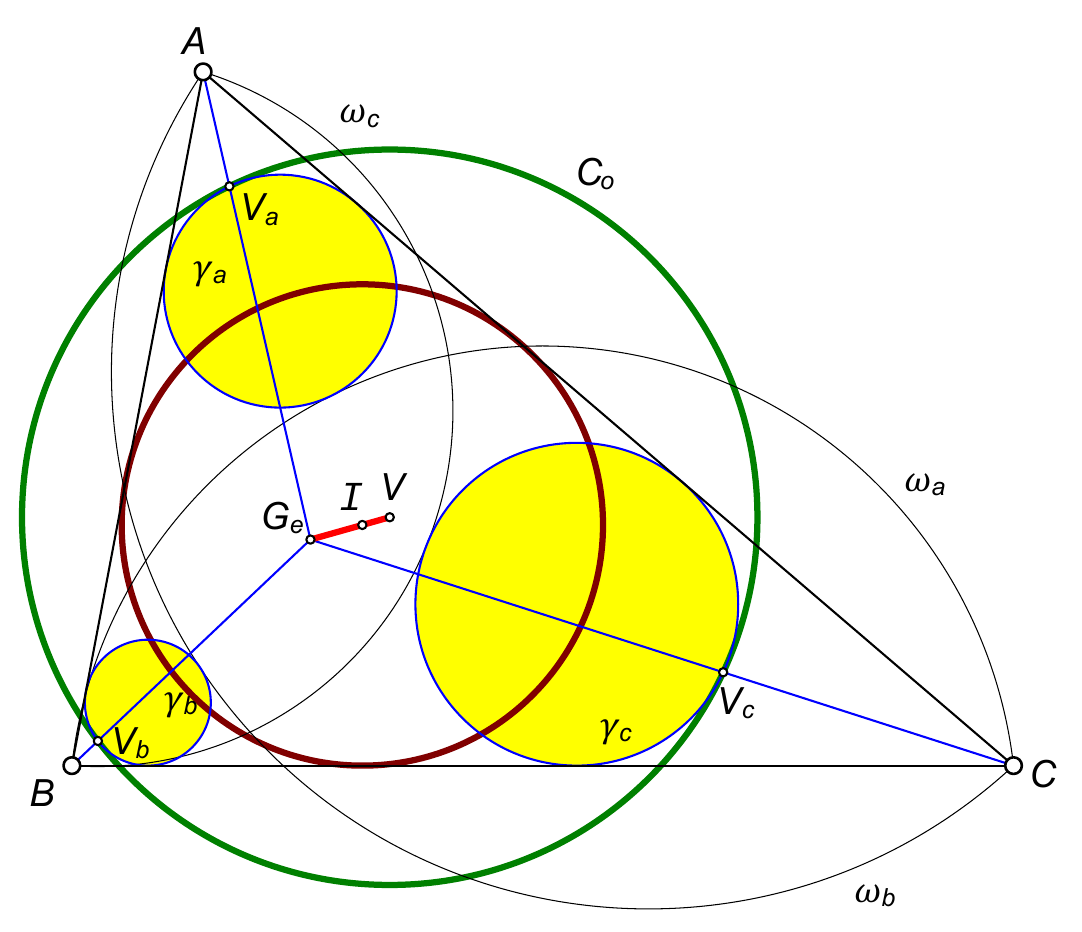}
\caption{$G_eV/G_eI=\rho_o/r$}
\label{fig:rhoor}
\end{figure}

\begin{proof}
The proof is the same as the proof of Theorem~\ref{thm:GeUI}.
\end{proof}


\begin{theorem}
For a general triad of circles associated with $\triangle ABC$,
let $U(\rho_i)$ and $V(\rho_o)$ be the inner and outer
Apollonius circles of $\gamma_a$, $\gamma_b$, and $\gamma_c$, respectively.
Let $G_e$ be the Gergonne point of $\triangle ABC$.
Then
$$\frac{G_eV}{G_eU}=\frac{\rho_o}{\rho_i}.$$
\end{theorem}

\begin{figure}[ht]
\centering
\includegraphics[scale=0.45]{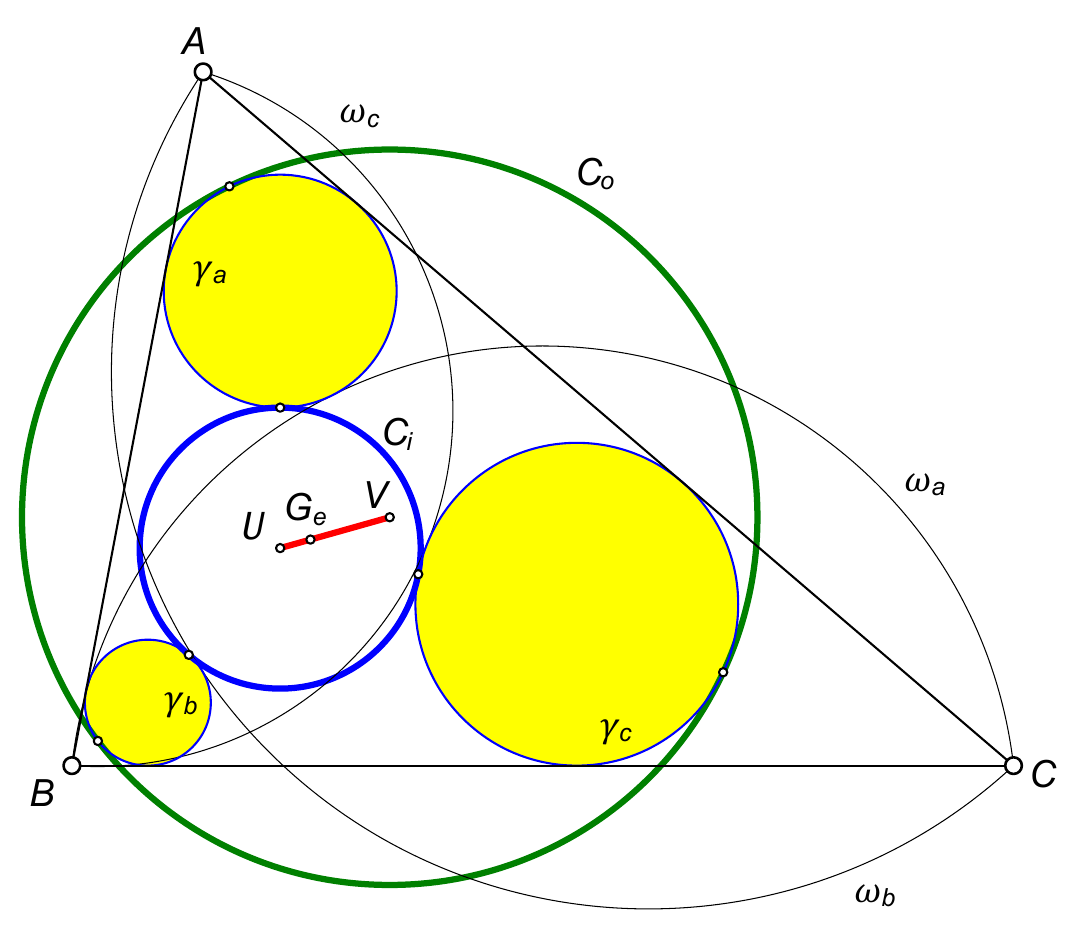}
\caption{$G_eV/G_eU=\rho_o/\rho_i$}
\label{fig:GEUI}
\end{figure}

\begin{proof}
This follows from Theorem~\ref{thm:oi}.
It also follows from Theorems \ref{thm:GeUI} and \ref{thm:GeVI}.
\end{proof}


\begin{theorem}
Let $\gamma_a$, $\gamma_b$, and $\gamma_c$ be
a general triad of circles associated with triangle $\triangle ABC$.
A circle externally tangent to each circle of the triad touches $\gamma_a$ at $U_a$.
Then the tangents from $U_a$ to $\gamma_b$ and $\gamma_c$ have the same length (Figure~\ref{fig:equalTangents}).
\end{theorem}

\begin{figure}[ht]
\centering
\includegraphics[scale=0.34]{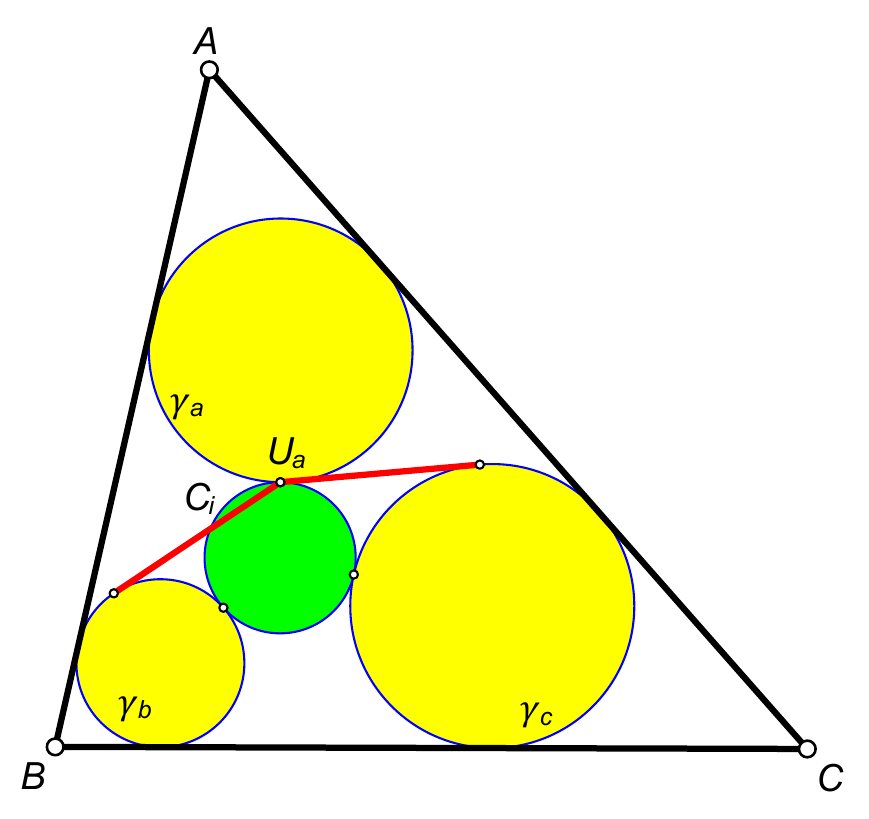}
\caption{red tangent lengths are equal}
\label{fig:equalTangents}
\end{figure}

\begin{proof}
By Theorem~\ref{thm:Ge}, $U_a$ lies on the Gergonne cevian from vertex $A$.
By Theorem~\ref{radixAxis}, this Gergonne cevian is the radical axis of circles $\gamma_b$ and $\gamma_c$.
Thus, the two tangents have the same length.
\end{proof}


\begin{theorem}[Miyamoto Analog]\label{thm:tangentCircles}
For a general triad of circles associated with $\triangle ABC$,
the inner Apollonius circle of $\gamma_a$, $\gamma_b$, $\gamma_c$ (blue circle in Figure~\ref{fig:Miyamoto}), is internally tangent
to the inner Apollonius circle of $\omega_a$, $\omega_b$, $\omega_c$ (green circle in Figure~\ref{fig:Miyamoto}).
\end{theorem}

\begin{figure}[ht]
\centering
\includegraphics[scale=0.4]{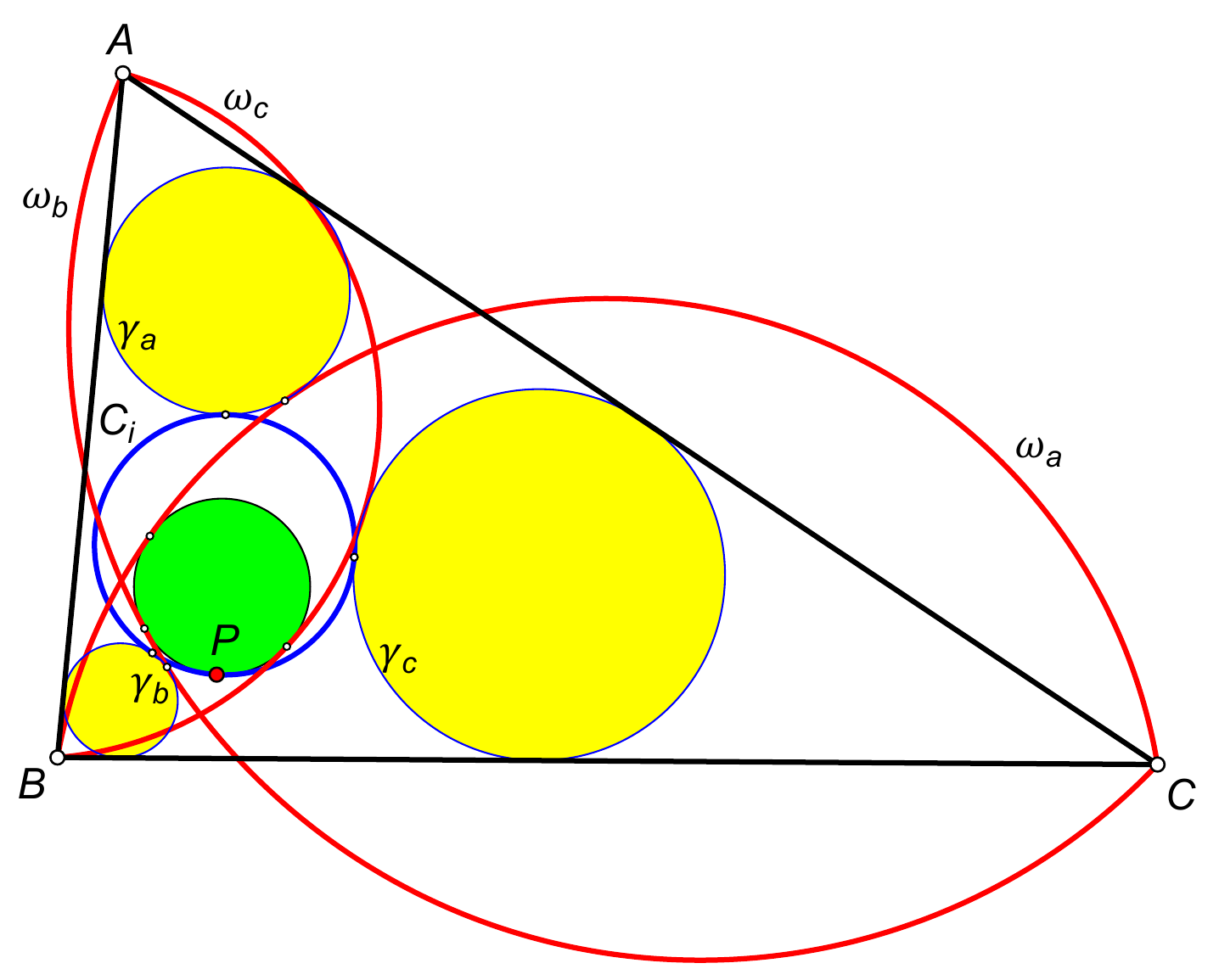}
\caption{green and blue circles touch at $P$}
\label{fig:Miyamoto}
\end{figure}

\begin{proof}
This is a special case of the following theorem which is stated in \cite{ETC52805}.
\end{proof}

\newpage

\begin{theorem}[Miyamoto Generalization]\label{thm:MiyamotoGen}
Let $\omega_a$ be any arc erected internally on side $BC$ of $\triangle ABC$.
Let $\gamma_a$ be the circle that is inside $\triangle ABC$, tangent to $AB$ and $AC$,
and tangent externally to $\omega_a$. Define $\omega_b$, $\omega_c$, $\gamma_b$, and $\gamma_c$
similarly.
Then the inner Apollonius circle of $\gamma_a$, $\gamma_b$, $\gamma_c$ (blue circle in Figure~\ref{fig:MiyamotoGen}), is internally tangent
to the inner Apollonius circle of $\omega_a$, $\omega_b$, $\omega_c$ (green circle in Figure~\ref{fig:MiyamotoGen}).
\end{theorem}

\begin{figure}[ht]
\centering
\includegraphics[scale=0.5]{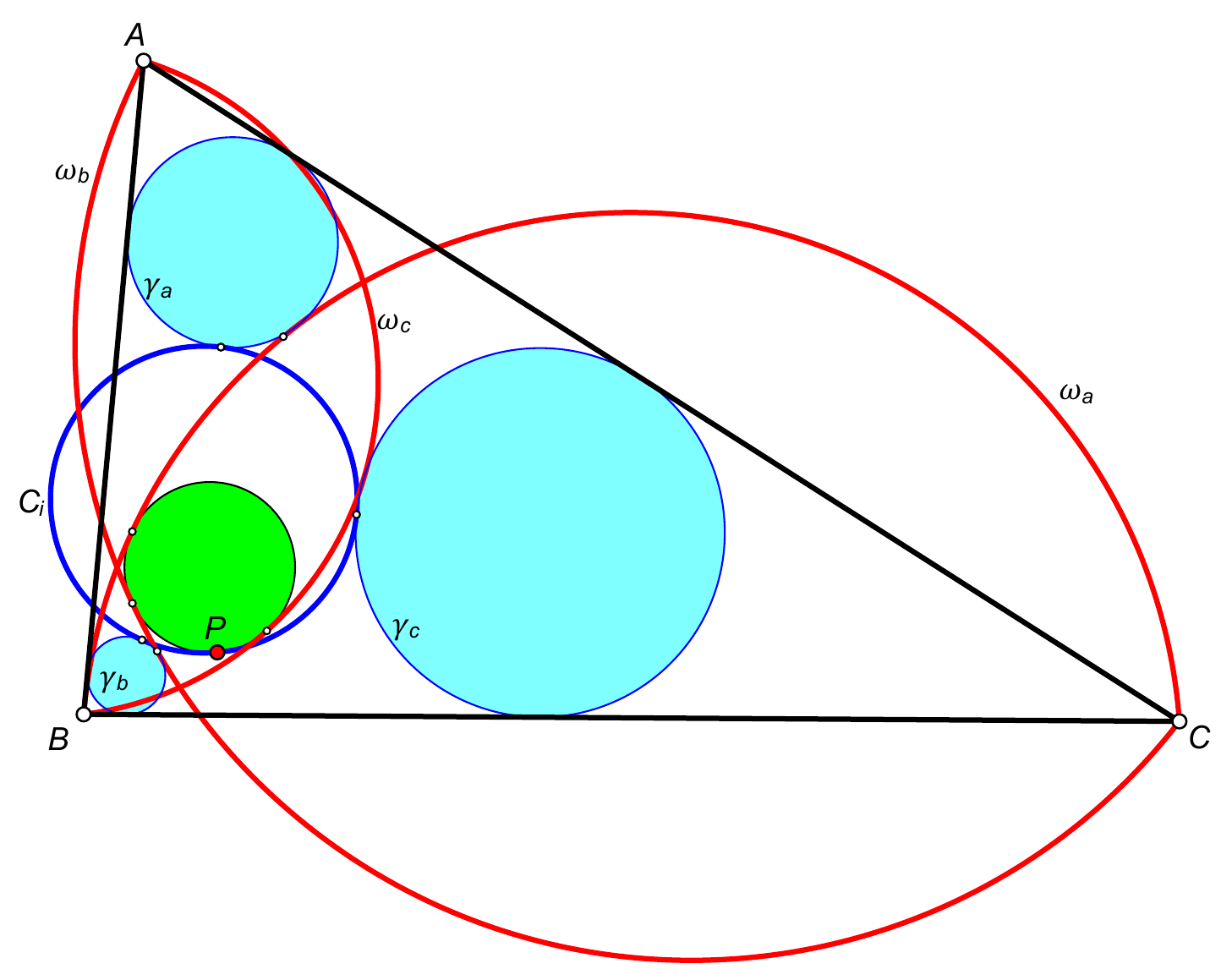}
\caption{green and blue circles touch at $P$. Red arcs have different angular measures.}
\label{fig:MiyamotoGen}
\end{figure}

\void{
Alternatively, from Theorem~\ref{thm:GeUI}
$$\frac{G_eU}{G_eI}=\frac{\rho_i}{r}.$$
Similarly,
$$\frac{G_eV}{G_eI}=\frac{\rho_o}{r}.$$
Dividing gives
$$\frac{G_eV}{G_eU}=\frac{\rho_o}{\rho_i}.$$
}


\begin{lemma}\label{lemma:rW}
We have
$$r\W=\frac{2 a b+2 b c+2 ca-a^2-b^2-c^2}{2(a+b+c)}.$$
\end{lemma}

\begin{proof}
Recall that $\W=(4R+r)/p$.
We use the well-known identities $r=\Delta/p$, $R=abc/(4\Delta)$,
and $\Delta=\sqrt{p(p-a)(p-b)(p-c)}$.
Then we have
\begin{align*}
r\W&=r\left(\frac{4R+r}{p}\right)\\
&=\left(\frac{\Delta}{p}\right)\left(4\cdot\frac{abc}{4\Delta}+\frac{\Delta}{p}\right)\Bigm/p\\
&=\left(\frac{abc}{p}+\frac{\Delta^2}{p^2}\right)\Bigm/p\\
&=\frac{1}{p}\left(\frac{abc}{p}+\frac{p(p-a)(p-b)(p-c)}{p^2}\right).
\end{align*}
Letting $p=(a+b+c)/2$ and simplifying, gives
\begin{equation*}
r\W=\frac{2 a b+2 b c+2 ca-a^2-b^2-c^2}{2 (a+b+c)}.\qedhere
\end{equation*}
\end{proof}

\newpage

\begin{lemma}[Length of $AG_e$]\label{lemma:AGe}
We have
$$AG_e=\frac{(p-a)\sqrt{a(p-a)[ap-(b-c)^2]}}{pr\W}.$$
\end{lemma}

\begin{proof}
The distance from a vertex of a triangle to its Gergonne point is known. From Property 2.1.1 in \cite{GergonneFormulary}, we have
\begin{equation}
AG_e=\frac{(b+c-a)\sqrt{a(b+c-a)[2ap-2(b-c)^2]}}{2 a b+2 b c+2 ca-a^2-b^2-c^2}.
\end{equation}
From Lemma~\ref{lemma:rW}, this can be written as
\begin{equation*}
AG_e=\frac{(b+c-a)\sqrt{a(b+c-a)[2ap-2(b-c)^2]}}{4pr\W}.
\end{equation*}
Noting that $b+c-a=2(p-a)$, gives us our result.
\end{proof}

\void{

\begin{lemma}\label{lemma:w}
We have
$$r^2p=(p-a)(a^2-ap+pr\W).$$
\end{lemma}

\begin{proof}
Since $W=(4R+r)/p$, we find
\begin{align*}
(p-a)(a^2-ap+pr\W)&=(p-a)[a^2-ap+r(4R+r)]\\
&=(p-a)\left[a^2-ap+r\left(\frac{abc}{K}+r\right)\right]\\
&=(p-a)\left[a^2-ap+r\left(\frac{abc}{rp}+\frac{K}{p}\right)\right]\\
&=(p-a)\left[a^2-ap+\frac{abc}{p}+\frac{rK}{p}\right]\\
&=(p-a)\left[a^2-ap+\frac{abc}{p}+K^2\right]\\
&=(p-a)\left[a^2-ap+\frac{abc}{p}+p(p-a)(p-b)(p-c)\right]\\
&=\frac18(a-b+c)(a+b-c)(b+c-a)\\
&=\frac18(a+b+c)(a-b+c)(a+b-c)(b+c-a)/(2p)\\
&=K^2/p\\
&=(rp)^2/p\\
&=r^2p\qedhere
\end{align*}
\end{proof}

}

\begin{lemma}\label{lemma:GeRatios}
Let the touch points of the incircle of $\triangle ABC$ with its sides be $L$, $M$, and $N$,
as shown in Figure~\ref{fig:GeRatios}.
Then
$$\frac{LG_e}{AG_e}=\frac{(p-b)(p-c)}{a(p-a)}.$$
\end{lemma}

\begin{figure}[ht]
\centering
\includegraphics[scale=0.5]{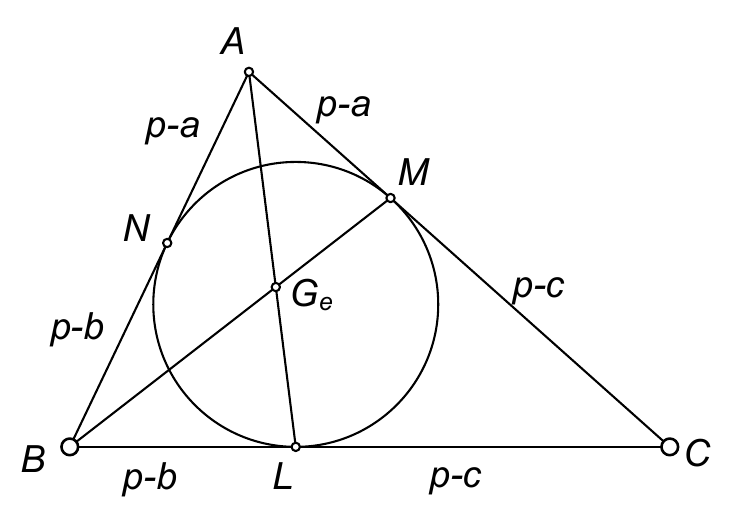}
\caption{}
\label{fig:GeRatios}
\end{figure}

\begin{proof}
By definition, $G_e$ is the intersection of $AL$ and $BM$.
Applying Menelaus' Theorem to $\triangle ALC$ with transversal $BM$ gives
$$AG_e\cdot BL\cdot CM=LG_e\cdot BC\cdot AM$$
or
$$(AG_e)(p-b)(p-c)=(LG_e)(a)(p-a)$$
which is equivalent to our desired result.
\end{proof}

\begin{corollary}\label{cor:GeRatios}
With the same terminology,
$$\frac{AL}{AG_e}=\frac{AG_e+LG_e}{AG_e}=1+\frac{LG_e}{AG_e}=1+\frac{(p-b)(p-c)}{a(p-a)}.$$
\end{corollary}


\begin{lemma}\label{lemma:w2}
We have
$$\W=\frac{r}{p-a}\left(\frac{a(p-a)}{(p-b)(p-c)}+1\right).$$
\end{lemma}

\begin{proof}
Using the well known formulas $\Delta^2=p(p-a)(p-b)(p-c)$ and $r=\Delta/p$, we get
\begin{align*}
\frac{r}{p-a}\left(\frac{a(p-a)}{(p-b)(p-c)}+1\right)&=\frac{ar}{(p-b)(p-c)}+\frac{r}{p-a}\\
&=r\cdot \frac{a(p-a)+(p-b)(p-c)}{(p-a)(p-b)(p-c)}\\
&=\frac{\Delta}{p}\cdot \frac{a(p-a)+(p-b)(p-c)}{(p-a)(p-b)(p-c)}\\
&=\frac{a(p-a)+(p-b)(p-c)}{\Delta}\\
&=\frac{a\cdot \frac{b+c-a}{2}+\frac{(a-b+c)(a+b-c)}{4}}{\Delta}\\
&=\frac{2ab+2bc+2ac-a^2-b^2-c^2}{4\Delta}\\
&=\frac{4pr\W}{4\Delta}\hspace{1.9in}\hbox{(by Lemma~\ref{lemma:rW})}\\
&=\W.\qedhere
\end{align*}
																		
\void{
Expanding the right side, we get
\begin{align*}
\frac{r}{p-a}\left(\frac{a(p-a)}{(p-b)(p-c)}+1\right)&=\frac{ar}{(p-b)(p-c)}+\frac{r}{p-a}\\
&=\frac{ar(p-a)+r(p-b)(p-c)}{(p-a)(p-b)(p-c)}\\
&=\frac{pr[a(p-a)+(p-b)(p-c)]}{p(p-a)(p-b)(p-c)}\\
&=\frac{pr[ap-a^2+p^2-pb-pc+bc]}{K^2}\\
&=\frac{pr[p^2-a^2+ap-pb-pc+bc]}{K^2}\\
&=\frac{pr[p^2-a^2+p(b+c-a)+bc]}{K^2}\\
&=\frac{pr[p^2-a^2+2p(p-a)+bc]}{K^2}\\
\end{align*}
Expanding the left side, we get
\begin{align*}
\W=\frac{4R+r}{p}=\frac{\frac{abc}{K}+r}{p}=\frac{abc+rK}{pK}=\frac{\frac{abc}{K}+r}{p}=\frac{abc+r\frac{K}{p}}{pK}
\end{align*}
}
\end{proof}


\begin{theorem}[Radius of Inner Apollonius Circle]\label{thm:rhoi}
Let $\rho_i$ be the radius of the inner Apollonius circle of $\gamma_a$, $\gamma_b$, and $\gamma_c$.
Then
$$\rho_i=rt\W-r.$$
\end{theorem}

\begin{proof}
By Corollary~\ref{cor:intouch}, $C_i$ and the incircle are homothetic
with $G_e$ being the external center of similitude.
Under this homothety, $U_a$ maps to $L$.
Thus
$$\frac{G_eU_a}{G_eL}=\frac{\rho_i}{r}.$$
We can write this as
$$\frac{\rho_i}{r}=\frac{AG_e-AU_a}{AL-AG_e}=\frac{AG_e-AL\cdot\frac{\rho_a}{r}}{AL-AG_e}
=\frac{1-\frac{AL}{AG_e}\cdot\frac{\rho_a}{r}}{\frac{AL}{AG_e}-1}$$
because $AU_a=AL\cdot\frac{\rho_a}{r}$ (from Theorem~\ref{thm:AL'L}).
Now $$\frac{AL}{AG_e}=\frac{AG_e+LG_e}{AG_e}=1+\frac{LG_e}{AG_e}.$$
From Lemma~\ref{lemma:GeRatios}, we have
$$\frac{LG_e}{AG_e}=\frac{(p-b)(p-c)}{a(p-a)}$$
so
$$\frac{\rho_i}{r}=\frac{1-(1+\frac{(p-b)(p-c)}{a(p-a)})\cdot\frac{\rho_a}{r}}{\frac{(p-b)(p-c)}{a(p-a)}}$$
which is equivalent to
$$\frac{\rho_i}{r}+1=\left(1-\frac{\rho_a}{r}\right)\left(\frac{a(p-a)}{(p-b)(p-c)}+1\right).$$

We also know that
$$1-\frac{\rho_a}{r}=\frac{rt}{p-a}$$
from Corollary~\ref{cor:rho}.
Thus
$$\frac{\rho_i}{r}+1=\frac{rt}{p-a}\left(\frac{a(p-a)}{(p-b)(p-c)}+1\right).$$
By Lemma~\ref{lemma:w2}, this reduces to
\begin{equation*}
\frac{\rho_i}{r}+1=t\W,
\end{equation*}
so $\rho_i/r=t\W-1$ or $\rho_i=rt\W-r$.
\end{proof}

\void{
\begin{proof}
By Corollary~\ref{cor:intouch}, $C_i$ and the incircle are homothetic
with $G_e$ being the external center of similitude.
Under this homothety, $U_a$ maps to $L$.
Thus
$$\frac{G_eU_a}{G_eL}=\frac{\rho_i}{r}.$$
We can write this as
$$\frac{\rho_i}{r}=\frac{AG_e-AU_a}{AL-AG_e}=\frac{AG_e-AL\cdot\frac{\rho_a}{r}}{AL-AG_e}$$
because $AU_a=AL\cdot\frac{\rho_a}{r}$.
The value of $AG_e$ is known from Lemma~\ref{lemma:AGe} and
the value of $AL$ is known from equation (\ref{eq:AL}).
Thus,
$$\frac{\rho_i}{r}=\frac{x-y(\rho_a/r)}{y-x}=\frac{x/y-\rho_a/r}{1-x/y}$$
where
$$x=AG_e=\frac{(p-a)\sqrt{a(p-a)[ap-(b-c)^2]}}{pr\W}$$
and
$$y=AL=\sqrt{\frac{(p-a)[ap-(b-c)^2]}{a}}.$$
Because of common factors inside the radicals, the expression $x/y$ simplifies.
$$\frac{x}{y}=\frac{a(p-a)}{pr\W}$$
We also know that
$$\frac{\rho_a}{r}=1-\frac{rt}{p-a}$$
from Corollary~\ref{cor:rho}.
Putting these all together gives us
$$\frac{\rho_i}{r}=\frac{k-(1-\frac{rt}{p-a})}{1-k}=-1+\frac{1}{1-k}\cdot\frac{rt}{p-a}$$
where $k=x/y$.
Replacing $k$ by $a(p-a)/(pr\W)$, we get
$$\frac{\rho_i}{r}=-1+\frac{1}{1-a(p-a)/(pr\W)}\cdot\frac{rt}{p-a}$$
or
$$\frac{\rho_i}{r}+1=\frac{r^2pt\W}{(p-a)(a^2-ap+pr\W)}.$$
By Lemma~\ref{lemma:w}, this reduces to
\begin{equation*}
\frac{\rho_i}{r}+1=t\W,
\end{equation*}
so $\rho_i/r=t\W-1$.
\end{proof}
}

\void{
\note{original proof}
\begin{proof}
The radius of the inner Apollonius circle is equal to the distance between
points $U_a$ and $U$. The barycentric coordinates for these points
were found in Theorems \ref{thm:coordUa} and \ref{thm:coordU}, respectively.
Using the formula for the distance between two points, we find that
\begin{equation}
r_i= \frac{\left(2 a b+2 b c+2 ca-a^2-b^2-c^2\right)t-2S}{2 (a+b+c)}
\end{equation}
where $S$ is twice the area of $\triangle ABC$ and $t=\tan(\theta/4)$.
The computation was performed using \textsc{Mathematica} and the details are omitted.
To put the result into the required form, first note that
$2S/(2(a+b+c))=4\Delta/(4p)=\Delta/p=r$.
Second, observe that, by Lemma~\ref{lemma:rW},
\begin{equation*}
r\W=\frac{2 a b+2 b c+2 ca-a^2-b^2-c^2}{2 (a+b+c)},
\end{equation*}
which completes the proof.
\end{proof}
}

Note that $r_i$ will be negative if the inner Apollonius circle is internally tangent
to $\gamma_a$, $\gamma_b$, and $\gamma_c$. This will happen if $t\W<1$.

\begin{open}
Is there a simpler proof of Theorem~\ref{thm:rhoi}?
\end{open}

\begin{corollary}
We have
$$\rho_i=2r-(\rho_a+\rho_b+\rho_c).$$
\end{corollary}

\begin{proof}
From Theorem~\ref{thm:1}, we have
$rt\W=3r-(\rho_a+\rho_b+\rho_c).$
Therefore, we have
$\rho_i=rt\W-r=3r-(\rho_a+\rho_b+\rho_c)-r=2r-(\rho_a+\rho_b+\rho_c)$.
\end{proof}


\begin{theorem}
The circles $\gamma_a$, $\gamma_b$, and $\gamma_c$ meet in a point (Figure~\ref{fig:tripleConcurGamma})
if and only if $t=1/\W$.
\end{theorem}

\begin{figure}[ht]
\centering
\includegraphics[scale=0.55]{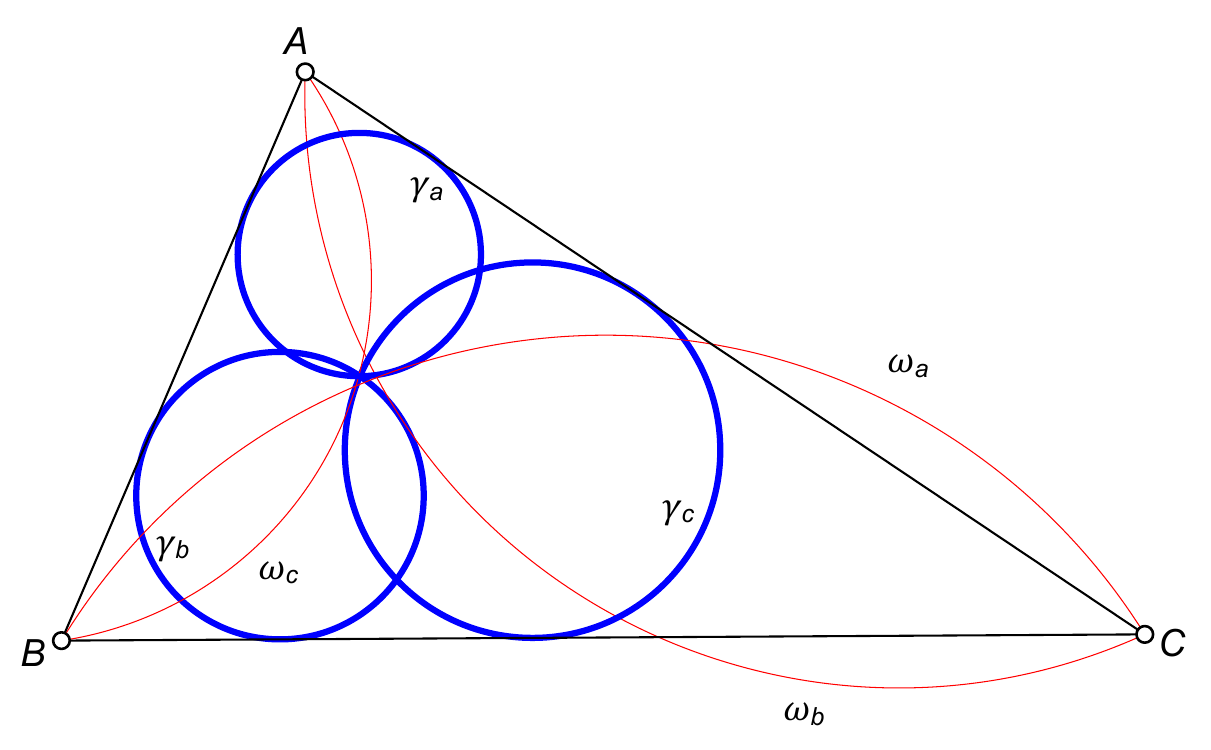}
\caption{$\gamma_a$, $\gamma_b$, $\gamma_c$ concur}
\label{fig:tripleConcurGamma}
\end{figure}

\begin{proof}
The three circles concur if and only if the radius of the inner Apollonius circle is 0, that is, when $\rho_i=0$.
By Theorem~\ref{thm:rhoi}, $\rho_i=rt\W-r.$ So $\rho_i=0$ if and only if $r=rt\W$ or $1=t\W$ since $r>0$.
In other words, when $t=1/\W$.
\end{proof}

\begin{corollary}
If $t=1/\W$, the circles $\gamma_a$, $\gamma_b$, and $\gamma_c$  all pass through
$G_e$, the Gergonne point of $\triangle ABC$.
\end{corollary}

\begin{proof}The common chord of each pair of circles is the radical axis of those two circles.
Since the three common chords meet at the point of concurrence of the three circles,
this point must be the radical center of the three circles.
By Theorem~\ref{thm:radicalCenter}, this is the Gergonne point of $\triangle ABC$.
\end{proof}

\begin{theorem}
The circles $\omega_a$, $\omega_b$, and $\omega_c$ meet in a point (Figure~\ref{fig:tripleConcurOmega})
if and only if $\theta=120\degrees$.
\end{theorem}

\begin{figure}[ht]
\centering
\includegraphics[scale=0.35]{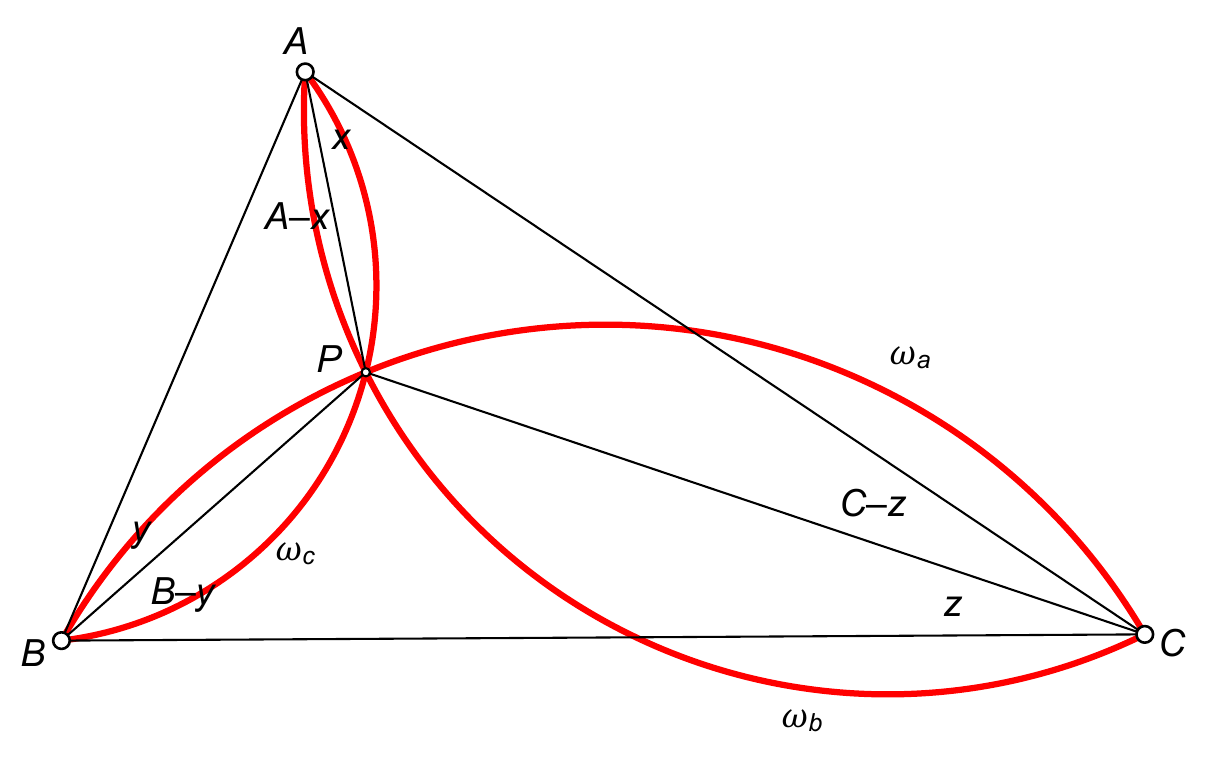}
\caption{$\omega_a$, $\omega_b$, $\omega_c$ concur}
\label{fig:tripleConcurOmega}
\end{figure}

\begin{proof}
Suppose the three arcs meet at $P$.
Let $\angle PAC=x$, $\angle PBA=y$, and $\angle PCB=z$.
Then $\angle BAP=A-x$, $\angle CBP=B-y$, and $\angle ACP=C-z$.
An angle inscribed in a circle is measured by half its intercepted arc.
So
\begin{align*}
2(A-x)+2y&=\theta,\\
2(B-y)+2z&=\theta,\\
2(C-z)+2x&=\theta.
\end{align*}
Adding these three equations gives
$$3\theta=2(A+B+C)=2(180\degrees)=360\degrees$$
or $\theta=120\degrees$.
\end{proof}


\begin{theorem}[Radius of Outer Apollonius Circle]\label{thm:rhoo}
Let $\rho_o$ be the radius of the outer Apollonius circle of $\gamma_a$, $\gamma_b$, and $\gamma_c$.
Then
$$\rho_o=\frac{r}{3}t\W+r.$$
\end{theorem}

\begin{proof}
Let $G_e$ be the Gergonne point of $\triangle ABC$.
Similar to Corollary~\ref{cor:intouch}, $C_o$ and the incircle are homothetic
with $G_e$ being the external center of similitude.
Let $U_a$ and $V_a$ be the points where $AG_e$ meets $\gamma_a$, with $V_a$ closer to $A$.
Let $W_a$ be the point where $U_aV_a$ intersects the incircle (Figure~\ref{fig:Va=X}).
\begin{figure}[ht]
\centering
\includegraphics[scale=0.3]{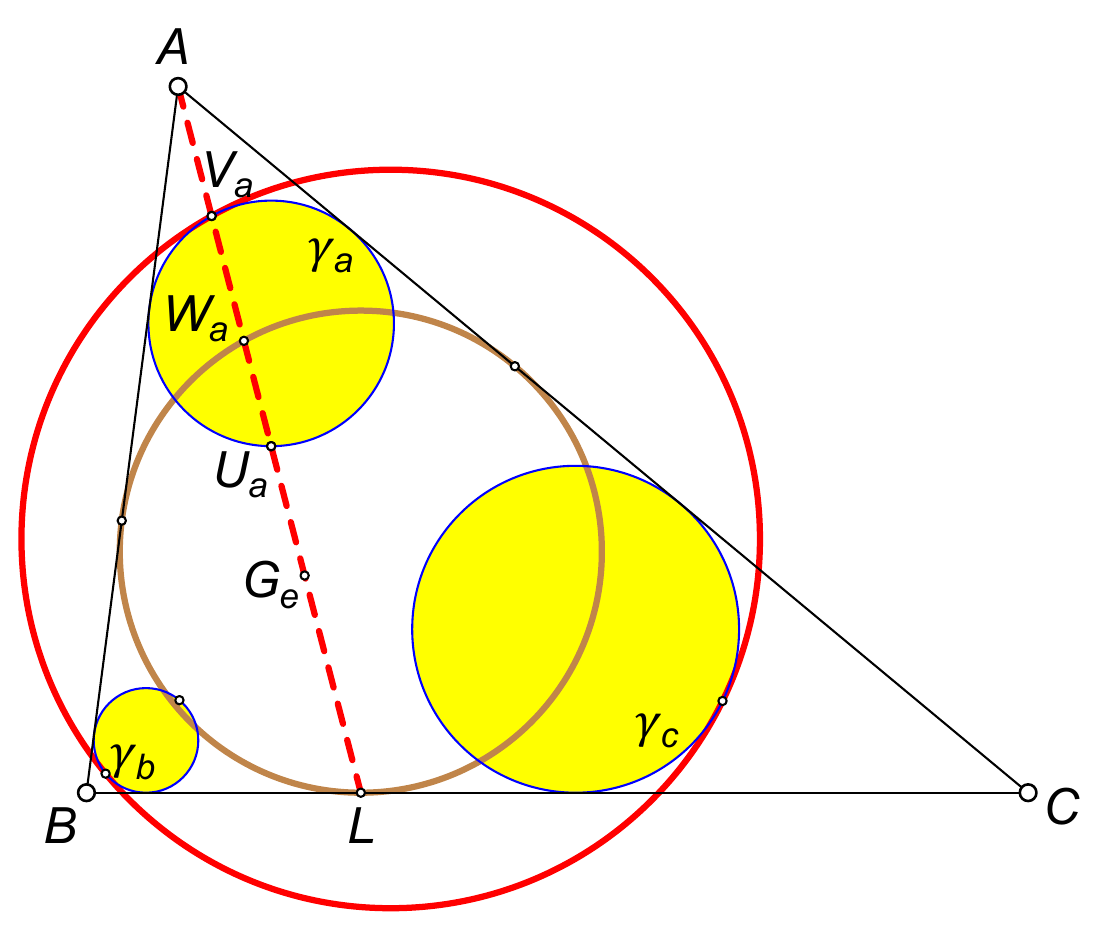}
\caption{brown and red circles are homothetic at $G_e$}
\label{fig:Va=X}
\end{figure}

\void{
Under this homothety, $V_a$ maps to $W_a$.
Thus
$$\frac{G_eV_a}{G_eW_a}=\frac{\rho_o}{r}.$$
We have
\begin{equation}
\frac{\rho_o}{r}=\frac{G_eV_a}{G_eW_a}=\frac{AG_e-AV_a}{AG_e-AW_a}.\label{eq:difs}
\end{equation}
From Theorem~\ref{thm:lengthOfAX}
$$AV_a=\frac{(p-a-rt)\sqrt{a(p-a)}}{\sqrt{ap-(b-c)^2}}.$$
The incircle and $\gamma_a$ are homothetic with homothetic center $A$.
Under this homothety, $V_a$ maps to $W_a$. Thus,
$$\frac{AW_a}{AV_a}=\frac{r}{\rho_a}.$$
Hence, $$AW_a=\frac{r\cdot AV_a}{\rho_a}.$$
From Lemma~\ref{lemma:AGe}
$$AG_e=\frac{(p-a)\sqrt{a(p-a)[ap-(b-c)^2]}}{pr\W}.$$
Substituting the values of $AG_e$, $AV_a$, and $AW_a$ into equation (\ref{eq:difs}) gives
\begin{align*}
\frac{\rho_o}{r}&=\frac{AG_e-AV_a}{AG_e-AW_a}\\
&=\frac{\frac{(p-a)\sqrt{a(p-a)[ap-(b-c)^2]}}{pr\W}-\frac{(p-a-rt)\sqrt{a(p-a)}}
{\sqrt{ap-(b-c)^2}}}{\frac{(p-a)\sqrt{a(p-a)[ap-(b-c)^2]}}{pr\W}-\frac{r}{\rho_a}\cdot\frac{(p-a-rt)\sqrt{a(p-a)}}{\sqrt{ap-(b-c)^2}}}\\
\end{align*}
}

Under this homothety, $V_a$ maps to $W_a$.
Thus
\begin{align}
\frac{\rho_o}{r}&=\frac{G_eV_a}{G_eW_a}=\frac{AG_e-AV_a}{AG_e-AW_a}\notag\\
&=\frac{AG_e/AL-AV_a/AL}{AG_e/AL-AW_a/AL}\label{eq:ratios},
\end{align}
From Theorem~\ref{thm:lengthOfAX},
$$AV_a=\frac{(p-a-rt)\sqrt{a(p-a)}}{\sqrt{ap-(b-c)^2}}.$$
From Corollary~\ref{cor:GeRatios},
$$\frac{AL}{AG_e}=1+\frac{(p-b)(p-c)}{a(p-a)}$$
so
$$\frac{AG_e}{AL}=\frac{a(p-a)}{a(p-a)+(p-b)(p-c)}.$$
From the homothety, center $A$ that maps $\gamma_a$ to the incircle,
$$\frac{AU_a}{AL}=\frac{\rho_a}{r}.$$
From Corollary~\ref{cor:AX/AL'}, we have
$$\frac{AV_a}{AU_a}=\frac{a(p-a)}{ap-(b-c)^2}$$
Multiplying the previous two equations gives
$$\frac{AV_a}{AL}=\frac{a(p-a)}{ap-(b-c)^2}\cdot\frac{\rho_a}{r}.$$
\void{
Multiplying this by the formula for $AL/AG_e$ gives
\begin{align*}
\frac{AV_a}{AG_e}&=\frac{a(p-a)}{ap-(b-c)^2}\cdot\frac{\rho_a}{r}\cdot\left(\frac{a(p-a)+(p-b)(p-c)}{a(p-a)}\right)\\
&=\frac{\rho_a}{ap-(b-c)^2}\left(\frac{a(p-a)+(p-b)(p-c)}{ar(p-a)}\right).
\end{align*}
}
From the homothety, center $A$ that maps $\gamma_a$ to the incircle,
$$\frac{AW_a}{AV_a}=\frac{r}{\rho_a}.$$
Multiplying the previous two equations gives
$$\frac{AW_a}{AL}=\frac{a(p-a)}{ap-(b-c)^2}.$$

Substituting the values for the ratios found into equation (\ref{eq:ratios}) gives
\begin{align*}
\frac{\rho_o}{r}&=\frac{AG_e/AL-AV_a/AL}{AG_e/AL-AW_a/AL},\\
&=\frac{\frac{a(p-a)}{a(p-a)+(p-b)(p-c)}-\frac{a(p-a)}{ap-(b-c)^2}\cdot\frac{\rho_a}{r}}{\frac{a(p-a)}{a(p-a)+(p-b)(p-c)}-\frac{a(p-a)}{ap-(b-c)^2}}.\\
\end{align*}
Simplifying this algebraically gives
$$\frac{\rho_o}{r}=1+\frac{2rt}{3}\cdot\frac{2ab+2bc+2ca-a^2-b^2-c^2}{(a+b-c)(b+c-a)(c+a-b)}.$$
Applying Lemma~\ref{lemma:rW} gives
\begin{align*}
\frac{\rho_o}{r}-1&=\frac{2rt}{3}\cdot\frac{4rp\W}{(a+b-c)(b+c-a)(c+a-b)}\\
&=\frac{2rt}{3}\cdot\frac{4rp\W}{8(p-c)(p-a)(p-b)}\\
&=\frac{t}{3}\cdot\frac{(rp)^2\W}{p(p-c)(p-a)(p-b)}\\
&=\frac{t}{3}\cdot\frac{\Delta^2\W}{\Delta^2}\\
&=\frac{t\W}{3},
\end{align*}
using the well-known formulas $\Delta=rp$ and $\Delta=\sqrt{p(p-a)(p-b)(p-c)}$.
Thus, $\rho_o=rt\W/3+r$.
\end{proof}

\begin{open}
Is there a simpler proof of Theorem~\ref{thm:rhoo}?
\end{open}

\void{
\begin{proof}
From Theorem~\ref{thm:1}, we have
$$\rho_a+\rho_b+\rho_c=3r-rt\left(\frac{4R+r}{p}\right)$$
where $t=\tan(\theta/4)$.
Substituting this value into Theorem~\ref{thm:A2} and using the value for $\rho_i$
found in Theorem~\ref{thm:rhoi}, gives
$$3\rho_o=2\left[3r-rt\left(\frac{4R+r}{p}\right)\right]+\frac{3r}{p}\Bigl[(4R+r)t-p\Bigr].$$
This simplifies to
$$3r_o=\frac{r}{p}\Bigl((4R+r)t+3p\Bigr)$$
and the result follows.
\end{proof}
}

\begin{corollary}\label{cor:roir}
We have
$$\frac{\rho_i+r}{\rho_o-r}=3.$$
\end{corollary}

\begin{proof}
This follows immediately from Theorems \ref{thm:rhoi} and \ref{thm:rhoo}.
\end{proof}

\void{
\begin{open}
Does this equation have any geometrical significance?
\end{open}
}

Remember when applying this result, that $\rho_i$ is to be considered negative when the inner
Apollonius circle is internally tangent to $\gamma_a$, $\gamma_b$, and $\gamma_c$, as
shown in Figure~\ref{fig:SoddyInverse}.


The line through the incenter of a triangle and the Gergonne point of that triangle
is called the \emph{Soddy line} of the triangle.

\begin{theorem}\label{thm:Soddy}
For a general triad of circles associated with $\triangle ABC$,
let $U$ be the center of the inner Apollonius circle of $\gamma_a$, $\gamma_b$, and $\gamma_c$.
Let $V$ be the center of the outer Apollonius circle of $\gamma_a$, $\gamma_b$, and $\gamma_c$.
Then $U$ and $V$ lie on the Soddy line of the triangle and $UI:IV=3:1$ (Figure~\ref{fig:Soddy}).
\end{theorem}

\begin{figure}[ht]
\centering
\includegraphics[scale=0.27]{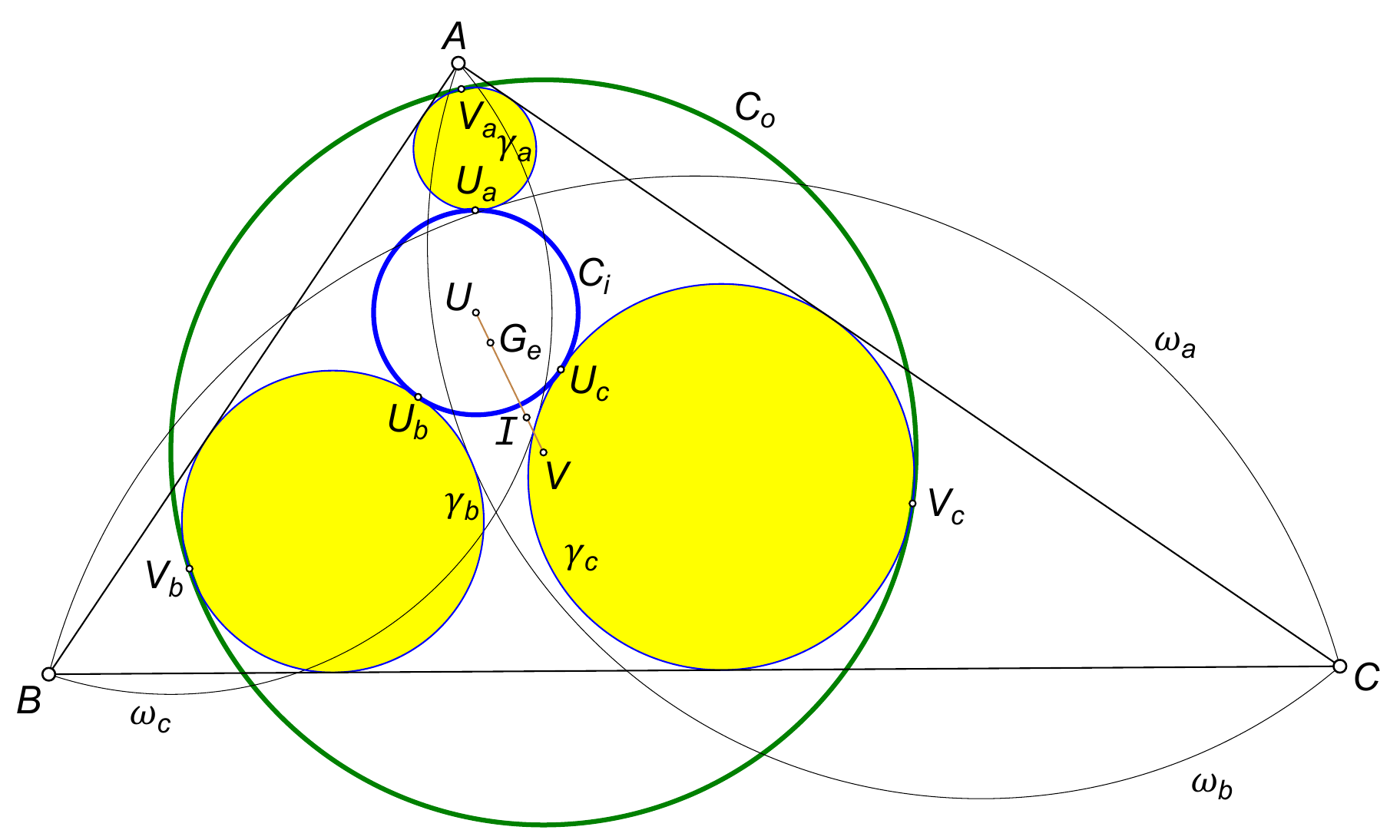}
\caption{$UI:IV=3:1$}
\label{fig:Soddy}
\end{figure}

\begin{proof}
All distances along the Soddy line will be signed.
We have
$$\frac{UI}{G_eI}=\frac{UG_e+G_eI}{G_eI}=\frac{UG_e}{G_eI}+1=\frac{\rho_i}{r}+1=\frac{\rho_i+r}{r}$$
and
$$\frac{IV}{G_eI}=\frac{G_eV-G_eI}{G_eI}=\frac{G_eV}{G_eI}-1=\frac{\rho_o}{r}-1=\frac{\rho_o-r}{r}.$$
Dividing gives
$$\frac{UI}{IV}=\frac{\rho_i+r}{\rho_o-r}.$$
The result now follows from Corollary~\ref{cor:roir}.
\end{proof}

\begin{open}
Is there a simple geometric proof that $UI/IV=3$?
\end{open}

\begin{figure}[ht]
\centering
\includegraphics[scale=0.55]{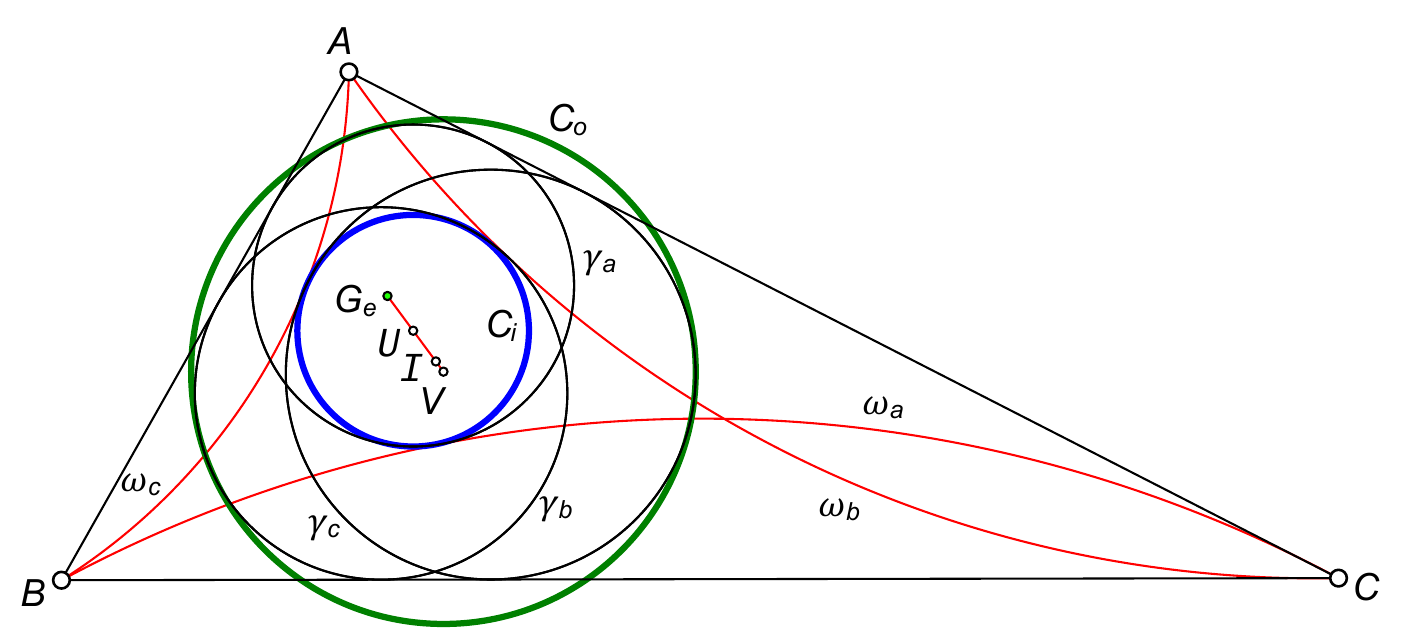}
\caption{$UI:IV=3:1$}
\label{fig:SoddyInverse}
\end{figure}


If the inner Apollonius circle is internally tangent to the three circles as in Figure~\ref{fig:SoddyInverse},
then $U$ and $V$ still lie on the Soddy line, but the points on that line
occur in the order $G_e$, $U$, $I$, $V$.

\void{
\begin{corollary}\label{cor:ugei}
We have
$$\frac{UG_e}{G_eI}=\frac{\rho_i}{r}.$$
\end{corollary}

\begin{corollary}\label{cor:vgei}
We have
$$\frac{G_eV}{G_eI}=\frac{\rho_o}{r}.$$
\end{corollary}
}

\begin{corollary}\label{cor:geiv}
We have
$$\frac{G_eI}{IV}=\frac{r}{\rho_o-r}.$$
\end{corollary}

\begin{corollary}\label{cor:ugeiv}
We have the extended proportion
$$UG_e:G_eI:IV=\rho_i:r:\rho_o-r.$$
\end{corollary}

It should be noted that the distance from $G_e$ to $I$ in terms of parts of the triangle
is known. From \cite[p.~184]{Rabinowitz2},
we have the following result.

\begin{theorem}
We have
$$G_eI=\frac{r}{4R+r}\sqrt{(4R+r)^2-3p^2}=r\sqrt{1-3/\W^2}.$$
\end{theorem}

This allows us to express any of the distances between the points $U$, $V$, $G_e$, and $I$ in terms
of $R$, $r$, and $p$ by using Corollary~\ref{cor:ugeiv}.

\begin{theorem}\label{thm:A1}
For a general triad of circles associated with $\triangle ABC$,
the inradius $r$ and the radii $\rho_i$ and $\rho_o$ of the Apollonius circles satisfy the relation
\begin{equation*}
   3\rho_o=\rho_i+4r.
\end{equation*}
\end{theorem}

\begin{proof}
This is algebraically equivalent to Corollary~\ref{cor:roir}.
\end{proof}

\begin{theorem}\label{thm:A2}
For a general triad of circles associated with $\triangle ABC$,
the radii $\rho_i$ and $\rho_o$ of the Apollonius circles and the radii $\rho_a$, $\rho_b$, and $\rho_c$
of the three circles in the triad satisfy the relation
\begin{equation*}
   3\rho_o=2(\rho_a+\rho_b+\rho_c)+3\rho_i.
\end{equation*}
\end{theorem}

\begin{proof}
From Theorem~\ref{thm:1}, we have
$$\rho_a+\rho_b+\rho_c=3r-rt\W.$$
From Theorems \ref{thm:rhoi} and \ref{thm:rhoo}, we have
$$3\rho_o-3\rho_i=6r-2rt\W=2(\rho_a+\rho_b+\rho_c)$$
as desired.
\end{proof}

\begin{theorem}
Let $\rho_i$ be the radius of the inner Apollonius circle of $\gamma_a$, $\gamma_b$, and $\gamma_c$.
Then
$$\rho_i^2=\rho_a^2+\rho_b^2+\rho_c^2+2r^2\left(t^2-1\right).$$
\end{theorem}

\begin{proof}
This follows algebraically by combining Theorems~\ref{thm:rhoi} and \ref{thm:3}.
\end{proof}

When $\theta=180\degrees$, the arcs become semicircles, $t=1$, and this result agrees with Theorem~6.1 in \cite{RabSup}.

\section{Relationship with Semicircles}

Many of the elements of our configuration are proportional to the corresponding elements
when $\omega_a$, $\omega_b$, and $\omega_c$ are semicircles (i.e. when $\theta=180\degrees$ or $t=1$).

If $x$ is any measurement or object, let $x^*$ denote the same measurement or object when $\theta=180\degrees$, i.e. when the
arcs are semicircles.

In \cite{RabSup}, it was found that
\begin{align*}
\rho_a^*&=r\left(1-\tan\frac{A}{2}\right)\\
\rho_b^*&=r\left(1-\tan\frac{B}{2}\right)\\
\rho_c^*&=r\left(1-\tan\frac{C}{2}\right)
\end{align*}
and in Theorem~\ref{thm:rho1}, we found that
\begin{align*}
\rho_a&=r\left(1-\tan\frac{A}{2}\tan\frac{\theta}{4}\right)\\
\rho_b&=r\left(1-\tan\frac{B}{2}\tan\frac{\theta}{4}\right)\\
\rho_c&=r\left(1-\tan\frac{C}{2}\tan\frac{\theta}{4}\right)
\end{align*}



In other words, if $t=\tan\frac{\theta}{4}$, then we have the following results.

\begin{theorem}\label{thm:rhoMinusR}
The following identities are true.
\begin{align*}
\rho_a-r&=t(\rho_a^*-r)\\
\rho_b-r&=t(\rho_b^*-r)\\
\rho_c-r&=t(\rho_c^*-r)
\end{align*}
\end{theorem}

\begin{theorem}\label{thm:rhoMinusRho}
The following identities are true.
\begin{align*}
\rho_a-\rho_b&=t(\rho_a^*-\rho_b^*)\\
\rho_b-\rho_c&=t(\rho_b^*-\rho_c^*)\\
\rho_c-\rho_a&=t(\rho_c^*-\rho_a^*)
\end{align*}
\end{theorem}

Let $T_{bc}$ denote the length of the common external tangent between circles $\gamma_b$ and $\gamma_c$.
Define $T_{ab}$ and $T_{ca}$ similarly.

In \cite{RabSup}, it was found that $T_{ab}^*=T_{bc}^*=T_{ca}^*=2r$.
In Theorem~\ref{thm:genTangents}, we found that $T_{ab}=T_{bc}=T_{ca}=2rt$.
This gives us the following result.

\begin{theorem}\label{thm:T}
The following identities are true.\begin{align*}
T_{ab}&=tT_{ab}^*\\
T_{bc}&=tT_{bc}^*\\
T_{ca}&=tT_{ca}^*
\end{align*}
\end{theorem}

Using these, we can prove the following new result.

\begin{theorem}\label{thm:D}
Let $D_{bc}$ denote the distance between the centers of $\gamma_b$ and $\gamma_c$.
Define $D_{ab}$ and $D_{ca}$ similarly.
Then the following identities are true.
\begin{align*}
D_{ab}&=tD_{ab}^*\\
D_{bc}&=tD_{bc}^*\\
D_{ca}&=tD_{ca}^*
\end{align*}
\end{theorem}

\begin{proof}
By symmetry, it suffices to prove the result for $D_{bc}$.
Let $E$ be the center of $\gamma_b$ and let $F$ be the center of $\gamma_c$.
Let the common external tangent along $BC$ be $XY$ as shown in Figure~\ref{fig:centers}.
Let the foot of the perpendicular from $E$ to $FY$ be $H$.
Then in right triangle $EHF$, we have $EH=XY=T_{bc}=tT_{bc}^*$ by Theorem~\ref{thm:T}.
We also have $FH=|\rho_c-\rho_b|$.
By Theorem~\ref{thm:rhoMinusRho}, $FH=t|\rho_c^*-\rho_a^*|$.
Since $\triangle EHF\sim\triangle E^*H^*F^*$, we must therefore have $D_{bc}=tD_{bc}^*$.
\end{proof}

\begin{figure}[ht]
\centering
\includegraphics[scale=0.35]{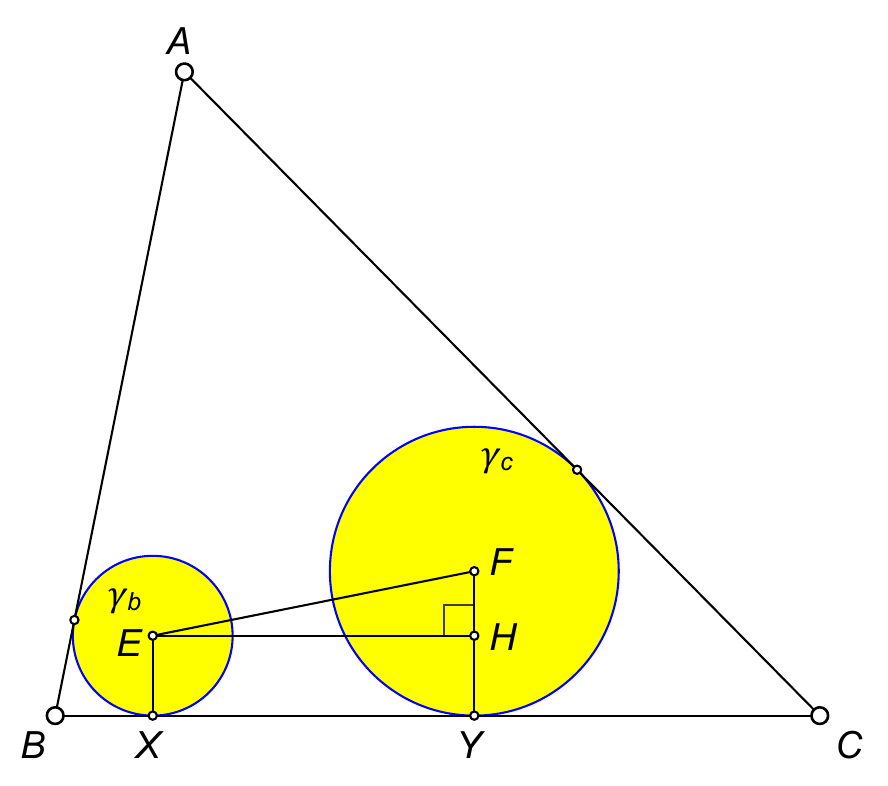}
\caption{case where $H$ lies between $F$ and $Y$}
\label{fig:centers}
\end{figure}

\begin{corollary}
The triangles formed by the centers of $\gamma_a$, $\gamma_b$, $\gamma_c$
and $\gamma_a^*$, $\gamma_b^*$, $\gamma_c^*$ are similar.
\end{corollary}

\begin{theorem}\label{thm:ri+r}
We have
$$\rho_i+r=t(\rho_i^*+r).$$
\end{theorem}

\begin{proof}
This follows immediately from Theorem~\ref{thm:rhoi}.
\end{proof}

\begin{theorem}\label{thm:ro-r}
We have
$$\rho_o-r=t(\rho_o^*-r).$$
\end{theorem}

\begin{proof}
This follows immediately from Theorem~\ref{thm:rhoo}.
\end{proof}

\begin{theorem}\label{thm:d}
Let $\gamma_a$, $\gamma_b$, and $\gamma_c$ be a general triad of circles associated
with $\triangle ABC$.
Let $U$ be the center of the inner Apollonius circle of the circles in the triad.
Let $d_{a}$ denote the distance between $U$ and the center of $\gamma_a$.
Define $d_b$ and $d_c$ similarly.
Then the following identities are true.
\begin{align*}
d_{a}-d_b&=t(d_a^*-d_b^*)\\
d_{b}-d_c&=t(d_b^*-d_b^*)\\
d_{c}-d_a&=t(d_c^*-d_a^*)
\end{align*}
\end{theorem}

\begin{proof}
Let $\rho_i$ be the radius of the inner Apollonius circle.
Then $d_a=\rho_i+\rho_a$ (Figure~\ref{fig:inner}).

\begin{figure}[ht]
\centering
\includegraphics[scale=0.5]{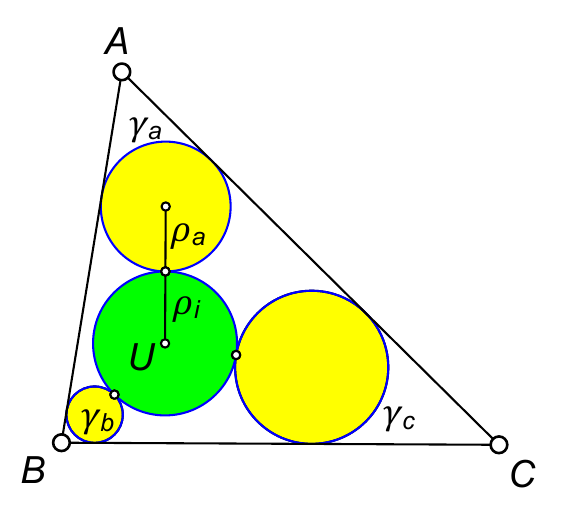}
\caption{$d_a=\rho_i+\rho_a$}
\label{fig:inner}
\end{figure}

Similarly, $d_b=\rho_i+\rho_b$.
Thus
$d_a-d_b=\rho_a-\rho_b$.
By Theorem~\ref{thm:rhoMinusRho},
$$d_a-d_b=t(\rho_a^*-\rho_b^*)=t(d_a^*-d_b^*).$$
The same argument works for $d_b-d_c$ and $d_c-d_a$.
\end{proof}



\begin{theorem}\label{lemma1}
If $u=EF$, $v=DF$, $w=DE$ are the distances between the centers of the circles $\gamma_a$, $\gamma_b$, and $\gamma_c$,
we have
\begin{align*}
u^2&=\frac{a(p-a)[ap-(b-c)^2]t^2}{p^2},\\
v^2&=\frac{b(p-b)[bp-(c-a)^2]t^2}{p^2},\\
w^2&=\frac{c(p-c)[ap-(a-b)^2]t^2}{p^2}.	
\end{align*}
\end{theorem}

\begin{proof}
This follows from Theorem~\ref{thm:D} and the simplified values of $u^*$, $v^*$, and $w^*$ found in \cite{RabSup}.
\end{proof}

\newpage

\section{Variations}

We have studied the case where Ajima circle $\gamma_a$ is inscribed in $\angle BAC$ and
is inside $\triangle ABC$ and is externally tangent to circle $\omega_a$.

There are actually four circles that are inscribed in $\angle BAC$ and
are tangent to circle $\omega_a$. These circles are shown in Figure~\ref{fig:4circles}.

\begin{figure}[ht]
\centering
\includegraphics[scale=0.5]{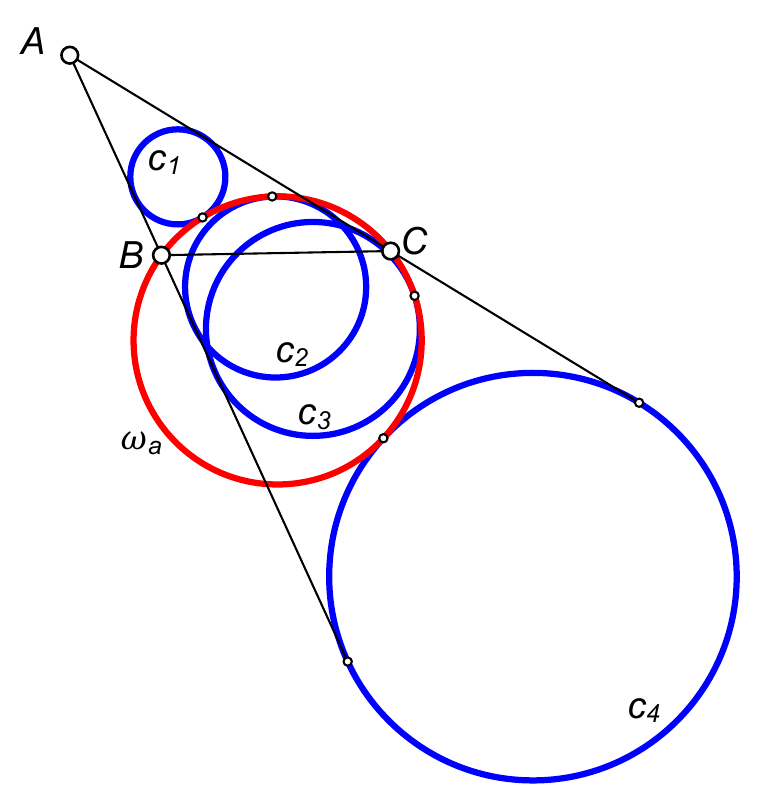}
\caption{circles inscribed in $\angle A$ and tangent to red circle}
\label{fig:4circles}
\end{figure}

Circle $c_1$ is variation 1 and is the variation already studied.
Note that circle $c_4$ is inscribed in $\angle BAC$ and
is \emph{outside} $\triangle ABC$ as well as being externally tangent to circle $\omega_a$.
Circle $c_2$ and $c_3$ are \emph{internally} tangent to $\omega_a$.
The touch point of $c_2$ and $\omega_a$ is inside $\triangle ABC$ while
the touch point of $c_3$ and $\omega_a$ is outside $\triangle ABC$.

Many of the results we found for variation 1 work for the other variations as well.
We present below a few of these results. Proofs are omitted because they
are similar to the proofs given for variation 1.
Variants 2, 3, and 4, are shown in Figure~\ref{fig:3variants}.

\begin{figure}[ht]
\centering
\includegraphics[scale=0.5]{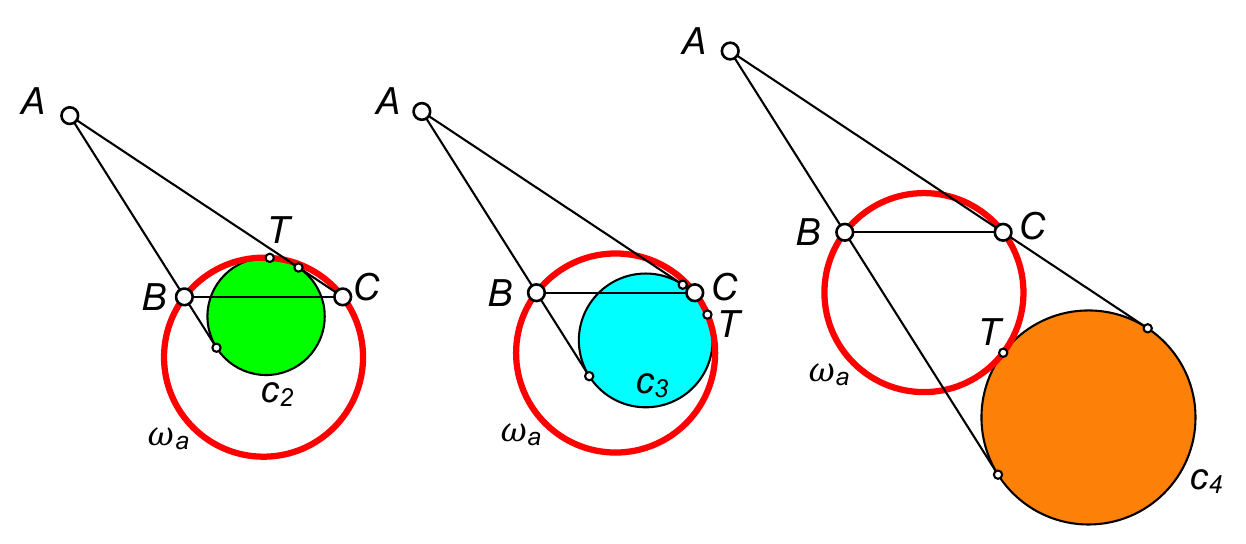}
\caption{variants 2 (green), 3 (blue), and 4 (orange)}
\label{fig:3variants}
\end{figure}

\newpage

\textbf{\large The Catalytic Lemma}

The Catalytic Lemma remains true except in some cases where the incenter is
replaced by an excenter.
Figure~\ref{fig:3variants-catalytic} shows variants 2, 3, and 4.
In each case, $B$, $K$, $T$, and an incenter/excenter lie on a circle.
\void{
For example, Figure~\ref{fig:outerCatalytic} shows
the case where $\gamma_a=c_4$ is external to the triangle and externally tangent
to $\omega_a$. In this case, $C$, $K$, $T$, and $I_a$ are concyclic where
$I_a$ is the excenter of $\triangle ABC$ opposite vertex $A$.
}

\begin{figure}[ht]
\centering
\includegraphics[scale=0.5]{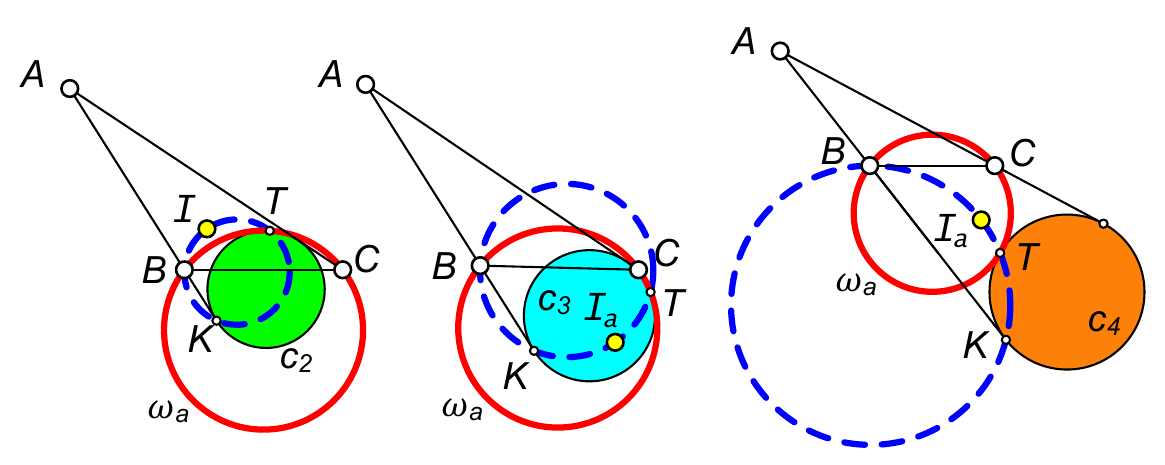}
\caption{four points lie on a circle}
\label{fig:3variants-catalytic}
\end{figure}

\textbf{\large Protasov's Theorem}

Protasov's Theorem remains true except in some cases where the incenter is
replaced by an excenter.
Figure~\ref{fig:3variants-Protasov} shows variants 2, 3, and 4.
In each case, the blue line bisects the angle formed by the dashed lines.

\begin{figure}[ht]
\centering
\includegraphics[scale=0.6]{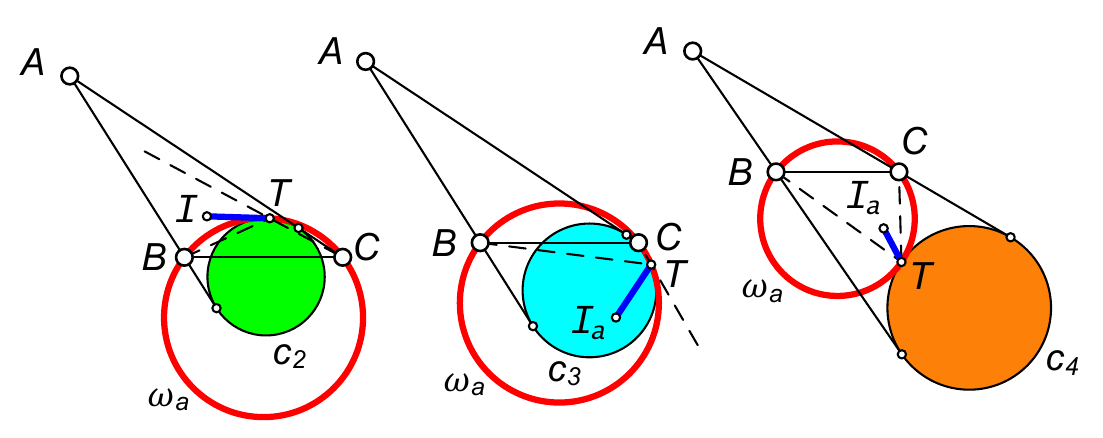}
\caption{blue line bisects angle formed by dashed lines}
\label{fig:3variants-Protasov}
\end{figure}

\void{
For example, in Figure~\ref{fig:outerProtassov},
$I_a$ is the excenter of $\triangle ABC$ opposite $A$. In this case, $TI_a$ bisects $\angle BTC$.
\begin{figure}[ht]
\centering
\includegraphics[scale=0.55]{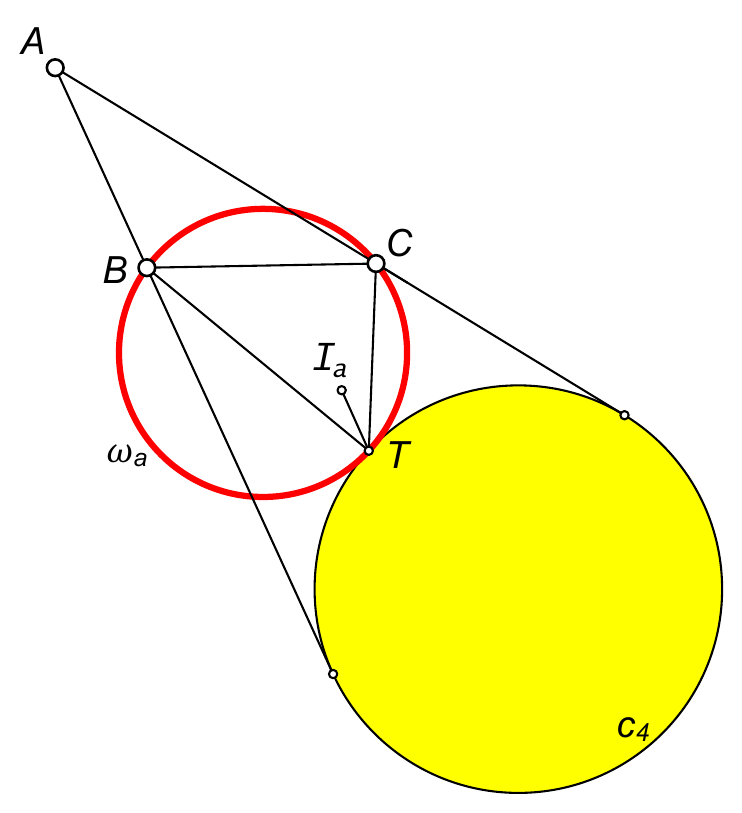}  
\caption{$TI_a$ bisects $\angle BTC$}
\label{fig:outerProtassov}
\end{figure}
}

\textbf{\large Theorem~\ref{thm:RG13066}}

Theorem~\ref{thm:RG13066} remains true except in some cases where the incenter is
replaced by an excenter.
Figure~\ref{fig:3variants-parallel} shows variants 2, 3, and 4.
In each case, the parallel to $BE$ through an incenter or excenter meets $AC$ at $F$,
where $E$ is the point where $\omega_a$ meets $AC$. Then $IF=FK$ where $K$ is the point
where $\gamma_a$ touches $AC$.

\begin{figure}[ht]
\centering
\includegraphics[scale=0.45]{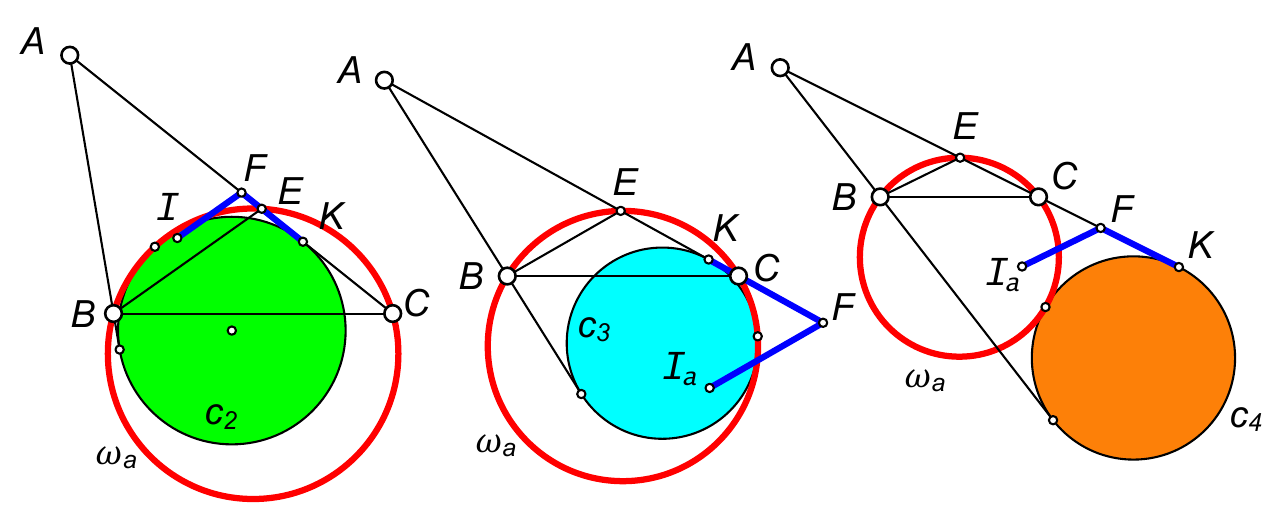}
\caption{blue lines are congruent}
\label{fig:3variants-parallel}
\end{figure}

\textbf{\large Theorem~\ref{thm:RG13069}}

Theorem~\ref{thm:RG13069} remains true except in some cases where the incenter is
replaced by an excenter.
Figure~\ref{fig:3variants-parallel} shows variants 2, 3, and 4.
In each case, the line through $T$ and an incenter or excenter meets $\omega_a$ at a point
on the perpendicular bisector of $BC$ (opposite $T$).

\begin{figure}[ht]
\centering
\includegraphics[scale=0.6]{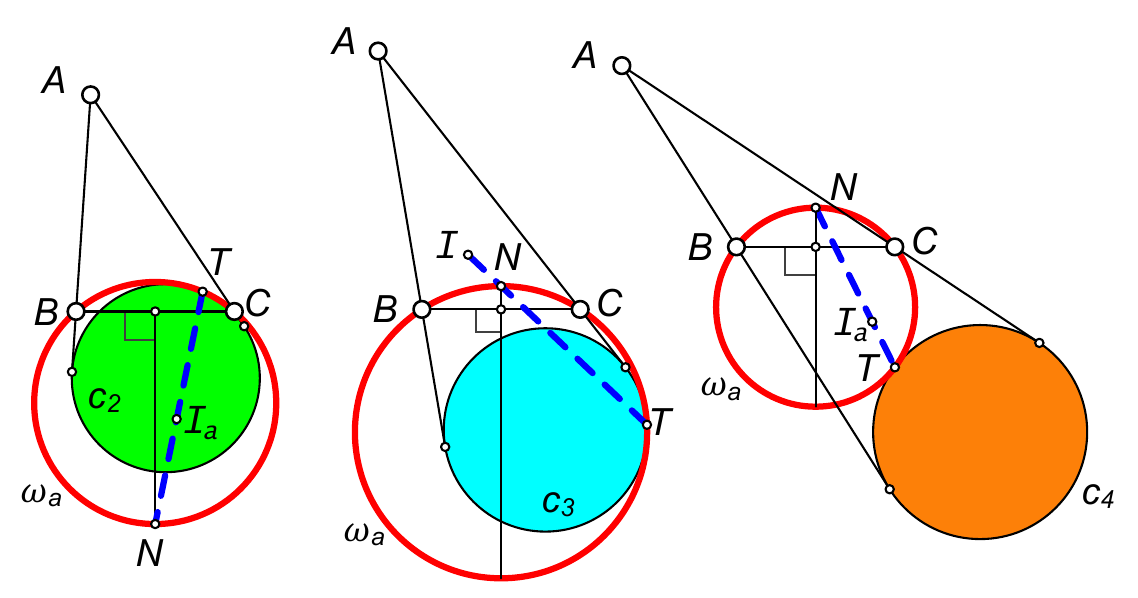}
\caption{$T$, $N$, and an incenter or excenter lie on a line.}
\label{fig:3variants-perpendicular}
\end{figure}

\bigskip
\textbf{\large Ajima's Theorem}

Similar formulas for the radii of the variant circles can be found
similar to Ajima's Theorem. These are shown in Figure~\ref{fig:3variants-Ajima},
where $r_a$ denotes the radius of the $A$-excircle of $\triangle ABC$.

\begin{figure}[ht]
\centering
\includegraphics[scale=0.55]{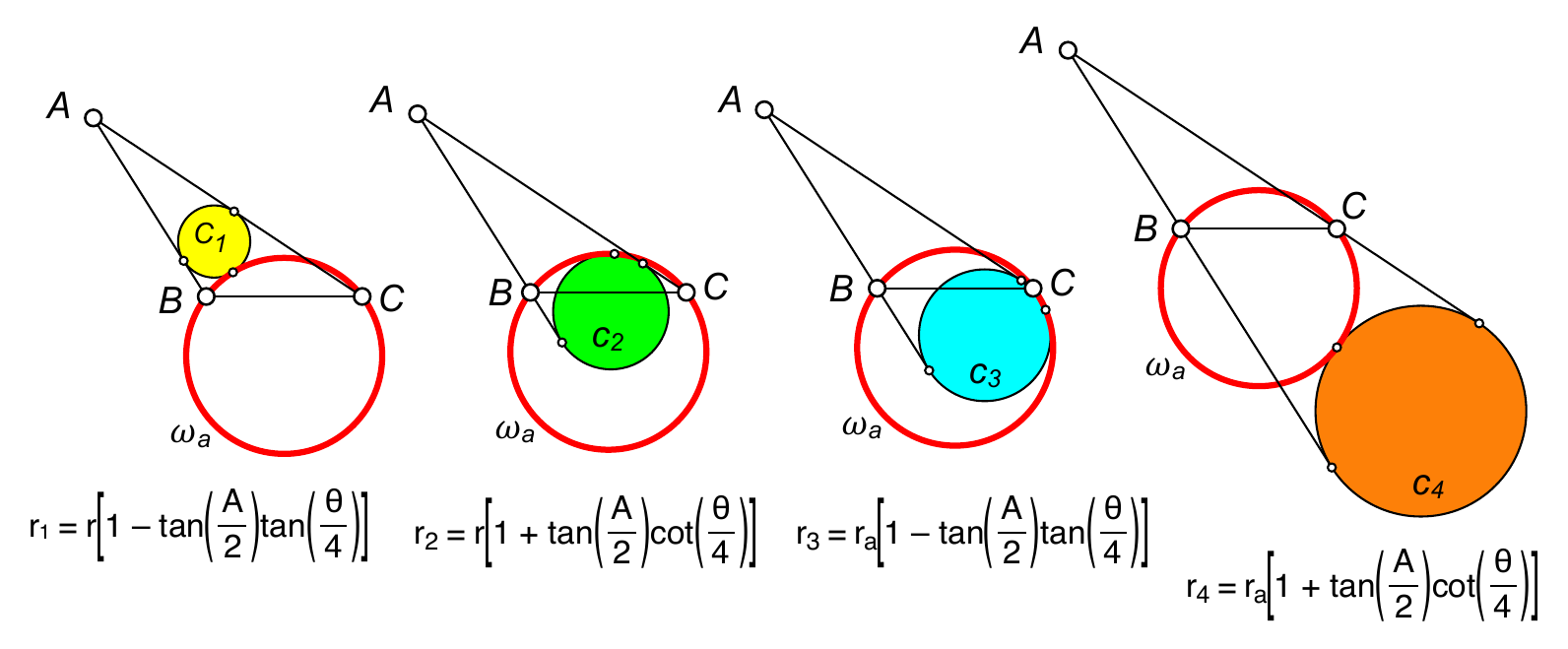}
\caption{}
\label{fig:3variants-Ajima}
\end{figure}

\bigskip

\newpage
\textbf{\large The Paasche Analog}

The Paasche Analog (Theorem~\ref{thm:Paasche}) remains true.
If $\gamma_a$ is any one of these variant circles,
and if $T_a$ is the touch point of $\gamma_a$ with $\omega_a$,
with $T_b$ and $T_c$ defined similarly, then $AT_a$, $BT_b$, and $CT_c$ are concurrent.
See Figure~\ref{fig:outerPaasche} for one case.

\begin{figure}[ht]
\centering
\includegraphics[scale=0.5]{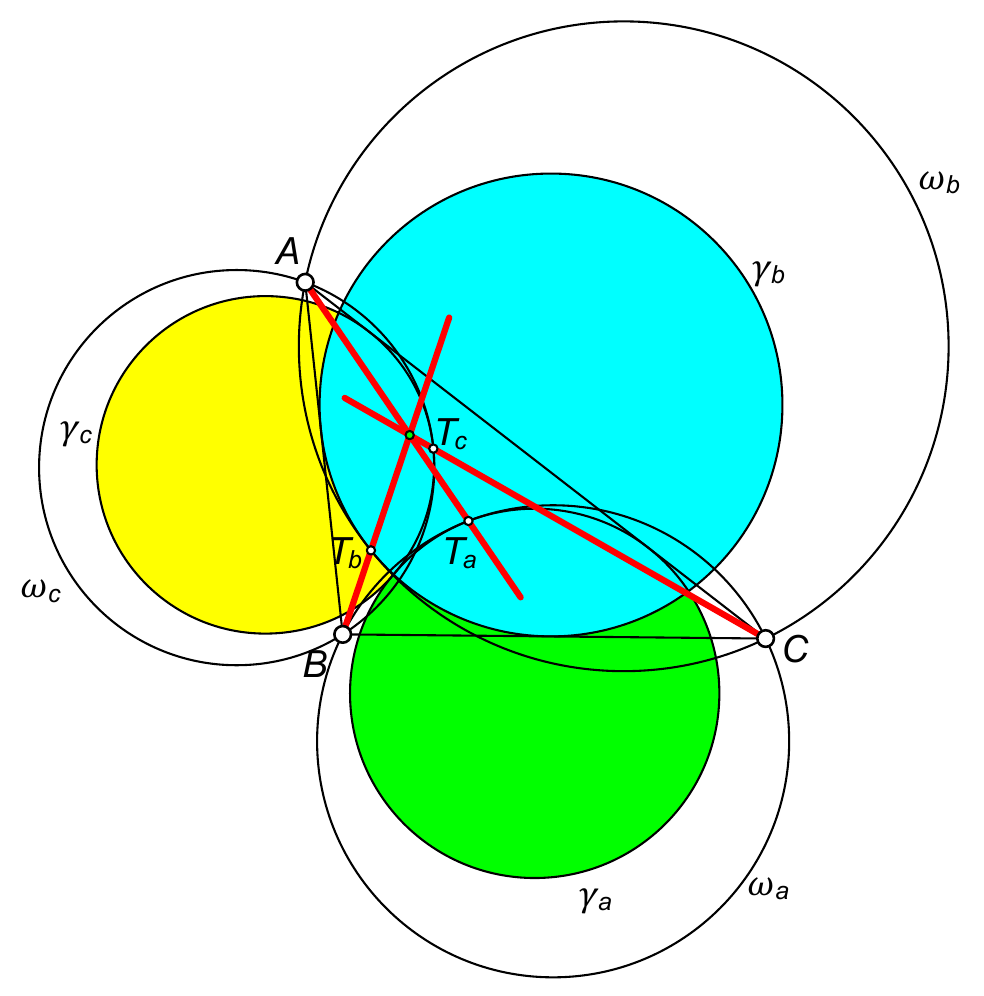}
\caption{}
\label{fig:outerPaasche}
\end{figure}

\void{
\begin{figure}[ht]
\centering
\includegraphics[scale=0.55]{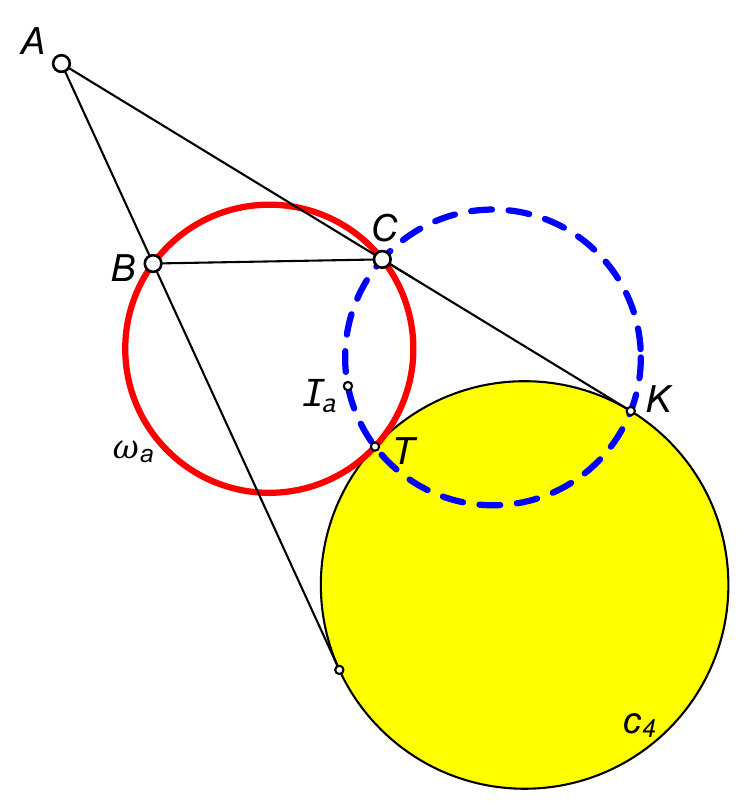}  
\caption{$C, K, T, I_a$ are concyclic}
\label{fig:outerCatalytic}
\end{figure}
}


\goodbreak

\end{document}